%% file: qle_draft.tex
\date{} 
\title{Quantum Loewner Evolution}
\author{Jason Miller and Scott Sheffield}
\documentclass[12pt,naturalnames]{article}
\usepackage{amsmath}
\usepackage{amssymb}
\usepackage{amsthm}
\usepackage{amsfonts}
\usepackage{graphicx}
\usepackage{tabularx}
\usepackage[labelformat=empty]{subfig}
\usepackage{enumerate}
\usepackage{color}
\usepackage{xparse}
\usepackage{comment}
\usepackage{morefloats}
\usepackage[normalem]{ulem}
\usepackage[margin=1.37in]{geometry}
\input epsf.tex

\def\@rst #1 #2other{#1}
\newcommand\MR[1]{\relax\ifhmode\unskip\spacefactor3000 \space\fi
  \MRhref{\expandafter\@rst #1 other}{#1}}
\newcommand{\MRhref}[2]{\href{http://www.ams.org/mathscinet-getitem?mr=#1}{MR#2}}

\allowdisplaybreaks

\newif\ifhyper\IfFileExists{hyperref.sty}{\hypertrue}{\hyperfalse}

\ifhyper\usepackage{hyperref}\fi

\newif\ifdraft
\drafttrue
\numberwithin{equation}{section}
\numberwithin{figure}{section}

\newtheorem{theorem}{Theorem}
\numberwithin{theorem}{section}
\newtheorem{corollary}[theorem]{Corollary}
\newtheorem{lemma}[theorem]{Lemma}
\newtheorem{proposition}[theorem]{Proposition}

\newtheorem{question}[theorem]{Question}

\theoremstyle{remark}\newtheorem{definition}[theorem]{Definition}
\theoremstyle{remark}\newtheorem{remark}[theorem]{Remark}

\newcommand{\R}{\mathbf{R}}
\newcommand{\C}{\mathbf{C}}
\newcommand{\D}{\mathbf{D}}
\newcommand{\Z}{\mathbf{Z}}
\newcommand{\N}{\mathbf{N}}
\newcommand{\Q}{\mathbf{Q}}

\newcommand{\CA}{\mathcal {A}}

\newcommand{\CC}{\mathcal {C}}
\newcommand{\CD}{\mathcal {D}}

\newcommand{\CF}{\mathcal {F}}

\newcommand{\CL}{\mathcal {L}}

\newcommand{\CN}{\mathcal {N}}

\newcommand{\CP}{\mathcal {P}}

\newcommand{\CG}{\mathcal {G}}
\newcommand{\CH}{\mathcal {H}}

\newcommand{\dirichlet}{{\mathrm D}}
\newcommand{\free}{{\mathrm F}}

\def\dist{\mathop{\mathrm{dist}}}

\newcommand{\CZ}{{\mathcal Z}}
\newcommand{\Fh}{{\mathfrak h}}
\newcommand{\var}{{\mathrm {var}}}
\newcommand{\re}{{\mathrm Re}}

\newcommand{\SLE}{{\rm SLE}}
\newcommand{\CLE}{{\rm CLE}}
\newcommand{\QLE}{{\rm QLE}}

\newcommand{\ol}{\overline}

\newcommand{\wh}{\widehat}
\newcommand{\wt}{\widetilde}

\newcommand{\one}{{\mathbf 1}}

\newcommand{\im}{{\rm Im}}

\newcommand{\cov}{\mathop{\mathrm {cov}}}
\newcommand{\confrad}{{\mathrm CR}}

\definecolor{purple}{rgb}{0.7,0,0.7}
\definecolor{gray}{rgb}{0.6,0.6,0.6}
\definecolor{dgreen}{rgb}{0.0,0.4,0.0}
\definecolor{dblue}{rgb}{0.0,0.0,0.5}

\newcommand{\pHarm}{P_{\rm harm}}
\newcommand{\pSupp}{P_{\rm supp}}

\def\Ito/{It\^o}

\def \p {{\mathbf P}}
\def \E {{\bf E}}

\begin{document}
\maketitle

\begin{abstract}
\input{tex/abstract.tex}
\end{abstract}

\newpage
{\setlength{\parskip}{0.0cm plus1mm minus1mm}
\tableofcontents}
\newpage

\medbreak {\noindent\bf Acknowledgments.}  We thank Hugo Duminil-Copin,  Ivan Corwin, Bertrand Duplantier, Steffen Rohde, Oded Schramm, Stanislav Smirnov, Wendelin Werner, David Wilson, and Hao Wu for helpful discussions.  We also thank Ewain Gwynne and Xin Sun for helpful comments on earlier versions of this article.
\vspace{-.2in}

\section{Introduction}
\label{sec::introduction}

\input{tex/introduction.tex}

\section{Discrete constructions}
\label{sec::discrete}

\input{tex/results.tex}

\section{Scaling exponents for continuum $\QLE$}
\label{sec::continuum}

\input{tex/continuum_constructions.tex}

\section{Preliminaries}
\label{sec::preliminaries}

\input{tex/preliminaries.tex}

\section{The reverse radial $\SLE$/GFF coupling}
\label{sec::couplings}

\input{tex/reverse_coupling.tex}

\section{Existence of QLE}
\label{sec::existence}

\input{tex/existence.tex}

\section{Limiting dynamics}
\label{sec::dynamics}

\input{tex/dynamics.tex}

\section{Sample path properties}
\label{sec::sample_path}

\input{tex/sample_path_properties.tex}

\section{Open questions}
\label{sec::questions}

\input{tex/open_questions.tex}

\appendix

\bibliographystyle{hmralphaabbrv}
\bibliography{qle}

\bigskip

\filbreak
\begingroup
\small
\parindent=0pt

\bigskip
\vtop{
\hsize=5.3in
Department of Mathematics\\
Massachusetts Institute of Technology\\
Cambridge, MA, USA } \endgroup \filbreak

\end{document}

%% file: tex/abstract.tex
{\small
What is the scaling limit of diffusion limited aggregation (DLA) in the plane?  This is an old and famously difficult question.  One can generalize the question in two ways: first, one may consider the {\em dielectric breakdown model} $\eta$-DBM, a generalization of DLA in which particle locations are sampled from the $\eta$-th power of harmonic measure, instead of harmonic measure itself.  Second, instead of restricting attention to deterministic lattices, one may consider $\eta$-DBM on random graphs known or believed to converge in law to a Liouville quantum gravity (LQG) surface with parameter $\gamma \in [0,2]$.

In this generality, we propose a scaling limit candidate called {\bf quantum Loewner evolution}, $\QLE(\gamma^2, \eta)$.  $\QLE$ is defined in terms of the radial Loewner equation like radial $\SLE$, except that it is driven by a measure valued diffusion $\nu_t$ derived from LQG rather than a multiple of a standard Brownian motion.  We formalize the dynamics of $\nu_t$ using an SPDE.  For each $\gamma \in (0,2]$, there are two or three special values of $\eta$ for which we establish the existence of a solution to these dynamics
and explicitly describe the stationary law of $\nu_t$.

We also explain discrete versions of our construction that relate DLA to loop-erased random walk and the Eden model to percolation.  A certain ``reshuffling'' trick (in which concentric annular regions are rotated randomly, like slot machine reels) facilitates explicit calculation.

We propose $\QLE(2,1)$ as a scaling limit for DLA on a random spanning-tree-decorated planar map, and $\QLE(8/3, 0)$ as a scaling limit for the Eden model on a random triangulation. We propose using $\QLE(8/3,0)$ to endow pure LQG with a distance function, by interpreting the region explored by a branching variant of $\QLE(8/3, 0)$, up to a fixed time, as a metric ball in a random metric space.}

%% file: tex/introduction.tex
\subsection{Overview}

The mathematical physics literature contains several simple ``growth models'' that can be understood as random increasing sequences of clusters on a fixed underlying graph $G$, which is often taken to be a lattice such as $\Z^2$.  These models are used to describe crystal formation, electrodeposition, lichen growth, lightning formation, coral reef formation, mineral deposition, cancer growth, forest fire progression, Hele-Shaw flow, water seepage, snowflake formation, oil dissipation, and many other natural processes.  Among the most famous and widely studied of these models are the Eden model (1961), first passage percolation (1965), diffusion limited aggregation (1981), the dielectric breakdown model (1984), and internal diffusion limited aggregation (1986) \cite{eden1961two, hammersley1965first, witten1981diffusion, witten1983diffusion,niemeyer1984fractal, meakin1986formation}, each of which was originally introduced with a different physical motivation in mind.

This paper mainly treats the dielectric breakdown model (DBM), which is a family of growth processes, indexed by a parameter $\eta$, in which new edges are added to a growing cluster according to the $\eta$-th power of harmonic measure, as we explain in more detail in Section~\ref{subsec::modelbackground}\footnote{In \cite{niemeyer1984fractal} growth is based on harmonic measure viewed from a specified boundary {\em set} within a regular lattice like $\Z^2$.  For convenience, one may identify the points in the boundary set and treat them as a single vertex $v$.  A cluster grows from a fixed interior vertex, and at each growth step, one considers the function $\phi$ that is equal to $1$ at $v$ and $0$ on the vertices of the growing cluster --- and is discrete harmonic elsewhere.  The harmonic measure (viewed from $v$) of an edge $e=(v_1,v_2)$, with $v_1$ in the cluster and $v_2$ not in the cluster, is defined to be proportional to $\phi(v_2)-\phi(v_1) = \phi(v_2)$.  We claim this is in turn proportional to the probability that a random walk started at $v$ first reaches the cluster via $e$ (which is the definition of harmonic measure we use for general graphs in this paper).  We sketch the proof of this standard observation here in this footnote.  On $\Z^2$, $\phi(v_2)$ is the probability that a random walk from $v_2$ reaches $v$ before the cluster boundary, i.e., $\phi(v_2) = \sum_P 4^{-|P|}$ where $P$ ranges over paths from $v_2$ to $v$ that do not hit the cluster or $v$ (until the end), and $|P|$ denotes path length.  Also, for each $P$, the probability that a walk from $v$ traces $P$ in the reverse direction and then immediately follows $e$ to hit the cluster is given by $4^{-|P|}/\mathrm{deg}(v)$.  Summing over $P$ proves that $\phi(v_2)$ is proportional to the probability that a walk starting from $v$ exits at $e$ without hitting $v$ a second time; this is in turn proportional to the overall probability that a walk from $v$ exits at $e$, which proves the claim. {\bf Variants:} One common variant is to consider the first time a walk from $v$ hits a cluster-adjacent vertex (instead of the first time it crosses a cluster-adjacent edge); this induces a harmonic measure on cluster-adjacent vertices and one may add new vertices via the $\eta$-th power of this measure.  The difference is analogous to the difference between site percolation and bond percolation.  Also, it is often natural to consider harmonic measure viewed from $\infty$ instead of from a fixed vertex $v$.}.
DBM includes some of the other models mentioned above as special cases: when $\eta = 0$, DBM is equivalent to the Eden model, and when $\eta = 1$, DBM is equivalent to diffusion limited aggregation (DLA), as noted in \cite{niemeyer1984fractal}.  Moreover, first passage percolation (FPP) is a growing family of metric balls in a metric space obtained by assigning i.i.d.\ positive weights to the edges of $G$ --- and when the law of the weights is exponential, FPP is equivalent (up to a time change) to the Eden model (see Section~\ref{subsubsection::eden}).

We would like to consider what happens when $G$ is taken to be a {\em random} graph embedded in the plane.  Specifically, instead of using $\Z^2$ or another deterministic lattice (which in some sense approximates the Euclidean structure of space) we will define the DBM on random graphs that in some sense approximate the random measures that arise in Liouville quantum gravity.

Liouville quantum gravity (LQG) was proposed in the physics literature by Polyakov in 1981, in the context of string theory, as a canonical model of a random two-dimensional Riemannian manifold \cite{MR623209, MR623210}, although it is too rough to be defined as a manifold in the usual sense.  By Riemann uniformization, any two-dimensional simply connected Riemannian manifold $\mathcal M$ can be conformally mapped to a planar domain~$D$.  If $\mu$ is the pullback to~$D$ of the area measure on $\mathcal M$, then the pair consisting of $D$ and $\mu$ completely characterizes the manifold $\mathcal M$.  One way to define an LQG surface is as the pair $D$ and $\mu$ with $\mu = e^{\gamma h(z)}dz$, where $dz$ is Lebesgue measure on $D$, $h$ is an instance of some form of the Gaussian free field (GFF) on $D$, and $\gamma \in [0,2)$ is a fixed parameter.  Since $h$ is a distribution, not a function, a regularization procedure is needed to make this precise \cite{ds2011kpz}.  It turns out that one can define the mean value of $h$ on a circle of radius $\epsilon$, call this $h_\epsilon(z)$, and then write $\mu = \lim_{\epsilon \to 0} \epsilon^{\gamma^2/2} e^{\gamma h_\epsilon(z)} dz$ \cite{ds2011kpz} (and a slightly different construction works when $\gamma=2$ \cite{duplantier2012critical,duplantier2012renormalization}).

Figure~\ref{fig::squaredecomposition} illustrates one way to tile $D$ with squares all of which have size of order $\delta$ (for some fixed $\delta>0$) in the random measure $\mu$.  Given such a tiling, one can consider a growth model on the graph whose vertices are the squares of this grid.  Another more isotropic approach to obtaining a graph from $\mu$ is to sample a Poisson point process with intensity given by some large multiple of $\mu$, and then consider the Voronoi tesselation corresponding to that point process.  A third approach, which we explain in more detail below, is to consider one of the {\em random planar maps} believed to converge to LQG in the scaling limit.

\begin{figure}
\vspace{-1.5cm}
\begin{center}
\subfloat[$\gamma=1/2$]{\includegraphics[width=0.49\textwidth]{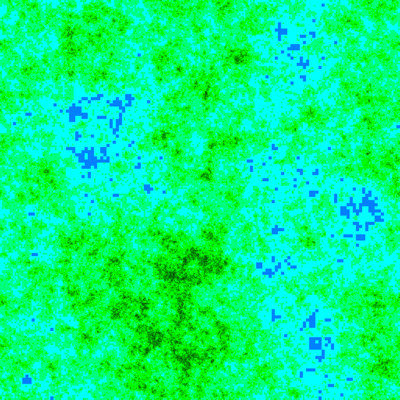}}
\hspace{0.01\textwidth}
\subfloat[$\gamma=1$]{\includegraphics[width=0.49\textwidth]{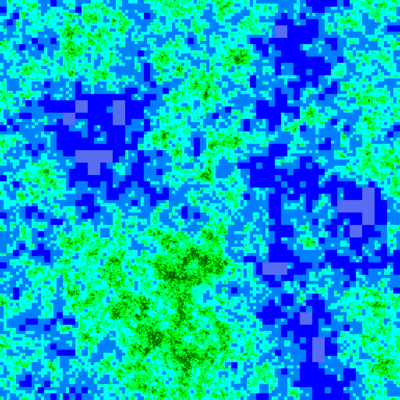}}

\subfloat[$\gamma=3/2$]{\includegraphics[width=0.49\textwidth]{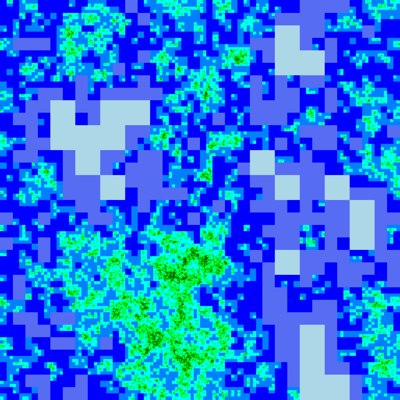}}
\hspace{0.01\textwidth}
\subfloat[$\gamma=2$]{\includegraphics[width=0.49\textwidth]{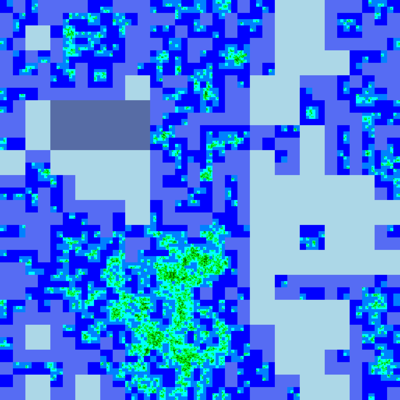}}

\subfloat[Number of subdivisions performed ranging from $0$ (left) to $12$ (right).]{\includegraphics[width=0.75\textwidth]{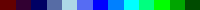}}
\end{center}
\vspace{-0.50cm}
\caption{\label{fig::squaredecomposition}{\small To construct the figures above, first an approximate $\gamma$-LQG measure $\mu$ was chosen by taking a GFF $h$ on a fine ($4096 \times 4096 = 2^{12} \times 2^{12}$) lattice and constructing the measure $e^{\gamma h(z)}dz$ where $dz$ is counting measure on the lattice (normalized so $\mu$ has total mass $1$).  Then a small constant $\delta$ is fixed (here $\delta = 2^{-16}$) and one divides the large square into four squares of equal Euclidean size, divides each of these into four squares of equal Euclidean size, etc., except that if any square's $\mu$-area is less than $\delta$, that square is not further divided.  Each square remaining in the end has $\mu$-area less than $\delta$, but the $\mu$-area of its diadic parent is greater than $\delta$.  Squares are colored by Euclidean size.}}
  \end{figure}

We are interested in all three approaches, but ultimately, the main purpose of this paper is to produce a {\em candidate} for the scaling limit of an $\eta$-DBM process on a $\gamma$-LQG surface (in the fine mesh, or $\delta \to 0$ scaling limit).  We expect that there is a universal scaling limit that does not depend on which approach we take (at least if the discrete setup is sufficiently isotropic; see the discussion in Section~\ref{subsubsec::DLAintro}).  Our goal is to show that (at least for some choices of $\gamma$ and $\eta$) there exists a process, which we call quantum Loewner evolution $\QLE(\gamma^2, \eta)$, that has the dynamic properties that we would expect a scaling limit to have.

For certain values of the parameters $\gamma^2$ and $\eta$, which are illustrated in Figure~\ref{fig::etavsgamma}, we will be able to explicitly describe a stationary law of the growth process in terms of quantum gravity.  We will see that this growth process is similar to SLE except that the point-valued ``driving function'' that one feeds into the Loewner differential equation to obtain SLE$_\kappa$ (namely $\sqrt \kappa$ times Brownian motion on a circle) is replaced by a {\em measure}-valued driving function $\nu_t$ whose stationary law corresponds to a certain boundary measure that appears in Liouville quantum gravity.  The time evolution of this measure is not nearly as easy to describe as the time evolution of Brownian motion, and making sense of this evolution is one of the main goals of this paper.

Let us explain this point a bit further.  We will fix $\gamma$ and an instance $h$ of a free boundary GFF (plus a deterministic multiple of the log function) on the unit disk $\D$.  We will interpret the pair $(\D, h)$ as a $\gamma$-LQG quantum surface and seek to define an increasing collection $(K_t)$ of closed sets, indexed by $t \in [0,T]$ for some $T$, starting with $K_0 = \partial \D$ and growing inward within $\D$ toward the origin.  We assume that each $K_t$ is a {\bf hull}, i.e., a subset of $\overline
\D$ whose complement is an open set containing the origin.  (Note that if a growth model grows outward toward infinity, one can always apply a conformal inversion so that the growth target becomes the origin.)  We will require that for each $t$ the set $K_t$ is a so-called {\em local set} of the GFF instance $h$.  This is a natural technical condition (see Section~\ref{sec::preliminaries}, or the more detailed treatment in \cite{ss2010contour}) that essentially states that altering $h$ on an open set $S \subseteq \D$ does not affect the way that $K_t$ grows {\em before} the first time that $K_t$ reaches $S$.  In order to describe these growing sets $K_t$, we will construct a solution to a type of differential equation imposed on a triple of processes, each of which is indexed by a time parameter $t \in [0,T]$, for some fixed $T>0$:
\begin{enumerate}
\item A measure $\nu$ on $[0,T] \times \partial \D$ whose first coordinate marginal is Lebesgue measure.  We write $\nu_t$ for the conditional probability measure (defined for almost all $t$) obtained by restricting $\nu$ to $\{ t \} \times \partial \D$.  Let $\CN_T$ be the space of measures $\nu$ of this type.
 \item A family $(g_t)$ of conformal maps $g_t \colon \D \setminus K_t \to \D$, where for each $t$ the set $K_t$ is a closed subset of $\D$ whose complement is a simply connected set containing the origin.  We require further that the sets $K_t$ are increasing, i.e.\ $K_s \subseteq K_t$ whenever $s \leq t$, and that for all $t \in [0,T]$ we have $g_t(0) = 0$ and $g_t'(0) = e^t$.  That is, the sets $(K_t)$ are parameterized by the negative $\log$ conformal radius of $\D \setminus K_t$ viewed from the origin.\footnote{$(-1)$ times the log of the conformal radius of $\D \setminus K_t$, viewed from the origin, is also called the {\em capacity} of the $K_t$ (though we caution that the term ``capacity'' has several other meanings in other contexts).}  Let $\CG_T$ be the space of families of maps $(g_t)$ of this type.
 \item A family $(\Fh_t)$ of harmonic functions on $\D$ with the property that $\Fh_t(0) = 0$ for all $t \in [0,T]$ and the map $[0,T] \times \D \to \R$ given by $(t,z) \to \Fh_t(z)$ is jointly continuous in $t$ and $z$.  Let $\CH_T$ be the space of harmonic function families of this type.
 \end{enumerate}
The differential equation on the triple $(\nu_t, g_t, \Fh_t)$ is a triangle of maps between the sets $\CN_T$, $\CG_T$, and $\CH_T$ that describes how the processes in the triple $(\nu_t, g_t, \Fh_t)$ are required to be related to each other, as illustrated in Figure~\ref{fig::QLEtriangle} and further explained below.  The triple involves a constant $\alpha$ that for now is unspecified.  The constant $\eta$ will actually emerge {\em a posteriori} as a scaling symmetry of the map from $\CH_T$ to $\CN_T$ that applies almost surely to the triples we construct.  We will see that when an LQG coordinate change is applied to $(\nu_t, \D)$ (a change that preserves quantum boundary length but changes harmonic measure viewed from zero) $\nu_t$ is locally rescaled by the derivative of the map to the $2+\eta$ power; Figure~\ref{fig::etascaling} explains heuristically why the scaling limit of $\eta$-DBM on a $\gamma$-LQG should have such a symmetry.  The definition of $\eta$ and its relationship to $\alpha$ will be explained in more detail in Section~\ref{subsec::mainresults} and Section~\ref{sec::continuum}.

\begin{figure}[ht!]
\begin{center}
\includegraphics[scale=0.85]{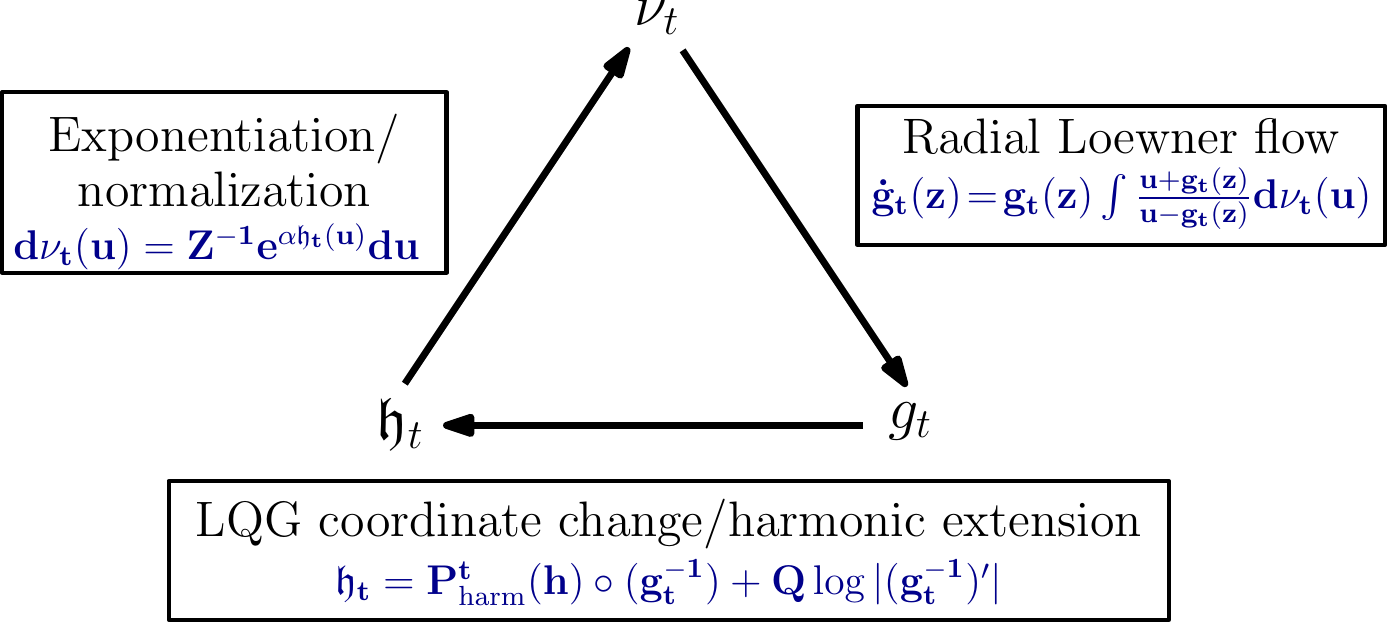}
\caption{\label{fig::QLEtriangle} Visual sketch of the differential equation for the $\QLE$ dynamics.  The map that takes the process $\nu_t$ to the process $g_t$ is the most straightforward to describe.  It is simply Loewner evolution, and works for any $\nu_t$ we would want to consider, see Theorem~\ref{thm::measurehullprocesscorrespondence}.  This is the ``differential'' part of the equation, since $\nu_t$ determines the time derivative of $g_t$. The map from $g_t$ to $\Fh_t$ is also fairly straightforward, assuming $h$ has been fixed in advance.  Here $\pHarm^t(h)$ is the harmonic extension of the values of $h$ from $\partial (\D \setminus K_t)$ to $\D \setminus K_t$.  (This notion is defined precisely in the case that $K_t$ is a local set in Section~\ref{sec::preliminaries}.)  We usually choose an additive constant for $\Fh_t$ so that $\Fh_t(0)=0$. Since the $\Fh_t$ of interest tend to blow up to $\pm \infty$ as one approaches $\partial \D$, a limiting procedure is required to make sense of the map from $\Fh_t$ to $\nu_t$.  One approach is to define a continuous approximation $\Fh_t^n$ to $\Fh_t$ using the first $n$ terms in the power series expansion of the analytic function with real part $\Fh_t$.  One can then let $\nu_t$ be the weak $n \to \infty$ limit of the measures $e^{\alpha \Fh^n_t(u)} du$ on $\partial \D$, normalized to be probability measures.  Such an approach makes sense provided that the process $\Fh_t$ almost surely spends almost all time on functions for which this limit exists.}
\end{center}
\end{figure}

\begin{enumerate}
\item $\CN_T \to \CG_T$:  the process $(g_t)$ is obtained as the radial Loewner flow driven by $(\nu_t)$,  as further explained in Section~\ref{subsec::measureloewner}.  It turns out (see Theorem~\ref{thm::measurehullprocesscorrespondence}) that Loewner evolution describes a one-to-one map from $\CN_T$ to $\CG_T$.
\item $\CG_T \to \CH_T$: for each $t$, the function $\Fh_t$ is obtained from $h$ by first letting $\pHarm^t(h)$ be the harmonic extension of the values of $h$ from $\partial (\D \setminus K_t)$ to $\D \setminus K_t$, and then letting $\Fh_t$ be the harmonic function on $\D$ defined by $\Fh_t = \pHarm^t(h) \circ (g_t^{-1}) + Q \log |(g_t^{-1})'|$, with the additive constant chosen to make $\Fh_t(0)=0$. (Here $Q = 2/\gamma + \gamma/2$ and the addition of $Q\log|(g_t^{-1})'|$ comes from the LQG coordinate change rule described in Section~\ref{subsubsec::LQGintro} below.)  Once $h$ is fixed, this step essentially describes a map from $\CG_T$ to $\CH_T$.  We say ``essentially'' because the harmonic extension $\pHarm^t(h)$ is not necessarily well-defined for {\em all} $h$ and $g_t$ pairs, but it {\em is} almost surely defined under the above-mentioned assumption that $K_t$ is local; see Section~\ref{sec::preliminaries} or \cite{ss2010contour}.
\item  $\CH_T \to \CN_T$: $\nu_t$ is obtained by exponentiating $\alpha \Fh_t$ on $\partial \D$, for a given value $\alpha$ (which depends on $\eta$ and $\gamma$).  Since the $\Fh_t$ we will be interested in are almost surely harmonic functions that blow up to $\pm \infty$ as one approaches $\partial \D$, we will have to use a limiting procedure: $d\nu_t = \lim_{n \to \infty} \CZ_{n,t}^{-1} e^{\alpha \Fh^n_t(u)}du$ where $du$ is Lebesgue measure on $\partial \D$ and $\Fh^n_t$ is (the real part of) the sum of the first $n$ terms in the power series expansion of the analytic function (with real part) $\Fh_t$, and $\CZ_{n,t} = \int_{\partial \D} e^{\alpha \Fh^n_t(u)}du$.  We would like to say that this step provides a map from $\CH_T$ to $\CN_T$, but in fact the map is only defined on the subset of $\CH_T$ for which these limits exist for almost all time.\footnote{Alternatively, one could define $\nu \in \CN_T$ as the weak $n \to \infty$ limit of the measures $\CZ_{n,t}^{-1} e^{\alpha \Fh^n_t(u)}dtdu$ on $[0,T] \times \partial \D$.   This limit could conceivably exist even in settings for which the $\nu_t$ did not exist for almost all $t$.  To avoid making any assumptions at all about limit existence, one could alternatively define a one-to-(possibly)-many map from each $\Fh_t$ process in $\CH_T$ to the set of {\em all} $\nu \in \CN_T$ obtained as weak $n \to \infty$ limit points of the sequence $\CZ_{n,t}^{-1} e^{\alpha \Fh^n_t(u)}dtdu$ of measures on $[0,T] \times \partial \D$.  With that approach, one might require only that $\nu$ be {\em one} of these limit points.  Although we do not address this point in this paper, we believe that it might be possible, using these alternatives, to extend the solution existence results of this paper to additional values of $\eta$ and $\gamma$.}
\end{enumerate}

\begin{figure}[ht!]
\begin{center}
\includegraphics[scale=0.85]{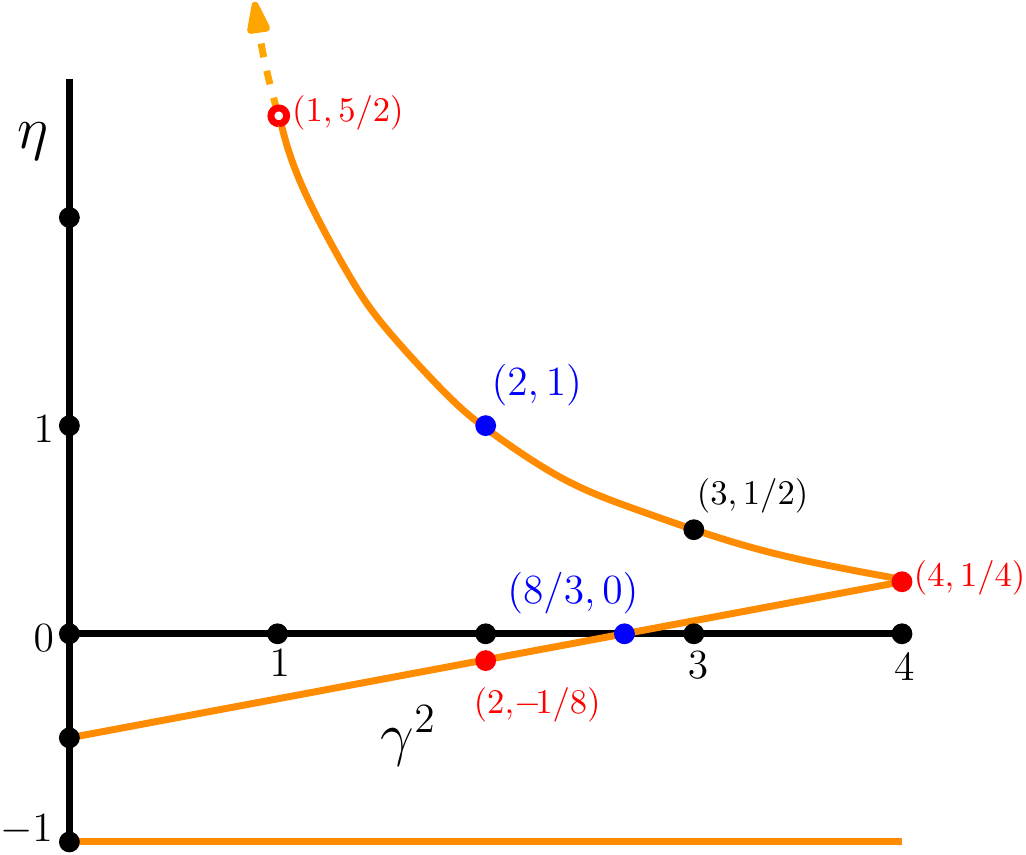}
\end{center}
\caption{\label{fig::etavsgamma}
The solid orange curves illustrate the $(\gamma^2,\eta)$ pairs for which we are able to construct and understand a $\QLE(\gamma^2,\eta)$ process most explicitly. The curves correspond to
$\eta \in \{-1 \, ,\, \, \frac{3\gamma^2}{16} - \frac12 \, ,\,\, \frac{3}{\gamma^2} - \frac12 \},$ where $\gamma^2 \in (0,4]$.  When $(\gamma^2,\eta)$ is on the middle curve, our construction involves radial $\SLE_\kappa$ with $\kappa = 16/\gamma^2$.  When $(\gamma^2, \eta)$ is on the upper curve, it involves radial $\SLE_\kappa$ with $\kappa = \gamma^2$.
The solid red dots are phase transitions of the $\SLE_\kappa$ curves used to construct $\QLE$: $(2,-1/8)$ corresponds to $\kappa=8$ and
$(4,1/4)$ corresponds to $\kappa=4$.
The point $(1,5/2)$ corresponds to $\kappa = 1$ and is a phase transition beyond which the $\QLE$ construction of this paper becomes trivial --- i.e., when $\kappa \leq 1$, the construction (carried out naively) produces a simple radial $\SLE$ curve independent of $h$ (and the measures $\nu_t$ are point masses for all $t$).  The blue dots are points we are especially interested in.  The point $(2,1)$ is related to DLA on spanning-tree-decorated random planar maps.  The point $(8/3,0)$ is related to the Eden model on undecorated random planar maps, and to the problem of endowing pure LQG with a distance function.
}
\end{figure}

\begin{figure}[ht!]
\begin{center}
\includegraphics[scale=0.85]{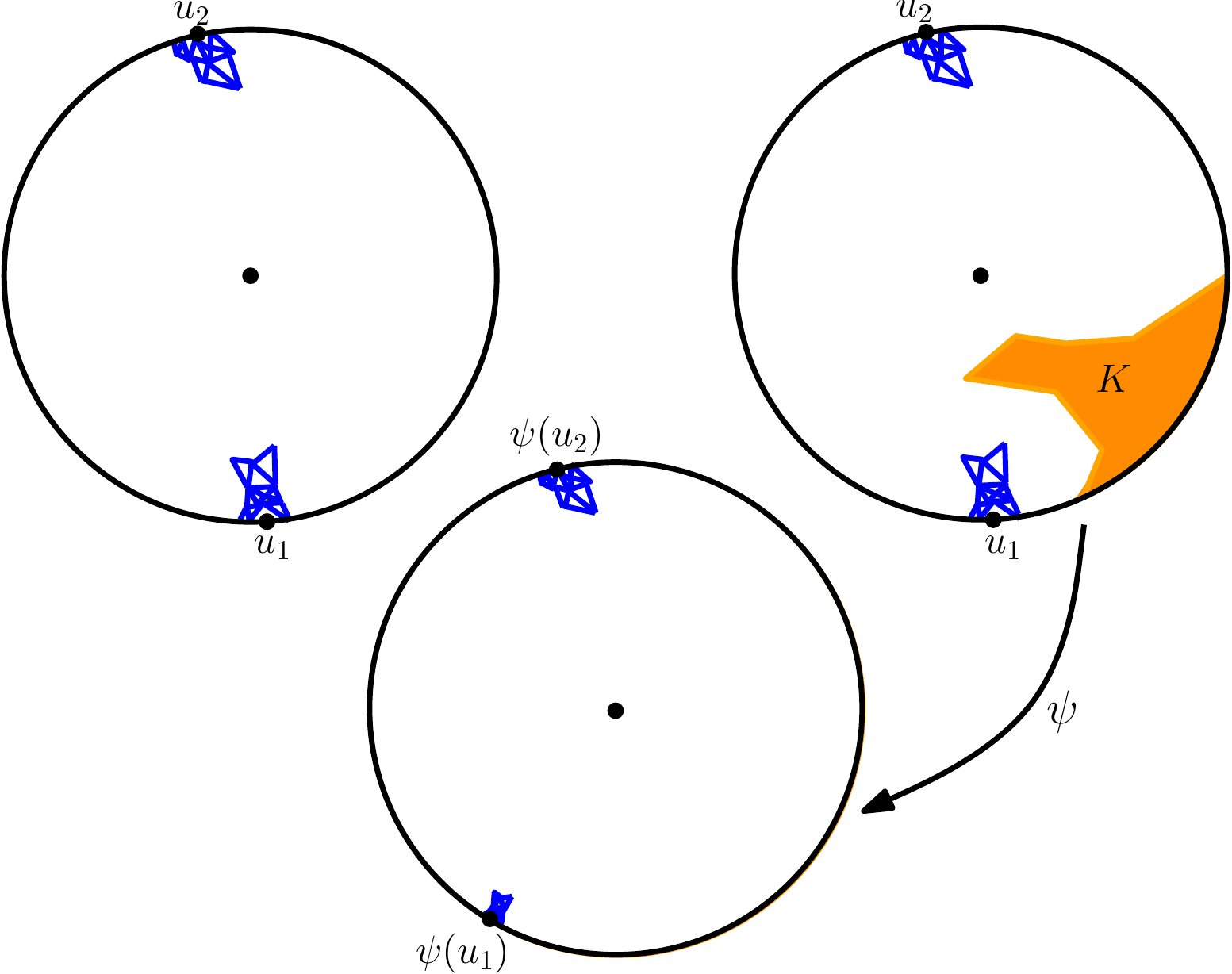}
\end{center}
\caption{\label{fig::etascaling}
{\bf Upper left:} Suppose a discrete random triangulation is conformally mapped to the disk, and the Eden model growing from the boundary inward takes about $N$ time units to absorb the cluster of triangles shown near $u_1$, and also about $N$ units to absorb the cluster near $u_2$.  (Other not-shown triangles scattered around the boundary are also added during that time.)
{\bf Upper right:} Now suppose that we modify the initial setup by designating the hull $K$ to be part of the boundary.  Intuitively, if the regions near $u_1$ and $u_2$ are small, this modification should not affect the {\em relative} rate at which growth happens near $u_1$ and $u_2$.  That is, there should be some $N'$ such that both clusters take about $N'$ steps to be absorbed.  {\bf Bottom:} A conformal map $\psi \colon \D \setminus K \to \D$ with $\psi(0)=0$ scales the region near $u_i$ by about $|\psi'(u_i)|$.  The capacity corresponding to the shown blue cluster near $\psi(u_i)$ is approximately $|\psi'(u_i)|^2$ times that of the original blue cluster near $u_i$.  This suggests that if $(\nu_t)$ is the driving measure in the bottom figure and $(\wt{\nu}_t)$ is the original driving measure in the upper left, and $I_i$ is a small interval about $u_i$, then $\nu_0(\psi(I_i))$ should be roughly proportional to $|\psi'(u_i)|^2 \wt{\nu}_0 (I_i)$.  In the $\eta$-DBM model, one replaces $|\psi'(u_i)|^2$ by $|\psi'(u_i)|^{2+\eta}$ because the rate at which particles reach $u_i$ should also change by a factor roughly proportional to $|\psi'(u_i)|^\eta$. }
\end{figure}

We remark that if we had $h=0$, then the triangle in Figure~\ref{fig::QLEtriangle} would say that $\Fh_t = Q \log|(g_t^{-1})'|$ and that $\nu_t$ is given (up to multiplicative constant) by $|(g_t^{-1}(u))'|^{\alpha Q} du$.  This is precisely the deterministic evolution associated with the DBM that is discussed, for example, in \cite{rohde2005some} (except that the exponent $\alpha Q$ given here is replaced by a single parameter $-\alpha$).  This deterministic evolution has some smooth trivial solutions (for example the constant circular growth given by letting $\nu$ be the uniform measure on $[0,T] \times \partial \D$, and taking $g_t(z) = e^t z$ and $\Fh_t(z) = 0$).  For these solutions, we would not need to use limits to construct $\nu_t$ from $\Fh_t$, since the measures $e^{\alpha \Fh_t(u)}du$ would be well defined.  However, if one starts with a generic harmonic function for $h$ that extends smoothly to $\partial \D$ (instead of simply $h = 0$) then the evolution can develop singularities in finite time, and once one encounters the singularities it is unclear how to continue the evolution; this issue and various regularization/approximation schemes to prevent singularity-formation are discussed in \cite{carleson2001aggregation, rohde2005some}.  Even in the $h=0$ case, Figure~\ref{fig::QLEtriangle} suggests an interesting alternative to the regularization approaches of \cite{carleson2001aggregation, rohde2005some}.  It suggests an {\em exact} (non-approximate) notion of what it means to be a solution to the dynamics that makes sense even when singularities are present; the approximation is only involved in making sense of the map from $\CH_T$ to $\CN_T$.  Since the real aim in the $h=0$ case is to define a natural probability measure on the space of fractal solutions to the dynamics (which should describe the scaling limit of DBM, at least in suitably isotropic formulations), one might hope that these solutions would have some nice properties (perhaps a sort of almost sure fractal self similarity, or long range approximate independence of $\Fh_t$ boundary values) that would allow the map from $\CH_T$ to $\CN_T$ to be almost surely well defined.

In this paper, we will take $h$ to be the GFF (plus a deterministic multiple of $\log| \cdot |$) and we will construct solutions to the dynamics of Figure~\ref{fig::QLEtriangle} for $\alpha$ and $Q$ values that correspond (in a way we explain later) to the $(\gamma^2, \eta)$ values that lie on the upper two curves in Figure~\ref{fig::etavsgamma}.  We will also argue that $\eta = -1$ corresponds to $\alpha = 0$, which yields a trivial solution corresponding to the bottom curve in Figure~\ref{fig::etavsgamma}.  We remark that although this solution is ``trivial'' in the continuum, the analogous statement about discrete graphs (namely that if a random planar map model, conformally mapped to the disk in some appropriate way, scales to LQG on the disk, then the $(-1)$-DBM on the random planar map has the dilating circle process as a scaling limit) is still an open problem.

We will produce non-trivial continuum constructions for (the solid portions of) the upper two curves in Figure~\ref{fig::etavsgamma} by taking subsequential limits of certain discrete-time approximate processes defined using a radial version of the quantum gravity zipper defined in \cite{sheffield2010weld}.  These approximate processes can themselves be interpreted as non-lattice-based variants of $\eta$-DBM on a $\gamma$-LQG surface that are designed to be isotropic and to have some extra conformal invariance symmetries (here one grows small portions of SLE curves instead of adding small particles of fixed Euclidean shape).  The similarity between our approximations and DBM seems to support the idea that (at least for some $(\gamma^2, \eta)$ pairs) the processes we construct are the ``correct'' continuum analogs of $\eta$-DBM on a $\gamma$-LQG surface.  The portion of the upper curve corresponding to $\gamma^2 \leq 1$ is degenerate in that the approximation procedure used to construct the process $\nu_t$, as described in Section~\ref{sec::existence}, would yield a point mass for almost all $t$ (although we will discuss this case in detail in this paper).

To each of these processes, we associate a discrete-time approximation of the triple $(\nu_t, g_t, \Fh_t)$, in which the time parameter takes values $0, \delta, 2\delta, \cdots$ for a constant $\delta$.  The most important property that these discrete-time processes have (which distinguishes them from, e.g., the Hastings-Levitov approximations described in \cite{hastings1998laplacian}) is that the stationary law of the $\nu_t$ and the $\Fh_t$ turn out to be {\em exactly} the same for each discrete-time approximation (even as the time step size varies).   This rather surprising property is what allows us to understand the stationary law of the $\delta \to 0$ limit (something that has never been possible, in the Euclidean $\gamma = 0 $ case, for DBM approximation schemes like Hastings-Levitov).  We find that the limiting stationary law is exactly the same as the common stationary law of the approximations, and this allows us to prove that the limit satisfies the dynamics of Figure~\ref{fig::QLEtriangle}, and to prove explicit results about this limit, which we state formally in Section~\ref{subsec::mainresults}.

We will see in Section~\ref{sec::discrete} that the procedure we use to generate the continuum process has discrete analogs, which give interesting relationships between percolation and the Eden model, and also between loop-erased random walk and DLA.  The reader who wishes to understand the key idea behind our construction (without delving into the analytical machinery behind the quantum zipper) might begin with Section~\ref{sec::discrete}.

Before we state our results more precisely, we present in Section~\ref{subsec::modelbackground} an overview of several of the models and mathematical objects that will be treated in this work.  We also present, in Figures~\ref{fig::triplegamma0.00} through~\ref{fig::triplegamma2.00}, computer simulations of the Eden model and DLA on $\gamma$-LQG square tilings such as those represented in Figure~\ref{fig::squaredecomposition}.  In each of these figures we have $\delta = 2^{-16}$ (as explained in the caption to Figure~\ref{fig::squaredecomposition}) which results in many squares of a larger Euclidean size (and hence a more pixelated appearance) for the larger $\gamma$ values.  Figures~\ref{fig::growingmetricball},~\ref{fig::large_metric2_squares}, and~\ref{fig::large_qdla2} show instances with larger $\gamma$ but smaller $\delta$ values.  Generally, the DLA simulations for larger $\gamma$ values appear to have characteristics in common with the $\gamma=0$ case, but there is more variability to the shapes when $\gamma$ is larger.  The large-$\gamma$, small-$\delta$ DLA simulations such as Figure~\ref{fig::large_qdla2} sometimes look a bit like Chinese dragons, with a fairly long and windy backbone punctuated by shorter heavily decorated limbs.

Figures~\ref{fig::large_metric2_seeds} and~\ref{fig::dlaseeds} show what happens when different instances of the Eden model or DLA are performed on top of the same instance of a LQG square decomposition.  These figures address an interesting question: how much of the shape variability comes from the randomness of the underlying graph, and how much from the additional randomness associated with the growth process?  We believe but cannot prove that in the Eden model case shown in Figure~\ref{fig::large_metric2_seeds}, the shape of the cluster is indeed determined, to first order (as $\delta$ tends to zero), by the GFF instance used to define the LQG measure.  The deterministic (given $h$) shape should be the metric ball in a canonical continuum metric space determined by the GFF.

On the continuum level, the authors are in the process of carrying out a program for using $\QLE(8/3,0)$ to endow a $\gamma = \sqrt{8/3}$ Liouville quantum gravity surface with metric space structure, and to show that the resulting metric space is equivalent in law to a particular random metric space called the Brownian map.  But this is not something we will achieve in this paper.  (We describe forthcoming works in more detail at the end of Section~\ref{sec::questions}.)

\begin{figure}[ht!]
\subfloat[Squares]{\includegraphics[width=0.32\textwidth]{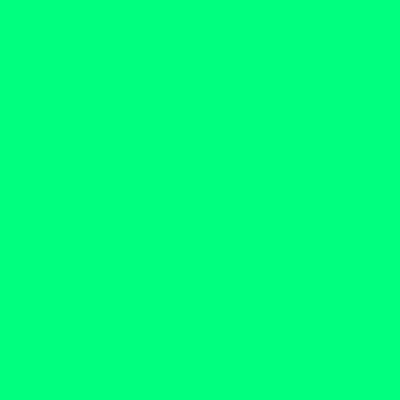}}
\hspace{0.01\textwidth}
\subfloat[Eden model]{\includegraphics[width=0.32\textwidth]{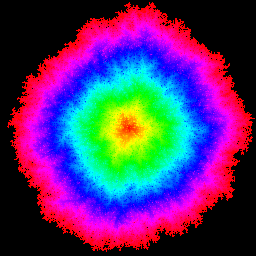}}
\hspace{0.01\textwidth}
\subfloat[DLA]{\includegraphics[width=0.32\textwidth]{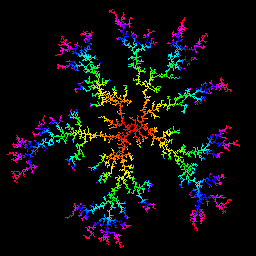}}
\caption{\label{fig::triplegamma0.00}  $\gamma = 0$}
\end{figure}

\begin{figure}[ht!]
\subfloat[Squares]{\includegraphics[width=0.32\textwidth]{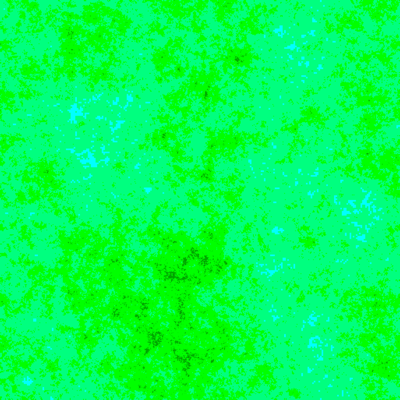}}
\hspace{0.01\textwidth}
\subfloat[Eden model]{\includegraphics[width=0.32\textwidth]{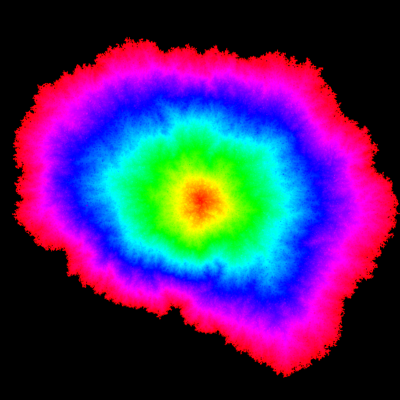}}
\hspace{0.01\textwidth}
\subfloat[DLA]{\includegraphics[width=0.32\textwidth]{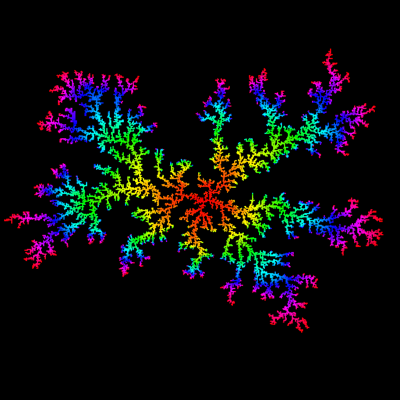}}
\caption{\label{fig::triplegamma0.25}$\gamma = 1/4$}
\end{figure}

\begin{figure}[ht!]
\subfloat[Squares]{\includegraphics[width=0.32\textwidth]{pictures/10-11-2013/small/resized/small_squares0_500000_0_062500_0_000000_0000_04096_0000.png}}
\hspace{0.01\textwidth}
\subfloat[Eden model]{\includegraphics[width=0.32\textwidth]{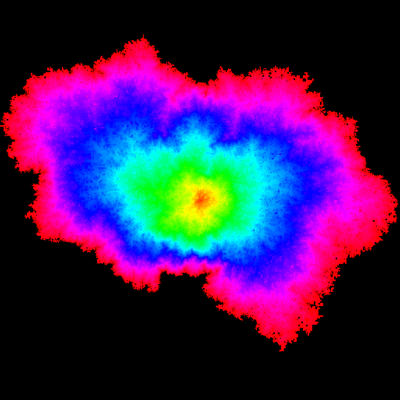}}
\hspace{0.01\textwidth}
\subfloat[DLA]{\includegraphics[width=0.32\textwidth]{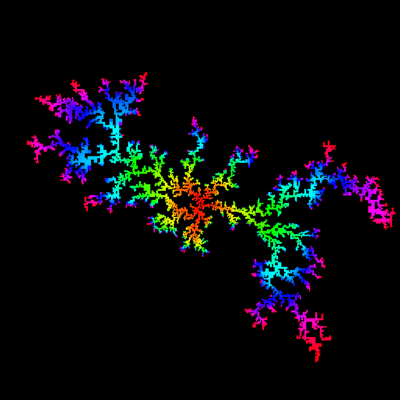}}
\caption{\label{fig::triplegamma0.50}$\gamma = 1/2$}
\end{figure}

\begin{figure}[ht!]
\subfloat[Squares]{\includegraphics[width=0.32\textwidth]{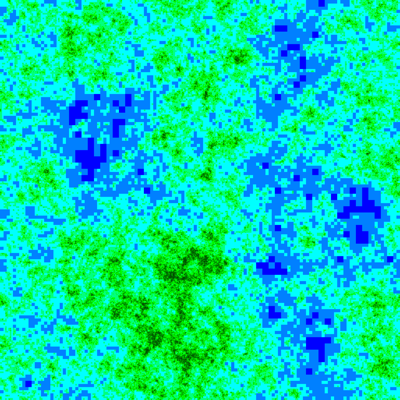}}
\hspace{0.01\textwidth}
\subfloat[Eden model]{\includegraphics[width=0.32\textwidth]{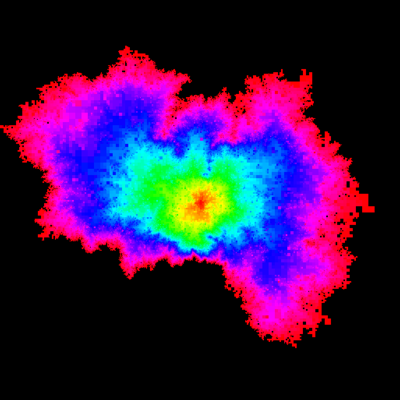}}
\hspace{0.01\textwidth}
\subfloat[DLA]{\includegraphics[width=0.32\textwidth]{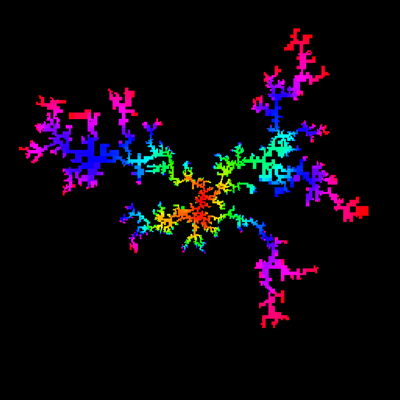}}
\caption{\label{fig::triplegamma0.75}$\gamma = 3/4$}
\end{figure}

\begin{figure}[ht!]
\subfloat[Squares]{\includegraphics[width=0.32\textwidth]{pictures/10-11-2013/small/resized/small_squares1_000000_0_062500_0_000000_0000_04096_0000.png}}
\hspace{0.01\textwidth}
\subfloat[Eden model]{\includegraphics[width=0.32\textwidth]{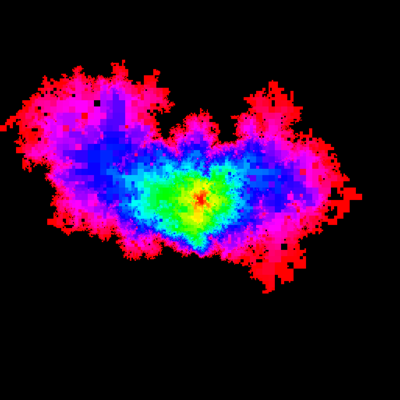}}
\hspace{0.01\textwidth}
\subfloat[DLA]{\includegraphics[width=0.32\textwidth]{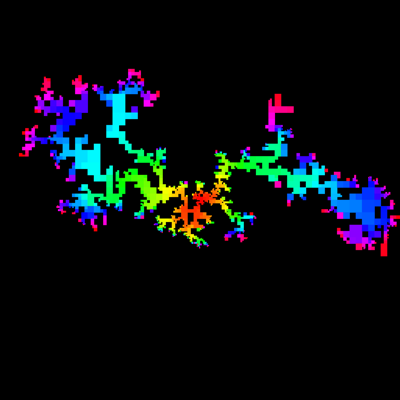}}
\caption{\label{fig::triplegamma1.00}$\gamma = 1$}
\end{figure}

\begin{figure}[ht!]
\subfloat[Squares]{\includegraphics[width=0.32\textwidth]{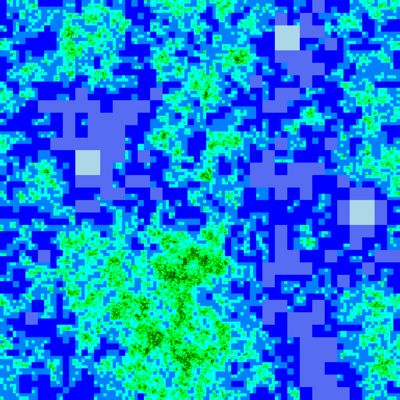}}
\hspace{0.01\textwidth}
\subfloat[Eden model]{\includegraphics[width=0.32\textwidth]{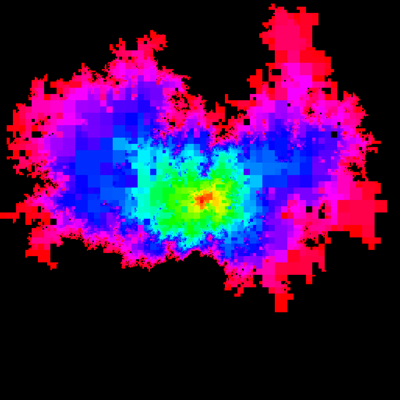}}
\hspace{0.01\textwidth}
\subfloat[DLA]{\includegraphics[width=0.32\textwidth]{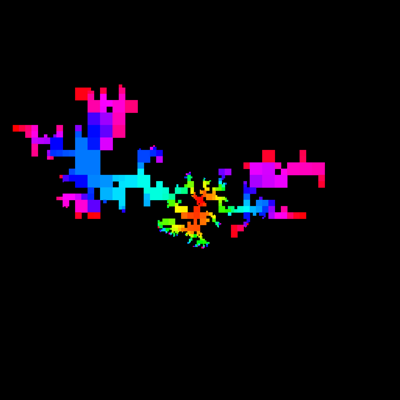}}
\caption{\label{fig::triplegamma1.25}$\gamma = 5/4$}
\end{figure}

\begin{figure}[ht!]
\subfloat[Squares]{\includegraphics[width=0.32\textwidth]{pictures/10-11-2013/small/resized/small_squares1_500000_0_062500_0_000000_0000_04096_0000.png}}
\hspace{0.01\textwidth}
\subfloat[Eden model]{\includegraphics[width=0.32\textwidth]{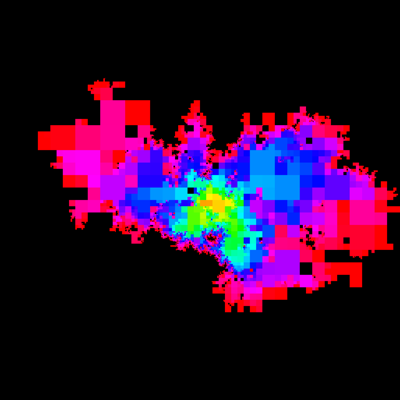}}
\hspace{0.01\textwidth}
\subfloat[DLA]{\includegraphics[width=0.32\textwidth]{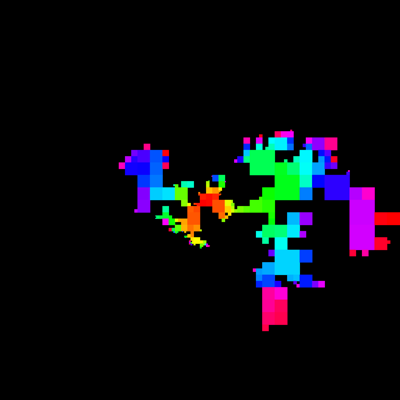}}
\caption{\label{fig::triplegamma1.50}$\gamma = 3/2$}
\end{figure}

\begin{figure}[ht!]
\subfloat[Squares]{\includegraphics[width=0.32\textwidth]{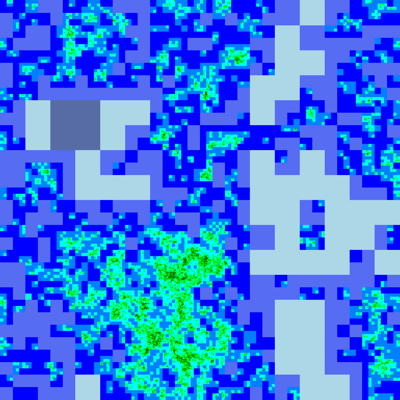}}
\hspace{0.01\textwidth}
\subfloat[Eden model]{\includegraphics[width=0.32\textwidth]{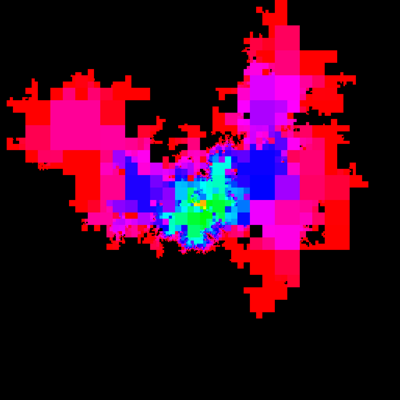}}
\hspace{0.01\textwidth}
\subfloat[DLA]{\includegraphics[width=0.32\textwidth]{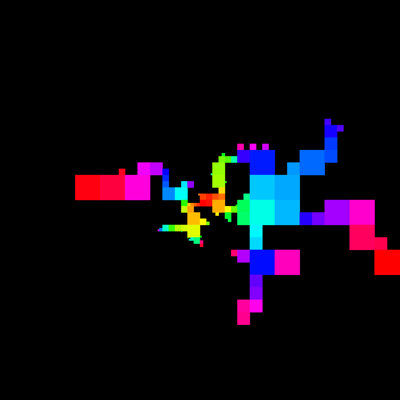}}
\caption{\label{fig::triplegamma1.75}$\gamma = 7/4$}
\end{figure}

\begin{figure}[ht!]
\subfloat[Squares]{\includegraphics[width=0.32\textwidth]{pictures/10-11-2013/small/resized/small_squares2_000000_0_062500_0_000000_0000_04096_0000.png}}
\hspace{0.01\textwidth}
\subfloat[Eden model]{\includegraphics[width=0.32\textwidth]{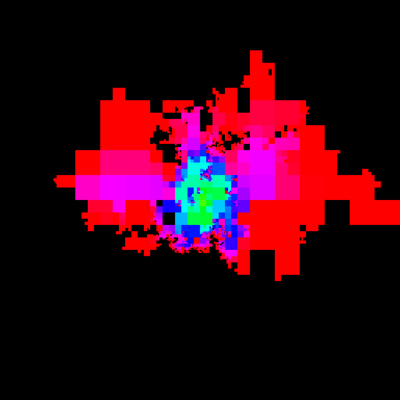}}
\hspace{0.01\textwidth}
\subfloat[DLA]{\includegraphics[width=0.32\textwidth]{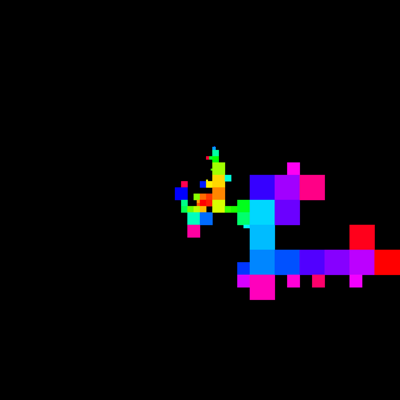}}
\caption{\label{fig::triplegamma2.00}$\gamma = 2$}
\end{figure}

\begin{figure}[ht!]
\vspace{-0.5cm}
\begin{center}
\subfloat[25\%]{\includegraphics[width=0.48\textwidth]{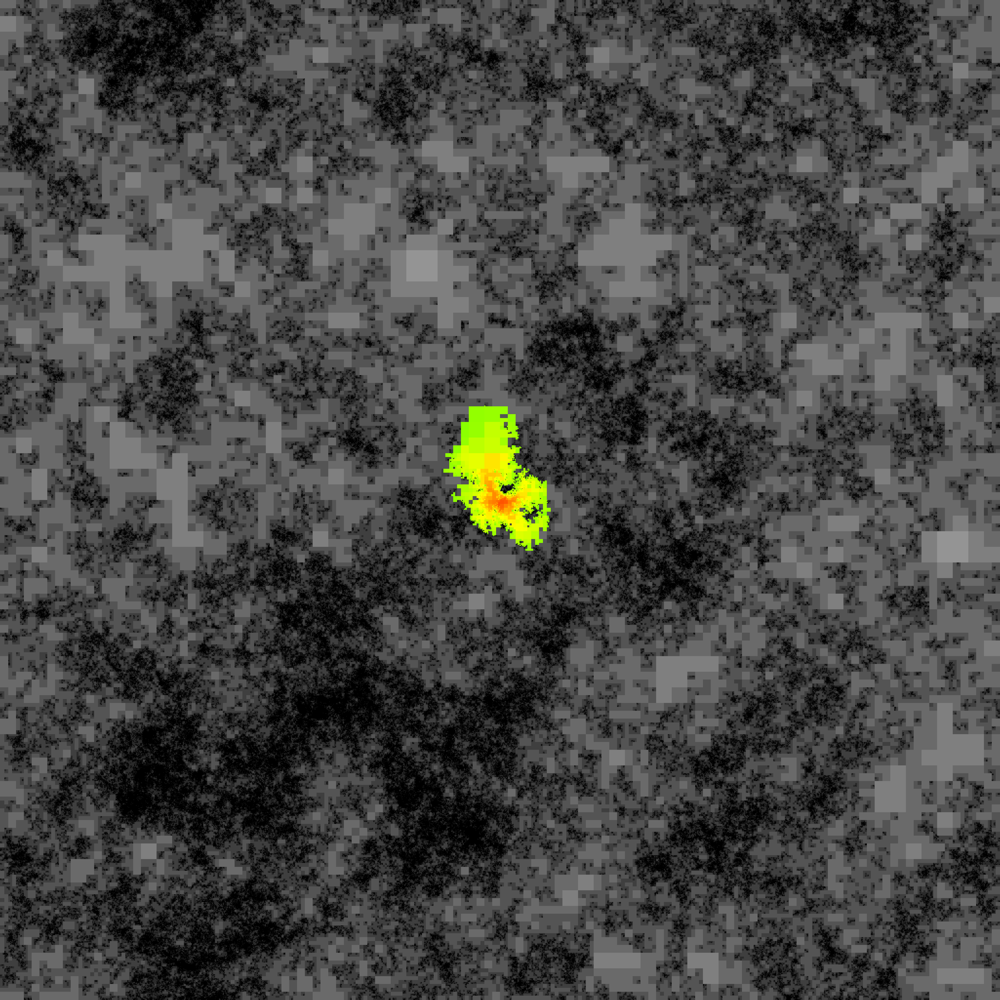}}
\hspace{0.017\textwidth}
\subfloat[50\%]{\includegraphics[width=0.48\textwidth]{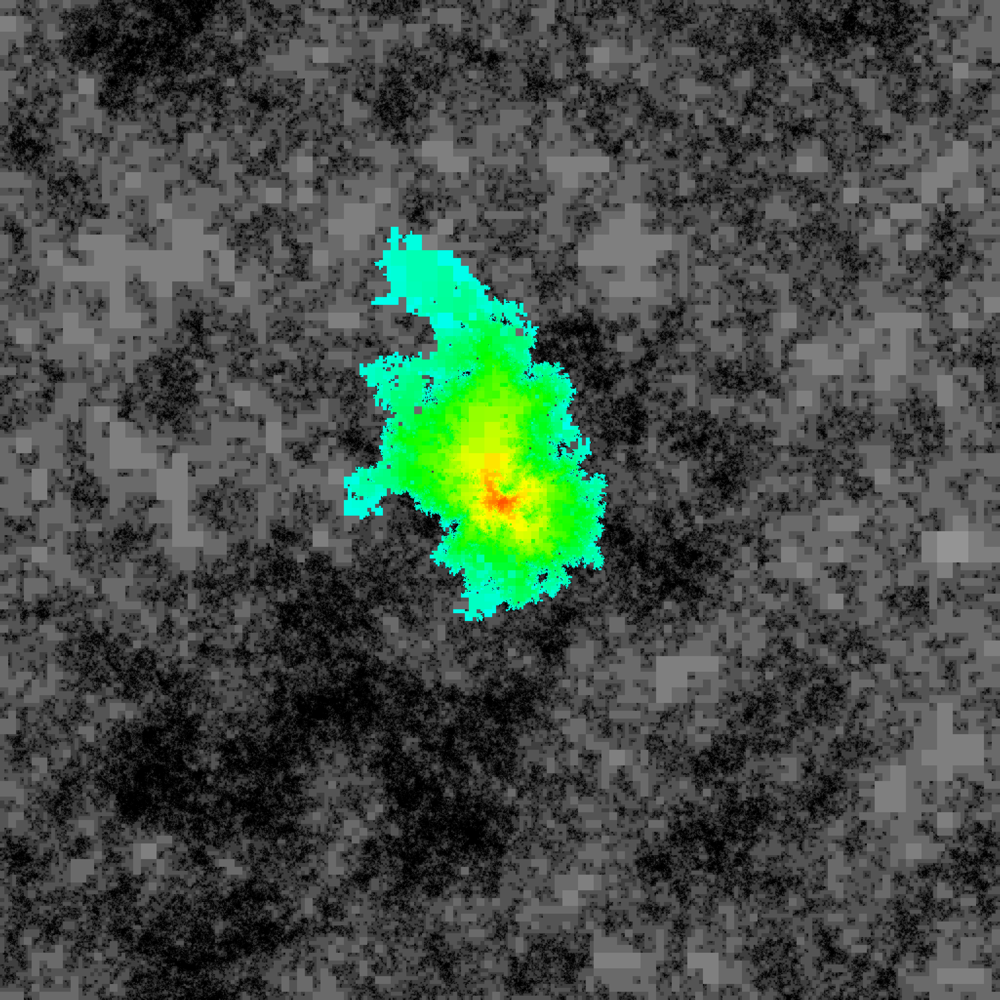}}

\subfloat[75\%]{\includegraphics[width=0.48\textwidth]{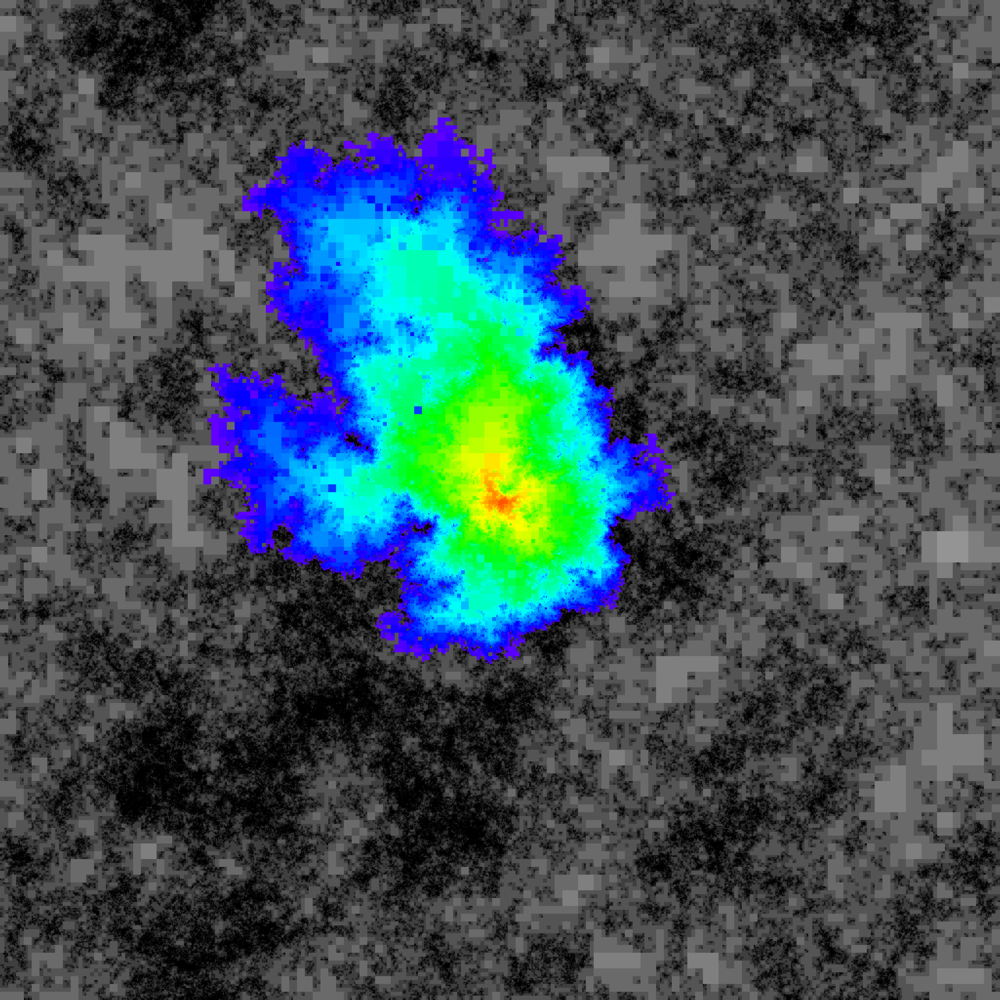}}
\hspace{0.017\textwidth}
\subfloat[100\%]{\includegraphics[width=0.48\textwidth]{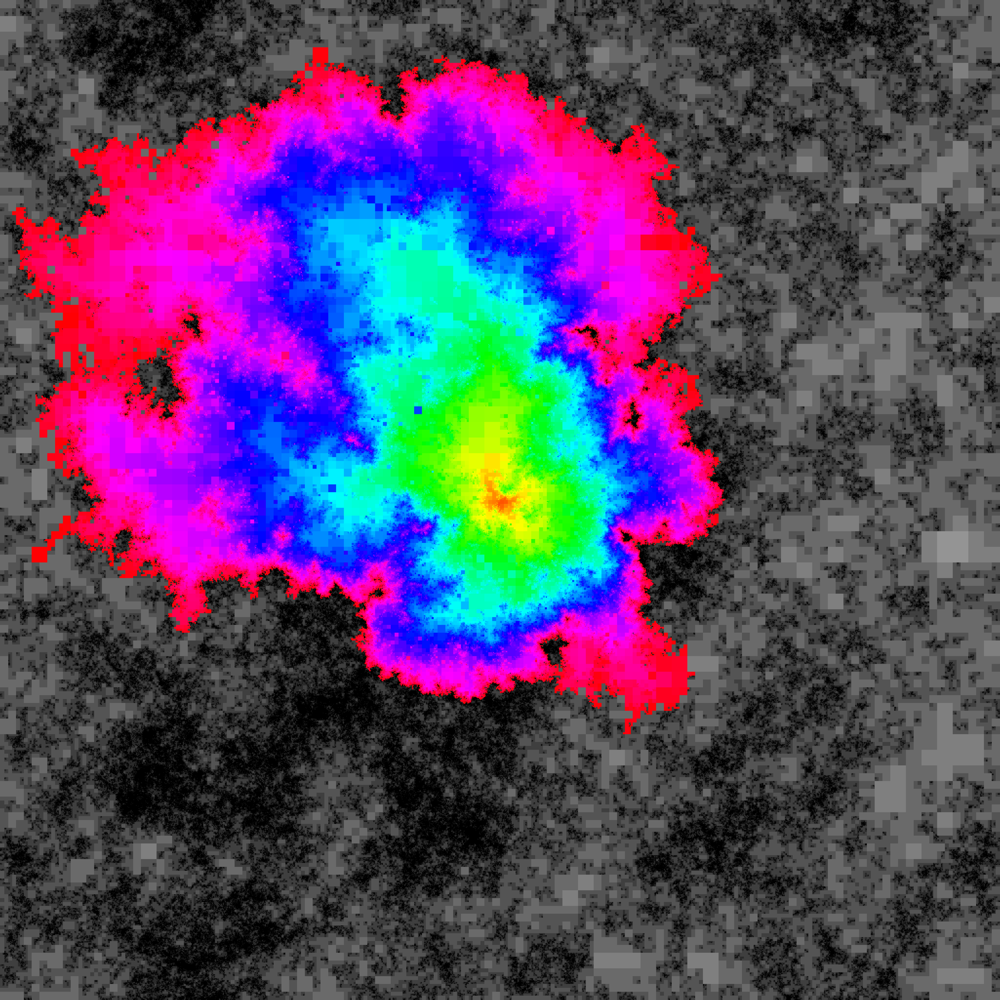}}

\hspace{0.025\textwidth}
\subfloat[Time-parameterization.]{\includegraphics[width=0.47\textwidth]{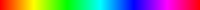}}
\end{center}
\vspace{-0.50cm}
\caption{\label{fig::growingmetricball} An Eden model instance on a $\sqrt{8/3}$-LQG generated with an $8192 \times 8192 = 2^{13} \times 2^{13}$ DGFF, where $\delta=2^{-24}$.  Shown in greyscale is the original square decomposition (squares of larger Euclidean size are colored lighter).  Using the scale shown above, the colors indicate the radius of the ball as it grows relative to the radius at which it first hits boundary of the square.  This simulation is a discrete analog of $\QLE(8/3,0)$. See also Figure~\ref{fig::large_metric2_squares} and Figure~\ref{fig::qletrace}.}
\end{figure}

\begin{figure}[ht!]
\begin{center}
\subfloat{\includegraphics[width=\textwidth]{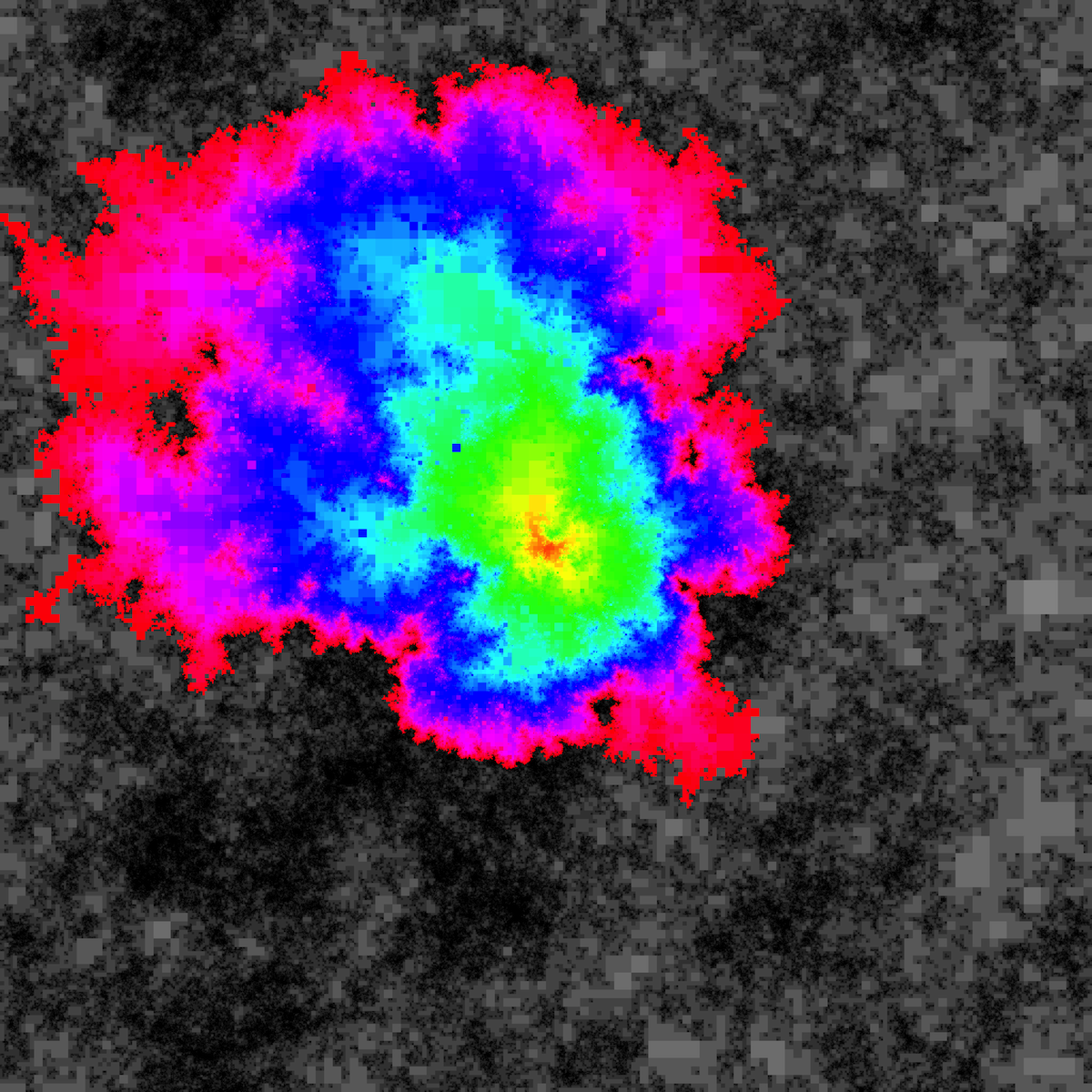}}

\subfloat[Time-parameterization.]{\includegraphics[width=0.47\textwidth]{figures/timescale.png}}
\end{center}
\vspace{-0.50cm}
\caption{\label{fig::large_metric2_squares} Enlargement of final box in Figure~\ref{fig::growingmetricball}.}
\end{figure}

\begin{figure}[ht!]
\vspace{-1cm}
\begin{center}
\subfloat{\includegraphics[width=\textwidth]{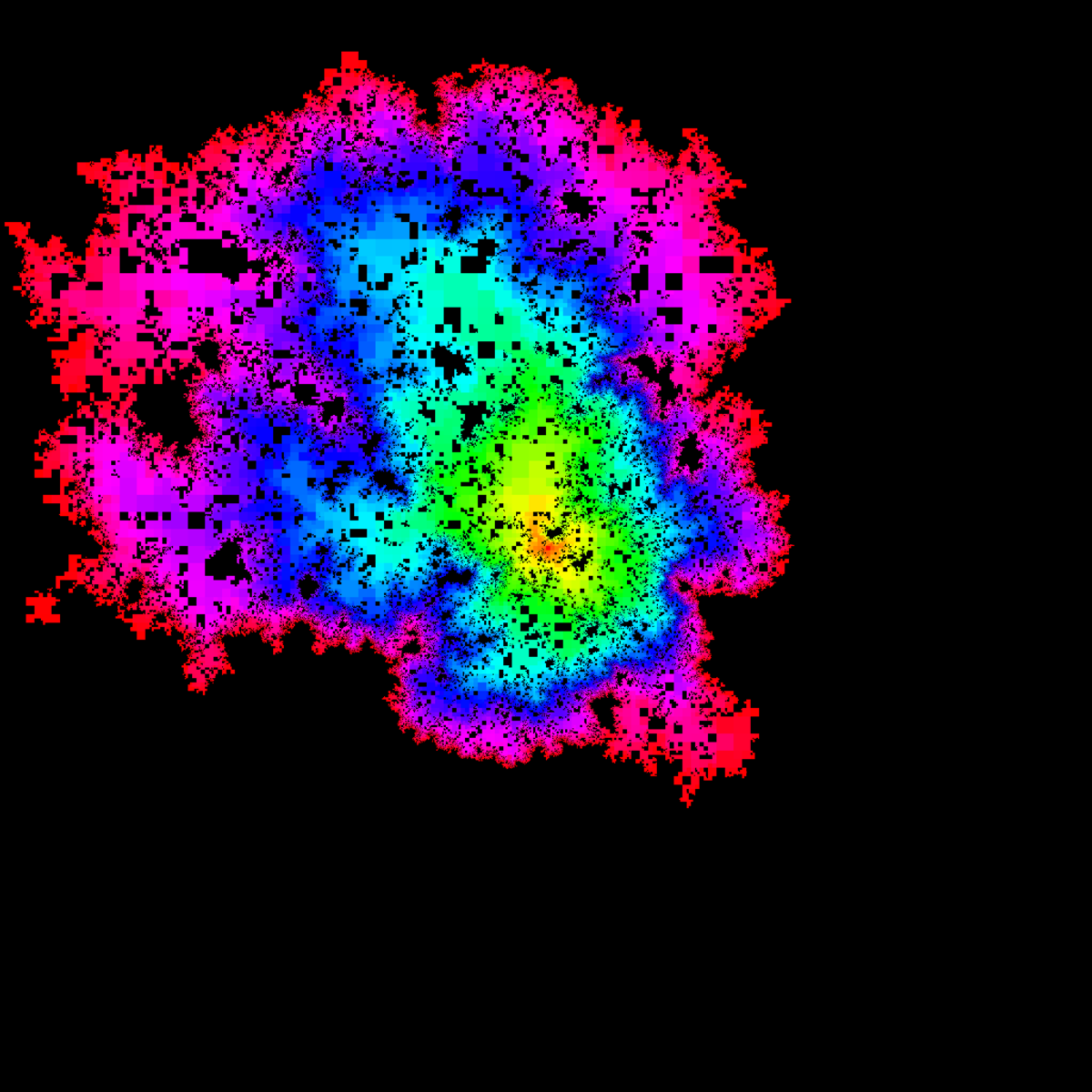}}

\subfloat[Time-parameterization.]{\includegraphics[width=0.47\textwidth]{figures/timescale.png}}
\end{center}
\vspace{-0.50cm}
\caption{\label{fig::qletrace} Eden model as in Figure~\ref{fig::large_metric2_squares} except that one only adds squares {\em on the outside} (i.e., reachable by paths from infinity that don't pass through the cluster).  The cluster appears to tend to hit regions with big squares but circumvent regions with tiny squares.  The number of colored squares is $213061 \approx 2^{17.7}$, and each has $\mu$ mass less than a $\delta = 2^{-24}$ fraction of the total, with one caveat: our simulation did not further subdivide the tiny $2^{-13} \times 2^{-13}$ squares, so these can have mass greater than a $2^{-24}$ fraction of the total.  There are $3008224 \approx 2^{21.5}$ squares (colored and non-colored) in this figure, and most of the $\mu$ mass lies in the tiny ones.}
\end{figure}

\begin{figure}[ht!]
\begin{center}
\subfloat{\includegraphics[width=\textwidth]{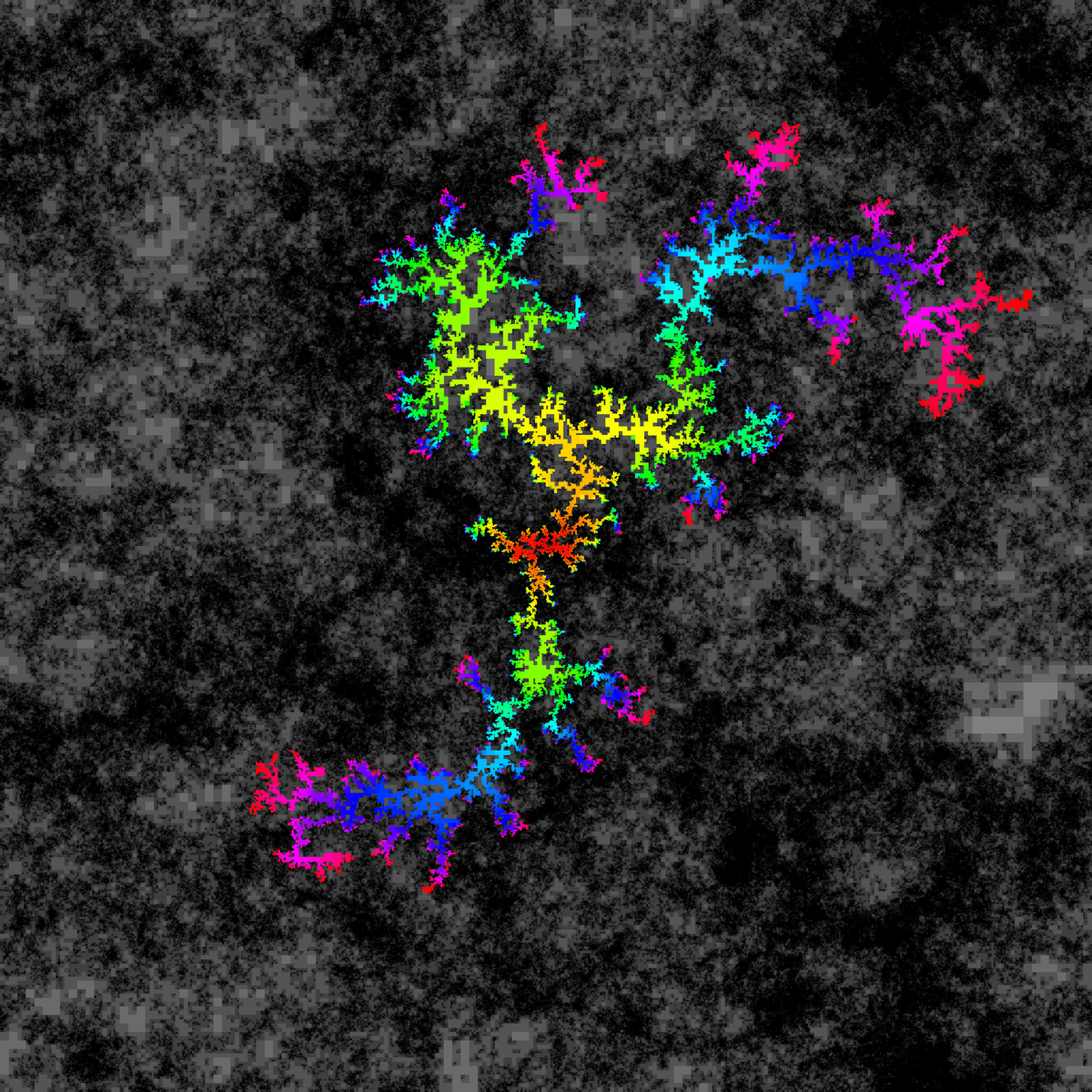}}

\subfloat[Time-parameterization.]{\includegraphics[width=0.47\textwidth]{figures/timescale.png}}
\end{center}
\vspace{-0.50cm}
\caption{\label{fig::large_qdla2}DLA on a $\sqrt{2}$-LQG generated with a $8192 \times 8192 = 2^{13} \times 2^{13}$ DGFF, with $\delta=2^{-26}$.  Time is parameterized by the ratio of the number of particles in the cluster over the number required for it to reach the concentric circle inside of the square and is indicated using the color scale shown above.  This simulation is a discrete analog of $\QLE(2,1)$.}
\end{figure}

\subsection{Background on several relevant models} \label{subsec::modelbackground}
\subsubsection{First passage percolation and Eden model} \label{subsubsection::eden}
The Eden growth model was introduced by Eden in 1961 \cite{eden1961two}.  One constructs a randomly growing sequence of clusters $C_n$ within a fixed graph $G = (V,E)$ as follows: $C_0$ consists of a single deterministic initial vertex $v_0$, and for each $n \in \N$, the cluster $C_n$ is obtained by adding one additional vertex to $C_{n-1}$.  To obtain this vertex, one samples uniformly from the set of edges that have exactly one endpoint in $C_{n-1}$, and adds the endpoint of this edge that does not lie in $C_{n-1}$.

First passage percolation (FPP) in turn was introduced by Hammersley and Welsh in 1965 \cite{hammersley1965first}.  We can construct a random metric on the vertices of the graph $G$ obtained by weighting all edges of $G$ with i.i.d.\ positive weights; the distance between any two vertices is defined to be the infimum, over all paths between them, of the weight sum along that path.  We can then let $C_t$ be the set of all vertices whose distance from an initial vertex $v_0$ is at most $t$.  If we think of the weight of an edge as representing the amount of time it takes a fluid to ``percolate across'' the edge, and we imagine that a fluid source is hooked up to a vertex $v_0$ at time $0$, then $C_t$ represents the set of vertices that ``get wet'' by time $t$.  It is instructive to think of $C_t$ as a growing sequence of balls in a random metric space obtained from the ordinary graph metric on $G$ via independent local perturbations.

For a discrete time parameterization of FPP, we can instead let $C_n$ be the set containing $v_0$ and the $n$ vertices that are closest to $v_0$ in this metric space.  When the law of the weight for each edge is that of an exponential random variable, it is not hard to see (using the ``memoryless'' property of exponential random variables) that the sequence of clusters $C_n$ obtained this way agrees in law with the sequence obtained in the Eden growth model.

An overall shape theorem was given by Cox and Durrett in \cite{MR624685} in 1981, which dealt with general first passage percolation on the square lattice and showed that under mild conditions on the weight distribution (which are satisfied in the case of exponential weights described above) the set $t^{-1} C_t$ converges to a deterministic shape (though not necessarily exactly a disk) in the limit.
Vahidi-Asl and Wierman proved an analogous result for first passage percolation on the Voronoi tesselation (and the related ``Delaunay triangulation'') later in the early $1990$'s \cite{MR1094141, MR1166620} and showed that in this case the limiting shape is actually a ball.

With a very quick glance at Figure~\ref{fig::triplegamma0.00}, one might guess that the limiting shape of the Eden model (whose existence is guaranteed by the Cox and Durrett theorem mentioned above) is circular; but early and subsequent computer experiments suggest that though the deterministic limit shape is ``roundish'' it is not exactly circular (e.g., \cite{freche1985surface, batchelor1991limits}).

The fluctuations of $t^{-1} C_t$ away from the boundary of the deterministic limit are of smaller order; with an appropriate rescaling, they are believed to have a scaling limit closely related to the KPZ equation introduced by Kardar, Parisi, and Zhang in 1986 \cite{kardar1986dynamic}.  Indeed, understanding growth models of this form was the original motivation for the KPZ equation \cite{kardar1986dynamic}, see Section~\ref{subsubsec::kardarparisizhang}.

\begin{figure}[ht!]
\begin{center}
\subfloat{\includegraphics[width=0.32\textwidth]{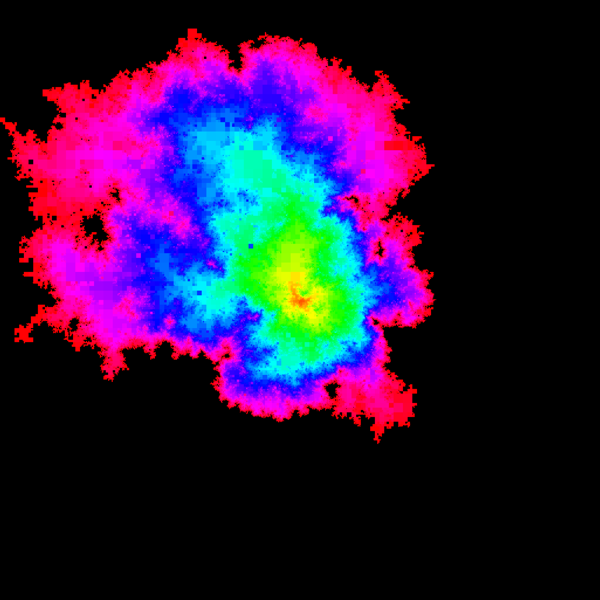}}
\hspace{0.01\textwidth}
\subfloat{\includegraphics[width=0.32\textwidth]{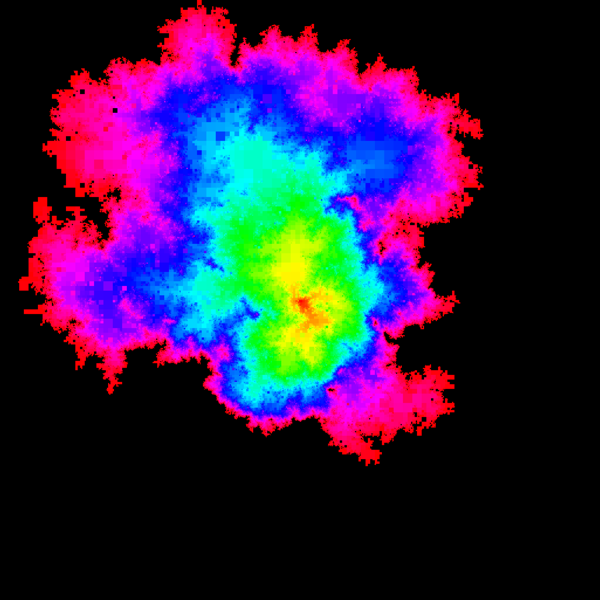}}
\hspace{0.01\textwidth}
\subfloat{\includegraphics[width=0.32\textwidth]{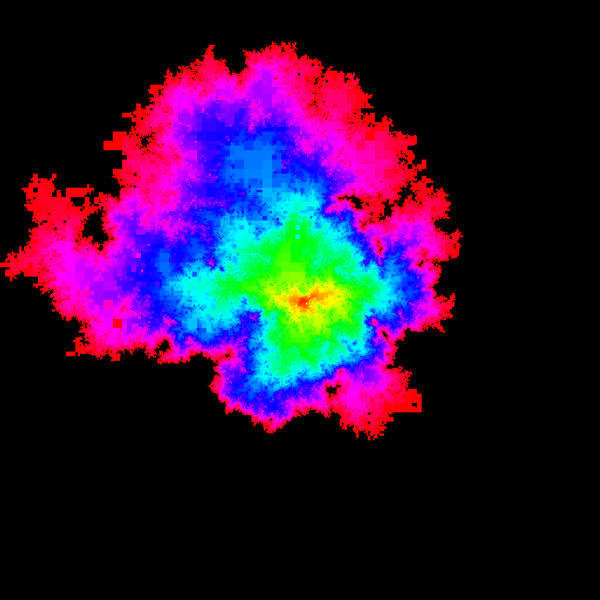}}

\subfloat{\includegraphics[width=0.32\textwidth]{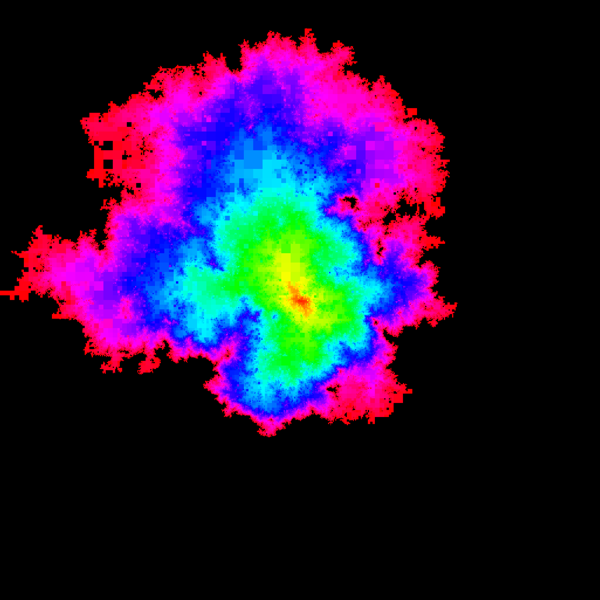}}
\hspace{0.01\textwidth}
\subfloat{\includegraphics[width=0.32\textwidth]{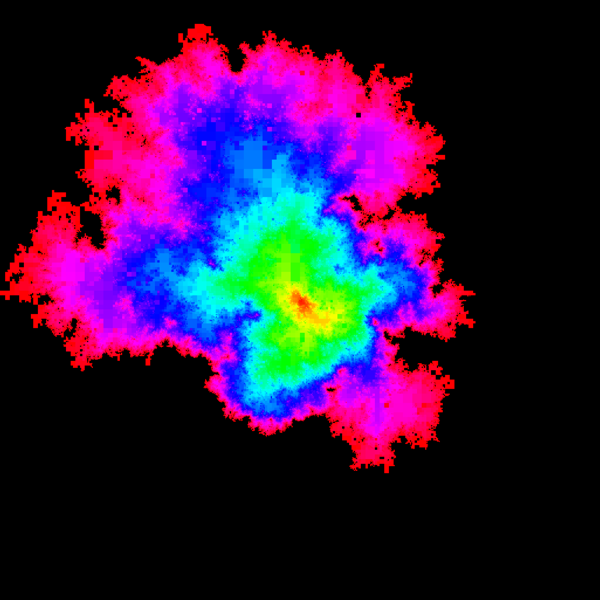}}
\hspace{0.01\textwidth}
\subfloat{\includegraphics[width=0.32\textwidth]{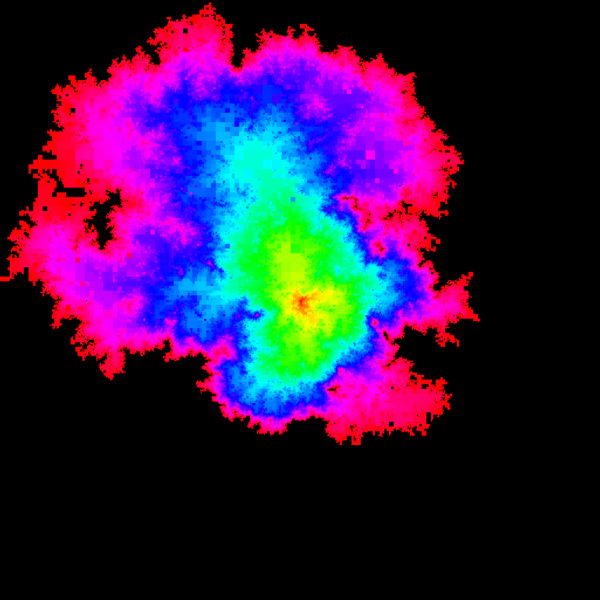}}

\subfloat{\includegraphics[width=0.32\textwidth]{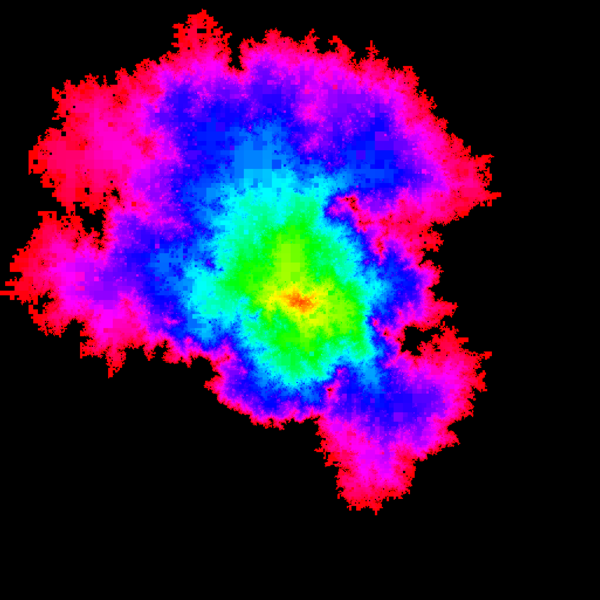}}
\hspace{0.01\textwidth}
\subfloat{\includegraphics[width=0.32\textwidth]{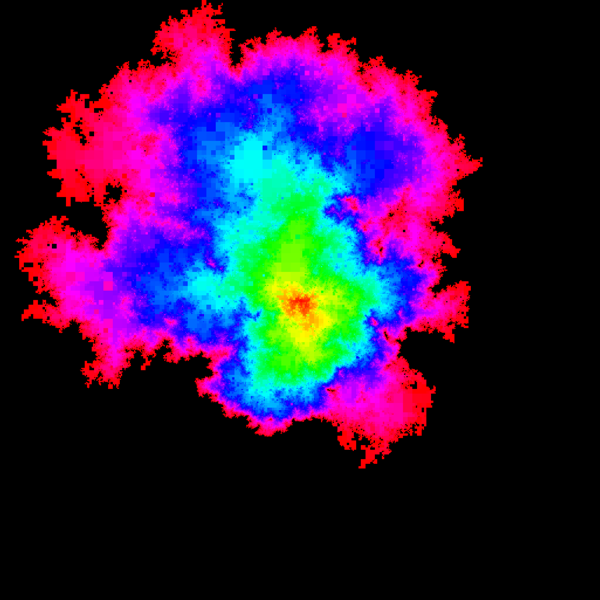}}
\hspace{0.01\textwidth}
\subfloat{\includegraphics[width=0.32\textwidth]{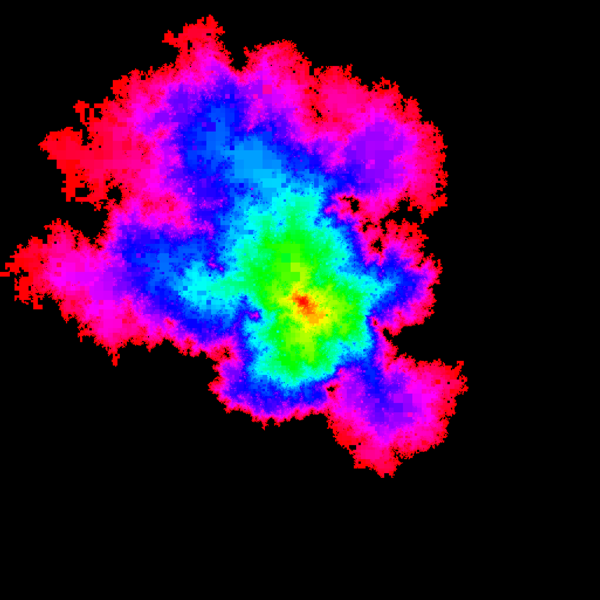}}
\end{center}
\caption{\label{fig::large_metric2_seeds} Different instances of the Eden model drawn on the square tiling shown in Figure~\ref{fig::large_metric2_squares}.  We expect that given an instance $h$ of the GFF (which determines the square decomposition for all $\delta$), it is almost surely the case that the shapes converge in probability to a limiting shape (depending only on $h$) as $\delta \to 0$.  The KPZ dynamics are conjecturally related to the fluctuations from the limit shape when $\gamma=0$ and $\delta$ tends to zero.  We do not have an analog of this conjecture for general $\gamma$.}
\end{figure}

\subsubsection{Diffusion limited aggregation (DLA)}
\label{subsubsec::DLAintro}

Diffusion limited aggregation (DLA) was introduced by Witten and Sander in 1981 and has been used to explain the especially irregular ``dendritic'' growth patterns found in coral, lichen, mineral deposits, and crystals \cite{witten1981diffusion, witten1983diffusion}.\footnote{Note that here (and throughout the remainder of this paper) we use the term DLA alone to refer to {\em external DLA}.  The so-called {\em internal DLA} is a growth process introduced by Meakin and Deutch in 1986 \cite{meakin1986formation} to explain the especially smooth growth/decay patterns associated with electropolishing, etching, and corrosion.  Internal DLA clusters grow spherically with very small (log order) fluctuations, much smaller than the fluctuations observed for the Eden model on a grid.  Although external DLA has had more attention in the physics literature, there has been much more mathematical progress on internal DLA, beginning with works by Diaconis and Fulton and by Lawler, Bramson, and Griffeath from the early 1990's \cite{MR1218674, lawler1992internal}.   More recently, the second author was part of an IDLA paper series with Levine and Jerison that describes the size and nature of internal DLA fluctuations in great detail and relates these fluctuations to a variant of the GFF \cite{MR2833484, 2010arXiv1012.3453J, 2011arXiv1101.0596J}, see also \cite{asselah2013logarithmic,asselah2013sublogarithmic}.}  Sander himself wrote a general overview of the subject in 2000 \cite{sander2000diffusion}; see also the review \cite{halsey2000diffusion}.

The most famous and substantial result about planar external DLA to date is Kesten's theorem \cite{kesten1987hitting}, which states that the diameter of the DLA cluster obtained after $n$ particles have been added almost surely grows asymptotically at most as fast as $n^{2/3}$.  Another way to say this is that by the time DLA reaches radius $n$ (for all $n$ sufficiently large), there are least $n^{3/2}$ particles in the cluster.  This seems to suggest (though it does not imply) that any scaling limit of DLA should have dimension at least $3/2$.\footnote{In his 2006 ICM paper, Schramm discussed the problem of understanding DLA on $\Z^2$ and wrote that Kesten's theorem ``appears to be essentially the only theorem concerning two-dimensional DLA,
though several very simplified variants of DLA have been successfully analysed'' \cite{MR2334202}.}  Although there is an enormous body of research on the behavior of DLA simulations, even the most basic questions about the scaling limit of DLA (such as whether the scaling limit is space-filling, or whether the scaling limit has dimension greater than $1$) remain unanswered mathematically.

The effects of lattice anisotropy on DLA growth also remain mysterious.  We mentioned above that limit shapes for FPP and Eden clusters need not be exactly round --- the anisotropy of the lattice can persist in the limit.  Intuitively, this makes sense: there is no particular reason, on a grid, to expect the rate of growth in the vertical direction to be exactly the same as the rate of growth in a diagonal direction.  In the case of DLA, effects of anisotropy can be rather subtle, and it is hard to detect anything anisotropic from a glance at a DLA cluster like the one in Figure~\ref{fig::triplegamma0.00}.  Nonetheless, simulations suggest that anisotropy may also affect scaling limits for DLA (perhaps by decreasing the overall scaling limit dimension from about 1.7 to about 1.6).  One recent overview of the scaling question (with many additional references) appears in \cite{menshutin2012scaling}, and effects of anisotropy are studied in \cite{menshutin2011morphological}.  There is some simulation-based evidence for universality among different isotropic ``off-lattice'' formulations of DLA (which involve differently-shaped dust particles performing Brownian motion until they attach themselves to a growing cluster) \cite{li2012diffusion}.  There is also some evidence that different types of isotropic models (such as DLA and the so-called viscous fingering) have common scaling limits \cite{mathiesen2006universality}.  Meakin proposed already in 1986 that off-lattice DLA and DLA on systems with five-fold or higher symmetry belong to one universality class, while DLA on systems with lower symmetry belong to one or more different universality classes \cite{meakin1986universality}.

In the DLA simulations generated in this paper, the square to add to a cluster is essentially chosen by running a Brownian motion from far away and choosing the first cluster-adjacent square the Brownian motion hits.  This is a little different from doing a simple random walk on the graph of squares started at a far away target vertex (and it was actually a little easier to code efficiently).  It is possible that our approach is somehow ``isotropic enough'' to ensure that the growth models in the simulations converge to a universal isotropic scaling limit as $\delta$ tends to zero, but we do not know how to prove this.  We stress that the QLE evolutions that we construct in this paper are rotationally invariant, and can thus only be scaling limits of growth models that have isotropic scaling limits.

\begin{figure}[ht!]
\begin{center}
\subfloat{\includegraphics[width=0.32\textwidth]{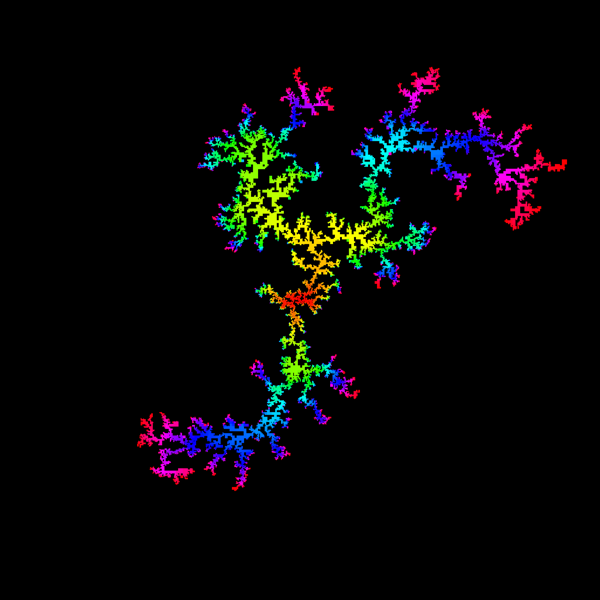}}
\hspace{0.01\textwidth}
\subfloat{\includegraphics[width=0.32\textwidth]{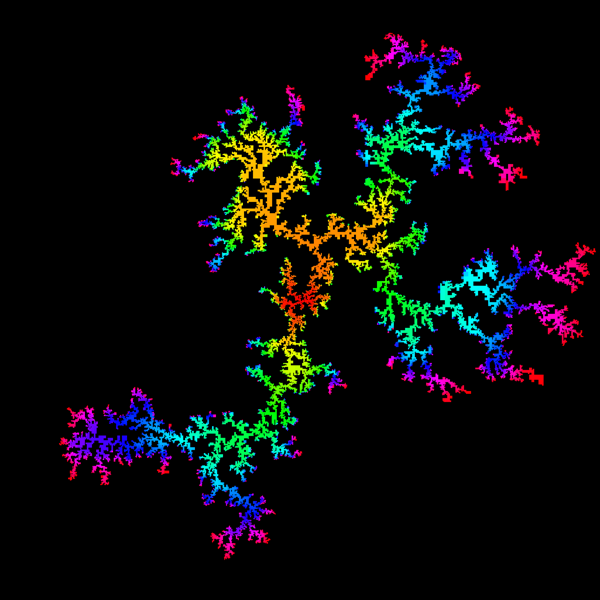}}
\hspace{0.01\textwidth}
\subfloat{\includegraphics[width=0.32\textwidth]{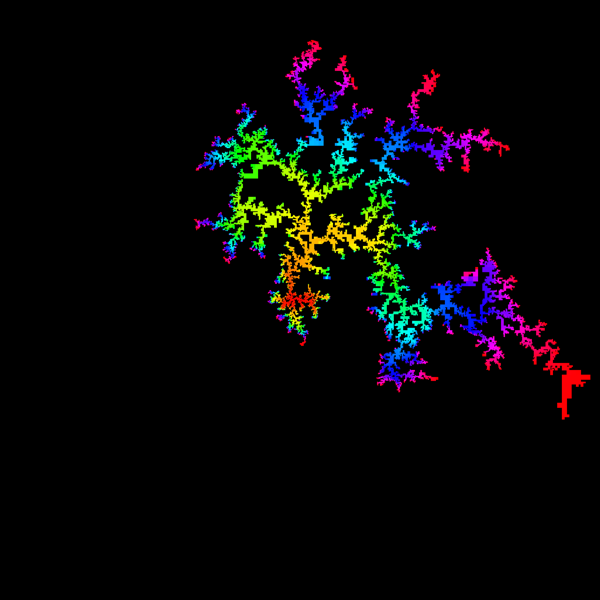}}

\subfloat{\includegraphics[width=0.32\textwidth]{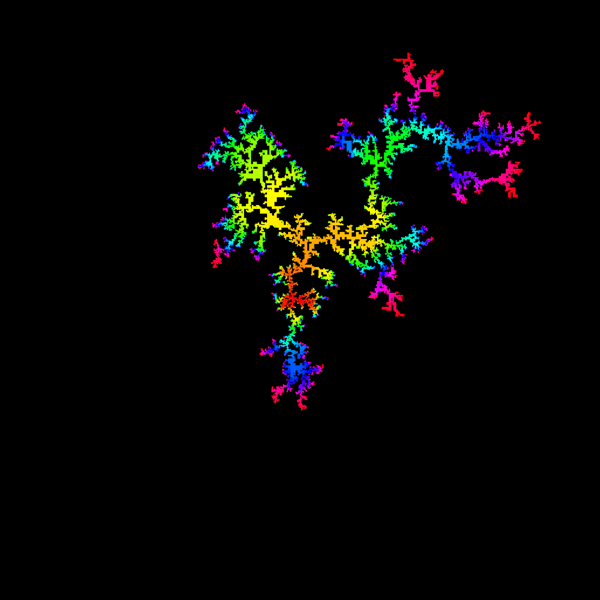}}
\hspace{0.01\textwidth}
\subfloat{\includegraphics[width=0.32\textwidth]{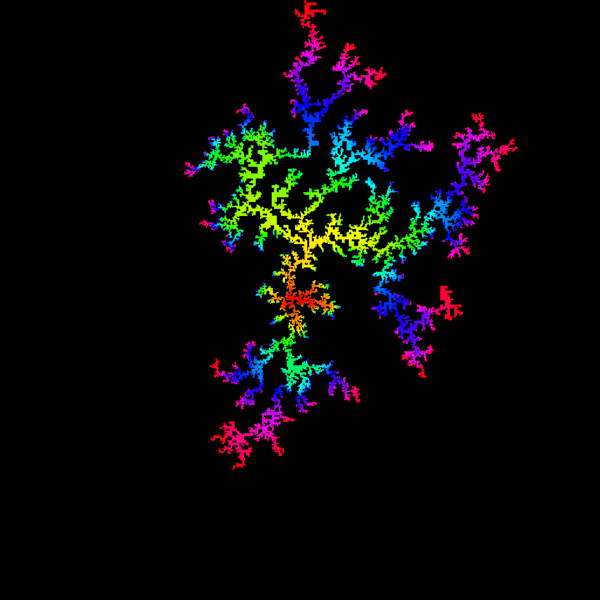}}
\hspace{0.01\textwidth}
\subfloat{\includegraphics[width=0.32\textwidth]{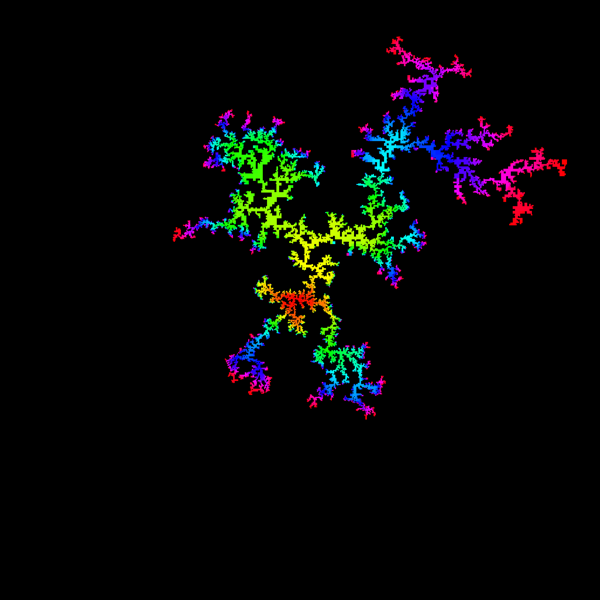}}

\subfloat{\includegraphics[width=0.32\textwidth]{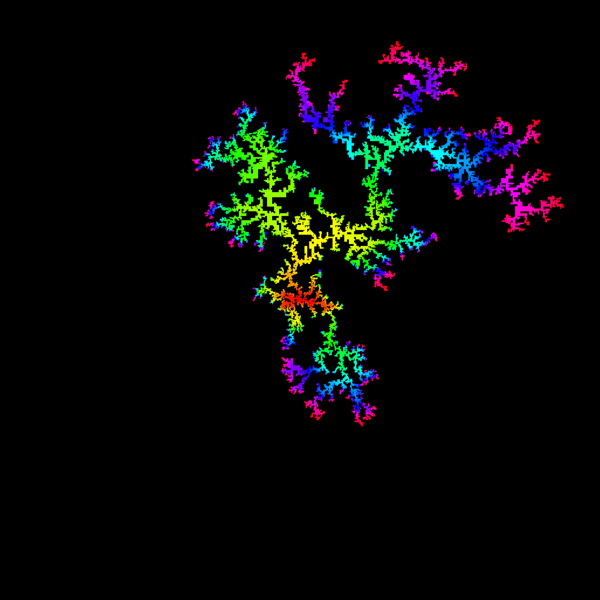}}
\hspace{0.01\textwidth}
\subfloat{\includegraphics[width=0.32\textwidth]{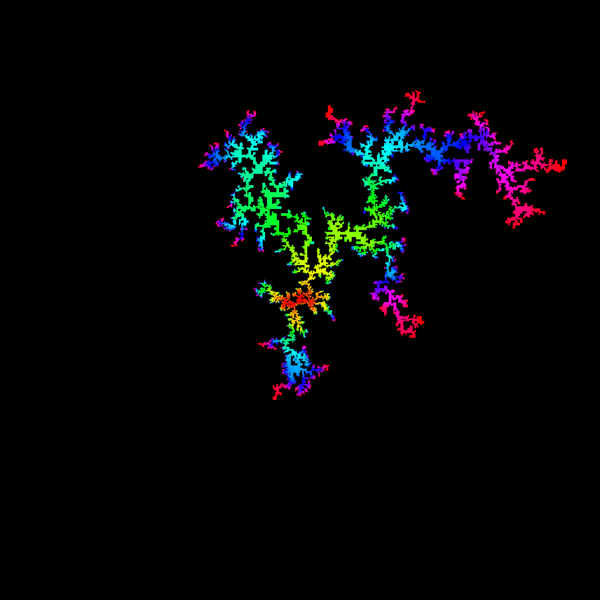}}
\hspace{0.01\textwidth}
\subfloat{\includegraphics[width=0.32\textwidth]{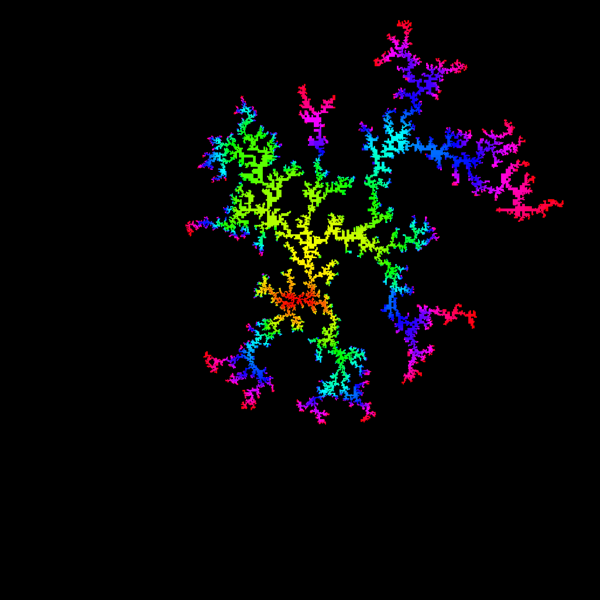}}
\end{center}
\caption{\label{fig::dlaseeds} Different instances of DLA on a common $\gamma = \sqrt{2}$ LQG tiling. Same $8192 \times 8192$ DGFF as in Figure~\ref{fig::large_qdla2} with the same value of $\delta$.  There are some macroscopic differences between the instances, but we do not know whether these differences will remain macroscopic in the limit as $\delta\to 0$.  Similarly, in our continuum formulation, we do not know whether $\QLE(2,1)$ is determined by the quantum surfaces on which it is drawn.}
\end{figure}

\subsubsection{Dielectric breakdown model and the Hastings-Levitov model}

As mentioned above, when FPP weights are exponential, the growth process selects new edges from counting measure on cluster-adjacent edges, i.e., according to the Eden model.
DLA is the same but with counting measure replaced by harmonic measure viewed from a special point (or from infinity).

Niemeyer, Pietronero, and Wiesmann introduced the dielectric breakdown model (DBM) in 1984 \cite{niemeyer1984fractal}.  Like SLE and LQG, it is a family of models indexed by a single real parameter, which in \cite{niemeyer1984fractal} is called $\eta$.
As noted in \cite{niemeyer1984fractal}, $\eta$-DBM can be understood as a hybrid between DLA and the Eden model.  If $\mu$ is counting measure on the harmonically exposed edges, and $\nu$ is harmonic measure, then the DBM model involves choosing a new edge from the measure $\mu$ weighted by $(\partial \nu/\partial \mu)^\eta$ (multiplied by a constant to produce a probability measure).  Equivalently, we can consider $\nu$ weighted by $(\partial \mu/\partial \nu)^{1-\eta}$, also multiplied by a normalizing constant to produce a probability measure.  Observe that $0$-DBM is then the Eden growth model, while $1$-DBM is external DLA.

The DBM models are believed to be related to the so-called $\alpha$-Hastings-Levitov model when $\alpha = \eta+1$\cite{hastings1998laplacian}.  (The $\alpha$ used in Hastings-Levitov is not the same as the $\alpha$ used in this paper describe QLE dynamics.)  The Hastings-Levitov model is constructed in the continuum using Loewner chains (rather than on a lattice).  It was introduced by Hastings and Levitov in 1998 as a plausible and simpler alternative to DLA and DBM, with the expectation that it would agree with these other models in the scaling limit but that it might be simpler to analyze \cite{hastings1998laplacian}.  In the Hastings-Levitov model one always samples the location of a new particle from harmonic measure, but the size of the new particle varies as the $\alpha$ power of the derivative of the normalizing conformal map at the location where the point is added.  This model itself is now the subject of a sizable literature that we will not attempt to properly survey here.  See for example works of Carleson and Makarov \cite{carleson2001aggregation} (obtaining growth bounds analogous to Kesten's bound for DLA), Rohde and Zinsmeister \cite{rohde2005some} (analyzing scaling limit dimension and other properties for various $\alpha \in [0,2]$, discussing the possibility of an $\alpha=1$ phase transition from smooth to turbulent growth), Norris and Turner \cite{norris2012hl} (proof of convergence in the $\alpha=0$ case to a growing disk and a connection to the Brownian web), and the reference text \cite{gustafsson2006conformal}.  In our terminology, the scaling limit of the $\alpha$-Hastings-Levitov model should correspond to $\QLE(0, \alpha-1)$, and the $\alpha \in [0,2]$ family studied in \cite{rohde2005some} should correspond to the points in Figure~\ref{fig::etavsgamma} along the vertical axis with $\eta \in [-1,1]$.

\subsubsection{Gaussian free field}

The Gaussian free field (GFF) is a Gaussian random distribution on a planar domain $D$, which can be interpreted as a standard Gaussian in the Hilbert space described by the so-called Dirichlet inner product.  It has free and fixed boundary analogs, as well as discrete variants defined on a grid; see the GFF survey \cite{sheffield2007gff}.  We defer a more detailed discussion of the GFF until Section~\ref{sec::preliminaries}

\subsubsection{Liouville quantum gravity} \label{subsubsec::LQGintro}

Liouville quantum gravity, introduced in the physics literature by Polyakov in 1981 in the context of string theory, is a canonical model of a random two-dimensional Riemannian manifold \cite{MR623209, MR623210}.  One version of this construction involves replacing the usual Lebesgue measure $dz$ on a smooth domain $D$ with a random measure $\mu_h = e^{\gamma h(z)}dz$, where $\gamma \in [0,2]$ is a fixed constant and $h$ is an instance of (for now) the free boundary GFF on $D$ (with an additive constant somehow fixed).  Since $h$ is not defined as a function on $D$, one has to use a regularization procedure to be precise.  Namely, one defines $h_\epsilon(z)$ to be the mean value of $h$ on the circle $\partial B(z,\epsilon)$, and takes the measure $\mu$ to be the weak limit of the measures
\[ \epsilon^{\gamma^2/2} e^{\gamma h_\epsilon(z)} dz\]
as $\epsilon$ tends to zero  \cite{ds2011kpz}. On a linear segment of $\partial D$, a boundary measure $\nu_h$ on $\partial D$ can be similarly defined as
\[ \lim_{\epsilon \to 0} \epsilon^{\gamma^2/4} e^{(\gamma/2) h_\epsilon(u)} du,\]
where in this case $h_\epsilon$ is the mean of $h$ on the semicircle $D \cap \partial B(u,\epsilon)$ \cite{ds2011kpz}.  (A slightly different procedure is needed to construct the measure in the critical case $\gamma = 2$ \cite{duplantier2012critical,duplantier2012renormalization}.)

We could also parameterize the same surface with a different domain $\wt D$.
Suppose $\psi \colon \wt D \to D$ is a conformal map.  Write $\wt h$ for the distribution on $\wt D$ given by $h \circ \psi + Q \log |\psi'|$
where $Q := \frac{2}{\gamma} + \frac{\gamma}{2}.$
Then it is shown in \cite{ds2011kpz} that $\mu_h$ is almost surely the image under $\psi$ of the measure $\mu_{\wt h}$. That is, $\mu_{\wt h}(A) = \mu_h(\psi(A))$ for $A \subseteq \wt{D}$.\footnote{The reader can also verify this fact directly; the first term in $Q$ is related to the ordinary change of measure formula, since the term $\frac{2}{\gamma} \log |\psi'|$ in $\wt h$ corresponds to a factor of $|\psi'|^2$ in the $\mu_{\wt h}$ definition.  The term $\frac{\gamma}{2} \log |\psi'|$ compensates for the rescaling of the $\epsilon$ that appears in the definition of $\mu_{\wt h}$.}
A similar argument to the one in \cite{ds2011kpz} mentioned above shows that the boundary length $\nu_h$ is almost surely the image under $\psi$ of the measure $\nu_{\wt h}$.  (This also allows us to make sense of $\nu_h$ on domains with non-linear boundary.)mc

We define a {\bf quantum surface} to be an equivalence class of pairs $(D,h)$ under the equivalence transformations
\begin{equation} \label{eqn::LQGcoordinatetransformation}(D,h) \to (\psi^{-1}(D), h \circ \psi + Q \log |\psi'|) = (\wt D, \wt h).\end{equation}
The measures $\mu_h$ and $\nu_h$ are almost surely highly singular objects with fractal structure, and thus we cannot understand LQG random surfaces as smooth manifolds.  Nonetheless quantum surfaces come equipped with well-defined notions of conformal structure, area, and boundary length.

\begin{figure}[ht!]
\begin{center}
\subfloat{\includegraphics[width=0.32\textwidth]{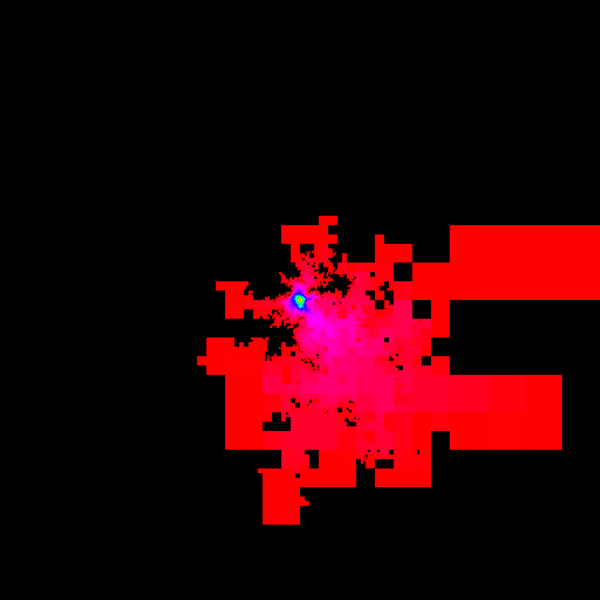}}
\hspace{0.01\textwidth}
\subfloat{\includegraphics[width=0.32\textwidth]{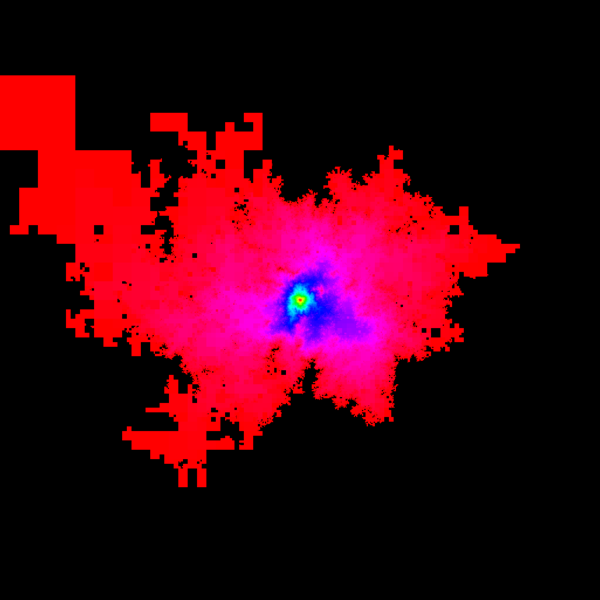}}
\hspace{0.01\textwidth}
\subfloat{\includegraphics[width=0.32\textwidth]{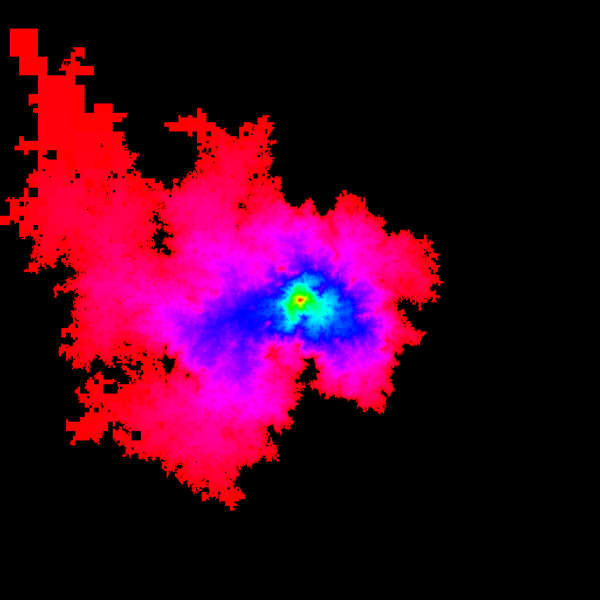}}

\subfloat{\includegraphics[width=0.32\textwidth]{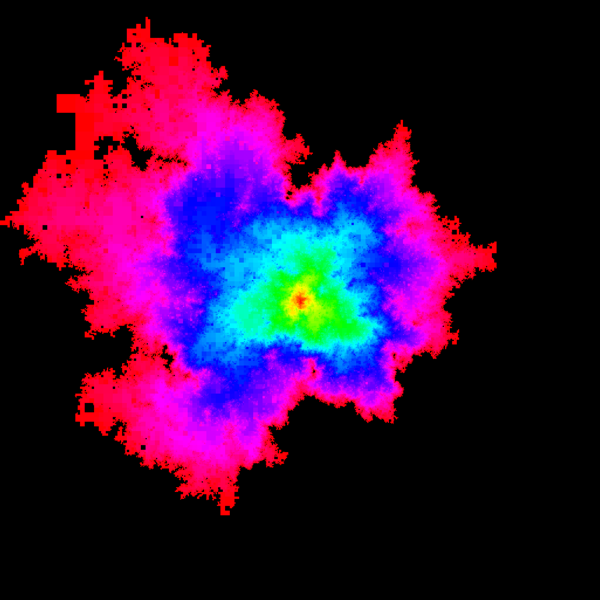}}
\hspace{0.01\textwidth}
\subfloat{\includegraphics[width=0.32\textwidth]{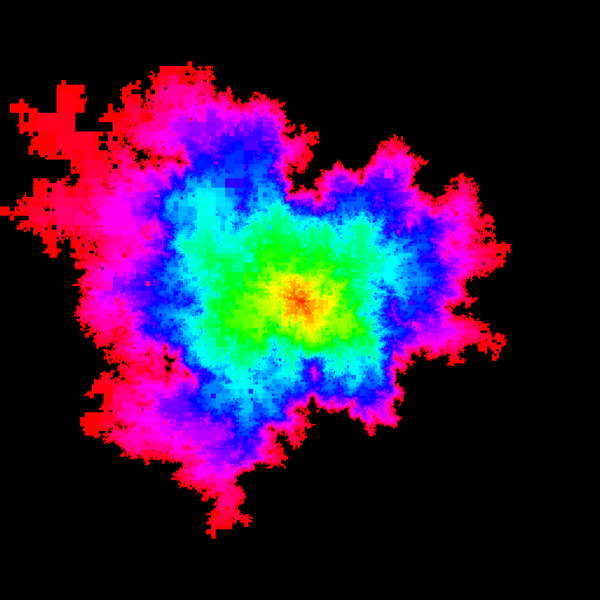}}
\hspace{0.01\textwidth}
\subfloat{\includegraphics[width=0.32\textwidth]{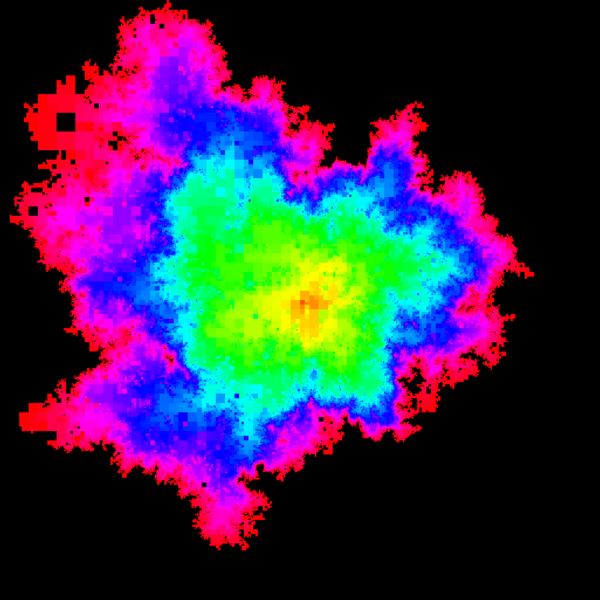}}

\subfloat{\includegraphics[width=0.32\textwidth]{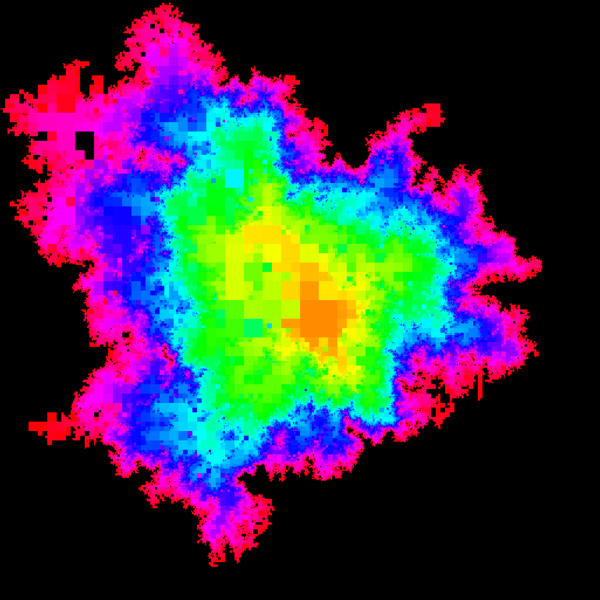}}
\hspace{0.01\textwidth}
\subfloat{\includegraphics[width=0.32\textwidth]{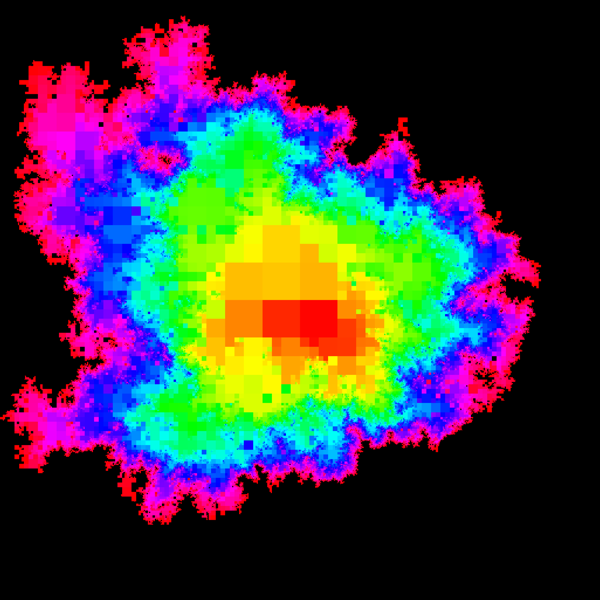}}
\hspace{0.01\textwidth}
\subfloat{\includegraphics[width=0.32\textwidth]{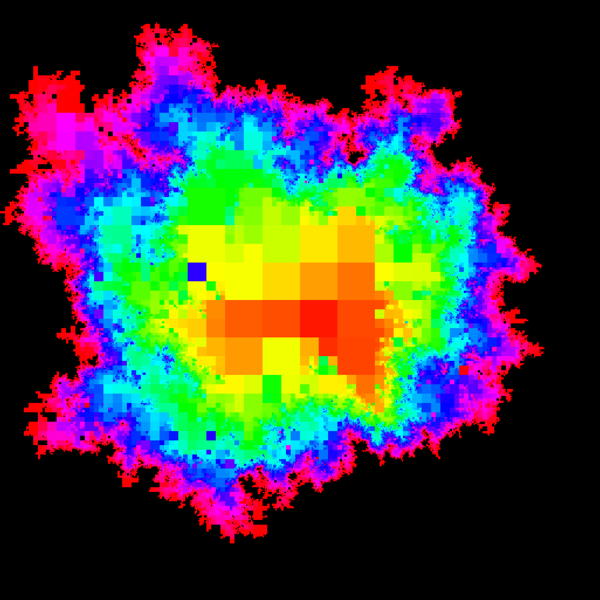}}
\end{center}
\caption{ \label{fig::edenwithlogs} $\gamma = \sqrt{8/3}$ Eden model on graph obtained when $h$ is the GFF plus $ j \log |\cdot|$, where $j \in \{-4,-3,-2, \ldots, 2, 3, 4 \}$ (read left to right, top to bottom).  Upper left figure has smaller squares in center, bigger squares on outside. Bottom right has bigger boxes in center, smaller boxes outside.}
\end{figure}

\begin{figure}[ht!]
\begin{center}
\subfloat{\includegraphics[width=0.32\textwidth]{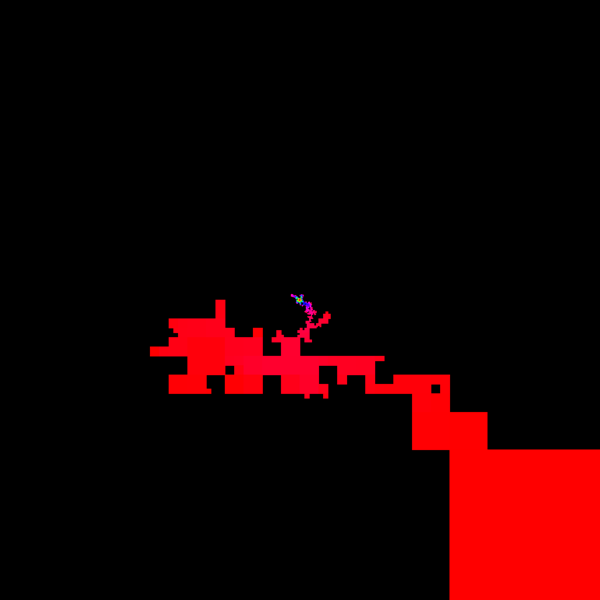}}
\hspace{0.01\textwidth}
\subfloat{\includegraphics[width=0.32\textwidth]{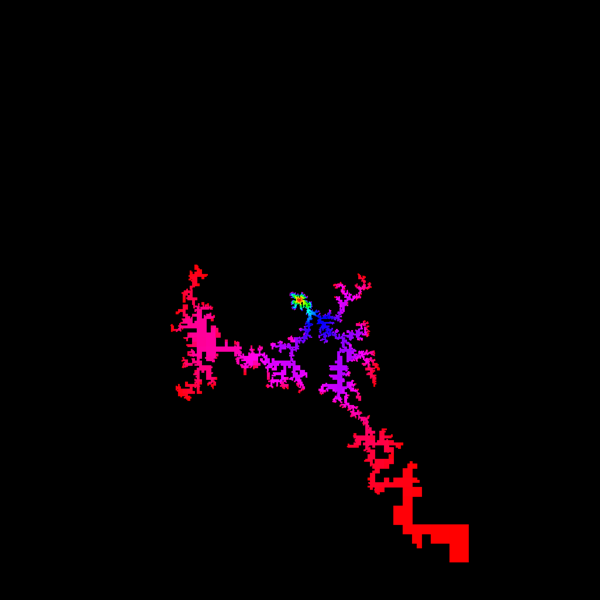}}
\hspace{0.01\textwidth}
\subfloat{\includegraphics[width=0.32\textwidth]{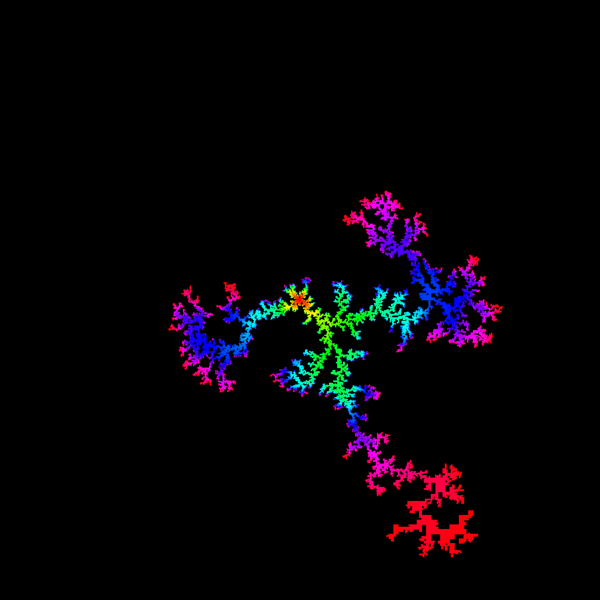}}

\subfloat{\includegraphics[width=0.32\textwidth]{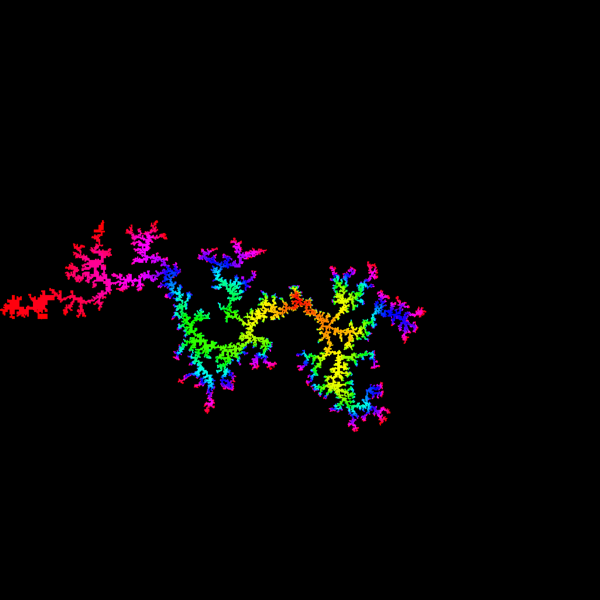}}
\hspace{0.01\textwidth}
\subfloat{\includegraphics[width=0.32\textwidth]{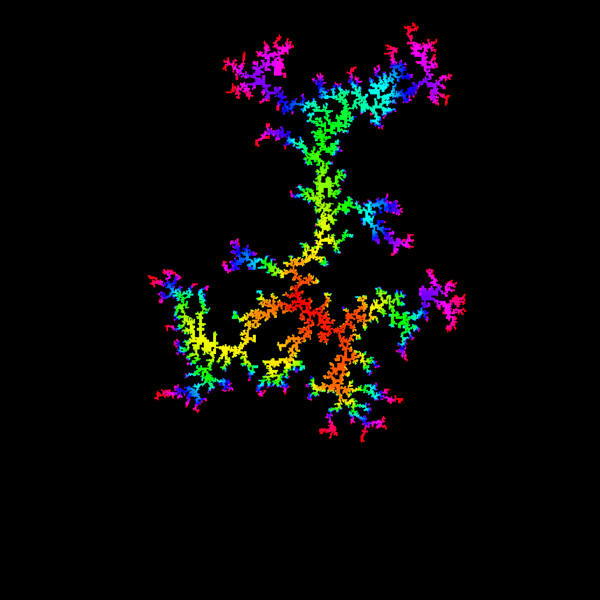}}
\hspace{0.01\textwidth}
\subfloat{\includegraphics[width=0.32\textwidth]{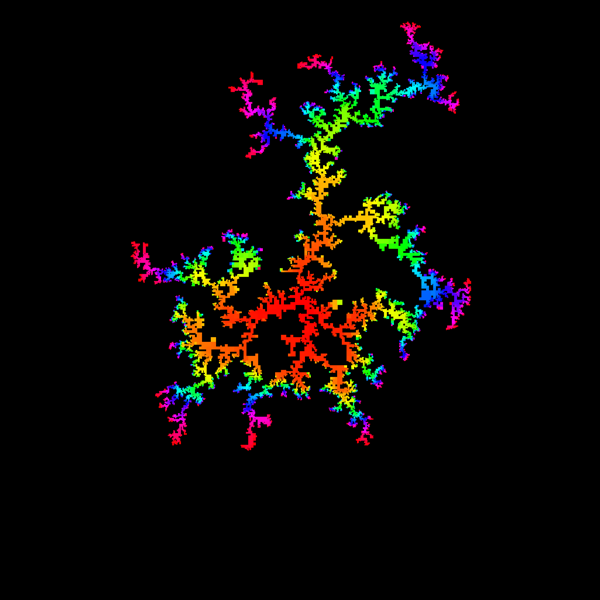}}

\subfloat{\includegraphics[width=0.32\textwidth]{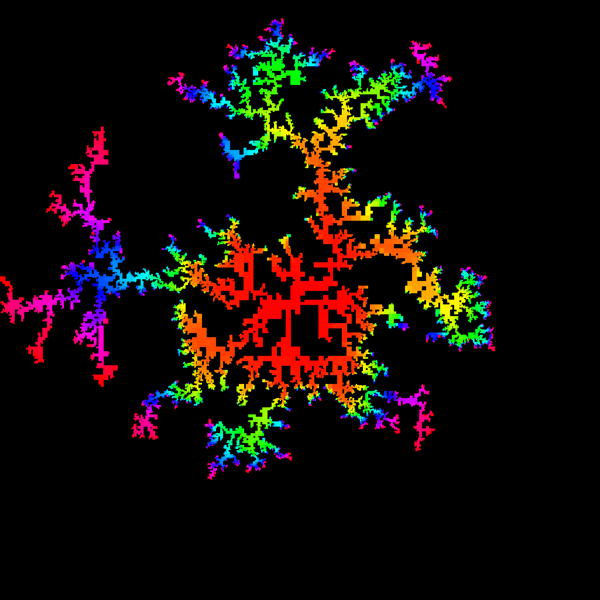}}
\hspace{0.01\textwidth}
\subfloat{\includegraphics[width=0.32\textwidth]{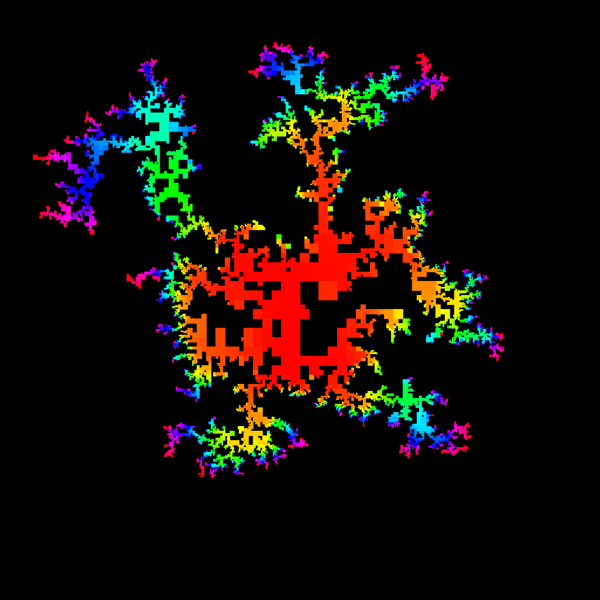}}
\hspace{0.01\textwidth}
\subfloat{\includegraphics[width=0.32\textwidth]{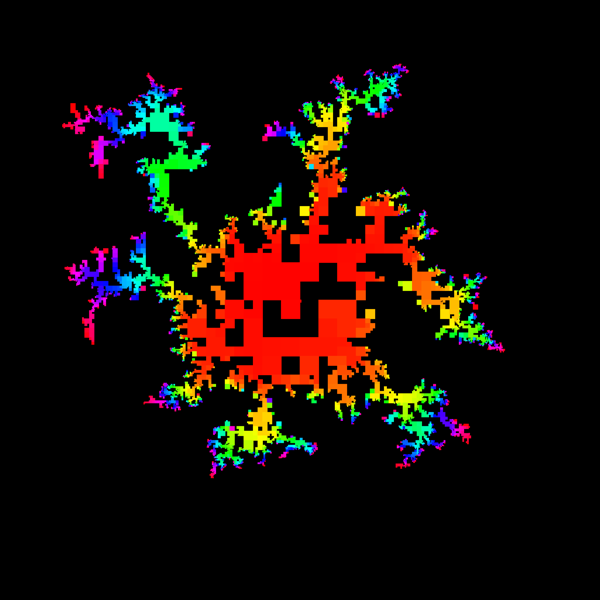}}
\end{center}
\caption{\label{fig::dlawithlogs} $\gamma = \sqrt{2}$ DLA drawn on graph obtained when $h$ is the GFF plus $ j \log |\cdot|$, where $j \in \{-4,-3,-2, \ldots, 2, 3, 4 \}$ (read left to right, top to bottom).  Upper left figure has smaller squares in center, bigger squares on outside. Bottom right has bigger boxes in center, smaller boxes outside.}
\end{figure}

\subsubsection{Random planar maps} \label{subsubsec::randomplanarmaps}

The number of planar maps with a fixed number of vertices and edges is finite, and there is an extensive literature on the enumeration of planar maps, beginning with the works of Mullin and Tutte in the 1960's \cite{tutte1962census,MR0205882,MR0218276}.  On the physics side, various types of random planar maps were studied in great detail throughout the 1980's and 1990's, in part because of their interpretation as ``discretized random surfaces.''  (See \cite{ds2011kpz} for a more extensive bibliography on planar maps and Liouville quantum gravity.)  The {\em metric space} theory of random quadrangulations begins with an influential bijection discovered by Schaeffer \cite{MR1465581}, and earlier by Cori and Vauquelin \cite{cori1981planar}.
Closely related bijections of Bouttier, Di Franceso, and Guitter \cite{bouttier2004planar} deal with planar maps with face size restrictions, including triangulations.  Subsequent works by Angel \cite{MR2024412} and by Angel and Schramm \cite{MR2013797} have explained the {\em uniform infinite planar triangulation} (UITP) as a subsequential limit of planar triangulations.

Although microscopic combinatorial details differ, there is one really key idea that underlies much of the combinatorial work in this subject: namely, that instead of considering a planar map alone, one can consider a planar map together with a spanning tree.  Given the spanning tree, one often has a notion of a dual spanning tree, and a path that somehow goes between the spanning tree and the dual spanning tree.  It is natural to fix a root vertex for the dual tree and an adjacent root vertex for the tree.  Then as one traverses the path, one can keep track of a pair of parameters in $\Z_+^2$: one's distance from a root vertex within the tree, and one's distance from the dual root within the dual tree.  Mullin in 1967 used essentially this construction to give a way of enumerating the pairs $(M,T)$ where $M$ is a rooted planar map on the sphere with $n$ edges and $T$ is distinguished spanning tree \cite{MR0205882}.  These pairs correspond precisely to walks of length $2n$ in $\Z_+^2$ that start and end at the origin.  (The bijection between tree-decorated maps and walks in $\Z_+^2$ was more explicitly explained by Bernardi in \cite{MR2285813}; see also the presentation and discussion in \cite{sheffield2011quantum}, as well as the brief overview in Section~\ref{subsubsec::DLALERW}.)  As $n$ tends to infinity and one rescales appropriately, one gets a Brownian excursion on $\R^2$ starting and ending at $0$.

The Mullin bijection gives a way of choosing a uniformly random $(M,T)$ pair, and if we ignore $T$, then it gives us a way to choose a random $M$ where the probability of a given $M$ is proportional to the number of spanning trees that $M$ admits.  If instead we had a way to choose randomly from a {\em subset} $\mathcal S$ of the set of pairs $(M,T)$, with the property that each $M$ belonged to at most {\em one} pair $(M,T) \in \mathcal S$, then this would give us a way to sample {\em uniformly} from some collection of maps $M$.  The Cori-Vauquelin-Schaeffer construction \cite{cori1981planar, MR1465581} suggests a way to do this: in this construction, $M$ is required to be a quadrangulation, and a ``tree and dual tree'' pair on $M$ are produced from $M$ in a deterministic way.  (The precise construction is simple but a bit more complicated than the Mullin bijection.  One of the trees is a breadth first search tree of $M$ consisting of geodesics, and the other is something like a dual tree defined on the same vertices, but with some edges that cross the quadrilaterals diagonally and some edges that overlap the tree edges.)  As one traces the boundary of the dual tree, the distance from the root in the dual tree changes by $\pm 1$ at each step, while the distance in the geodesic tree changes by either $0$ or $\pm 1$.  Schaeffer and Chassaing showed that this distance function should scale to a two-dimensional continuum random path called the Brownian snake, in which the first coordinate is a Brownian motion (and the second coordinate comes from a Brownian motion indexed by the continuum random tree defined by the first Brownian motion) \cite{MR2031225}.

Another variant due to the second author appears in \cite{sheffield2011quantum}, where the trees are taken to be the exploration trees associated with a random planar map together with a random FK random cluster on top of it.  In fact, the construction in \cite{sheffield2011quantum}, described in terms of a ``hamburgers and cheeseburgers'' inventory management process, is a generalization of the work of Mullin \cite{MR0205882}.  We stress that the walks on $\Z_+^2$ that one finds in both \cite{MR0205882} and \cite{sheffield2011quantum} have as scaling limits forms of two-dimensional Brownian motion (in \cite{sheffield2011quantum} the diffusion rate of the Brownian motion varies depending on the FK parameter), unlike the walks on $\Z_+^2$ given in \cite{MR1465581, MR2031225} (which scale to the Brownian snake described above).

\subsubsection{The Brownian map}

The Brownian map is a random metric space equipped with an area measure.  It can be constructed from the Brownian snake, and is believed to be in some sense equivalent to a form of Liouville quantum gravity when $\gamma = \sqrt{8/3}$.  The idea of the Brownian map construction has its roots in the combinatorial works of Schaeffer and of Chassaing and Schaeffer \cite{MR1465581,MR2031225}, as discussed just above in Section~\ref{subsubsec::randomplanarmaps}, where it was shown that certain types of random planar maps could be described by a random tree together with a random labeling that determines a dual tree, and that this construction is closely related to the Brownian snake.

The Brownian map was introduced in works by Marckert and Mokkadem and by Le Gall and Paulin \cite{marckert2006limit,le2008scaling}.  For a few years, the term ``Brownian map'' was often used to refer to any one of the subsequential Gromov-Hausdorff scaling limits of random planar maps.  Because it was not known whether the limit was unique, the phrase ``a Brownian map'' was sometimes used in place of ``the Brownian map''.
Works by Le Gall and by Miermont established the uniqueness of this limit and showed that it is equivalent to a natural metric space constructed directly from the Brownian snake \cite{le2010geodesics,miermont2011brownian,gall2011uniqueness}.  Infinite planar quadrangulations on the half plane or the plane and the associated infinite volume Brownian maps are discussed in \cite{curien2012uniform,curien2012brownian}.

\subsubsection{KPZ: Kardar-Parisi-Zhang} \label{subsubsec::kardarparisizhang}

As mentioned briefly in Section~\ref{subsubsection::eden}, Kardar, Parisi, and Zhang introduced a formal stochastic partial differential equation in 1986 in order to describe the fluctuations from the deterministic limit shape that one finds in the Eden model on a grid (as in Figure~\ref{fig::triplegamma0.00}) or in related models such as first passage percolation \cite{kardar1986dynamic}.  As described in \cite{kardar1986dynamic}, the equation is a type of ill-posed stochastic partial differential equation,  but one can interpret the log of the stochastic heat equation with multiplicative noise as in some sense solving this equation (this is called the Hopf-Cole solution). The Eden model fluctuations are believed to scale to a ``fixed point'' of the dynamics defined this way; see the discussion by Corwin and Quastel in \cite{corwin2011renormalization}, as well as the survey article \cite{MR2930377}.  Other recent discussions of this point include e.g.\ \cite{cieplak1996invasion, alves2011universal}.

One interesting question for us is what the analog of the KPZ growth equation should be for the random graphs described in this paper.  Figure~\ref{fig::large_metric2_squares} shows different instances of the Eden model drawn on the square tiling shown in Figure~\ref{fig::large_metric2_squares}.  Although they appear to be roughly the same shape, there are clearly random fluctuations and at present we do not have a way to predict the behavior or even the magnitude of these fluctuations (though we would guess that the magnitude decays like some power of $\delta$).

\subsubsection{KPZ: Knizhnik-Polyakov-Zamolodchikov}

A natural question is whether discrete models for random surfaces (built combinatorially by randomly gluing together small squares or triangles) have Liouville quantum gravity as a scaling limit.  Polyakov became convinced in the affirmative in the 1980's after jointly deriving, with Knizhnik and Zamolodchikov, the so-called {\em KPZ formula} for certain Liouville quantum gravity scaling dimensions and comparing them with known combinatorial results for the discrete models \cite{knizhnik1988fractal,polyakov2008quarks}.  Several precise conjectures along these lines appear in \cite{ds2011kpz, sheffield2010weld} and the KPZ formula was recently formulated and proved mathematically in \cite{ds2011kpz}; see also \cite{sheffdup2009prl}.

To describe what the KPZ formula says, suppose that a constant $\gamma \in [0,2]$, a fractal planar set $X$, and an instance $h$ of the GFF are all fixed.  The set $X$ can be either deterministic or random, as long as it is chosen independently from $h$.  Then for any $\delta$ one can generate a square decomposition of the type shown in Figure~\ref{fig::squaredecomposition} and ask whether the expected number of squares intersecting $X$ scales like a power of $\delta$.  One form of the KPZ statement proved in \cite{ds2011kpz} is that if the answer is yes when $\gamma = 0$ (the Euclidean case) then the answer is also yes for any fixed (positive) $\gamma \in (0,2)$, and the Euclidean and quantum exponents satisfy a particular quadratic relationship (depending on $\gamma$).  Formulations of this statement in terms of Hausdorff dimension (and a quantum-surface analog of Hausdorff dimension) in one and higher dimensions appear respectively in \cite{benjamini2009kpz, rhodes2011kpz}; see also \cite{duplantier2012critical,duplantier2012renormalization} for the case $\gamma=2$.

One important thing to recognize for this paper is that the KPZ formula only applies when $X$ and $h$ are chosen independently of one another.  This independence assumption is natural in many contexts---for example, one sometimes expects the scaling limit of a random planar map decorated with a path (associated to some statistical physics model) to be an LQG surface decorated with an SLE-curve that is in fact independent of the field $h$ describing the LQG surface \cite{sheffield2010weld,duplantier2011schramm}.
However, we do not expect the Euclidean and quantum dimensions of the QLE traces constructed in this paper to be related by the KPZ formula, because these random sets are not independent of the GFF instance $h$.

\subsubsection{Schramm Loewner evolution}

$\SLE_\kappa$ ($\kappa > 0$) is a one-parameter family of conformally invariant random curves, introduced by Oded Schramm in \cite{schramm2000sle} as a candidate for (and later proved to be) the scaling limit of loop erased random walk \cite{lsw2004lerw} and the interfaces in critical percolation \cite{smirnov2001percolation, cn2006percolation}.  Schramm's curves have been shown so far also to arise as the scaling limit of the macroscopic interfaces in several other models from statistical physics: \cite{smirnov2010ising,cs2012ising,ss2005harmonicexplorer,ss2009dgffcontour,miller2010sle}.  More detailed introductions to $\SLE$ can be found in many excellent survey articles of the subject, e.g., \cite{werner2004stflour, lawler2005conformally}.

Given a simply connected planar domain $D$ with boundary points $a$ and $b$ and a parameter $\kappa \in [0,\infty)$, the {\em chordal} {\bf Schramm-Loewner evolution} SLE$_\kappa$ is
a random non-self-crossing path in $\overline D$ from $a$ to $b$.  In this work, we will be particularly concerned with the so-called {\em radial} $\SLE_\kappa$, which is a random non-self-crossing path from a fixed point on $\partial D$ to a fixed interior point in $D$.  Like chordal SLE, it is completely determined by certain conformal symmetries \cite{schramm2000sle}.

The construction of SLE is rather interesting.  When $D = \D$ is the unit disk, the radial SLE curve can be parameterized by a function $U \colon [0, \infty) \to \partial \D$.  However, instead of constructing the curve directly, one constructs for each $t$ the conformal map $g_t \colon \D_t \to \D$, where $\D_t$ is the complementary component\footnote{Here ``complementary component of'' means ``component of the complement of''.} of the curve drawn up to time $t$ which contains $0$, with $g_t(0) = 0$ and $g_t'(0) > 0$.  For $u \in \partial \D$ and $z \in \D$, let
\begin{equation}
\label{eqn::radial_functions_intro}
\Psi(u,z) = \frac{u+z}{u-z} \quad \text{and} \quad\Phi(u,z)  = z \Psi(u,z).
\end{equation}
For each fixed $z$, the value $g_t(z)$ is defined as the solution to the ODE
\begin{equation}
\label{eqn::radialloewner}
\dot{g}_t(z) = \Phi(U_t,g_t(z)),
\end{equation}
where $U_t = e^{i\sqrt{\kappa} B_t}$ and $B_t$ is a standard Brownian motion.  More introductory material about $\SLE$ appears in Section~\ref{subsec::sle}.

$\SLE$ is relevant to this paper primarily because of its relevance to Liouville quantum gravity and the so-called {\em quantum gravity zipper} described by the second author in \cite{sheffield2010weld}.  Roughly speaking, the constructions there allow one to form one LQG surface by ``cutting'' another LQG surface along an $\SLE$ path.  In fact, one can do this in such a way that the new (cut) surface has the same law as the original (uncut) surface.  This will turn out to be extremely convenient as we construct and study the quantum Loewner evolution.

\subsection{Measure-driven Loewner evolution}\label{subsec::measureloewner}

We consider an analog of Loewner evolution, also called the {\em Loewner-Kufarev evolution}, in which the point-valued driving function is replaced by a measure-valued driving function:
\begin{equation}
\label{eqn::lke}
\dot{g}_t(z) =  \int_{\partial \D} \Phi(u,g_t(z))d\nu_t(u),
\end{equation}
(recall \eqref{eqn::radial_functions_intro}) where, for each time $t$, the measure $\nu_t$ is a probability measure on $\partial \D$.  For each time $t$, the map $g_t$ is the unique conformal map from $\D \setminus K_t$ to $\D$ with $g_t(0) = 0$ and $g_t'(0) > 0$, for some hull $K_t$.  Time is parameterized so that $g_t'(0) = e^t$ (this is the reason that $\nu_t$ is normalized to be a probability measure).  That is, the log conformal radius of $\D \setminus K_t$, viewed from the origin, is given by $-t$.  Given any measure $\nu$ on $[0,T] \times \partial \D$ whose first coordinate is given by Lebesgue measure, we can define $\nu_t$ to be the conditional measure obtained on $\partial \D$ by restricting the first coordinate to $t$.

Unlike the space of point-valued driving functions indexed by $[0,T]$, the space of measure-valued driving functions indexed by $[0,T]$ has a natural topology with respect to which it is compact: namely the topology of weak convergence of measures on $[0,T] \times \partial \D$.

We now recall a standard result, which can be found, for example, in \cite{johansson2012scaling}.  (A slightly more restrictive statement is found in \cite{lawler2005conformally}.)  Essentially it says that $\CN_T$ together with with the notion of weak convergence, corresponds to the space of capacity-parameterized growing hull processes in $\D$ (indexed by $t \in [0,T]$), with the notion of Carath\'eodory convergence for all~$t$.  (Recall that a sequence of hulls $K^1, K^2, \ldots$ converges to a hull $K$ in the Carath\'eodory sense if the conformal normalizing maps from $\D \setminus K^j$ to $\D$ converge uniformly on compact subsets of $\D \setminus K$ to the conformal normalizing map from $\D \setminus K$ to $\D$.)

\begin{theorem}
\label{thm::measurehullprocesscorrespondence}
Consider the following:
\begin{enumerate}[(i)]
 \item\label{it::measures} A measure $\nu \in \CN_T$.
 \item\label{it::hulls} An increasing family $(K_t)$ of hulls in $\D$, indexed by $t \in [0,T]$, such that $\D \setminus K_t$ is simply connected and includes the origin and has conformal radius $e^{-t}$, viewed from the origin.  (In other words, for each $t$, there is a unique conformal map $g_t \colon \D \setminus K_t \to \D$ with $g_t(0) = 0$ and $g_t'(0) = e^t$.)
\end{enumerate}
There is a one-to-one correspondence between objects of type \eqref{it::measures} and \eqref{it::hulls}.  In this correspondence, the maps $g_t$ are obtained from $\nu$ via \eqref{eqn::lke}, where $\nu_t$ is taken to be the conditional law of the second coordinate of $\nu$ given that the first coordinate is equal to $t$.
Moreover, a sequence of measures $\nu^1, \nu^2, \ldots$ in $\CN_T$ converges weakly to a limit $\nu$ if and only if for each $t$ the functions $g^1_t, g^2_t, \ldots$ corresponding to $\nu_i$ converge uniformly to the function $g_t$ corresponding to $\nu$ on any compact set in the interior of $\D \setminus K_t$.
\end{theorem}

For completeness, we will provide a proof of Theorem~\ref{thm::measurehullprocesscorrespondence} in Section~\ref{sec::existence}.  The reader may observe that the notion of Carath\'eodory convergence for all $t \in [0,T]$ is equivalent to the notion of Carath\'eodory convergence for all $t$ in a fixed countable dense subset of $[0,T]$.  This can be used to give a direct proof of compactness of the set of hull families described Theorem~\ref{thm::measurehullprocesscorrespondence}, using the topology of Carath\'eodory convergence for all $t$.

\subsection{Main results} \label{subsec::mainresults}

\subsubsection{Subsequential limits and compactness}

The main purpose of this paper is to construct a candidate for what should be the scaling limit of $\eta$-DLA on a $\gamma$-LQG surface (at least in sufficiently isotropic formulations) for the $(\gamma^2,\eta)$ pairs which lie on the top two solid curves from Figure~\ref{fig::etavsgamma}.

Before presenting these results, let us explain one path that we will {\em not} pursue in this paper.  One natural approach would be to take a subsequential limit of $\eta$-DLA on $\delta$-approximations of $\gamma$-LQG (perhaps using an inherently isotropic setting, such as the one involving Voronoi cells of a Poisson point process associated with the LQG measure) and to simply {\em define} the limit to be a $\QLE(\gamma^2, \eta$).  Using Theorem~\ref{thm::measurehullprocesscorrespondence} and the weak compactness of $\CN_T$, it should not be hard to construct a triple $(\nu_t, g_t, \Fh_t)$ coupled with a free field instance $h$, as in the context of Figure~\ref{fig::QLEtriangle}, with the property that
\begin{enumerate}
\item The sets $K_t$ corresponding to $g_t$ are local.
\item The maps from $\nu_t$ to $g_t$, and from $g_t$ to $\Fh_t$ are as described in Figure~\ref{fig::QLEtriangle}.
\end{enumerate}
The natural next step would then be to show that $\Fh_t$ determines $\nu_t$ in the manner of Figure~\ref{fig::QLEtriangle}.  We consider this to be an interesting problem, and one that might potentially be solvable by understanding (using the discrete approximations) how $\nu_t$ restricted to a boundary arc would change when one added a constant to $\Fh_t$ on that boundary arc (see this list of open problems in Section~\ref{sec::questions}).

However, we stress that even if this problem were solved, it would not immediately give us an explicit description of the stationary law of $\nu_t$.  The main contribution of this article is to construct a solution to the dynamics of Figure~\ref{fig::QLEtriangle} for the $(\gamma^2, \eta)$ pairs illustrated in Figure~\ref{fig::etavsgamma} and to explicitly describe the stationary law of the corresponding $\nu_t$.  The construction is explicit enough to enable us to describe basic properties of the QLE growth.

\subsubsection{Theorem statements}

Before presenting our main results, we need to formalize the scaling symmetry illustrated in Figure~\ref{fig::etascaling}, which in the continuum should be a statement (which holds for any fixed $t$) about how the boundary measure $\nu_t$ changes when $\Fh_t$ is locally transformed via an LQG coordinate change.  It is a bit delicate to formulate this, since this should be an almost sure statement (i.e., it should hold almost surely for the $\Fh_t$ that one observes in a random solution, but not necessarily for {\em all} possible $\Fh_t$ choices) and one would not necessarily expect a coordinate change such as the one described in Figure~\ref{fig::etascaling} to preserve the probability measure on $\Fh_t$, or even that the law of the image would be absolutely continuous with respect to the law of the original.  However, we believe that it {\em would} be reasonable to expect the law of the restriction of $\Fh_t$ to the intervals $I_i$ in Figure~\ref{fig::etascaling} to change in an absolutely continuous way.  (This is certainly the case when $\Fh_t$ is a free boundary Gaussian free field plus a smooth deterministic function; see the many similar statements in \cite{ss2010contour}.)  In this case, one can couple two instances of the field in such a way that one looks like a quantum coordinate change of the other (via a map such as the one described in Figure~\ref{fig::etascaling}) with positive probability.  Given a coupling of this type one can formalize the $\eta$-DBM scaling symmetry, as we do in the following definition:

\begin{definition}
\label{def::eta_scaling}
We say that a triple $(\nu_t, g_t, \Fh_t)$ that forms a solution to the dynamics described in Figure~\ref{fig::QLEtriangle} satisfies {\bf $\eta$-DBM scaling} if the following is true.  Suppose that we are given any two instances $(\nu_t, g_t, \Fh_t)$ and  $(\wt \nu_t, \wt g_t, \wt \Fh_t)$ coupled in such a way that for a fixed value of $t_0 \geq 0$ and a fixed conformal map $\psi$ from a subset of $\D$ to $\D$, there is a positive probability of the event $\mathcal A$ that
$\wt \Fh_{t_0}(u) = \Fh_{t_0} \circ \psi(u) + Q \log|\psi'(u)|$ for all $u \in I$ where $I$ is an arc of $\partial \D$.  More precisely, this means that
\[ \lim_{\substack{u \to I \\ u \in \D}} \bigg( \wt \Fh_{t_0}(u)- \Fh_{t_0}\circ \psi(u) - Q \log|\psi'(u)| \bigg) = 0\]
and it says that $\Fh_{t_0}$ and $\wt \Fh_{t_0}$ are related by an LQG quantum coordinate change (as in \eqref{eqn::LQGcoordinatetransformation}).  Then we have almost surely on $\CA$ that
\begin{align} \label{eqn::etascaling}
A \mapsto \nu_{t_0}(\psi(A)) \quad\text{and}\quad
A \mapsto \int_A |\psi'(u)|^{2+\eta}  d\wt{\nu}_{t_0}(u)
\end{align}
agree as measures on $I$, up to a global multiplicative constant.
\end{definition}

Our first result is the existence of stationary solutions to the dynamics described in Figure~\ref{fig::etavsgamma} that satisfy $\eta$-DBM scaling for appropriate $\eta$ values.  (The existence of the trivial solution corresponding to $\alpha = 0$, $\nu_t$ given by uniform Lebesgue measure for all $t$, and $\eta = -1$, i.e.\ to the bottom line in Figure~\ref{fig::etavsgamma}, is obvious and hence omitted from the theorem statement, since in this case the measures $\nu_t$ do not depend on $h$ and \eqref{eqn::etascaling} is a straightforward change of coordinates.)

The particular law of $h$ described in the theorem statement below (a free boundary GFF with certain logarithmic singularity at the origin and another logarithmic singularity at a prescribed boundary point) may seem fairly specific.   Both singularities are necessary for our particular method of constructing a solution to the $\QLE$ dynamics (which uses ordinary radial SLE and the quantum gravity zipper).  However, we stress that once one obtains a solution for this particular law for $h$, one gets for free a solution corresponding to {\em any} random $h$ whose law is absolutely continuous with respect to that law, since one can always weight the law of the collection $\bigl( h, (\nu_t, g_t, \Fh_t) \bigl)$ by a Radon-Nikodym derivative depending only on $h$ without affecting any almost sure statements.

In particular, it turns out that adding the logarithmic singularity (which is not too large) centered at the uniformly chosen boundary point changes the overall law of $h$ in an absolutely continuous way (in fact the Radon-Nikodym derivative has an explicit interpretation in terms of the total mass of a certain LQG boundary measure; see the discussion in Section~\ref{sec::preliminaries} and Section~\ref{sec::couplings}, or in \cite{ds2011kpz}).  Also, adding any finite Dirichlet energy function to $h$ changes the law in an absolutely continuous way. In particular, one could add to $h$ a finite Dirichlet energy function that agrees with a multiple of $\log|\cdot|$ outside a neighborhood $U$ of the origin; a corresponding $\QLE$ would then be well defined up until the process first reaches $U$.  Since this can be done for any arbitrarily small $U$, one can obtain in this way a (not-necessarily-stationary) solution to the $\QLE$ dynamics that involves replacing the multiple of the logarithm in the definition of $h$ with another multiple of the logarithm (or removing this term altogether).

Figures~\ref{fig::edenwithlogs} and~\ref{fig::dlawithlogs} illustrate the changes that occur in the simulations when different multiples of $\log|\cdot|$ are added to $h$.  As explained in \cite[Section~1.6]{sheffield2010weld}, adding a $\log|\cdot|$ singularity to the GFF has the interpretation of first starting off with a cone and then conformally mapping to $\C$ with the conic singularity sent to the origin.  Adding a negative multiple of $\log|\cdot|$ corresponds to an opening angle smaller than $2\pi$ and a positive multiple corresponds to an opening angle larger than $2\pi$.  This is why the simulations of the Eden model (resp.\ DLA) in Figure~\ref{fig::edenwithlogs} (resp.\ Figure~\ref{fig::dlawithlogs}) appear more and more round (resp.\ have more arms) as one goes from left to right and then from top to bottom.

\begin{theorem}
\label{thm::existence}
For each $\gamma \in (0,2]$ and $\eta$ such that $(\gamma^2, \eta)$ lies on one of the upper two curves in Figure~\ref{fig::etavsgamma}, there is a $(\nu_t, g_t, \Fh_t)$ triple that forms a solution to the dynamics described in Figure~\ref{fig::QLEtriangle} and that satisfies $\eta$-DBM scaling. The triple can be constructed using an explicit limiting procedure that involves radial $\SLE_\kappa$, where $\kappa = \gamma^2$, when $(\gamma^2, \eta)$ lies on the upper curve in Figure~\ref{fig::etavsgamma}, and $\kappa = 16/\gamma^2$ when $(\gamma^2, \eta)$ lies on the middle curve.  In this solution, the $\alpha$ appearing in Figure~\ref{fig::etavsgamma} is equal to $-\tfrac{1}{\sqrt{\kappa}}$ and $h$ is a free boundary GFF on $\D$ minus $\frac{\kappa+6}{2\sqrt{\kappa}} \log|\cdot|$ (which will turn out to mean that there is an infinite amount of quantum mass in any neighborhood of the origin) plus $\tfrac{2}{\sqrt{\kappa}} \log|z - \cdot|$ where $z$ is a uniformly chosen random point on $\partial \D$ independent of the GFF.  The pair $(\nu_t, \Fh_t)$ is stationary with respect to capacity (i.e., minus log conformal radius) time.
\end{theorem}

The solutions described in Theorem~\ref{thm::existence} will be constructed as subsequential limits of certain approximations involving $\SLE$.  Although we cannot prove that the limits are unique, we can prove that {\em every} subsequential limit of these approximations has the properties described in Theorem~\ref{thm::existence} (and in particular has the same stationary distribution, described in terms in of the GFF).  We will write $\QLE(\gamma^2, \eta)$ to refer to one of these solutions.  That these solutions satisfy the $\eta$-DBM property will turn out to follow easily from the fact that $\Fh_t$, for each $t \geq 0$, is given by the harmonic extension of the boundary values of a form of the GFF and $\nu_t$ is simply a type of LQG quantum measure corresponding to that GFF instance; these points will be explained in Section~\ref{sec::continuum} and Theorem~\ref{thm::existence}.

In Section~\ref{sec::dynamics}, we derive an infinite dimensional SDE which describes the dynamics in time of the harmonic component $(\Fh_t)$ of the $\QLE(\gamma^2,\eta)$ solutions we construct in the proof of Theorem~\ref{thm::existence}.  We will not restate the result here but direct the reader to Section~\ref{sec::dynamics} for the precise form of the equation.

Our next result is the H\"older continuity of the complementary component of a $\QLE(\gamma^2,\eta)$ which contains the origin.

\begin{theorem}
\label{thm::qle_continuity}
Fix $\gamma \in (0,2)$, let $Q = 2/\gamma + \gamma/2$, and let
\begin{equation}
\label{eqn::continuity_holder_exponent}
\ol{\Delta} = \frac{Q-2}{Q+2\sqrt{2}}.
\end{equation}
Fix $\Delta \in (0,\ol{\Delta})$.  Suppose that $(\nu_t,g_t,\Fh_t)$ is one of the $\QLE(\gamma^2,\eta)$ processes described in Theorem~\ref{thm::existence}.  For each $t \geq 0$, let $\D_t = \D \setminus K_t$.  Then $\D_t$ is almost surely a H\"older domain with exponent $\Delta$.  That is, for each $t \geq 0$, $g_t^{-1} \colon \D \to D_t$ is almost surely H\"older continuous with exponent $\Delta$.
\end{theorem}

In fact, the proof of Theorem~\ref{thm::qle_continuity} will only use the fact the stationary law of $\Fh_t$ is given by the harmonic extension of the boundary values of a form of the GFF; if we could somehow construct other solutions to the $\QLE$ dynamics with this property, then this theorem would apply to those solutions as well.

Theorem~\ref{thm::qle_continuity} is a special case of a more general result which holds for any random closed set $A$ which is coupled with $h$ in a certain manner.  This is stated as Theorem~\ref{thm::continuity} in Section~\ref{sec::sample_path}.  Another special case of this result is the fact that the complementary components of $\SLE_\kappa$ for $\kappa \neq 4$ are H\"older domains.  This fact was first proved by Rohde and Schramm in \cite[Theorem~5.2]{rs2005sle} in a very different way.  We will state this result formally and give additional examples in Section~\ref{sec::sample_path}.

Suppose that $K \subseteq \D$ is a closed set.  Then $K$ is said to be {\bf conformally removable} if the only maps $\varphi \colon \D \to \C$ which are homeomorphisms of $\D$ and conformal on $\D \setminus K$ are the maps which are conformal transformations of $\D$.  The removability of the curves coupled with the GFF which arise in this theory is important because it is closely related to the question of whether the curve is almost surely determined by the GFF \cite{sheffield2010weld}.  One important consequence of Theorem~\ref{thm::qle_continuity} and \cite[Corollary~2]{js2000remove} is the removability of component boundaries of a $\QLE(\gamma^2,\eta)$ when $(\gamma^2,\eta)$ lies on one of the upper two curves of Figure~\ref{fig::etavsgamma} and $\gamma \in (0,2)$.

\begin{corollary}
\label{cor::qle_removable}
Suppose that we have the same setup as in Theorem~\ref{thm::qle_continuity}.  For each $t \geq 0$ we almost surely have that $\partial \D_t$ is conformally removable.
\end{corollary}

\subsection{Interpretation and conjecture when $\eta$ is large}

In the physics literature, there has been some discussion and debate about what happens to the $\eta$-DBM model (in the Euclidean setting, i.e., $\gamma=0$) when $\eta$ is large.  Generally, it is understood that when $\eta$ is large, there could be a strong enough preference for growth to occur at the ``tip'' that the scaling limit of $\eta$-DBM could in principle be a simple path.  There has been some discussion on the matter of whether one actually obtains a one-dimensional path when $\eta$ is above some critical value.  Some support for this idea with a critical value of about $\eta = 4$ appears in \cite{PhysRevLett.87.175502} and \cite{PhysRevLett.88.235505}.  (The latter contains a figure depicting a simulation of the $\eta=3$ DBM.)  However, a later study estimates the dimension of $\eta$-DBM in more detail and does not find evidence for a phase transition at $\eta = 4$, and concludes that the dimension of $\eta$-DBM is about $1.08$ when $\eta = 4$  \cite{PhysRevE.77.066203}.  Another reasonable guess might be that the scaling limit of $\eta$-DBM is indeed a simple path when $\eta$ is large enough, but that the simple path may be an $\SLE$ with a small value of $\kappa$ (and not necessarily a straight line).

As a reasonable toy model for this scenario, and a model that is also interesting in its own right, one may consider a variant of $\eta$-DBM in which, at each step, one {\em conditions} on having the next edge added begin {\em exactly} at the tip of where the last edge was added (so that a simple path is produced in the end).  That is, instead of choosing a new edge from the set of all cluster adjacent edges (with probability proportional to harmonic measure to the $\eta$ power) one chooses a new edge from the set of edges beginning at the current tip of the path (again with probability proportional to harmonic measure to the $\eta$ power).  This random non-self-intersecting walk is sometimes called the {\em Laplacian random walk} (LRW) with parameter $\eta$.\footnote{The term ``Laplacian-$b$ random walk'', with parameter $b = \eta$, is also used.}  Lawler has proposed (citing early calculations by Hastings) that the $\eta$-LRW should have $\SLE$ as a scaling limit (on an ordinary grid) with
\begin{equation}
\label{eqn::bvskappa} \eta = (6-\kappa)/(2\kappa),
\end{equation}
\cite{MR2247843, lyklema1986laplacian, hastings2002exact} at least when $\eta \geq 1/4$, which corresponds to $\kappa \in (0, 4]$.  Interestingly, if we set $\kappa = \gamma^2$ and $\eta = b$, then \eqref{eqn::bvskappa} states that $\eta = 3/\gamma^2 -1/2$, which corresponds to the upper curve in Figure~\ref{fig::etavsgamma}.  Simulations have shown that $\eta$-LRW for large $\eta$ looks fairly similar to a straight line \cite{0295-5075-92-3-36004}, as one would expect.

At this point, there are two natural guesses that come to mind:
\begin{enumerate}
\item Maybe the conjecture about $\eta$-LRW on a grid scaling to $\SLE_\kappa$ holds in more generality, so that the scaling limit of $\eta$-LRW on a $\gamma$-LQG is also given by $\SLE_\kappa$ with $\kappa$ as in \eqref{eqn::bvskappa}.  (Note that it is often natural to guess that processes that converge to $\SLE$ on fixed lattices also converge to $\SLE$ when drawn on $\gamma$-LQG type random graphs, assuming the latter are embedded in the plane in a conformal way \cite{ds2011kpz,sheffield2010weld}.)
\item Maybe, for each fixed $\gamma$, it is the case that when $\eta$ is large enough, the scaling limit of $\eta$-DBM on a $\gamma$-LQG is the same as the scaling of $\eta$-LRW on a $\gamma$-LQG.  (If in the $\eta$-DBM model, the growth tends to take place near the tip, maybe the behavior does not change so much when one requires the growth to take place {\em exactly} at the tip.)
\end{enumerate}

The authors do not have a good deal of evidence supporting these guesses.  However it is interesting to observe that if these guesses are correct, then for sufficiently large $\eta$, the $\eta$-DBM on a $\gamma$-LQG has a scaling limit given by $\SLE_\kappa$ for the $\kappa$ obtained from \eqref{eqn::bvskappa}, and this scaling limit does not actually depend on the value of $\gamma$.  If this is the case, then (at least for $\eta$ sufficiently large) the dotted line in Figure~\ref{fig::etavsgamma} represents $(\gamma^2, \eta)$ pairs for which the scaling limit of $\eta$-DBM on a $\gamma$-LQG is described by the ordinary radial quantum gravity zipper, which we describe in Section~\ref{sec::couplings}.  We remark that the $\eta$-DBM property as formulated above might be satisfied in a fairly empty way for this process (when $\nu_t$ is a point mass for all $t$), but the property may have some content if one considers an initial configuration in which there are two or more distinct tips (i.e., $\nu_t$ contains atoms but is not entirely concentrated at a single point).

We will not further speculate on the large $\eta$ case or further discuss scenarios in which $\nu_t$ might contain atoms.  Indeed, throughout the remainder of the paper, we will mostly limit our discussion to the solid portions of the upper two curves in Figure~\ref{fig::etavsgamma}, and the $\nu_t$ we construct in these settings will be almost surely non-atomic for all almost all $t$.

\subsection{Outline}

The remainder of this article is structured as follows.  In Section~\ref{sec::discrete}, we will describe several discrete constructions which motivate our definition of $\QLE$ as well as our interpretation of $\QLE(2,1)$ and $\QLE(8/3,0)$.  Next, in Section~\ref{sec::continuum}, we will provide arguments in the continuum that support the same interpretation for $\QLE(2,1)$ and $\QLE(8/3,0)$ as well as relate the other parameter pairs indicated in orange in Figure~\ref{fig::etavsgamma} to $\eta$-DBM.  The purpose of Section~\ref{sec::preliminaries} is to review some preliminaries (radial SLE, GFF with Dirichlet and free boundary conditions, local sets, and LQG boundary measures).  In Section~\ref{sec::couplings} we will establish a coupling between reverse radial $\SLE_\kappa$ and the GFF, closely related to the so-called quantum zipper describe in \cite{sheffield2010weld}.  This will then be used in Section~\ref{sec::existence} to prove the existence of $\QLE(\gamma^2,\eta)$ for $(\gamma^2,\eta)$ which lies on one of the upper two curves from Figure~\ref{fig::etavsgamma}.  We will then derive an equation for the stochastic dynamics of the measure which drives a $\QLE(\gamma^2,\eta)$ in Section~\ref{sec::dynamics}.  In Section~\ref{sec::sample_path} we will establish continuity and removability results about $\QLE(\gamma^2,\eta)$ and discussion the problem of describing certain phase transitions for $\QLE$.  Finally, in Section~\ref{sec::questions} we will state a number of open questions and describe some current works in progress.

%% file: tex/results.tex
\subsection{Reshuffled Markov chains}

At the heart of our discrete constructions lies a very simple observation about Markov chains.
Consider a measure space $S$ which is a disjoint union of spaces $S_1, S_2, \ldots, S_N$.  In the examples of this section, $S$ will be a finite set.  Suppose we have a measure $\mu_i$ defined on each $S_i$.  Let $X = (X_k)$ be any Markov chain on $S$ with the property that for any $i$, $j$, and $\mathcal S \subseteq S_i$, we have
\[ \p[ X_j \in \mathcal S | X_j \in S_i ] = \mu_i(\mathcal S).\]
This property in particular implies that the conditional law of $X_0$, given that it belongs to $S_i$, is given by $\mu_i$.

Then there is a {\em reshuffled Markov chain} $Y = (Y_k)$ defined as follows.  First, $Y_0$ has the same law as $X_0$.  Then, to take a step in the reshuffled Markov chain from a point $x$, one first chooses a point $y \in S$ according to the transition rule for the Markov chain $X$ (from the point $x$), and then one chooses a new point $z$ from $\mu_j$, where $j$ is the value for which $y \in S_j$.  The step from $x$ to $z$ is a step in the reshuffled Markov chain (and subsequent steps are taken in the same manner).  Intuitively, one can think of the reshuffled Markov chain as a Markov chain in which one imposes a certain degree of forgetfulness: if we are given the value $Y_i$, then in order to sample $Y_{i+1}$ we can imagine that we first take a transition step from $Y_i$ and then we ``forget'' everything we know about the new location except which of the sets $S_j$ it belongs to.

The reshuffled chain induces a Markov chain on $\{S_1, S_2, \ldots, S_N \}$.  Now suppose that $A$ is a union of some of the $S_i$ and is a sink of this Markov chain (i.e., once the Markov chain enters $A$, it almost surely does not leave).  Then we have the following:

\begin{proposition}
\label{prop::reshuffling}
In the context described above, for each fixed $j \geq 0$, the law of $X_j$ is equivalent to the law of $Y_j$.  Moreover, the law of $\min \{j: X_j \in A \}$ agrees with the law of $\min \{j: Y_j \in A \}$.
\end{proposition}
\begin{proof}
The first statement holds for $j=0$ and follows for all $j>0$ by induction.  The second statement follows from the first, since for any $j$ the probability that $A$ has been reached by step $j$ is the same for both Markov chains.
\end{proof}

Our aim in the next two subsections is to show two things:
 \begin{enumerate}
 \item The Eden model on a random triangulation can be understood as a reshuffled percolation interface exploration on that triangulation.
 \item DLA on a random planar map can be understood as a reshuffled loop-erased random walk on that map.
 \end{enumerate}
In order to establish these results, and to apply Proposition~\ref{prop::reshuffling}, we will need to decide in each setting what information we keep track of (i.e., what information is contained in the state $Y_i$) and what information we forget (i.e., what information we lose when we remember which $S_j$ the state $Y_i$ belongs to and forget everything else).  In both settings, the information we keep track of will be
 \begin{enumerate}
 \item the structure of the ``unexplored'' region of a random planar map,
 \item the location of a ``target'' within that region,
 \item the location of a ``tip'' on the boundary of the unexplored region.
\end{enumerate}
Also in both settings, the information that we ``forget'' is the location of the tip.  Thus we will replace a path that grows continuously with a growth process that grows from multiple locations.  In both cases, a natural sink (to which one could apply Theorem~\ref{prop::reshuffling}) is the ``terminal'' state obtained when the exploration process reaches its target.  The total number of steps in the exploration path agrees in law with the total number of steps in the reshuffled variant.  We will also find additional symmetries (and a ``slot machine'' decomposition) in the percolation/Eden model setting.

\subsection{The Eden model and percolation interface}

\subsubsection{Finite volume Eden/percolation relationship}

The following definitions and basic facts are lifted from the overview of planar triangulations given by Angel and Schramm in \cite{MR2013797} (which cites many of these results from other sources, including \cite{angel2003growth}).  Throughout this section we consider only so-called ``type II triangulations,'' i.e., triangulations whose graphs have no loops but may have multiple edges.  For integers $n, m \geq 0$, \cite{MR2013797} defines $\phi_{n,m}$ to be the number of triangulations of a disc (rooted at a boundary edge) with $m+2$ boundary edges and $n$ internal vertices, giving in \cite[Theorem~2.1]{MR2013797} the explicit formula\footnote{In \cite{MR2013797} a superscript $2$ is added to $\phi_{n,m}$ to emphasize that the statement is for type II triangulations.  We omit this superscript since we only work with triangulations of this type.}:
\begin{equation}
\label{eqn::phi_n_m}
\phi_{n,m} = \frac{2^{n+1} (2m+1)!(2m+3n)!}{(m!)^2n!(2m+2n+2)}.
\end{equation}
By convention $\phi_{0,0} = 1$ because when the external face is a $2$-gon, one possible way to ``fill in'' the inside is simply to glue the external edges together, with no additional vertices, edges, or triangles inside (and this is in fact the only possibility).  As $n \to \infty$,
\begin{equation}
\label{eqn::phi_n_m_asymp}
\phi_{n,m} \sim C_m \alpha^n n^{-5/2},
\end{equation}
where $\alpha = 27/2$ and
\begin{equation}
\label{eqn::c_m}
 C_m = \frac{\sqrt{3}(2m+1)!}{2 \sqrt{\pi}(m!)^2}(9/4)^m \sim C 9^m m^{1/2}.
\end{equation}
(Both \eqref{eqn::phi_n_m_asymp} and \eqref{eqn::c_m} are stated just after \cite[Theorem~2.1]{MR2013797}.)

\begin {figure}[ht!]
\begin {center}
\includegraphics [width=4.5in]{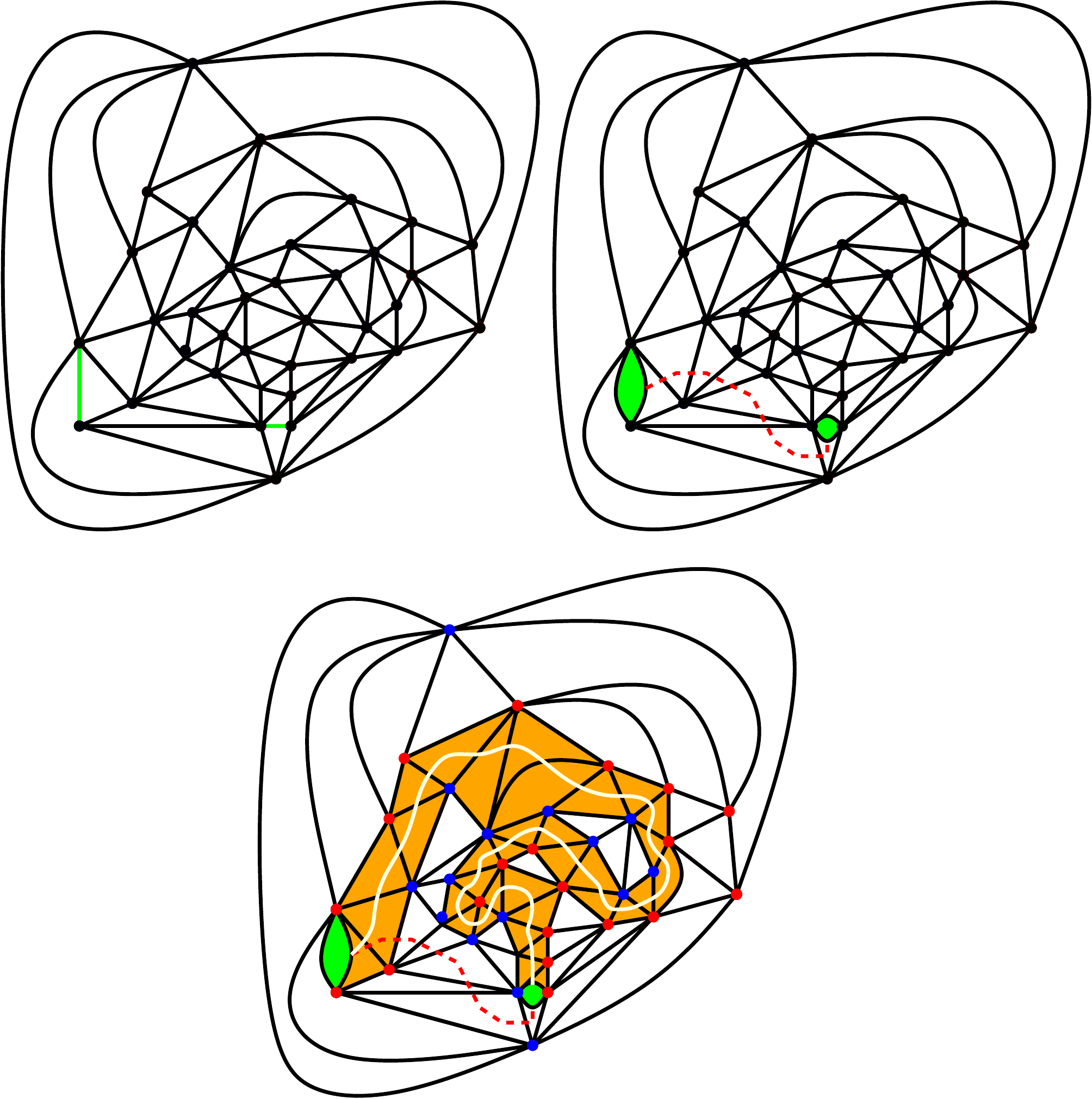}
\caption {\label{triangulation} {\bf Upper left:} a triangulation of the sphere together with two distinguished edges colored green.  {\bf Upper right:} It is conceptually useful to ``fatten'' each green edge into a $2$-gon.  We fix a distinguished non-self-intersecting dual-lattice path $p$ (dotted red line) from one $2$-gon to the other.  {\bf Bottom:} Vertices are colored red or blue with i.i.d.\ fair coins.  There is then a unique dual-lattice path from one $2$-gon to the other (triangles in the path colored orange) such that each edge it crosses either has opposite-colored endpoints and {\em does not} cross $p$, or has same-colored endpoints and {\em does} cross $p$.  The law of the orange path does not depend on the choice of $p$, since shifting $p$ across a vertex has the same effect as flipping the color of that vertex.  (Readers familiar with this terminology will recognize the orange path as a percolation interface of an antisymmetric coloring of the double cover of the complement of the $2$-gons.  Here ``antisymmetric'' means the two liftings of a vertex have opposite colors.)  When the triangulations are embedded in the sphere in a conformal way, the conjectural scaling limit of the path is a whole plane SLE$_6$ between the two endpoints.}
\end {center}
\end {figure}

\begin {figure}[ht!]
\begin {center}
\includegraphics [width=5.5in]{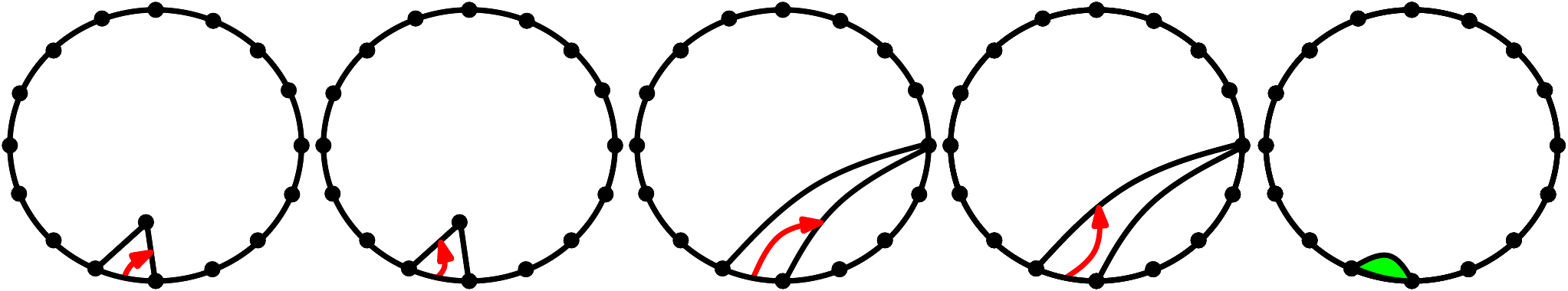}
\caption {\label{peelinstep} Begin with a polygon with $m+2$ edges (for some $m \geq 0$) and a fixed seed edge on the boundary (from which the exploration will take place).  Suppose we wish to construct a triangulation of the polygon with $n \geq 0$ additional vertices in the interior.  Observe by an easy induction argument that $n$ and $m$ together determine the number of triangles in this triangulation: $m+2n$.  They also determine the number of edges (including boundary edges): $m+2+3n$.  The total number of possible triangulations is $\phi_{m,n}$, and for each triangulation there are $(m+2+3n)$ choices for the location of the green edge.  The exploration ends if the face incident to the seed edge is the green $2$-gon, as in the right figure, which has probability $(m+2+3n)^{-1}$.  Conditioned on this not occurring, the probability that we see a triangle with a new vertex (as in the left two figures) is given by $\phi_{m+1,n-1}/\phi_{m,n}$, and given this, the two directions are equally likely (and depend on the coin toss determining the vertex color).  In the third and fourth pictures, the exploration step involves deciding both the location of the new vertex (how many steps it is away from the seed edge, counting clockwise) {\em and} how many of the remaining interior vertices will appear on the right side.  We can work out the number of triangulations consistent with each choice: it is given by the product $\phi_{m_1, n_1} \phi_{m_2,n_2}$ where $(m_i,n_i)$ are the new $(m,n)$ values associated to the two unexplored regions.  (The choices are constrained by $m_1 + m_2 = m-1$ and $n_1 + n_2 = n$.)  The probability of such a choice is therefore given by this value divided by $\phi_{m,n}$.  Once that choice is made, we have to decide whether the step corresponds to the third or fourth figure shown --- i.e., whether the green edge is somewhere in the left unexplored region or the right unexplored region.  The probability of it being in the first region is the number of edges in that region divided by the total number of edges (excluding the seed edge, since we are already conditioning on the seed not being the target): $(m_1+2+3n_1)/(m+1+3n)$.  In each of the first four figures, we end up with a new unexplored polygon-bounded region known to contain the target green edge, and a new $(m,n)$ pair.  We may thus begin a new exploration step starting with this pair and continue until the target is reached.}
\end {center}
\end {figure}

\begin {figure}[!ht]
\begin {center}
\includegraphics [width=3.5in]{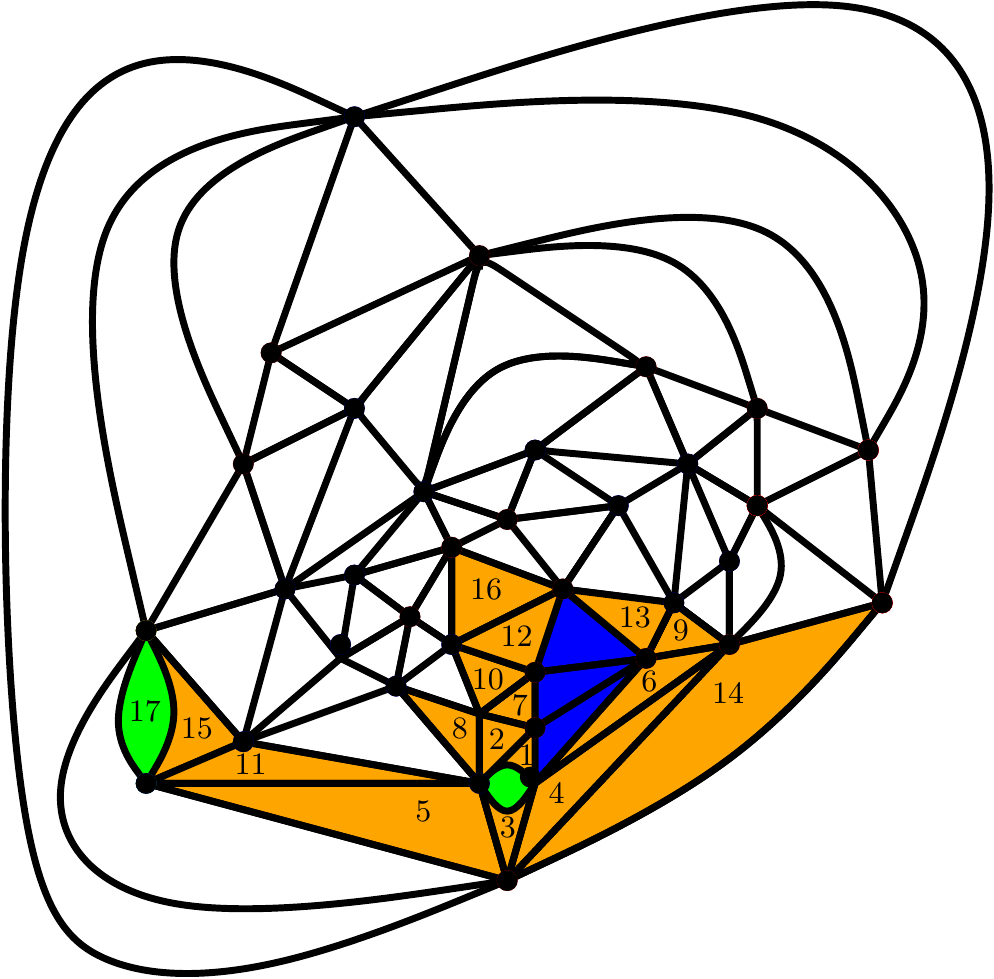}
\caption {\label{edentriangulation} Same as Figure~\ref{triangulation} except that one explores using the Eden model instead of percolation.  At each step, one chooses a uniformly random edge on the boundary of the unexplored region containing the target and explores the face incident to that edge.  The faces are numbered according to the order in which they were explored.  When the unexplored region is divided into two pieces, each with one or more triangles, the piece without the target is called a {\em bubble} and is never subsequently explored by this process.  In this figure there is only one bubble, which is colored blue.}
\end {center}
\end {figure}

Figure~\ref{triangulation} shows a triangulation $T$ of the sphere with two distinguished edges $e_1$ and $e_2$, and the caption describes a mechanism for choosing a random path in the dual graph of the triangulation, consisting of distinct triangles $t_1, t_2, \ldots, t_k$, that goes from $e_1$ to $e_2$.  It will be useful to imagine that we begin with a single $2$-gon and then grow the path dynamically, exploring new territory as we go.  At any given step, we keep track of the total number edges on the boundary of the already-explored region and the number of vertices remaining to be seen in the component of the unexplored region that contains the target edge.  The caption of Figure~\ref{peelinstep} explains one step of the exploration process.  This exploration procedure is closely related to the peeling process described in \cite{angel2003growth}, which is one mechanism for sampling a triangulation of the sphere by ``exploring'' new triangles one at a time.  The exploration process induces a Markov chain on the set of pairs $(m,n)$ with $m \geq 0$ and $n \geq 0$.  In this chain, the $n$ coordinate is almost surely non-increasing, and the $m$ coordinate can only increase by $1$ when the $n$ coordinate decreases by~$1$.

Now consider the version of the Eden model in which new triangles are only added to the unexplored region containing the target edge, as illustrated Figure~\ref{edentriangulation}.  In both Figure~\ref{triangulation} and Figure~\ref{edentriangulation}, each time an exploration step separates the unexplored region into two pieces (each containing at least one triangle) we refer to the one that does not contain the target as a {\em bubble}.  The exploration process described in Figure~\ref{triangulation} created two bubbles (the two small white components), and the exploration process described in Figure~\ref{edentriangulation} created one (colored blue).   We can interpret the bubble as a triangulation of a polygon, rooted at a boundary edge (the edge it shares with the triangle that was observed when the bubble was created).

The specific growth pattern in Figure~\ref{edentriangulation} is very different from the one depicted in Figure~\ref{triangulation}.  However, the analysis used in Figure~\ref{peelinstep} applies equally well to both scenarios.  The only difference between the two is that in Figure~\ref{edentriangulation} one re-randomizes the seed edge (choosing it uniformly from all possible values) after each step.

In either of these models, we can define $C_k$ to be the boundary of the target-containing unexplored region after $k$ steps.  If $(M_k, N_k)$ is the corresponding Markov chain, then the length of $C_k$ is $M_k + 2$ for each $k$.  Let $D_k$ denote the union of the edges and vertices in $C_k$, the edges and vertices in $C_{k-1}$ and the triangle and bubble (if applicable) added at step $k$, as in Figure~\ref{slotmachine}.  We refer to each $D_k$ as a {\em necklace} since it typically contains a cycle of edges together with a cluster of one or more triangles hanging off of it.  The  analysis used in Figure~\ref{peelinstep} (and discussed above) immediately implies the following (parts of which could also be obtained from Proposition~\ref{prop::reshuffling}):

\begin{proposition}
\label{prop::edenandpercolationinterface}
Consider a random rooted triangulation of the sphere with a fixed number $n>2$ of vertices together with two distinguished edges chosen uniformly from the set of possible edges.  (Using the Euler characteristic and the fact that edges and faces are in $2$ to $3$ correspondence, it is clear that this triangulation contains $2(n-2)$ triangles and $3(n-2)$ edges.)  If we start at one edge and explore using the Eden model as in Figure~\ref{edentriangulation}, or if we explore using the percolation interface of Figure~\ref{triangulation}, we will find that the following are the same:
\begin{enumerate}[(i)]
\item The law of the Markov chain $(M_k,N_k)$ (which terminates when the target 2-gon is reached).
\item The law of the total number of triangles observed before the target is reached.
\item The law of the sequence $D_k$ of necklaces.
\end{enumerate}
Indeed, one way to construct an instance of the Eden model process is to start with an instance of the percolation interface exploration process and then randomly rotate the necklaces in the manner illustrated in Figure~\ref{slotmachine}.
\end{proposition}

\begin {figure}[!ht]
\begin {center}
\includegraphics [width=5.5in]{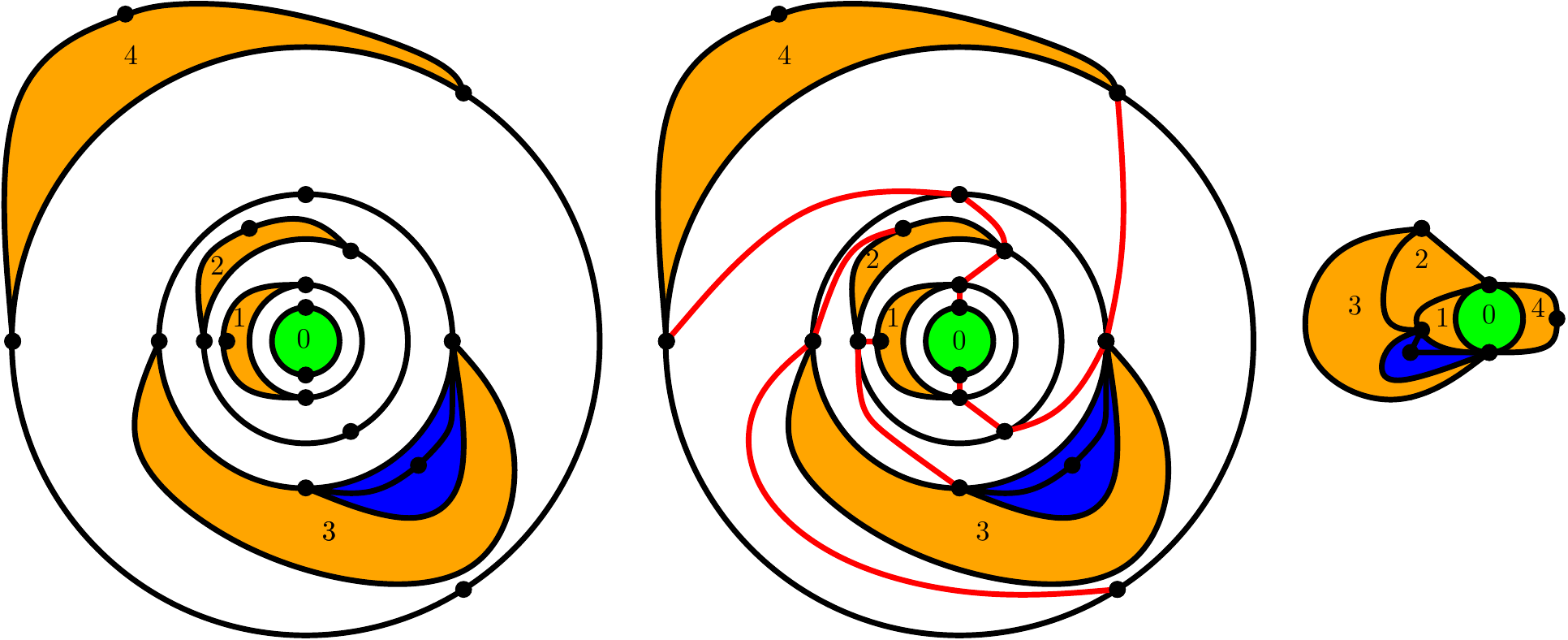}
\caption {\label{slotmachine} {\bf Left:} the first four necklaces (separated by white space) generated by an Eden model exploration.  {\bf Middle:} one possible way of identifying the vertices on the outside of each necklace with those on the inside of the next necklace outward.  {\bf Right:} The map with exploration associated to this identification.  If a necklaces has $n$ vertices on its outer boundary, then there are $n$ ways to glue this outer boundary to the inner boundary of the next necklace outward.  It is natural to choose one of these ways uniformly at random, independently for each consecutive pair of necklaces.  Intuitively, we imagine that before gluing them together, we randomly spin the necklaces like the reels of a slot machine, as in Figure~\ref{actualslotmachine}.  A fanciful interpretation of Proposition~\ref{prop::edenandpercolationinterface} is that if we take a percolation interface exploration as in Figure~\ref{triangulation} (which describes a sequence of necklaces) and we pull the slot machine lever, then we end up with an Eden model exploration of the type shown in Figure~\ref{edentriangulation}.  In later sections, this paper will discuss a continuum analog of ``pulling the slot machine lever'' that involves SLE and LQG.}
\end {center}
\end {figure}

\begin {figure}[!ht]
\begin {center}
\includegraphics [width=3.5in]{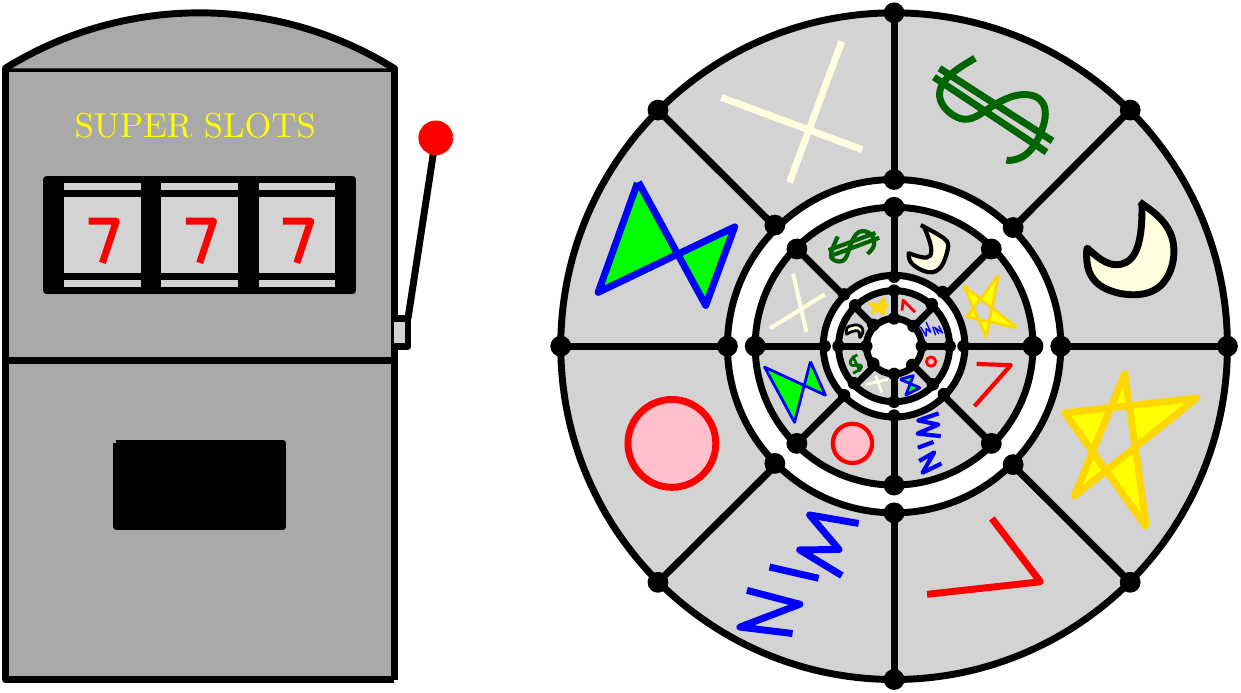}
\caption {\label{actualslotmachine} {\bf Left:} sketch of an actual slot machine.  {\bf Right:} sketch of slot machine reels conformally mapped from cylinder to plane.  When the lever is pulled, each of the reels rotates a random amount.}
\end {center}
\end {figure}

\subsubsection{Infinite volume Eden/percolation relationship}

In \cite{angel2003growth} Angel gives a very explicit construction of the {\em uniform infinite planar triangulation} (UIPT), which is further investigated by Angel and Schramm in \cite{MR2013797}.  The authors in \cite{MR2013797} define $\tau_n$ to be the uniform distribution on rooted type II triangulations of the sphere with $n$ vertices and show that the measures $\tau_n$ converge as $n\to \infty$ (in an appropriate topology) to an infinite volume limit called the {\em uniform infinite planar triangulation} (UIPT).  (Related work on infinite planar quadrangulations appears in \cite{curien2012view}.)  The convenient property that this process possesses is that the number $n$ of remaining vertices is always infinite, and hence, in the analog of the Markov chain described in Figure~\ref{peelinstep}, it is only necessary to keep track of the single number $m$, instead of the pair $(m,n)$.  A very explicit description of this Markov chain and the law of the corresponding necklaces appears in \cite{angel2003growth,MR2013797}.  As in the finite volume case, the sequence of necklaces has the same law in the UIPT Eden model as in the UIPT percolation interface exploration.  One can first choose the necklaces associated to a UIPT percolation interface model and then randomly rotate them (by ``pulling the slot machine lever'') to obtain an instance of the UIPT Eden model, as in Figure~\ref{slotmachine}.

Later in this paper, we will interpret the version of $\QLE(8/3,0)$ that we construct as a continuum analog of the Eden model on the UIPT.  The construction will begin with the continuum analog of the percolation exploration process on the UIPT, which is a radial $\SLE_6$ exploration on a certain type of LQG surface.  We will then ``rerandomize the tip'' at discrete time intervals, and we will then find a limit of these processes when the interval size tends to zero.

Finally, we remark that Gill and Rohde have recently established parabolicity of the Riemann surfaces obtained by gluing triangles together \cite{gill2011riemann}, which implies that the UIPT as a triangulation can be conformally mapped onto the entire complex plane, as one would expect.

\subsection{DLA and the loop-erased random walk}

\subsubsection{Finite volume DLA/LERW relationship} \label{subsubsec::DLALERW}

The uniform spanning-tree-decorated random planar map is one of the simplest and most elegant of the planar map models, due to the relationship with simple random walks described by Mullin in 1967 \cite{MR0205882} (and explained in more detail by Bernardi in \cite{MR2285813}) which we briefly explain in Figure~\ref{planarmapfigs} and Figure~\ref{planarmapfigstree} (which are lifted from a more detailed exposition in \cite{sheffield2011quantum}).  As the caption to Figure~\ref{planarmapfigs} explains, one first observes a correspondence between planar maps and quadrangulations: there is a natural quadrangulation such that each edge of the original map corresponds to a quadrilateral (whose vertices correspond to the two endpoints and the two dual endpoints of that edge).  As the caption to Figure~\ref{planarmapfigstree} explains, one may then draw diagonals in these quadrilaterals corresponding to edges of the tree or the dual tree.

If an adjacent vertex and dual vertex are fixed and designated as the root and dual root (big dots in Figure~\ref{planarmapfigstree}) then one can form a cyclic path starting at that edge that passes through each green edge once, always with blue on the left and red on the right.  To the $k$th green edge that the path encounters (after one spanning root and dual root) we assign a pair of integers $(x_k, y_k)$, where $x_k$ is the distance of the edge's left vertex to the root within the tree, and $y_k$ is the distance from its right vertex to the dual root within the dual tree.  If $n$ is the number of edges in the original map, then the sequence $(x_0,y_0), (x_1,y_1), (x_2, y_2), \ldots, (x_{2n}, y_{2n})$ is a walk in $\Z_+^2$ beginning and ending at the origin, and it is not hard to see that there is a one-to-one correspondence between walks of this type and rooted spanning-tree-decorated maps with $n$ edges, such as the one illustrated in Figures~\ref{planarmapfigs} and~\ref{planarmapfigstree}.  The walks of this type with $2m$ left-right steps and $2n-2m$ up-down steps correspond to the planar maps with $m$ edges in the tree (hence $m+1$ vertices total in the original planar map) and $n-m$ edges in the dual tree (hence $n-m+1$ faces total in the original planar map).
Once $m$ and $n$ are fixed (it is natural to take $m \approx n/2$), it is easy to sample the spanning-tree-decorated rooted planar map by sampling the corresponding random walk.

\begin {figure}[!ht]
\begin {center}
\includegraphics [width=4.5in]{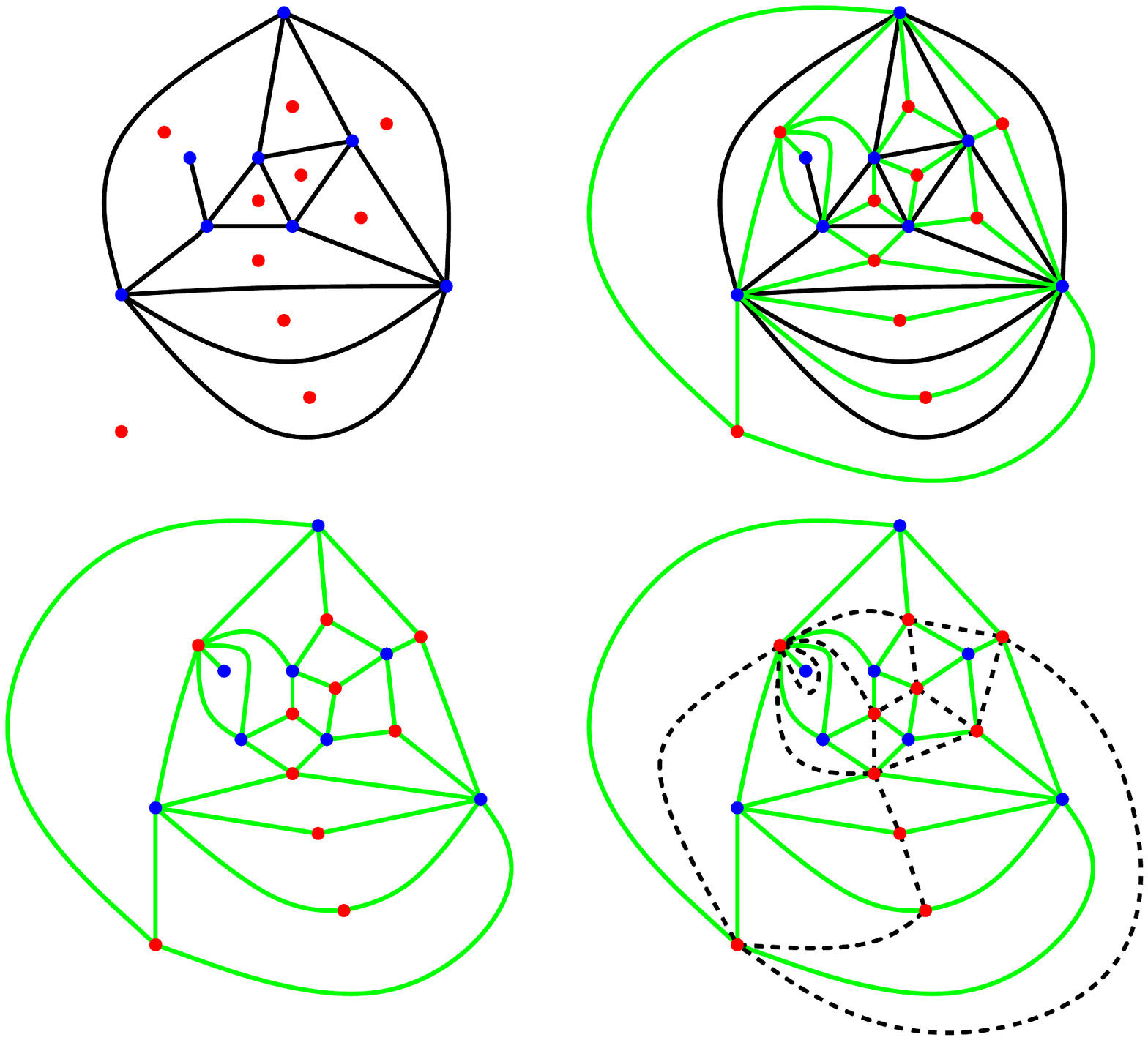}
\caption {\label{planarmapfigs} {\bf Upper left:} a planar map $M$ with vertices in blue and ``dual vertices'' (one for each face) shown in red.  {\bf Upper right:} the quadrangulation $Q=Q(M)$ formed by adding a green edge joining each red vertex to each of the boundary vertices of the corresponding face.  {\bf Lower left:} quadrangulation $Q$ shown without $M$.  {\bf Lower right:} the dual map $M'$ corresponding to the same quadrangulation, obtained by replacing the blue-to-blue edge in each quadrilateral with the opposite (red-to-red) diagonal.}
\end {center}
\end {figure}

\begin {figure}[!ht]
\begin {center}
\includegraphics [width=4.5in]{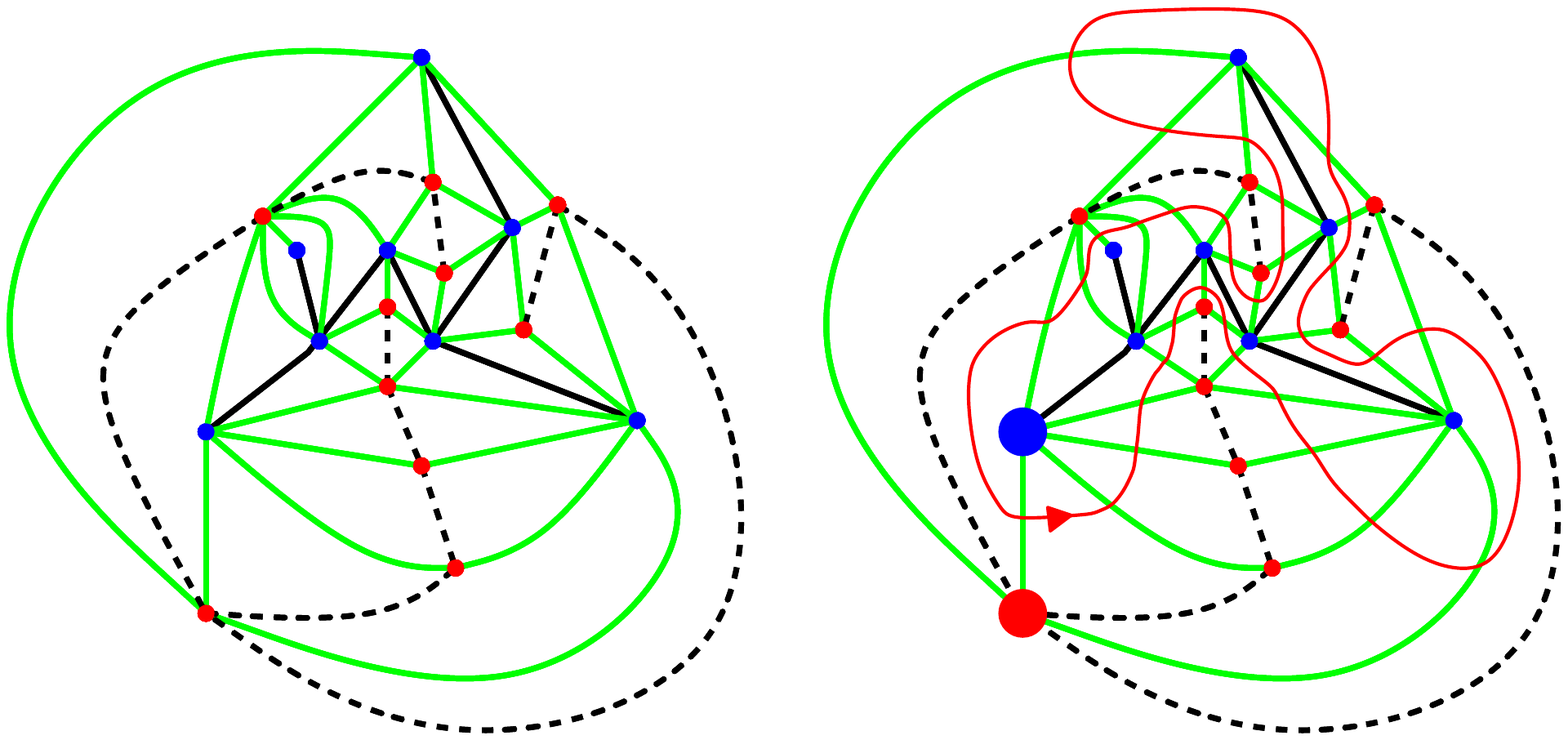}
\caption {\label{planarmapfigstree} {\bf Left:} in each quadrilateral we either draw an edge (connecting blue to blue) or the corresponding dual edge (connecting red to red).  In this example, the edges drawn form a spanning tree of the original (blue-vertex) graph, and hence the dual edges drawn form a spanning tree of the dual (red-vertex) graph.   {\bf Right:} designate a ``root'' (large blue dot) and an adjacent ``dual root'' (large red dot).  The red path starts at the midpoint of the green edge between the root and the dual root and crosses each of the green edges once, keeping the blue endpoint to the left and red endpoint to the right, until it returns to the starting position.  Each endpoint corresponds to a pair of vertices}
\end {center}
\end {figure}

As shown in Figure~\ref{treesplicing}, if we endow the map with two distinguished vertices, a ``seed'' and a ``target'' then there is a path from the seed to the target and a deterministic procedure for ``unzipping'' the edges of the path one at a time, to produce (at each step) a new planar map with a distinguished grey polygon that has a marked tip vertex (``zipper handle'') on its boundary.  This procedure is also reversible --- i.e., if we see one of the later decorated maps in Figure~\ref{treesplicing}, then we have enough information to recover the earlier figures.

It is possible to consider the same procedure but keep track of less information: one can imagine a version of Figure~\ref{treesplicing} in which all of the edges colored black or green (except those on the boundary of the grey polygon) were colored red, like the first two maps shown in Figure~\ref{treesplicingdla}.  To put ourselves in the context of Proposition~\ref{prop::reshuffling}, we can let $X_i$ be the decorated planar map (the planar map endowed with a distinguished grey face with a marked blue tip on its boundary, and a distinguished green target vertex) obtained after unzipping $i$ steps.  By Wilson's algorithm \cite{wilson1996generating}, if one is given the first $k$ steps of the path from the seed to the target, then the conditional law of the remaining edges is the law of the loop erasure of a simple random walk started at the target and conditioned to hit the grey polygon for the first time at the blue tip vertex (whereupon the walk is terminated).  In particular, this tells us how to perform the Markov transition step from $X_i$ to $X_{i+1}$.  Namely, one chooses an edge incident to the tip with probability proportional to its harmonic measure (viewed from the target), colors that edge green, and ``unzips'' it by sliding up the blue tip, as in the first transition step shown in Figure~\ref{treesplicingdla}.

\begin {figure}[!ht]
\begin {center}
\includegraphics [width=5.5in]{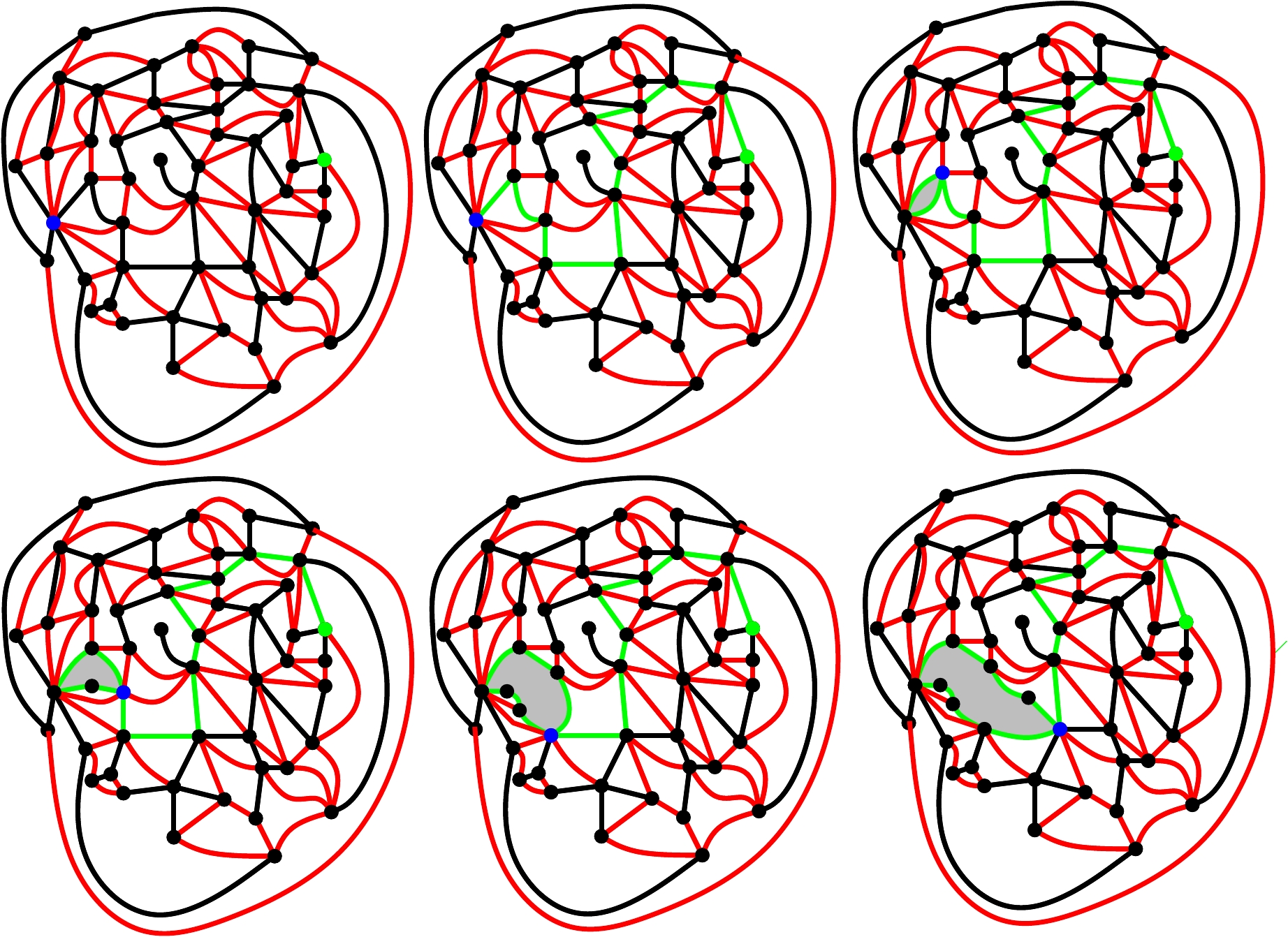}
\caption {\label{treesplicing} {\bf Upper left:} A planar map with distinguished spanning tree (tree edges black, other edges red) along with distinguished ``seed'' and ``target'' vertices (colored green).  Assume the tree-decorated map is chosen uniformly from the set of tree-decorated maps with a given number of vertices and edges, and that the seed and target are then uniformly chosen vertices.  {\bf Upper middle:} tree path from seed to target colored green.  {\bf Upper right:} think of the blue dot as a zipper handle, and the green path as the closed zipper; we slide the blue dot up one step and ``unzip'' the first edge by splitting it in two to form 2-gon (with inside colored grey).  {\bf Lower left to lower right:} second, third, fourth edges along path are similarly unzipped, to produce 4-gon, 6-gon, 8-gon.  Given the initial tree-decorated map and seed/target vertices, the unzipping procedure is deterministic.
}
\end {center}
\end {figure}

\begin {figure}[!ht]
\begin {center}
\includegraphics [width=5.5in]{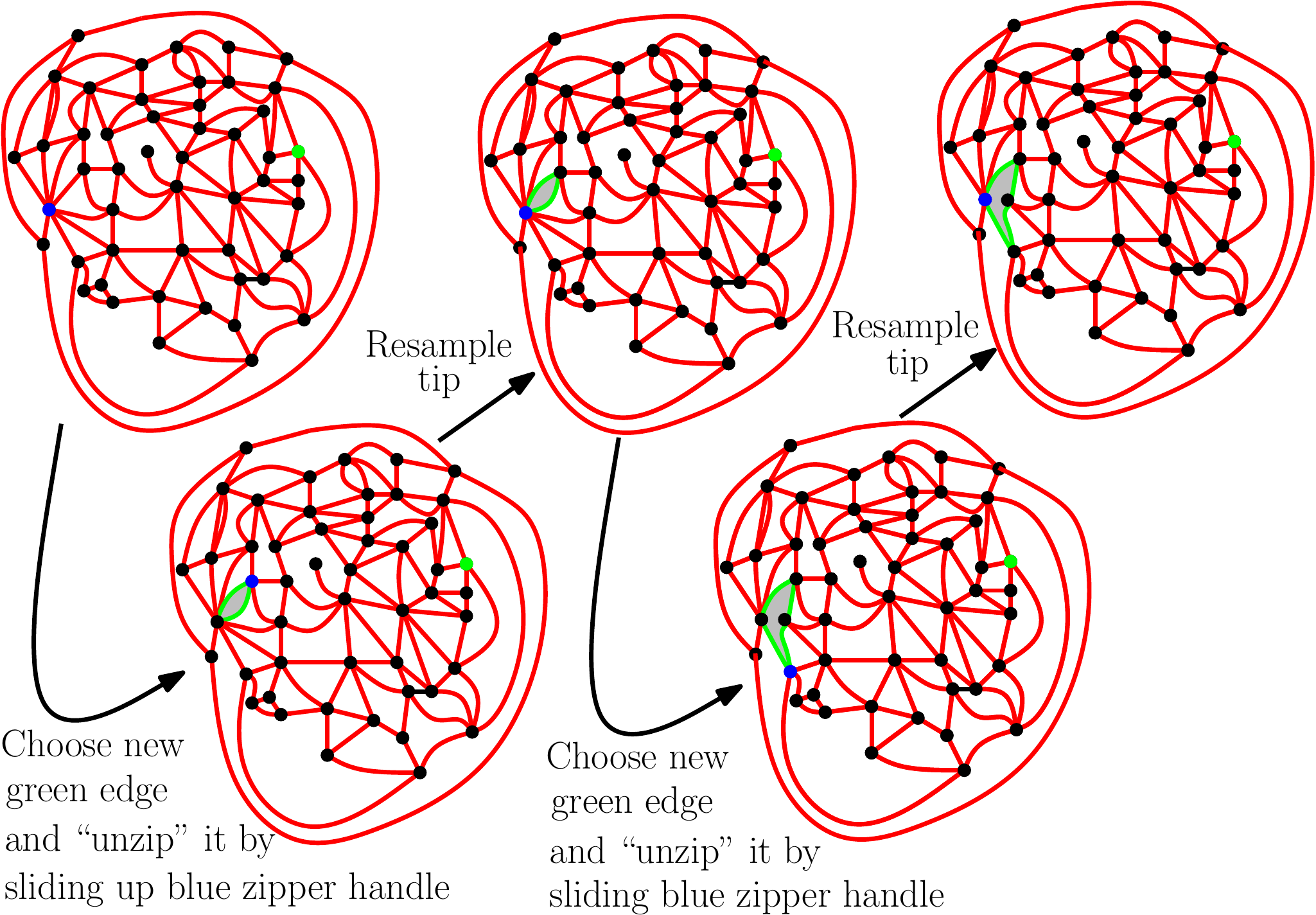}
\caption {\label{treesplicingdla} We could have drawn the images in Figure~\ref{treesplicing} with a different coloring --- showing all edges red except for those around the grey polygon.  With such a coloring, we could imagine that we do not know the tree in advance: we only discover the path from seed to tip one edge at a time.  Conditioned on the first $k$ edges, Wilson's algorithm implies that the probability that a given tip-adjacent edge $e$ is the next edge in the path is proportional to the probability that a random walk from the target first reaches the grey polygon via $e$.  After selecting an edge, we color it green and unzip it by sliding up the blue zipper handle, tracing a path whose overall law is that of a LERW from tip to seed.  The process shown in the figure is a ``reshuffled'' version of the one just described.  After an edge is drawn, we ``resample'' the blue vertex according to harmonic measure viewed from the target and then choose a sample green edge from that vertex.  We can equivalently combine the resample-tip and pick-new-edge steps by performing a random walk from the target and picking the last edge traversed before the grey polygon is hit.  The order in which edges are ``unzipped'' in this reshuffled form of LERW is the same as the order in which edges are discovered in edge-growth DLA.  In a sense, DLA is nothing more than reshuffled LERW.
}
\end {center}
\end {figure}

Using the notation of Proposition~\ref{prop::reshuffling}, we can let the sets $S_j$ be the equivalence classes of decorated maps, where two maps are considered equivalent if they agree except that their blue tips are at different locations (on the boundary of the same grey polygon).  Conditioned on $X_i \in S_j$, it is not hard to see that the conditional law of the tip location is given by harmonic measure viewed from the target. This is because, once we condition on $S_j$, one can treat the grey polygon as a single vertex, and note that all spanning trees of the collapsed graph are equally likely; hence, one can therefore use Wilson's algorithm to sample the path from the target to the grey polygon, and the law of the location at which it exits is indeed given by harmonic measure.  Thus, within each $S_j$ we can define the measure $\mu_j$ on decorated maps such that sampling from $\mu_j$ amounts to re-sampling the seed vertex from this harmonic measure.

One can now define the reshuffled Markov chain $Y_0, Y_1, \ldots$ using precisely the procedure described in Proposition~\ref{prop::reshuffling}.  This chain has the same transition law as the unreshuffled chain except that after each step we resample the blue tip from the harmonic measure viewed from the target, as explained in Figure~\ref{treesplicingdla}.  As explained in Figure~\ref{treesplicingdla}, this reshuffling procedure converts loop-erased random walk (LERW) to diffusion limited aggregation (DLA).  The following is now immediate from Proposition~\ref{prop::reshuffling}.

\begin{proposition}
\label{prop::finitedlaresult}
Consider a random rooted planar map $M$ with $n$ edges, $m+1$ vertices, and $n-m+1$ faces, with two of the vertices designated ``seed'' and ``target'' chosen uniformly among all such decorated maps except that the probability of a given decorated map is proportional to the number of spanning trees the map has.  Conditional on $M$, one may generate a loop-erased random walk $L$ from the seed to target.  Given $M$, one may also generate an edge-based DLA growth process, which yields a random tree $D$ containing the seed and the target.
\begin{enumerate}[(i)]
\item The number of edges in $D$ agrees in law with the number of edges in $L$.
\item The law of the map obtained by unzipping the first $k$ steps of $L$ (to produce the grey polygon with distinguished tip, as in Figure~\ref{treesplicing}) is the same as the law of the map obtained by unzipping the first $k$ steps of $D$, as in Figure~\ref{treesplicing}.
\item The law of the map obtained by unzipping all of the edges of $L$ agrees in law with the map obtained by unzipping all of the edges of $D$.
\end{enumerate}
\end{proposition}

\subsubsection{Infinite volume DLA/LERW relationship}

In this section, we will observe that the constructions of the previous section can be extended to the so-called  {\em uniform infinite planar tree-decorated map} (UIPTM).  We present this infinite volume construction partly because of its intrinsic interest, and partly because we believe that the form of $\QLE(2,1)$ that we construct in this paper is the scaling limit of DLA on the UIPTM.

We define the UIPTM to be the infinite volume limit of the models of random rooted planar maps described in the previous section as $n \to \infty$ and $m=n/2$.  More discussion of this model appears in \cite{sheffield2011quantum} and in the work of Gill and Rohde in \cite{gill2011riemann}.  The latter showed that the Riemannian surface defined by gluing together the triangles in the UIPTM is parabolic (like the analogous surface defined using the UIPT).  Gurel-Gurevich and Nachmias also recently proved a very general recurrence statement for random planar maps, which implies that if we forget the spanning tree on the UIPTM and simply run a random walk on the vertices of the underlying graph, then this walk is almost surely recurrent \cite{2012arXiv1206.0707G}.  (Their work also implies that random walk on the UIPT is almost surely recurrent, and extends the earlier recurrence results obtained by Benjamini and Schramm in \cite{MR1873300}.)

Given the walk $(x_k, y_k)$ described in the previous section, we may write $I_k = (x_k,y_k) - (x_{k-1}, y_{k-1})$.  The $I_k$ are random variables taking values in $$\{ (-1,0), (1,0), (0,-1), (0,1) \}.$$  There is a one-to-one correspondence between steps of type $(1,0)$ and vertices $v$ of the planar map (discounting the root vertex) since the red path in Figure~\ref{planarmapfigstree} first encounters a green edge incident to a vertex $v$ at step $k$ if and only if $I_k = (1,0)$.

If $k$ is such that $I_k$ corresponds to a chosen seed vertex $v$, then we may recenter time so that this vertex corresponds to the first increment.  That is, we define a new centered increment process: $\wt I_j = I_{j-k}$.  It is not hard to see that in the limit as $n \to \infty$ and $m = n/2$, the $\wt I_j$ converge to a process indexed by $\Z$ in which $I_1 = (1,0)$ almost surely but the other $I_i$ are i.i.d.\ uniformly chosen elements from $\{ (-1,0), (1,0), (0,-1), (0,1) \}$.  The use of doubly infinite sequences of this form to describe random surfaces is discussed in more detail in \cite{sheffield2011quantum}.  It is easy to see from this construction that there is almost surely a unique infinite simple path in the tree in the UIPTM that extends from the seed vertex to infinity.

\begin{proposition} \label{prop::UIPTMwilson}
If one samples a UIPTM (which is an infinite rooted planar map $M$ endowed with a spanning tree $T$ and a root vertex $v$) and then samples a tree $T'$ on $M$ according to Wilson's algorithm, then the law of $(M,T',v)$ is again that of a UIPTM.
\end{proposition}

\begin{proof}
It is shown in Theorem 5.6 of \cite{benjamini2001special} that for any recurrent graph the tree generated by Wilson's algorithm (with any choice of vertex order) agrees in law with the so-called wired spanning forest, and also with the so-called free spanning forest.  In particular, this implies that Wilson's algorithm determines a unique random tree on $M$ (independent of the vertex order) and we just have to show that the law of this tree agrees with the conditional law of $T$ given $M$.

Let $M_n$ be the random tree-decorated rooted planar map obtained with $n$ edges and $m=n/2$ vertices, $T_n$ the corresponding spanning tree, and $v_n$ the corresponding seed vertex.  The proposition will follow from the fact that $(M_n, T_n, v)$ converges in law to $(M,T,v)$, that $M$ is almost surely recurrent, and that for any $n$ one can first sample $(M_n, v)$ and then use Wilson's algorithm to sample~$T_n$.

To explain this in more detail, note that it suffices to show that for any $N>0$ the law of $(M,T)$ restricted to the ball of radius $N$ about $v$ agrees with the law of $(M,T')$ restricted to the ball of radius $N$ about $v$.  Now, the recurrence of $M$ implies that for any $\delta > 0$ we can choose $N'$ large enough so that if we run Wilson's algorithm starting at all points within $B(v,N)$, to obtain the shortest tree path from each of these points to $v$, we find that the probability that any of these paths reaches distance $N'$ from $v$ is at most $\delta$.

Now take $n$ large enough so that $(M_n,T_n,v_n)$ and $(M,T,v)$ can be coupled in such a way that their restrictions to the $N'$ ball about their origin vertices agree with probability at least $1-\delta$.  It then follows that we can couple $(M_n, T_n, v_n)$ with $(M,T',v)$ so that they agree within a radius $N$ ball of their origin vertices with probability at least $1-2\delta$.  Since $\delta$ can be made arbitrarily small (by taking $N'$ large enough) this completes the argument.\end{proof}

Now, based on Proposition~\ref{prop::UIPTMwilson}, we find that the law of the branch of the tree from the origin to $\infty$ can be obtained as a limit of the law of the loop erased random walk from $w$ to $v$, as the distance from $w$ to $v$ tends to $\infty$.  In particular, the limit of the harmonic measure of the possible next edges to be added to this LERW (as measured from $w$ as this distance from $v$ to $w$ tends to infinity) exists, and one can grow the branch from $v$ to $\infty$ one step at a time by sampling from the tip according to this measure, using the procedure indicated in Figure~\ref{treesplicing}.

Noting that $(M,T,v)$ is the limit of the $(M_n, T_n, v_n)$, we also find that the conditional law of the location of the tip (given the grey polygon and the map but not tip location) is given by harmonic measure, and hence we can obtain the infinite volume analog of
Proposition~\ref{prop::finitedlaresult} using the same argument used in the proof of Proposition~\ref{prop::finitedlaresult}.

\subsection{``Capacity'' time parameterization} \label{subsec::capacitytime}

In each of the models in this section, for any edge $e$ on the boundary of the cluster, we can let $b(e)$ denote the harmonic measure at edge $e$ as viewed from the target.  When considering possible scaling limits of the discrete models in this section, we should keep in mind that heuristically the ``capacity'' added to the cluster by putting in the new edge should be roughly proportional to $b(e)^2$.  (In the continuum, drawing a slit of length $\epsilon$ from $\partial \D$ towards the origin changes the conformal radius of the remaining domain viewed from the origin by order $\epsilon^2$.)  Thus, the amount of ``capacity'' time corresponding to a given step in a discrete model is random.  One might therefore try to reparameterize time in the discrete models in such a way that one might expect to obtain a scaling limit parameterized by capacity (i.e., negative the log conformal radius).

One way to do this with the $\eta$-DBM model is as follows: suppose that $b(e)$ represents the harmonic measure at an edge $e$ viewed from the target.  Then at each step, we choose a new edge to add with probability proportional $b(e)^{2+\eta}$, but then after choosing the new edge we toss a coin that is heads with probability proportional to $b(e)^{-2}$, and we only add the edge to the cluster if the coin comes up heads.  This construction is equivalent to the usual $\eta$-DBM model (up to a random time change) but now the expected amount of capacity we add to the cluster at each time step is of the same order.  Another approach is to say that after an edge is selected, instead of flipping a coin that is heads with probability proportional to $b(e)^{-2}$, we simply only add a ``$b(e)^{-2}$ sized portion of the edge'' (i.e., we don't consider an edge to have been ``added'' until it has been hit multiple times, and the sum of all of these fractional contributions exceeds some large constant).

We mention these alternatives, because the approximations to the continuum construction of $\QLE$ we present in this paper will involve random increments of constant capacity (i.e., constant change to the log conformal radius), and the scaling limit will be parameterized by capacity.  One could modify the continuum construction (adding increments of constant quantum length instead of constant capacity) but this will not be our first approach.

%% file: tex/continuum_constructions.tex
The $\QLE$ dynamics described in Figure~\ref{fig::QLEtriangle} involves two parameters: $\gamma$ and $\alpha$.  Here $\gamma$ describes the type of LQG surface on which the growth process takes place and $\alpha$ determines the multiple of~$\Fh_t$ used in the exponentiation that generates~$\nu_t$.  As discussed in Section~\ref{subsec::mainresults}, once one has a solution to the dynamics for a given $\alpha$ and $\gamma$ pair, one can seek to verify that the solution satisfies $\eta$-DBM scaling, as defined in Definition~\ref{def::eta_scaling}, for some value of $\eta$.

It is natural to wonder whether, for each $\gamma$ value, there is a one-to-one correspondence between $\alpha$ and $\eta$ values (at least over some range of the parameters).  This is not a question we will settle in this paper, as we will only construct (and determine $\alpha$ and $\eta$ for) certain families of QLE processes, and these correspond to points on the curves in Figure~\ref{fig::etavsgamma}.

However, in Section~\ref{subsec::scaling} we will propose a relationship between $\alpha$, $\beta$, $\gamma$, and $\eta$ where~$\beta$ (introduced below) is an additional parameter that appears in the regularization used to make sense of $e^{\alpha \Fh_t}$, and in some sense encodes how fast $e^{\alpha \Fh_t}$ blows up near $\partial \D$.

In full generality, this calculation should be taken as a heuristic (since we do not know that $\beta$ is defined for general solutions to the QLE dynamics) but it can be made rigorous under some assumptions --- for example, if one assumes that the stationary law of $\Fh_t$ is given by a free boundary Gaussian free field (restricted to $\partial \D$ and harmonically extended to $\D$).  This latter assumption will turn out to imply that $\beta = \alpha^2$ and hence (for each fixed $\gamma$) it determines a relationship between $\alpha$ and $\eta$.  This assumption turns out to hold for the solutions we construct from the quantum zipper (corresponding to the upper two curves in Figure~\ref{fig::etavsgamma}) and this gives us a way to recover $\eta$ from $\alpha$ for these solutions, as we discuss in Section~\ref{subsec::zipperspecialization}.

In Section~\ref{subsec::free_boundary_scaling} we will argue that when $\eta = 0$ the $\beta = \alpha^2$ assumption leads to a prediction of the dimension for the $\gamma$-LQG surfaces when these surfaces are understood as metric spaces.  (We stress that endowing a $\gamma$-LQG surface with a metric space structure has never been done rigorously, but we {\em believe} that such a metric should exist and that a ball in this metric, whose radius increases in time, should be described by a $\QLE(\gamma^2, 0)$ process.)  The dimension prediction we obtain agrees with a prediction made in the physics literature by Watabiki in \cite{watabiki1993analytic}.  (The fact that our formula agrees with Watabiki's derivation was pointed out to us by Duplantier.)  As mentioned above, the $\beta = \alpha^2$ assumption would hold if the stationary law of $\Fh_t$ for a $\QLE(\gamma^2, 0)$ process were given by a free boundary Gaussian free field (harmonically extended from $\partial \D$ to $\D$).  However, we do not currently have a compelling heuristic to suggest why a stationary law for $\QLE(\gamma^2, 0)$ should have this form.

\subsection{Scaling exponents: a relationship between $\alpha$, $\beta$, $\gamma$, $\eta$}
\label{subsec::scaling}

The caption to Figure \ref{fig::QLEtriangle} describes a particular way to make sense of the map from $\Fh_t$ to $\nu_t$.  Precisely, we let $\nu_t$ be the $n \to \infty$ limit of the measures $e^{\alpha \Fh^n_t(u)} du$ on $\partial \D$, normalized to be probability measures; recall that the $\Fh^n_t$ are obtained by throwing out all but the first $n$ terms in the power series expansion of the analytic function with real part $\Fh_t$.  (This can be understood as a projection of the GFF onto a finite dimensional subspace.)

Instead of using the power series approximations or other projections of the GFF onto finite dimensional subspaces, another natural approach would be to use approximations $\Fh^\epsilon_t$ to $\Fh_t$ defined by ``something like'' convolving $\Fh_t$ with a bump function supported (or mostly supported) on an interval with length of order $\epsilon$.  For example, we could write $\Fh^\epsilon_t(u) = \Fh_t\bigl((1-\epsilon)u\bigr)$ for each $u \in \partial \D$.  (This is equivalent to convolving with a bump function related to the Poisson kernel.) Or we could let $\Fh^\epsilon_t(u)$ be the mean of $\Fh_t$ on $\partial B(u,\epsilon) \cap \D$ for each $u \in \partial \D$.  To describe another approach (which involves more of the unexplored field than just the harmonic projection), let us simplify notation for now by writing $h$ for the sum of $\Fh_t$ and an independent zero boundary GFF on $\D$ and let $\Fh_t^\epsilon = h^\epsilon$ be the mean value of $h$ on $\partial B(u,\epsilon) \cap \D$.  (The latter definition of $\Fh_t^\epsilon$ is essentially what is used in  \cite{ds2011kpz} to define boundary measures when $h$ is an instance of the free boundary GFF.)

For now, let us assume that the boundary values of $h$ are such that it is possible to make sense of the average $h^\epsilon(z)$ of $h$ on $\partial B(z,\epsilon)$ for each $z \in \ol{\D}$ and $\epsilon > 0$.  We also assume that $h^\epsilon(u)$ blows up to $\pm \infty$ almost surely for each $u \in \partial \D$ as $\epsilon \to 0$ (as is the case when $h$ is given by the form of the free boundary GFF considered in Theorem~\ref{thm::existence}).  Now, let us assume we have a constant $\beta$ such that the following limit exists and is almost surely a non-zero finite measure:
\begin{equation}
\label{eqn::boundary_measure_limit}
\nu_h = \lim_{\epsilon \to 0} \epsilon^\beta e^{\alpha h^\epsilon(u)} du.
\end{equation}
(This limit turns out not to depend on the zero-boundary GFF used in the definition of $h$ \cite{ds2011kpz}.)
In a sense, $\beta$ encodes the growth rate of $e^{\alpha \Fh_t}$ near $\partial \D$.  Note that when describing the dynamics of Figure \ref{fig::QLEtriangle}, we avoided having to specify a regularizing factor such as $\epsilon^\beta$ (or an analogous factor depending on $n$) because we normalized to make each approximation a probability measure.

In the case that $h$ is the free boundary GFF and $\alpha \in (-1,1)$ so that $\nu_h$ is given by the $2\alpha$-LQG boundary measure, $\beta$ is given by $(2\alpha)^2/4 = \alpha^2$ \cite{ds2011kpz}. \footnote{In \cite[Section~6]{ds2011kpz}, the existence of the limit \eqref{eqn::boundary_measure_limit} is proved when $h$ is given by the free boundary GFF on a domain with piecewise linear boundary while here we are taking our domain to be $\D$.  It is easy to see, however, that the argument of \cite{ds2011kpz} also goes through in the case that the domain is $\D$.}

For $\gamma  > 0$ given, let $Q_\gamma = 2/\gamma + \gamma/2$.  Recall that $Q_\gamma$ is the factor the appears in front of the $\log$-derivative in the $\gamma$-LQG coordinate change described in \eqref{eqn::LQGcoordinatetransformation}.  We are going to derive the following relationship between $\alpha$, $\gamma$, $\eta$, and~$\beta$:
\begin{align}
\label{eqn::alphabetaetagamma}
\alpha Q_\gamma = \beta - \eta -1.
\end{align}
Once three of the variables $\alpha$, $\beta$, $\gamma$, and $\eta$ are fixed we can use \eqref{eqn::alphabetaetagamma} to determine the fourth.  Moreover, once $\gamma$ is fixed, \eqref{eqn::alphabetaetagamma} gives an affine relationship between $\alpha$, $\beta$, and $\eta$.

Let $\psi \colon \D \to \wt{D}$ be a conformal change of coordinates.  Let
\begin{align}
\wt Q :=& \frac{1}{\alpha} + \frac{\beta}{\alpha} = \frac{1+\beta}{\alpha} \notag
\intertext{and let $\wt{h}$ be the distribution on $\wt{D}$ given by}
 \wt{h} =& h \circ \psi^{-1} + \wt{Q} \log |(\psi^{-1})'|. \label{eqn::tilde_h}
\end{align}
Let $\nu_{\wt{h}}$ be the boundary measure as in \eqref{eqn::boundary_measure_limit} defined in terms of $\wt{h}$.  Then it is not hard to see (at least if $\psi$ is linear) from \eqref{eqn::boundary_measure_limit} that $\nu_{\wt{h}}$ is almost surely the image under $\psi$ of $\nu_{h}$.   That is, $\nu_{h}(A) = \nu_{\wt{h}}(\psi(A))$ for $A \subseteq \partial \D$.  To see this, observe that $e^{\alpha \tilde Q \log |(\psi^{-1})'|} = |(\psi^{-1})'|^{1+\beta}$, which is $|(\psi^{-1})'|$ (the ordinary coordinate change term) times $|(\psi^{-1})'|^\beta$.

When $\alpha =\gamma/2$ and $\beta = \alpha^2$, the definition \eqref{eqn::tilde_h} is the same as the usual change of coordinates formula for the LQG boundary measure \cite{ds2011kpz}.

Let $\nu_\gamma$ be the measure on $\partial \wt{D}$ which is constructed by replacing $\wt{Q}$ in the definition~\eqref{eqn::tilde_h} of $\wt{h}$ with $Q_\gamma$.  Replacing $\wt{Q}$ with $Q_\gamma$ makes it so that the change of coordinates by $\psi$ preserves the $\gamma$-LQG boundary measure defined from $h$ (as opposed to the boundary measure with scaling exponent $\beta$ as defined in~\eqref{eqn::boundary_measure_limit}).  Then the Radon-Nikodym derivative between $\nu_\gamma$ and $\nu_{\wt{h}}$ is (formally) given by a constant times
\[ \exp(\alpha(Q_\gamma - \wt Q) \log |(\psi^{-1})'|) = |(\psi^{-1})'|^{\alpha(Q_\gamma - \wt Q)}.\]

The application of the conformal transformation $\psi$ scales the harmonic measure of a small region near $\partial \D$ by the factor $|\psi'|$.  Recalling the discussion in the caption of Figure~\ref{fig::etascaling}, we want $\nu_\gamma$ to be given by scaling $\nu_{\wt{h}}$ by the factor $|\psi'|^{2+\eta}$.  We therefore want
\[ -\alpha(Q_\gamma - \wt Q) = 2+\eta.\]
Plugging in the definition for $\wt{Q}$, we have $-\alpha Q_\gamma + 1 + \beta = 2+\eta$.  Rearranging gives~\eqref{eqn::alphabetaetagamma}.

\subsection{Free boundary GFF and quantum zipper $\alpha$} \label{subsec::zipperspecialization}

Fix $\gamma \in (0,2]$.  Using the quantum zipper machinery, we will find in later sections that it is natural to consider a setting in which $\beta = \alpha^2$ and we have one additional constraint, namely, $\alpha \in \{-\gamma/4, -1/\gamma \}$.  These two facts and \eqref{eqn::alphabetaetagamma} together imply the relationship between $\eta$ and $\gamma$ described by the upper two curves in Figure~\ref{fig::etavsgamma}\footnote{There is also another heuristic way to determine what $\alpha$ must be when $\eta$ and $\gamma$ are given (in the case that $\Fh_0$ is a harmonically projected GFF, so that $\beta = \alpha^2$), which would give an alternate derivation of \eqref{eqn::alphabetaetagamma}.  This heuristic was shown to us by Bertrand Duplantier.  Consider the discrete $\eta$-DBM interpretation described in Section 2 in which one samples a boundary face (or edge) of the planar map from harmonic measure to the $\eta + 2$ power, and then adds a unit of capacity near the chosen face. Recall that the measure that assigns a unit mass to each face is (conjecturally) supposed to have approximately the form $e^{\gamma h(z)}dz$ for a type of free boundary GFF $h$.  Now, what does the field look like near a ``typical'' face chosen from harmonic measure to the $\eta + 2$ power?  According to the KPZ formalism as applied to ``negative dimensional'' sets (see the discussion in \cite{ds2011kpz} on non-intersecting Brownian paths), if the face is centered at a point $u$, then the field near $u$ should look approximately like an ordinary free boundary GFF plus $2\alpha \log|u - \cdot|$, where $\alpha$ and $\eta$ are related in precisely the manner described here.  We hope to explain this point in more detail in a future joint work with Duplantier.}.  Very roughly speaking, the reason is that for these values the $2\alpha$-LQG boundary measure is supported on ``thick points'' $u$ near which the field behaves like $-2\alpha \log|u - \cdot|$ where $2\alpha \in \{-\gamma/2, -2/\gamma\}$ (see \cite[Proposition~3.4]{ds2011kpz} for the bulk version of this statement as well as Proposition~\ref{prop::free_boundary_weighted} below for the version which will be relevant for this article), and these values have the form $-2/\sqrt{\kappa}$ for $\kappa \in \{ 16/\gamma^2, \gamma^2 \}$, which correspond to the singularities that appear in the capacity invariant quantum zipper.

In these settings we will also find a stationary law of $\Fh_t$ given by the harmonic extension of the boundary values of a form of the free boundary GFF on $\D$, and as mentioned earlier, in this setting one has $\beta = \alpha^2$.

If we plug in $\alpha = -1/\gamma$ and $\beta = \alpha^2$ into~\eqref{eqn::alphabetaetagamma} then we obtain:
\[ -\frac{1}{\gamma} \left(\frac{2}{\gamma} + \frac{\gamma}{2}\right) = \frac{1}{\gamma^2} -\eta - 1,\]
or equivalently
\begin{equation}
\label{eqn::gammaetaupper}
\eta = \frac{3}{\gamma^2} - \frac{1}{2}.
\end{equation}
This describes the upper curve in Figure~\ref{fig::etavsgamma}

If we plug $\alpha = -\gamma/4$ and $\beta = \alpha^2$ into~\eqref{eqn::alphabetaetagamma} then we obtain
\[ -\frac{\gamma}{4}\left(\frac{2}{\gamma}+\frac{\gamma}{2}\right) = \frac{\gamma^2}{16 }-\eta- 1,\]
or equivalently
\begin{equation}
\label{eqn::gammaetaupper}
\eta = \frac{3 \gamma^2}{16} - \frac{1}{2}.
\end{equation}
This describes the middle curve in Figure~\ref{fig::etavsgamma}.  Note that the lower curve in Figure~\ref{fig::etavsgamma} corresponds to $\alpha = \beta = 0$ and $\eta = -1$, which is trivially a solution to~\eqref{eqn::alphabetaetagamma} for any $\gamma$.

\subsection{Free boundary scaling $\beta = \alpha^2$ and $\eta = 0$} \label{subsubsec::betaalpha^2}
\label{subsec::free_boundary_scaling}

\begin{figure}[ht!]
\begin{center}
\includegraphics[width=4.0in]{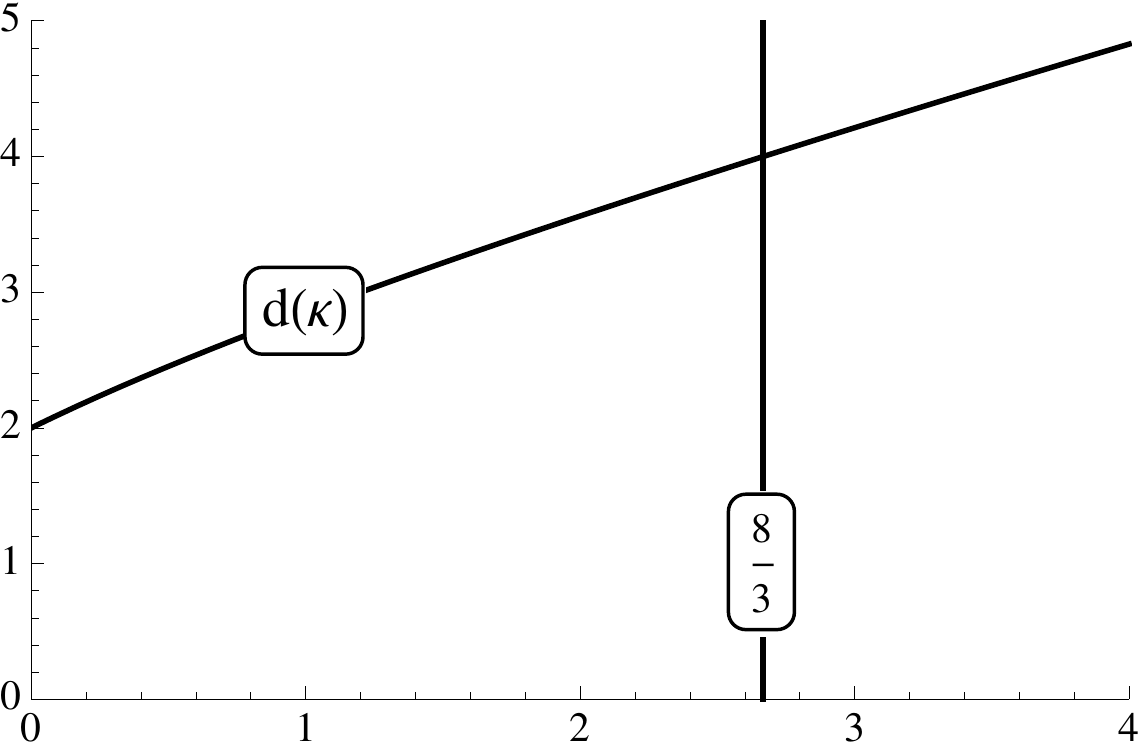}
\caption{ \label{fig::dversuskappaplot}The value $d$ as a function of $\kappa=\gamma^2$, as defined by~\eqref{eqn::heuristicd}.  Although the graph is not a straight line, it appears ``almost straight'' and it takes the value $2$ for $\kappa = 0$ and $4$ for $\kappa = 8/3$.}
\end{center}
\end{figure}

In Theorem~\ref{thm::existence} we prove the existence of stationary $\QLE(\gamma^2,\eta)$ processes for $(\gamma^2,\eta)$ pairs which are on one of the upper two curves in Figure~\ref{fig::etavsgamma} with $\beta = \alpha^2$.  It is natural to wonder whether this is just a coincidence, or whether there are other $(\gamma^2, \eta)$ pairs for which there exist $\QLE$ solutions with $\beta = \alpha^2$.  (This would be the case, for example, if the $\Fh_t$ turned out to have stationary laws described by the harmonic extension of the boundary values from $\partial \D$ to $\D$ of a form of the free boundary GFF.)  We observe that if we simply plug in $\beta = \alpha^2$, then~\eqref{eqn::alphabetaetagamma} becomes
\begin{align*}
\alpha Q_\gamma &= \alpha^2-\eta-1,
\intertext{or equivalently}
\eta &= \alpha^2 - \alpha Q_\gamma -1.
\end{align*}
One can also solve this for $\alpha$ to obtain
\[\alpha = \frac{Q_\gamma \pm \sqrt{Q_\gamma^2 + 4 + 4\eta}}{2}.\]

We now introduce a parameter $d = -\gamma/\alpha$, which can interpreted as a sort of ``dimension'', at least in the $\eta = 0$ case.\footnote{One way to define the dimension of a metric space is as the value $d$ such that the number of radius $\delta$ balls required to cover the space scales as $\delta^{-d}$.  (Hausdorff dimension is a variant of this idea.)  If the metric space comes endowed with a measure (and is homogenous, in some sense) then one might guess that each of these balls would have area of order $\delta^d$.  In fact, if there is a natural notion of ``rescaling'' the metric space so that its diameter changes by a factor of $\delta$ (and the measure is also defined for the rescaled version), then one can define $d$ to be such that the area scales as $\delta^d$.  In the QLE setting with $\eta = 0$, if we consider a small neighborhood $U$ of a point $u \in \partial \D$, and we rescale the quantum surface restricted to $U$ (by modifying $h$ on $U$) then we expect the ``length of time it takes a QLE to traverse $U$'' to scale by approximately the same factor as the diameter of $U$ (assuming a metric space structure on the quantum surface is defined).}
Let $A = e^{\gamma C}$.  This represents the factor by which the $\gamma$-LQG area in a small neighborhood of a boundary point $u \in \partial \D$ changes when we add a function to $h$ that is equal to a constant $C$ in that neighborhood.  Then $e^{\alpha C}$ represents the factor by which the $\QLE$ driving measure changes, which suggests that the time it takes to traverse the neighborhood should scale like $T = e^{-\alpha C}$.  Now $d$ is the value such that $A = T^{-\gamma/\alpha} = T^d$.

Computing this, we have
\begin{align*}
     d
& = -\frac{2\gamma}{Q_\gamma \pm \sqrt{Q^2_\gamma+4+4\eta}}
   = \frac{2\gamma\left( Q_\gamma \pm \sqrt{Q^2_\gamma+4+4\eta}\right)}{4+4\eta} \\
&= \frac{1}{1+\eta}\left(1 + \frac{\gamma^2}{4} \pm \sqrt{\frac{\gamma^4}{16} + \frac{3\gamma^2}{2} +1 + \eta \gamma^2}\right).
\end{align*}

Setting $\kappa = \gamma^2 \in (0,4)$ this is equivalently equal to
\[  \frac{1}{1+\eta}\left( 1 + \frac{\kappa}{4} \pm \sqrt{\frac{\kappa^2}{16} + \frac{3\kappa}{2} + 1 + \eta\kappa} \right).\]
In the case $\eta = 0$, the positive root can be written as
\begin{equation}
\label{eqn::heuristicd}
d =1 + \frac{\kappa}{4} + \frac{1}{4} \sqrt{(4+\kappa)^2 + 16 \kappa}.
\end{equation}
The graph of $d$ as a function of $\kappa = \gamma^2$ is illustrated in Figure \ref{fig::dversuskappaplot}.  The plot matches a physics literature prediction made by Watabiki in 1993 for the fractal dimension of $\gamma$-LQG quantum gravity when understood as a metric space \cite[Equation~(5.13)]{watabiki1993analytic}.\footnote{The quantities $\alpha_1$ and $\alpha_{-1}$ which appear in \cite[Equation~(5.13)]{watabiki1993analytic} are defined in \cite[Equation~(4.15)]{watabiki1993analytic}.  These, in turn, are defined in terms of the central charge $c$.  The central charge $c$ corresponding to an $\SLE_\kappa$ is $(8-3\kappa)(\kappa-6)/(2\kappa)$; see the introduction of \cite{lsw2003chordal}.}  However, we stress again that our calculation was made under the assumption that $\beta = \alpha^2$, and that we do not currently have even a heuristic argument for why there should exist QLE processes satisfying this relationship for $\eta = 0$ and a given $\gamma \in (0,2]$ (though of course the reader may consult the explanation given in \cite{watabiki1993analytic}).  The exception is the case $\gamma = \sqrt{8/3}$, since $(8/3,0)$ is one of the $(\gamma^2,\eta)$ pairs for which we construct solutions to the QLE dynamics.  In this case, our arguments do support the notion the Hausdorff dimension of Liouville quantum gravity should be $4$ for $\gamma = \sqrt{8/3}$, though we will not prove this statement in this paper.  This is consistent with the dimension of the Brownian map \cite{MR2031225,lg2007topological}

%% file: tex/preliminaries.tex
\subsection{Forward and reverse radial $\SLE_\kappa$}
\label{subsec::sle}

For $u \in \partial \D$ and $z \in \D$, let
\begin{equation}
\label{eqn::radial_functions}
\Psi(u,z) = \frac{u+z}{u-z} \quad \text{and} \quad\Phi(u,z)  = z \Psi(u,z).
\end{equation}
A radial $\SLE_\kappa$ in $\D$ starting from~$1$ and targeted at~$0$ is described by the random family of conformal maps obtained by solving the radial Loewner ODE
\begin{equation}
\label{eqn::radial_loewner}
\dot{g}_t(z) = \Phi(e^{i W_t}, g_t(z)),\quad g_0(z) = z
\end{equation}
where $W_t = \sqrt{\kappa} B_t$ and $B$ is a standard Brownian motion.  We refer to $e^{iW_t}$ as the {\bf driving function} for $(g_t)$.  For each $z \in \D$ let $\tau(z) = \sup\{t \geq 0: |g_t(z)| < 1\}$ and write $K_t := \{z \in \D : \tau(z) \leq t\}$.  For each $t \geq 0$, $g_t$ is the unique conformal map which takes $\D_t := \D \setminus K_t$ to $\D$ with $g_t(0) = 0$ and $g_t'(0) > 0$.  Time is parameterized by $\log$ conformal radius so that $g_t'(0) = e^t$ for each $t \geq 0$.  Rohde and Schramm showed that there almost surely exists a curve~$\eta$ (the so-called $\SLE$ {\bf trace}) such that for each $t \geq 0$ the domain $\D_t$ of $g_t$ is equal to the connected component of $\D \setminus \eta([0,t])$ which contains $0$.  The (necessarily simply connected and closed) set $K_t$ is called the ``filling'' of $\eta([0,t])$ \cite{rs2005sle}.

In our construction of $\QLE$, it will sometimes be more convenient to work with \emph{reverse} radial $\SLE_\kappa$ rather than \emph{forward} radial $\SLE_\kappa$ (as defined in \eqref{eqn::radial_loewner}).  A reverse $\SLE_\kappa$ in  $\D$ starting from~$1$ and targeted at~$0$ is the random family of conformal maps obtained by solving the reverse radial Loewner ODE
\begin{equation}
\label{eqn::reverse_radial_loewner}
\dot{g}_t(z) = -\Phi(e^{i W_t}, g_t(z)),\quad g_0(z) = z
\end{equation}
where $W_t = \sqrt{\kappa} B_t$ and $B$ is a standard Brownian motion.  As in the forward case, we refer to $e^{i W_t}$ as the {\bf driving function} for $(g_t)$.

\begin{remark}
\label{rem::forward_reverse_radial}
Forward and reversal radial $\SLE_\kappa$ are related in the following manner.  Suppose that $(g_t)$ solves the reverse radial Loewner equation \eqref{eqn::reverse_radial_loewner} with driving function $W_t = \sqrt{\kappa} B_t$ and $B$ a standard Brownian motion.  Fix $T > 0$ and let $f_t = g_{T-t}$ for $t \in [0,T]$. Then $(f_t)$ solves the forward radial Loewner equation with driving function $t \mapsto W_{T-t}$ and with initial condition $f_0(z) = g_T(z)$.  Equivalently, we can let $q_t$ for $t \in [0,T]$ solve \eqref{eqn::radial_loewner} with driving function $t \mapsto W_{T-t}$ and $q_0(z) = z$ and then take $f_t = q_t \circ g_T$.  Indeed, this follows from standard uniqueness results for ODEs.  Then
\[ z = g_0(z) = f_T(z) = q_T(g_T(z)).\]
That is, $q_T$ is the inverse of $g_T$.  This implies that the image of $g_T$ can be expressed as the complementary component containing zero of an $\SLE_\kappa$ curve $\eta_T$ in $\D$ drawn up to $\log$ conformal radius time $T$.  We emphasize here that the path $\eta_T$ changes with~$T$.  (See the end of the proof of \cite[Theorem~4.14]{lawler2005conformally} for a similar discussion.)
\end{remark}

\subsection{Gaussian free fields}
\label{subsec::gff}

We will now describe the construction of the two-dimensional GFF (with either Dirichlet, free, or mixed boundary conditions) as well as some properties that will be important for us later.  We refer the reader to \cite{sheffield2007gff} as well as \cite[Section~3]{sheffield2010weld} for a more detailed introduction.

\subsubsection{Dirichlet inner product}
\label{subsubsec::gff_dirichlet}

Let $D$ be a domain in $\C$ with smooth, {\bf harmonically non-trivial} boundary.  The latter means that the harmonic measure of $\partial D$ is positive as seen from any point in $D$.  Let $C_0^\infty(D)$ denote the set of $C^\infty$ functions compactly supported in~$D$.  The {\bf Dirichlet inner product} is defined by
\begin{equation}
\label{eqn::dirichlet}
(f,g)_\nabla = \frac{1}{2\pi} \int_D \nabla  f(x) \cdot \nabla g(x) dx \quad \text{for}\quad f,g \in C_0^\infty(D).
\end{equation}
More generally, \eqref{eqn::dirichlet} makes sense for $f,g \in C^\infty(D)$ with $L^2$ gradients.

\subsubsection{Distributions}
\label{subsubsec:gff_spaces}

We view $C_0^\infty(D)$ as a space of test functions and equip it with the topology where a sequence $(\phi_k)$ in $C_0^\infty(D)$ satisfies $\phi_k \to 0$ if and only if there exists a compact set $K \subseteq D$ such that the support of $\phi_k$ is contained in $K$ for every $k \in \N$ and $\phi_k$ as well as all of its derivatives converge uniformly to zero as $k \to \infty$.  A {\bf distribution} on $D$ is a continuous linear functional on $C_0^\infty(D)$ with respect to the aforementioned topology.  A {\bf modulo additive constant distribution} on $D$ is a continuous linear functional on the subspace of functions $f$ of $C_0^\infty(D)$ with $\int_D f(x) dx = 0$ with the same topology.

\subsubsection{GFF with Dirichlet and mixed boundary conditions}
\label{subsubsec::gff_dirichlet}

We let $H_0(D)$ be the Hilbert-space closure of $C_0^\infty(D)$ with respect to the Dirichlet inner product \eqref{eqn::dirichlet}.  The GFF $h$ on $D$ with zero Dirichlet boundary conditions can be expressed as a random linear combination of an $(\cdot,\cdot)_\nabla$-orthonormal basis $(f_n)$ of $H_0(D)$:
\begin{equation}
\label{eqn::gff_series}
h = \sum_n \alpha_n f_n,\quad (\alpha_n) \quad \text{i.i.d.}\quad N(0,1).
\end{equation}
Although this expansion of $h$ does not converge in $H_0(D)$, it does converge almost surely in the space of distributions or (when $D$ is bounded) in the fractional Sobolev space $H^{-\epsilon}(D)$ for each $\epsilon > 0$ (see \cite[Proposition 2.7]{sheffield2007gff} and the discussion thereafter).  If $f,g \in C_0^\infty(D)$ then an integration by parts gives $(f,g)_\nabla = -(2\pi)^{-1} ( f,\Delta g)$.  Using this, we define
\[ (h,f)_\nabla = -\frac{1}{2\pi}(h, \Delta f) \quad\text{for}\quad f \in C_0^\infty(D).\]
Observe that $(h,f)_\nabla$ is a Gaussian random variable with mean zero and variance $(f,f)_\nabla$.  Hence $h$ induces a map $C_0^\infty(D) \to \CG$, $\CG$ a Gaussian Hilbert space, that preserves the Dirichlet inner product.  This map extends uniquely to $H_0(D)$ and allows us to make sense of $(h,f)_\nabla$ for all $f \in H_0(D)$ and, moreover,
\[
\cov((h,f)_\nabla,(h,g)_\nabla) = (f,g)_\nabla \quad\text{for all}\quad f,g \in H_0(D).
\]

For fixed $x \in D$ we let $\wt{G}_x(y)$ be the harmonic extension of $y \mapsto -\log|x-y|$ from $\partial D$ to $D$.
The {\bf Dirichlet Green's function} on $D$ is defined by
\[ G^\dirichlet(x,y) = -\log|y-x| - \wt{G}_x(y).\]
When $x \in D$ is fixed, $G^\dirichlet(x,\cdot)$ may be viewed as the distributional solution to $\Delta G^\dirichlet(x,\cdot) = - 2\pi \delta_x(\cdot)$ with zero boundary conditions.  When $D = \D$, we have that
\begin{equation}
\label{eqn::green_dirichlet_disk}
 G^\dirichlet(x,y) = \log\left|\frac{1-x \ol{y}}{y-x}\right|.
\end{equation}
Repeated applications of integration by parts also imply that
\begin{align*}
      \cov((h,f),(h,g))
&= (2\pi)^2 \cov( (h,\Delta^{-1} f)_\nabla, (h,\Delta^{-1} g)_\nabla)\\
&= \iint_{D \times D} f(x) G^\dirichlet(x,y) g(y) dx dy
\end{align*}
where $G^\dirichlet$ is the Dirichlet Green's function on $D$.  If $h$ is a zero-boundary GFF on $D$ and $F \colon D \to \R$ is harmonic, then $h+F$ is the GFF with Dirichlet boundary conditions given by those of $F$.

More generally, suppose that that $D \subseteq \C$ is a domain and $\partial D = \partial^\dirichlet \cup \partial^\free$ where $\partial^\dirichlet \cap \partial^\free = \emptyset$.  We also assume that the harmonic measure of $\partial^\dirichlet$ is positive as seen from any point $z \in D$.  The GFF on $D$ with Dirichlet (resp.\ free) boundary conditions on $\partial^\dirichlet$ (resp.\ $\partial^\free$) is constructed using a series expansion as in \eqref{eqn::gff_series} except the space $H_0(D)$ is replaced with the Hilbert space closure with respect to $(\cdot,\cdot)_\nabla$ of the subspace of functions in $C^\infty(D)$ which have an $L^2$ gradient and vanish on $\partial^\dirichlet$.  The aforementioned facts for the GFF with only Dirichlet boundary conditions also hold verbatim for the GFF with mixed Dirichlet/free boundary conditions.  In the case that $D$ is a smooth Jordan domain and $\partial^\dirichlet$, $\partial^\free$ are each  non-degenerate intervals of $\partial D$, the Green's function $G$ is taken to solve $\Delta G(x,\cdot) = - 2\pi \delta_x(\cdot)$ with $n \cdot \nabla G(x,\cdot) = 0$ on $\partial^\free$ and $G(x,\cdot) = 0$ on $\partial^\dirichlet$ for $x \in D$.  See also the discussion in \cite[Section~6.2]{ds2011kpz} for the GFF with mixed boundary conditions.

\subsubsection{GFF with free boundary conditions}
\label{subsubsec::gff_free}

The GFF with free boundary conditions on $D \subseteq \C$ is constructed using a series expansion as in \eqref{eqn::gff_series} except we replace $H_0(D)$ with the Hilbert space closure $H(D)$ of the subspace of functions $f \in C^\infty(D)$ with $\| f \|_\nabla^2 := (f,f)_\nabla < \infty$ with respect to the Dirichlet inner product \eqref{eqn::dirichlet}.  Since the constant functions are elements of $C^\infty(D)$ but have $\|\cdot \|_\nabla$-norm zero, in order to make sense of this object, we will work in the space of distributions modulo additive constant.  As in the case of the ordinary GFF, it is not difficult to see that the series converges almost surely in this space.  As in Section~\ref{subsubsec::gff_dirichlet}, we can view $(h,f)_\nabla$ for $f \in H(D)$ as a Gaussian Hilbert space where
\[ \cov( (h,f)_\nabla, (h,g)_\nabla) = (f,g)_\nabla \quad\text{for all}\quad f,g \in H(D).\]
Note that we do not need to restrict to mean zero test functions here due to the presence of gradients.

The {\bf Neumann Green's function} on $D$ is defined by
\[ G^\free(x,y) = -\log|y-x| - \wh{G}_x(y)\]
where for $x \in D$ fixed, $y \mapsto \wh{G}_x(y)$ is the function in on $D$ such that the normal derivative of $G^\free(x,y)$ along $\partial D$ is equal to $1$.
(The reason for the superscript ``$\free$'' is that, as explained below, $G^\free$ gives the covariance function for the GFF with free boundary conditions.)  When $x \in D$ is fixed, $G^\free$ may be viewed as the distributional solution to $\Delta G^\free(x,\cdot) = - 2\pi \delta_x(\cdot)$ where the normal derivative of $G^\free(x,\cdot)$ is equal to $1$ at each $y \in \partial D$.  When $D = \D$, we have that
\begin{equation}
\label{eqn::green_neumann_disk}
   G^\free(x,y)
= -\log \left|(x-y)(1-x\ol{y}) \right|.
\end{equation}
Assuming that $f,g$ have mean zero, repeated applications of integration by parts yield that
\begin{align*}
      \cov((h,f),(h,g))
&= (2\pi)^2 \cov( (h,\Delta^{-1} f)_\nabla, (h,\Delta^{-1} g)_\nabla)\\
&= \iint_{D \times D} f(x) G^\free(x,y) g(y) dx dy
\end{align*}
where $G^\free$ is the Neumann Green's function on $D$.

\subsubsection{Markov property}
\label{subsubsec::gff_markov}

We are now going to explain the Markov property enjoyed by the GFF with Dirichlet, free, or mixed boundary conditions.  For simplicity, for the present discussion we are going to assume that $h$ is a GFF with zero boundary conditions (though the proposition stated below is general and so is the following argument).  Suppose that $W \subseteq D$ with $W \neq D$ is open.  There is a natural inclusion $\iota$ of $H_0(W)$ into $H_0(D)$ where
\[ \iota(f)(x) = \begin{cases} f(x) \text{ if } x \in W,\\ 0 \text{ otherwise.} \end{cases}\]
If $f \in C_0^\infty(W)$ and $g \in C_0^\infty(D)$, then as $(f,g)_\nabla = -(2\pi)^{-1}(f,\Delta g)$ it is easy to see that $H_0(D)$ admits the $(\cdot,\cdot)_\nabla$-orthogonal decomposition $H_0(W) \oplus H_0^\perp(W)$ where $H_0^\perp(W)$ is the subspace of functions in $H_0(D)$ which are harmonic in $W$.  Thus we can write
\[ h = h_W + h_{W^c} = \sum_n \alpha_n^W f_n^W + \sum_n \alpha_n^{W^c} f_n^{W^c}\]
where $(\alpha_n^W),(\alpha_n^{W^c})$ are independent i.i.d.\ sequences of standard Gaussians and $(f_n^W)$, $(f_n^{W^c})$ are orthonormal bases of $H_0(W)$ and $H_0^\perp(W)$, respectively.  Observe that $h_W$ is a zero-boundary GFF on $W$, $h_{W^c}$ is the harmonic extension of $h|_{\partial W}$ from $\partial W$ to $W$, and $h_W$ and $h_{W^c}$ are independent.  We arrive at the following proposition:

\begin{proposition}[Markov Property]
\label{prop::gff_markov}
Suppose that $h$ is a GFF with Dirichlet, free, or mixed boundary conditions.  The conditional law of $h|_W$ given $h |_{ D \setminus W}$ is that of the sum of a zero boundary GFF on $W$ plus the harmonic extension of $h|_{\partial W}$ from $\partial W$ to $W$.
\end{proposition}

The orthogonality of $H_0(W)$ and the set of functions in $H_0(D)$ which are harmonic in $W$ is also proved in \cite[Theorem~2.17]{sheffield2007gff} and it is explained thereafter how this is related to the Markov property of the GFF.

\begin{remark}
\label{rem::free_zero}
Proposition~\ref{prop::gff_markov} implies that if~$h$ is a free boundary GFF on~$D$ then we can write~$h$ as the sum of the harmonic extension of its boundary values from~$\partial D$ to~$D$ and an independent zero boundary GFF in~$D$.
\end{remark}

\begin{remark}
\label{rem::orthogonal_projection}
Proposition~\ref{prop::gff_markov} implies that for each fixed $W \subseteq D$ open we can almost surely define the orthogonal projection of a GFF $h$ onto the subspaces of functions which are harmonic in and supported in $W$.  We will indicate these by $\pHarm(h;W)$ and $\pSupp(h;W)$, respectively.  If $W$ is clear from the context, we will simply write $\pHarm(h)$ and $\pSupp(h)$.
\end{remark}

\subsection{Local sets}
\label{subsec::local_sets}

The theory of local sets, developed in \cite{ss2010contour}, extends the Markovian structure of the field (Proposition~\ref{prop::gff_markov}) to the setting of conditioning on the values it takes on a \emph{random set} $A \subseteq D$.  More precisely, suppose that $(A,h)$ is a coupling of a GFF (with either Dirichlet, free, or mixed boundary conditions)~$h$ on~$D$ and a random variable~$A$ taking values in the space of closed subsets of~$\ol{D}$, equipped with the Hausdorff metric.  Then~$A$ is said to be a {\bf local set} of~$h$ \cite[Lemma 3.9, part (4)]{ss2010contour} if there exists a law on pairs $(A,h_1)$ where $h_1$ takes values in the space of distributions on $D$ with $h_1|_{D \setminus A}$ harmonic is such that a sample with the law $(A,h)$ can be produced by
\begin{enumerate}
\item choosing the pair $(A,h_1)$,
\item then sampling an instance~$h_2$ of the zero boundary GFF on $D \setminus A$ and setting $h=h_1+h_2$.
\end{enumerate}
There are several other characterizations of local sets which are discussed in \cite[Lemma 3.9]{ss2010contour}.  These are stated and proved for the GFF with Dirichlet boundary conditions, however the argument goes through verbatim for the GFF with either free or mixed boundary conditions.

For a given local set~$A$, we will write $\CC_A$ for $h_1$ as above.  We can think of $\CC_A$ as being given by $\pHarm(h;D \setminus A)$.  We can also interpret~$\CC_A$ as the conditional expectation of~$h$ given~$A$ and~$h|_A$.  In the case that $h$ is a GFF with free boundary conditions, $\CC_A$ is defined modulo additive constant.

Throughout this article, we will often work with increasing families of closed subsets $(K_t)_{t \geq 0}$ each of which is local for a GFF $h$.  The following is a restatement of \cite[Proposition~6.5]{ms2010imag1} and describes the manner in which $\CC_{K_t}$ evolves with $t$.  In the following statement, for a domain $U \subseteq \C$ with simply-connected components and $z \in U$, we write $\confrad(z;U)$ for the conformal radius of the component $U_z$ of $U$ containing $z$ as seen from $z$.  That is, $\confrad(z;U) = \phi'(0)$ where $\phi$ is the unique conformal map which takes $\D$ to $U_z$ with $\phi(0) = z$ and $\phi'(0) > 0$.

\begin{proposition}
\label{prop::cond_mean_continuous}
Suppose that $D \subseteq \C$ is a non-trivial simply connected domain.  Let $h$ be a GFF on $D$ with either Dirichlet, free, or mixed boundary conditions.  Suppose that $(K_t)_{t \geq 0}$ is an increasing family of closed sets such that $K_\tau$ is local for $h$ for every $(K_t)$ stopping time $\tau$ and $z \in D$ is such that $\confrad(z;D \setminus K_t)$ is almost surely continuous and monotonic in $t$.  Then $\CC_{K_t}(z) - \CC_{K_0}(z)$ has a modification which is a Brownian motion when parameterized by $\log \confrad(z;D \setminus K_0)-\log \confrad(z;D \setminus K_t)$ up until the first time $\tau(z)$ that $K_t$ accumulates at $z$.  In particular, $\CC_{K_t}(z)$ has a modification which is almost surely continuous in $t \geq 0$.  (In the case that $h$ has free boundary conditions, we use the normalization $\CC_{K_0}(z) = 0$.)
\end{proposition}

\subsection{Quantum boundary length measures}
\label{subsec::lqg_boundary_measures}

We are now going to summarize a few important facts which are based on the discussion in \cite[Section~6]{ds2011kpz} regarding the Liouville quantum gravity boundary length measure.  Suppose that $h$ is a GFF with mixed Dirichlet/free boundary conditions on a Jordan domain $D \subseteq \C$ where both the Dirichlet and free parts $\partial^\dirichlet$ and $\partial^\free$, respectively, of $\partial D$ are non-degenerate boundary arcs.  In \cite[Theorem~6.1]{ds2011kpz}, it is shown how to construct the measure $\nu_h^\gamma = \exp(\tfrac{\gamma}{2} h(u))du$ on $\partial^\free$ for $\gamma \in (-2,2)$ fixed.  Formally, this means that the Radon-Nikodym derivative of $\nu_h^\gamma$ with respect to Lebesgue measure on $\partial^\free$ is given by $\exp(\tfrac{\gamma}{2}h)$.  This does not make literal sense because $h$ does not take values in the space of functions and, in particular, does not take on a specific value at a given point in $\partial^\free$.  One can make this rigorous as follows.  First, suppose that $\partial^\free$ consists of a single linear segment.  For each $z \in \partial^\free$ and $\epsilon > 0$, let $h_\epsilon(z)$ be the average of $h$ on the semi-circle $\partial B(z,\epsilon) \cap D$ (see \cite[Section~3]{ds2011kpz} for background on the circle average process).  For each $\gamma \in (-2,2)$, the measure $\exp(\tfrac{\gamma}{2}h(u)) du$ is defined as the almost sure limit
\begin{equation}
\label{eqn::boundary_measure}
\nu_h^\gamma = \lim_{\epsilon \to 0} \epsilon^{\gamma^2/4} \exp(\tfrac{\gamma}{2} h_\epsilon(u)) du
\end{equation}
along powers of two as $\epsilon \to 0$ with respect to the weak topology.  Upon showing that the limit in~\eqref{eqn::boundary_measure} exists for linear $\partial^\free$, the boundary measure for other domains is defined via conformal mapping and applying the change of coordinates rule for quantum surfaces.

One can similarly make sense of the limits~\eqref{eqn::boundary_measure} in the case that $h$ has free boundary conditions, i.e.\ $\partial^\dirichlet = \emptyset$.  If we consider $h$ as a distribution defined modulo additive constant, then the measure $\nu_h^\gamma$ will only be defined up to a multiplicative constant.  We can ``fix'' the additive constant in various ways, in which case $\nu_h^\gamma$ is an actual measure.

One also has the following analog of \cite[Proposition~1.2]{ds2011kpz} for the boundary measures associated with the free boundary GFF on $\D$.  Suppose that $(f_n)$ is any orthonormal basis consisting of smooth functions for the Hilbert space used to define $h$.  For each $n \in \N$, let $h^n$ be the orthogonal projection of $h$ onto the subspace spanned by $\{f_1,\ldots,f_n\}$.  One can similarly define $\nu_h^\gamma$ as the almost sure limit
\begin{equation}
\label{eqn::boundary_measure_onb}
 \nu_h^\gamma = \lim_{n \to \infty} \exp\left(\tfrac{\gamma}{2} h^n(u) - \tfrac{\gamma^2}{4} \var(h^n(u))\right) du.
\end{equation}
That these two definitions for $\nu_h^\gamma$ almost surely agree is not explicitly stated in \cite{ds2011kpz} for boundary measures however its proof is exactly the same as in the case of bulk measures which is given in \cite[Proposition~1.2]{ds2011kpz}.

\begin{proposition}
\label{prop::free_boundary_weighted}
Fix $\gamma \in (-2,2)$.  Consider a random pair $(u,h)$ where $u$ is sampled uniformly from $\partial \D$ using Lebesgue measure and, given $u$, the conditional law of $h$ is that of a free boundary GFF on $\D$ plus $-\gamma \log|\cdot-u|$ viewed as a distribution defined modulo additive constant.  Let $\nu_h^\gamma$ denote the $\gamma$ boundary measure associated with $h$.  Then given $h$, the conditional law of $u$ is that of a point uniformly sampled from $\nu_h^\gamma$ ($\nu_h^\gamma$ is only defined up to a multiplicative constant, but can be normalized to be a probability measure).
\end{proposition}

\begin{proof}
Let $A_r$ for $0<r<1$ be the annulus $\D \setminus \overline{B(0,r)}$.  Let $\wt{A}_r$ be the larger annulus $B(0,1/r) \setminus \overline{B(0,r)}$.  Let $dh$ be the law of an instace $h$ be the GFF on $A_r$ with zero boundary conditions on the inner boundary circle $\partial B(0,r)$ and free boundary conditions on $\partial \D$.  Let $\nu_h^\gamma$ denote the corresponding boundary $\gamma$-LQG measure on $\partial \D$.  Then it is not hard to see that the following ways to produce a random pair $u, h$ are equivalent (and very similar statements are proved in \cite[Section 6]{ds2011kpz}):
\begin{enumerate}
\item First sample $u$ uniformly on $\partial \D$ and the let $h$ be a sample from the law described above {\em plus} the deterministic function $f_{u,r}(\cdot) = \gamma G_{\wt{A}_r}(u,\cdot)$ where $\gamma G_{\wt{A}_r}$ is the Green's function on $\wt{A}_r$.  (This function is harmonic on $\wt{A}_r \setminus \{u \}$ except at the point $u$.)
\item First sample $h$ from the measure $\nu_h^\gamma(\partial \D) dh$ and then, conditioned on $h$, sample $u$ from the boundary measure $\nu_h$ (normalized to be a probability measure).
\end{enumerate}
The lemma can be obtained by taking the limit as $r$ goes to zero (with the corresponding $h$ being considered modulo additive constant).  Note that on the set $\partial \D$, the functions $f_{r,u}(\cdot)$ (treated modulo additive constant) converge uniformly to $-\gamma \log|u - \cdot|$ as $r$ tends to zero.
\end{proof}

%% file: tex/reverse_coupling.tex
\begin{figure}
\begin{center}
\includegraphics[scale=0.85]{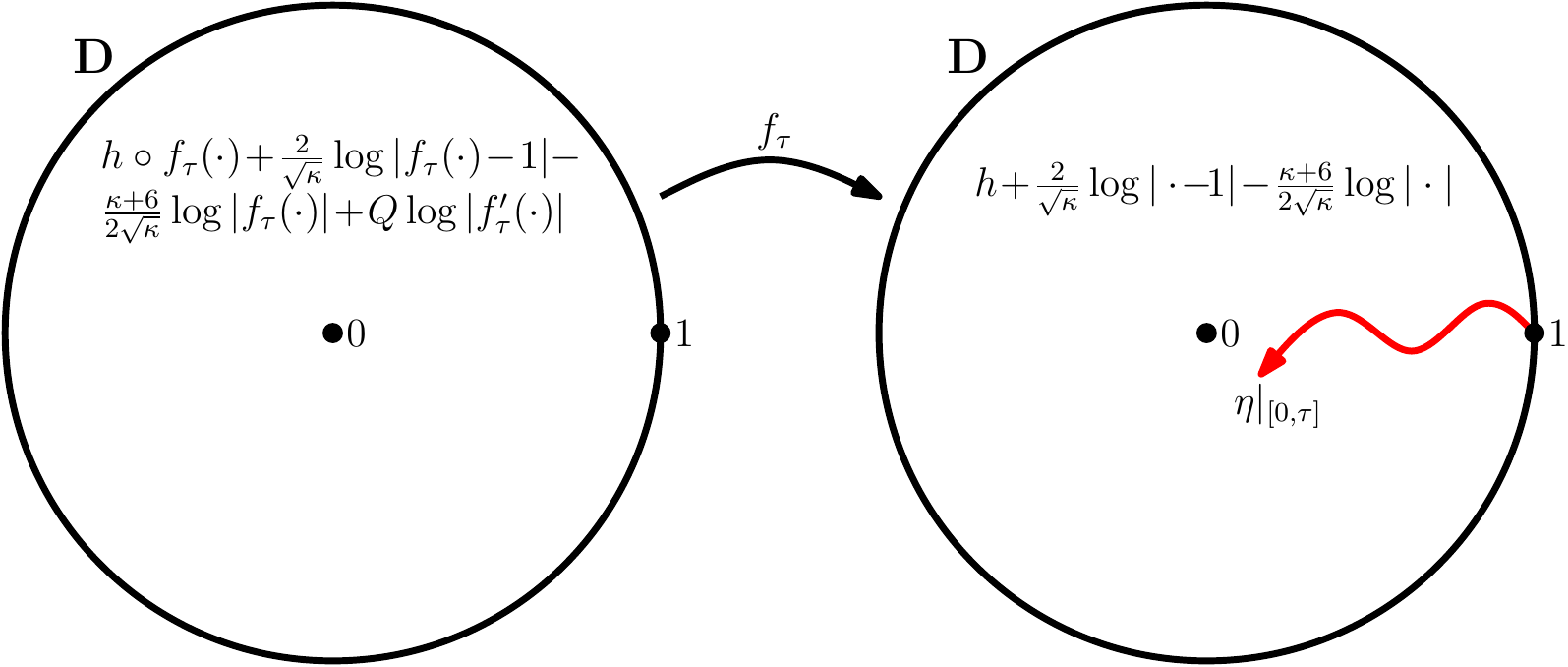}
\end{center}
\caption{\label{fig::reverse_radial_coupling}
Illustration of the coupling of reverse radial $\SLE_\kappa$ in $\D$ starting from~$1$ and targeted at~$0$ with a free boundary GFF $h$ on $\D$.  Here, $Q = 2/\gamma + \gamma/2$ for $\gamma=\min(\sqrt{\kappa},\sqrt{16/\kappa})$ and $f_\tau$ is the centered reverse radial $\SLE_\kappa$ Loewner flow evaluated at a stopping time $\tau$.  Theorem~\ref{thm::radial_coupling_existence} implies that the distributions on the left and right above have the same law.}
\end{figure}

The purpose of this section is to establish the radial version of the reverse coupling of $\SLE_\kappa$ with the free boundary GFF.  It is a generalization of the coupling with reverse chordal $\SLE_\kappa$ with the free boundary GFF established in \cite[Theorem~1.2]{sheffield2010weld}.  Suppose that $B_t$ is a standard Brownian motion, $W_t = \sqrt{\kappa} B_t$, and $U_t = e^{iW_t}$.  Let $(g_t)$ solve the reverse radial Loewner ODE \eqref{eqn::reverse_radial_loewner} driven by $U_t$.  The centered reverse $\SLE_\kappa$ is given by the centered conformal maps $f_t = U_t^{-1} g_t$.  We note that
\begin{align}
\label{eqn::centered_reverse_radial_loewner}
\begin{split}
 df_t(z)
&= U_t^{-1} dg_t(z) - i U_t^{-1} g_t(z) dW_t - \frac{\kappa}{2} U_t^{-1} g_t(z) dt\\
&= -f_t(z) \left( \frac{1+f_t(z)}{1-f_t(z)} + \frac{\kappa}{2}\right) dt - i f_t(z) dW_t\\
&= -\left( \Phi(1,f_t(z)) + \frac{\kappa}{2} f_t(z) \right) dt - i f_t(z) dW_t
\end{split}
\end{align}
(recall~\eqref{eqn::radial_functions}).

\begin{theorem}
\label{thm::radial_coupling_existence}
Fix $\kappa > 0$.  Suppose that $h$ is a free boundary GFF on $\D$, let $B$ be a standard Brownian motion which is independent of $h$, and let $(f_t)$ be the centered reverse radial $\SLE_\kappa$ Loewner flow which is driven by $U_t = e^{i W_t}$ where $W = \sqrt{\kappa} B$ as in \eqref{eqn::centered_reverse_radial_loewner}.  For each $t \geq 0$ and $z \in \D$ we let\footnote{The function $\Fh_t$ in the statement of Theorem~\ref{thm::radial_coupling_existence} is not the same as the harmonic component in the definition of $\QLE$.  We are using this notation in this section to be consistent with the notation used in \cite{sheffield2010weld}.}
\begin{equation}
\label{eqn::fh_definition}
\Fh_t(z) = \frac{2}{\sqrt{\kappa}} \log|f_t(z)-1| - \frac{\kappa+6}{2 \sqrt{\kappa}} \log|f_t(z)| + Q \log|f_t'(z)|
\end{equation}
where $Q = 2/\gamma + \gamma/2$ and $\gamma = \min(\sqrt{\kappa},\sqrt{16/\kappa})$.  Let $\tau$ be an almost surely finite stopping time for the filtration generated by $W$.  Then
\begin{equation}
\label{eqn::coupling}
h + \Fh_0 \stackrel{d}{=} h \circ f_\tau + \Fh_\tau
\end{equation}
where we view the left and right sides as distributions defined modulo additive constant.
\end{theorem}

Theorem~\ref{thm::radial_coupling_existence} states that the law of $h+\Fh_0$ is invariant under the operation of sampling an independent $\SLE_\kappa$ process $\eta$ and then drawing it on top of $h+\Fh_0$ up until some time $t$ and then applying the change of coordinates formula for quantum surfaces using the forward radial Loewner flow for $\eta$ at time $t$.  An illustration of the setup for Theorem~\ref{thm::radial_coupling_existence} is given in Figure~\ref{fig::reverse_radial_coupling}.

We include the following self-contained proof of Theorem~\ref{thm::radial_coupling_existence} for the convenience of the reader which follows the strategy of \cite{sheffield2010weld}.  The first step, carried out in Lemma~\ref{lem::green_change}, is to compute the Ito derivatives of some quantities which are related to the right side of~\eqref{eqn::coupling}.  Next, we show in Lemma~\ref{lem::test_function_distribution} that the random variable on the right hand side of \eqref{eqn::coupling} takes values in the space of distributions and, when integrated against a given smooth mean-zero test function, yields a process which is continuous in time.  We then compute another Ito derivative in Lemma~\ref{lem::fh_quadratic_variation} and afterwards combine the different steps to complete the proof.

Let $G$ denote the Neumann Green's function for $\Delta$ on $\D$ given in \eqref{eqn::green_neumann_disk}.  Suppose that $(g_t)$ is the reverse radial $\SLE_\kappa$ Loewner flow and $(f_t)$ is the corresponding centered flow as in Theorem~\ref{thm::radial_coupling_existence}.  Throughout, we let
\[ G_t(y,z) = G(f_t(y),f_t(z)) = G(g_t(y),g_t(z)) \quad\text{for each}\quad t \geq 0.\]
We also let $\CP$ (resp.\ $\ol{\CP}$) denote $2\pi$ times the Poisson (resp.\ conjugate Poisson) kernel on $\D$.  Explicitly,
\begin{align}
\label{eqn::poisson_kernel}
   \big(\CP + i \ol{\CP}\big) (z,w) = \frac{w+z}{w-z} = \Psi(w,z).
\end{align}
That is, $\CP$ (resp.\ $\ol{\CP}$) is given by the real (resp.\ imaginary) part of the expression in the right side above. 

\begin{lemma}
\label{lem::green_change}
Suppose that we have the same setup as in Theorem~\ref{thm::radial_coupling_existence}.  There exists a smooth function $\phi \colon \D \to \R$ such that the following is true.  For each $y,z \in \D$ we have that
\begin{align}
 dG_t(y,z) &= \bigg(\phi(f_t(y)) + \phi(f_t(z))  -\ol{\CP}(1,f_t(y)) \ol{\CP}(1,f_t(z)) \bigg)dt \quad\text{and} \label{eqn::green_change}\\
 d \Fh_t(z) &= \frac{1}{\sqrt{\kappa}} dt - \ol{\CP}(1,f_t(z)) dB_t. \label{eqn::fh_change}
\end{align}
\end{lemma}
When we apply Lemma~\ref{lem::green_change} later in this section, we will consider $G_t(y,z)$ and $\Fh_t(z)$ integrated against mean zero test functions.  In particular, the terms involving $\phi$ for $dG_t(y,z)$ and the term $(1/\sqrt{\kappa}) dt$ in $d\Fh_t(z)$ will drop out.
\begin{proof}[Proof of Lemma~\ref{lem::green_change}]
Both \eqref{eqn::green_change} and \eqref{eqn::fh_change} follow from applications of Ito's formula.  In particular,
\begin{align}
 d \log(g_t(y) - g_t(z)) &= \frac{ g_t(y) g_t(z) - U_t(g_t(y) + g_t(z)) - U_t^2}{(U_t - g_t(z))(U_t - g_t(y))} dt \quad\text{and} \label{eqn::green_change1}\\
 d\log( 1- \ol{g_t(y)} g_t(z)) &= \frac{2 g_t(z) \ol{g_t(y)}}{(\ol{U}_t - \ol{g_t(y)})(U_t - g_t(z))} dt. \label{eqn::green_change2}
\end{align}
We note that \eqref{eqn::green_change1} and \eqref{eqn::green_change2} do not depend on the choice of driving function.  A tedious calculation thus shows that $d G_t(y,z) + \ol{\CP}(1,f_t(y)) \ol{\CP}(1,f_t(z)) dt$ can be written as $\phi(f_t(y))+\phi(f_t(z))$ where $\phi$ is a smooth function.  This gives \eqref{eqn::green_change}.

For \eqref{eqn::fh_change}, we fix $z \in \D$ and write $f_t = f_t(z)$.  Then we can express $\Fh_t(z)$ in terms of the real part of
\begin{equation}
\label{eqn::radial_mean_terms}
\log(f_t-1), \quad \log(f_t),\quad\text{and} \quad \log(f_t').
\end{equation}
The Ito derivative of $f_t$ is given in \eqref{eqn::centered_reverse_radial_loewner}.  Differentiating this with respect to~$z$ yields
\begin{equation}
\label{eqn::centered_reverse_radial_loewner_deriv}
 df_t' = - f_t'\left(\frac{1+f_t}{1-f_t} + \frac{2f_t}{(1-f_t)^2} + \frac{\kappa}{2} \right) dt - i f_t' d W_t.
\end{equation}
Applying \eqref{eqn::centered_reverse_radial_loewner} and \eqref{eqn::centered_reverse_radial_loewner_deriv}, we see that the Ito derivatives of the terms in \eqref{eqn::radial_mean_terms} are given by
\begin{align*}
      d \log(f_t-1)
&= \left(\frac{(1+\tfrac{\kappa}{2})f_t + f_t^2}{(1-f_t)^2} \right) dt + \frac{if_t}{1-f_t} dW_t,\\
    d \log(f_t)
&= -\left(\frac{1+f_t}{1-f_t}\right) d t - i dW_t, \quad\text{and}\\
d \log(f_t')
&= \left(1-\frac{2}{(1-f_t)^2} \right) dt - i dW_t.
\end{align*}
This implies that $d \Fh_t(z)$ is given by the real part of
\[ \frac{1}{\sqrt{\kappa}} dt + i \left(\frac{1+f_t}{1-f_t} \right) dB_t,\]
from which \eqref{eqn::fh_change} follows.
\end{proof}

\begin{lemma}
\label{lem::test_function_distribution}
Suppose that we have the same setup as in Theorem~\ref{thm::radial_coupling_existence}.  For each $t \geq 0$, the random variable $h \circ f_t + \Fh_t$ takes values in the space of distributions defined modulo additive constant.  Moreover, for any fixed $\rho \in C_0^\infty(\D)$ with $\int_\D \rho(z) dz = 0$, both $(h \circ f_t + \Fh_t,\rho)$ and $(\Fh_t,\rho)$ are almost surely continuous and the latter is a square-integrable martingale.
\end{lemma}
\begin{proof}
Fix $\rho \in C_0^\infty(\D)$ with $\int_\D \rho(z) dz =0$.  We first note that it is clear that $h \circ f_t$ takes values in the space of distributions modulo additive constant and that $(h \circ f_t,\rho)$ is almost surely continuous in time from how it is defined.  This leaves us to deal with $\Fh_t$.  It follows from \eqref{eqn::fh_change} of Lemma~\ref{lem::green_change} that
\begin{equation}
\label{eqn::fh_qv}
d \langle \Fh_t(z) \rangle = \big( \ol{\CP}(1,f_t(z)) \big)^2 dt.
\end{equation}
By the Schwarz lemma, we note that $|f_t(z)| \leq |z|$ for all $z \in \D$ and $t \geq 0$.  Consequently, it follows from \eqref{eqn::poisson_kernel} that for each $r \in (0,1)$ there exists $C_r \in (0,\infty)$ such that
\begin{equation}
\label{eqn::fh_qv_bound}
\sup_{z \in r \D} \langle \Fh_t(z) - \Fh_u(z) \rangle \leq C_r (t-u) \quad\text{for all}\quad 0 \leq u \leq t < \infty.
\end{equation}
It therefore follows from the Burkholder-Davis-Gundy inequality that for each $p \geq 1$ and $r \in (0,1)$ there exists $C_p, C_{\kappa,r,p} \in (0,\infty)$ such that for all $0 \leq u \leq t$ we have
\begin{align}
      &\sup_{z \in r \D} \E\left[ \sup_{u \leq s \leq t} |\Fh_s(z) - \Fh_u(z)|^p\right] \notag\\
\leq &C_p \left( \sup_{z \in r \D} \E \left[ \langle \Fh_t(z) - \Fh_u(z) \rangle^{p/2} \right] + \frac{1}{\kappa^{p/2}} (t-u)^p\right) \notag \\
\leq &C_{\kappa,r,p} \bigg( (t-u)^p + (t-u)^{p/2} \bigg). \label{eqn::maximal}
\end{align}
It is easy to see from \eqref{eqn::maximal} with $u=0$ and Fubini's theorem that for each $r \in (0,1)$ we have $\Fh_t|_{r \D}$ is almost surely in $L^p(r \D)$.  By combining \eqref{eqn::maximal} with a large enough value of $p > 1$ and the Kolmogorov-Centsov theorem, it is also easy to see that that $t \mapsto (\Fh_t,\rho)$ is almost surely continuous for any $\rho \in C_0^\infty(\D)$ with $\int_\D \rho(z) dz = 0$.  Lastly, it follows from \eqref{eqn::maximal} and \eqref{eqn::fh_change} of Lemma~\ref{lem::green_change} that $(\Fh_t,\rho)$ is a square-integrable martingale.  This completes the proof of both assertions of the lemma.
\end{proof}

For each $\rho \in C_0^\infty(\D)$ with $\int_\D \rho(z) dz = 0$ and $t \geq 0$ we let
\[ E_t(\rho) = \int_{\D}\int_\D \rho(y) G_t(y,z) \rho(z) dy dz\]
be the conditional variance of $(h \circ f_t, \rho)$ given $f_t$.

\begin{lemma}
\label{lem::fh_quadratic_variation}
For each $\rho \in C_0^\infty(\D)$ with $\int_\D \rho(z) dz = 0$ we have that
\[ d \langle (\Fh_t,\rho) \rangle = -dE_t(\rho).\]
\end{lemma}
\begin{proof}
Since $(\Fh_t,\rho)$ is a continuous $L^2$ martingale, the process $\langle (\Fh_t,\rho)\rangle$ is characterized by the property that
\[ (\Fh_t,\rho)^2 - \langle (\Fh_t ,\rho) \rangle\]
is a continuous local martingale in $t \geq 0$.  Thus to complete the proof of the lemma, it suffices to show that
\[ (\Fh_t,\rho)^2 + E_t(\rho)\]
is a continuous local martingale.  It follows from \eqref{eqn::green_change} and \eqref{eqn::fh_change} of Lemma~\ref{lem::green_change} that
\[ \Fh_t(y) \Fh_t(z) + G_t(y,z)\]
evolves as the sum of a martingale in $t \geq 0$ plus a drift term which can be expressed as a sum of terms one of which depends only on $y$ and the other only on $z$.  These drift terms cancel upon integrating against $\rho(y)\rho(z) dy dz$ which in turn implies the desired result.
\end{proof}

\begin{proof}[Proof of Theorem~\ref{thm::radial_coupling_existence}]
Fix $\rho \in C_0^\infty(\D)$ with $\int_\D \rho(z) dz = 0$.  Let $\CF_t$ be the filtration generated by $f_t$.  Note that $\Fh_t$ is $\CF_t$-measurable and that, given $\CF_t$, $(h \circ f_t,\rho)$ is a Gaussian random variable with mean zero and variance $E_t(\rho)$.  Let $I_t(\rho) = (h \circ f_t + \Fh_t,\rho)$.  For $\theta \in \R$ we have that:
\begin{align*}
   & \E[ \exp(i \theta I_t(\rho))]
=  \E[ \E[ \exp(i \theta I_t(\rho)) | \CF_t]]\\
=& \E[ \E[ \exp(i \theta  (h \circ f_t,\rho)) | \CF_t] \exp(i \theta (\Fh_t,\rho))]\\
=& \E[ \exp( i \theta (\Fh_t,\rho) -\tfrac{\theta^2}{2} E_t(\rho))]\\
=& \exp( i \theta (\Fh_0,\rho) - \tfrac{\theta^2}{2} E_0(\rho)).
\end{align*}
Therefore $I_t(\rho) \stackrel{d}{=} I_0(\rho)$ for each $\rho \in C_0^\infty(\D)$ with $\int_\D \rho(z) dz = 0$.  The result follows since this holds for all such test functions $\rho$ and $\rho \mapsto I_0(\rho)$ has a Gaussian distribution.
\end{proof}

Reverse radial $\SLE_\kappa(\rho)$ is a variant of reverse radial $\SLE_\kappa$ in which one keeps track of an extra marked point on $\partial \D$.  It is defined in an analogous way to reverse radial $\SLE_\kappa$ except the driving function $U_t$ is taken to be a solution to the SDE:
\begin{align}
\label{eqn::reverse_radial_sle_kappa_rho}
\begin{split}
 dU_t &= -\frac{\kappa}{2} U_t dt + i \sqrt{\kappa} U_t dB_t + \frac{\rho}{2}\Phi(V_t,U_t) dt\\
  dV_t &= -\Phi(U_t,V_t) dt.
  \end{split}
\end{align}
Observe that when $\rho = 0$ this is the same as the driving SDE for ordinary reverse radial $\SLE_\kappa$.  This is analogous to the definition of forward radial $\SLE_\kappa(\rho)$ up to a change of signs (see, for example, \cite[Section~2]{sw2005coordinate}).  In analogy with Theorem~\ref{thm::radial_coupling_existence}, it is also possible to couple reverse radial $\SLE_\kappa(\rho)$ with the GFF (the chordal version of this is \cite[Theorem~4.5]{sheffield2010weld}).

\begin{theorem}
\label{thm::radial_rho_coupling_existence}
Fix $\kappa > 0$.  Suppose that $h$ is a free boundary GFF on $\D$ and let $(f_t)$ be the centered reverse radial $\SLE_\kappa(\rho)$ Loewner flow which is driven by the solution $U$ as in \eqref{eqn::reverse_radial_sle_kappa_rho} with $V_0 = v_0 \in \partial \D$ taken to be independent of $h$.  For each $t \geq 0$ and $z \in \D$ we let
\begin{equation}
\label{eqn::fh_definition}
\begin{split}
\Fh_t(z) =& \frac{2}{\sqrt{\kappa}} \log|f_t(z)-1| - \frac{\kappa+6 - \rho}{2 \sqrt{\kappa}} \log|f_t(z)| -\\
 &\frac{\rho}{\sqrt{\kappa}} \log|f_t(z)-V_t| + Q \log|f_t'(z)|
 \end{split}
\end{equation}
where $Q = 2/\gamma + \gamma/2$ and $\gamma = \min(\sqrt{\kappa},\sqrt{16/\kappa})$.  Let $\tau$ be an almost surely finite stopping time for the filtration generated by $W$ which occurs before the first time $t$ that $f_t(v_0) = 1$.  Then
\begin{equation}
\label{eqn::coupling_rho}
h + \Fh_0 \stackrel{d}{=} h \circ f_\tau + \Fh_\tau
\end{equation}
where we view the left and right sides as distributions defined modulo additive constant.
\end{theorem}
\begin{proof}
This result is proved in the same manner as Theorem~\ref{thm::radial_coupling_existence}; the only difference is that the calculations needed to verify that the analogy of the assertion of \eqref{eqn::fh_change} from Lemma~\ref{lem::green_change} also holds in the setting of the present theorem.  As in the proof of Lemma~\ref{lem::green_change}, we will not spell out all of the calculations but only indicate the high level steps.  Fix $z \in \D$ and write $f_t = f_t(z)$.  We also let 
\[ Z_t = U_t^{-1} V_t \quad\text{and}\quad A_t = \frac{\rho}{2} \Phi(Z_t,1).\]
We will now explain how to show that
\begin{equation}
\label{eqn::radial_sle_kappa_rho_fh_drift}
d\Fh_t(z) = -\re\left(\frac{(A_t-1)(2-\rho)}{2 \sqrt{\kappa}} \right) dt - \ol{\CP}(1,f_t) dB_t.
\end{equation}
Note that the diffusion term does not depend on $\rho$.  Moreover, the drift term does not depend on $z$ and so integrates to zero against any mean-zero test function.

First, we note that
\begin{equation}
\label{eqn::centered_radial_loewner_sle_kappa_rho}
   df_t = -f_t \left(\frac{1+f_t}{1-f_t} + A_t + \frac{\kappa}{2}\right) dt - i \sqrt{\kappa} f_t dB_t.
\end{equation}
Applying this for $z = v_0$ also gives $dZ_t$.  Differentiating both sides with respect to $z$ yields
\begin{equation}
\label{eqn::centered_radial_loewner_deriv_sle_kappa_rho}
  d f_t' = -f_t' \left( \frac{1+f_t}{1-f_t} + \frac{2 f_t}{(1-f_t)^2} +  A_t + \frac{\kappa}{2} \right) dt - i \sqrt{\kappa} f_t' dB_t.
\end{equation}
Using \eqref{eqn::centered_radial_loewner_sle_kappa_rho} and \eqref{eqn::centered_radial_loewner_deriv_sle_kappa_rho}, we thus see that
\begin{align*}
      d \log(f_t-1)
&= \left(\frac{(1+\tfrac{\kappa}{2})f_t + f_t^2}{(1-f_t)^2} + \frac{A_t f_t}{1-f_t} \right) dt + \frac{if_t}{1-f_t} \sqrt{\kappa} dB_t,\\
    d \log(f_t)
&= -\left(\frac{1+f_t}{1-f_t} + A_t\right) d t - i \sqrt{\kappa} dB_t,\\
d \log(f_t')
&= \left(1-\frac{2}{(1-f_t)^2} - A_t \right) dt - i \sqrt{\kappa} dB_t,\quad\text{and}\\
  d \log(f_t - Z_t)
&= \left( \frac{Z_t + 1}{Z_t -1} \cdot \frac{1}{1-f_t} - \frac{f_t}{1-f_t} - A_t \right) dt - i \sqrt{\kappa} dB_t.
\end{align*}
Adding these expressions up gives \eqref{eqn::radial_sle_kappa_rho_fh_drift}.
\end{proof}

%% file: tex/existence.tex
The purpose of this section is to prove Theorem~\ref{thm::existence}.  Throughout, we suppose that the pair $(\gamma^2,\eta)$ is on one of the upper two lines from Figure~\ref{fig::etavsgamma}.  We are going to construct a triple $(\nu_t, g_t, \Fh_t)$ which satisfies the dynamics described in Figure~\ref{fig::QLEtriangle} where
\begin{equation}
\label{eqn::alpha_kappa}
\alpha_\kappa = -\frac{1}{\sqrt{\kappa}}
\end{equation}
for $\kappa > 1$.  We will first give a careful definition of the spaces in which our random variables take values in Section~\ref{subsec::spaces}.  
We will then prove Theorem~\ref{thm::measurehullprocesscorrespondence} in Section~\ref{subsec::measure_hull_correspondence}.  We next introduce approximations $(\varsigma_t^\delta,g_t^\delta,\Fh_t^\delta)$ to $\QLE(\gamma^2,\eta)$ in Section~\ref{subsec::approximation}.  Throughout, we reserve using the symbol $\nu$ to denote a measure which is constructed using exponentiation.  This is why the Loewner driving measure for the approximation is referred to as~$\varsigma_t^\delta$.  We will then show that each of the elements of $(\varsigma_t^\delta,g_t^\delta, \Fh_t^\delta)$ is tight on compact time intervals with respect to a suitable topology in Section~\ref{subsec::tightness}.  Finally, we will show that the subsequentially limiting triple $(\nu_t,g_t,\Fh_t)$ satisfies the dynamics from Figure~\ref{fig::QLEtriangle} in Section~\ref{subsec::existence}.  This will complete the proof of Theorem~\ref{thm::existence}.

\subsection{Spaces, topologies, and $\sigma$-algebras}
\label{subsec::spaces}

We are going to recall the spaces $\CN_T$, $\CG_T$, and $\CH_T$ (and their infinite time versions) from the introduction as well as introduce a certain subspace of the space of distributions.  We will then equip each of these spaces with a metric and the corresponding Borel $\sigma$-algebra.  We emphasize that each of the spaces that we consider is separable.  This will be important later since we will make use of the Skorohod representation theorem for weak convergence.

{\bf Measures.} We let $\CN_T$ be the space of measures $\varsigma$ on $[0,T] \times \partial \D$ whose marginal on $[0,T]$ is given by Lebesgue measure.  We equip $\CN_T$ with the topology given by weak convergence.  That is, we say that a sequence $(\varsigma^n)$ in $\CN_T$ converges to $\varsigma \in \CN_T$ if for every continuous function $\phi$ on $[0,T] \times \partial \D$ we have that $\int_{[0,T] \times \partial \D} \phi(s,u) d\varsigma^n(s,u) \to \int_{[0,T] \times \partial \D} \phi(s,u) d\varsigma(s,u)$.  Equivalently, we can equip $\CN_T$ with the Levy-Prokhorov metric $d_{\CN,T}$.  We let $\CN$ be the space of measures $\varsigma$ on $[0,\infty) \times \partial \D$ whose marginal on $[0,\infty)$ is given by Lebesgue measure.  Note that there is a natural projection $P_T \colon \CN \to \CN_T$ given by restriction.  We equip $\CN$ with the following topology.  We say that a sequence $(\varsigma^n)$ in $\CN$ converges to $\varsigma$ if $(P_T(\varsigma^n))$ converges to $P_T(\varsigma)$ as a sequence in $\CN_T$ for each $T \geq 0$.  Equivalently, we can equip $\CN$ with the metric $d_{\CN}$ given by $\sum_{n=1}^\infty 2^{-n} \min(d_{\CN,n}(\cdot,\cdot),1)$.  Then $(\CN,d_\CN)$ is a separable metric space and we equip $\CN$ with the Borel $\sigma$-algebra.

{\bf Families of conformal maps.}  We let $\CG_T$ be the space of families of conformal maps $(g_t)$ where, for each $0 \leq t \leq T$, $g_t \colon \D \setminus K_t \to \D$ is the unique conformal transformation with $g_t(0) = 0$ and $g_t'(0) > 0$.  We assume further that $g_t'(0) = e^t$ so that time is parameterized by $\log$ conformal radius.  We define $\CG$ analogously except time is defined on the interval $[0,\infty)$.  We say that a sequence of families $(g_t^n)$ in $\CG$ converges to $(g_t)$ if $(g_t^n)^{-1} \to g_t^{-1}$ locally uniformly in space and time.  In other words, for each compact set $K \subseteq \D$ and $T \geq 0$ we have that $(g_t^n)^{-1} \to g_t^{-1}$ uniformly on $[0,T] \times K$.  We can construct a metric which is compatible with this notion of convergence by taking $d_{\CG,n}$ to be the uniform distance on functions defined on $B(0,1-1/n) \times [0,n]$ and then taking $d_\CG$ to be $\sum_{n=1}^\infty 2^{-n} \min(d_{\CG,n}(\cdot,\cdot),1)$.  Then $(\CG,d_{\CG})$ is a separable metric space and we equip $\CG$ with the corresponding Borel $\sigma$-algebra.

{\bf Families of harmonic functions.}  We let $\CH_T$ be the space of families of harmonic functions $(\Fh_t)$ where, for each $t \in [0,T]$, $\Fh_t \colon \D \to \R$ is harmonic, $\Fh_t(0) = 0$, and $(t,z) \mapsto \Fh_t(z)$ is continuous.  We define $\CH$ similarly with $T = \infty$.  We equip $\CH$ with the topology of local uniform convergence.  That is, if $(\Fh_t^n)$ is a sequence in $\CH$ then we say that $(\Fh_t^n)$ converges to $(\Fh_t)$ if for each compact set $K \subseteq \D$ and $T \geq 0$ we have that $\Fh_t^n \to \Fh_t$ uniformly on $[0,T] \times K$.  We can construct a metric $d_\CH$ which is compatible with this notion of convergence in a manner which is analogous to $d_\CG$ and we equip $\CH$ with the corresponding Borel $\sigma$-algebra.

{\bf Distributions.}  Suppose that $(f_n)$ are the eigenvectors of $\Delta$ with Dirichlet boundary conditions on $\D$ with negative eigenvalues $(\lambda_n)$.  By the spectral theorem, $(f_n)$ properly normalized gives an orthonormal basis of $L^2(\D)$.  Thus for $f \in C_0^\infty(\D)$ we can write $f = \sum_n \alpha_n f_n$ and, for $a \in \R$, we define $(-\Delta)^a f = \sum_n \alpha_n (-\lambda_n)^a f_n$.  We let $(-\Delta)^a L^2(\D)$ denote the Hilbert space closure of $C_0^\infty(\D)$ with respect to the inner product $(f,g)_a = ((-\Delta)^{-a} f, (-\Delta)^{-a}g)$ where $(\cdot,\cdot)$ is the $L^2(\D)$ inner product; see \cite[Section~2.3]{sheffield2007gff} for additional discussion of this space.  We equip $(-\Delta)^a L^2(\D)$ with the Borel $\sigma$-algebra associated with the norm generated by $(\cdot,\cdot)_a$.

The GFF with zero boundary conditions takes values in $(-\Delta)^a L^2(\D)$ for each $a > 0$ \cite{sheffield2007gff} (see also \cite[Section~4.2]{ss2010contour}).  By Proposition~\ref{prop::gff_markov}, we can write the GFF on $\D$ with either mixed or free boundary conditions as the sum of a harmonic function and an independent zero-boundary GFF on $\D$.  It therefore follows that for each $\epsilon > 0$, each of these fields restricted to $(1-\epsilon)\D$ take values in $(-\Delta)^a L^2((1-\epsilon)\D)$.  (In the case of free boundary conditions, we can either consider the space modulo additive constant or fix the additive constant in a consistent manner by taking, for example, the mean of the field on $\D$ to be zero.)  We let $\CD_a^\epsilon$ be the subspace of distributions on $\D$ which are elements of $(-\Delta)^a L^2((1-\epsilon)\D)$ and let $d_{a,\epsilon}$ be the metric on $\CD_a^\epsilon$ induced by the $(\cdot,\cdot)_a$ inner product.  Let $\CD_a = \cap_{\epsilon > 0} \CD_a^\epsilon$ and equip $\CD_a$ with the metric given by $d_a(\cdot,\cdot) = \sum_n 2^{-n}\min(d_{a,n^{-1}}(\cdot,\cdot),1)$.  Since each each $\CD_a^\epsilon$ is separable, so is $\CD_a$ and we equip it with the Borel $\sigma$-algebra.

\subsection{Proof of Theorem~\ref{thm::measurehullprocesscorrespondence}}
\label{subsec::measure_hull_correspondence}

Recall that Theorem~\ref{thm::measurehullprocesscorrespondence} has three assertions.  For the convenience of the reader, we restate them here and then give the precise location of where each is established below.
\begin{enumerate}[(i)]
\item For any $\varsigma \in \CN$ there exists a unique solution to the radial Loewner equation (in integrated form) driven by $\varsigma$.  This is proved in Proposition~\ref{prop::solve_radial}.
\item If we have any increasing family of compact hulls $(K_t)$ in $\ol{\D}$ parameterized by $\log$ conformal radius as seen from $0$ then there exists a unique measure $\varsigma \in \CN$ such that the complement of the domain in $\ol{\D}$ of the solution to the radial Loewner equation driven by $\varsigma$ at time $t$ is given by $K_t$.  This is proved in Proposition~\ref{prop::hulls_measure}.
\item The convergence of a sequence $(\varsigma^n)$ in $\CN$ to a limiting measure $\varsigma \in \CN$ is equivalent to the Caratheodory convergence of the families of compact hulls in $\ol{\D}$ parameterized by $\log$ conformal radius associated with the corresponding radial Loewner chains.  That the convergence of such measures implies the Caratheodory convergence of the hulls is proved as part of Proposition~\ref{prop::solve_radial}.  The reverse implication is proved in Proposition~\ref{prop::convergence_equivalent}.
\end{enumerate}

We establish the first assertion of Theorem~\ref{thm::measurehullprocesscorrespondence} in the following proposition.

\begin{proposition}
\label{prop::solve_radial}
Suppose that $\varsigma \in \CN$.  Then there exists a unique solution $(g_t)$ to the radial Loewner evolution driven by $\varsigma$.  That is, $(g_t)$ solves
\begin{equation}
\label{eqn::radial_loewner_sigma}
g_t(z) = \int_{[0,t] \times \partial \D} \Phi(u,g_s(u)) d\varsigma(s,u),\quad g_0(z) = z.
\end{equation}
Moreover, suppose that $(\varsigma^n)$ is a sequence in $\CN$ converging to $\varsigma \in \CN$.  For each $n \in \N$, let $(g_t^n)$ solve the radial Loewner equation driven by $\varsigma^n$ and likewise let $(g_t)$ solve the radial Loewner equation driven by $\varsigma$.  Then $(g_t^n) \to (g_t)$ as $n \to \infty$ in $\CG$.
\end{proposition}

Before we prove Proposition~\ref{prop::solve_radial}, we first collect the following two lemmas.

\begin{lemma}
\label{lem::continuous_measures_dense}
Suppose that $\varsigma \in \CN$.  Then there exists a sequence $(\varsigma_t^n)$ where, for each $n \in \N$ and $t \geq 0$, $\varsigma_t^n$ is a probability measure on $\partial \D$ such that the following are true.
\begin{enumerate}[(i)]
\item For each $n \in \N$, $t \mapsto \varsigma_t^n$ is continuous with respect to the weak topology on measures on $\partial \D$.
\item We have that $d \varsigma_t^n dt \to \varsigma$ as $n \to \infty$ in $\CN$.
\end{enumerate}
\end{lemma}
\begin{proof}
We define $\varsigma_t^n$ by averaging the first coordinate of $\varsigma$ as follows: for $\phi \colon \partial \D \to \R$ continuous, we take
\[  \int_{\partial \D} \phi(u) d\varsigma_t^n(u) = n \int_{[t,t+n^{-1}] \times \partial \D} \phi(u) d\varsigma(s,u).\]
Then it is easy to see that the sequence $(\varsigma_t^n)$ has the desired properties.
\end{proof}

\begin{lemma}
\label{lem::loewner_tightness}
If $(\varsigma^n)$ is a sequence in $\CN$ and, for each $n \in \N$, $t \mapsto g_t^n$ solves the radial Loewner equation driven by $\varsigma^n$, then the following is true.  There exists a family of conformal transformations $(g_t)$ which are continuous in both space and time and each of which maps $\D$ into itself and a subsequence $(g_t^{n_k})$ of $(g_t^n)$ such that $(g_t^{n_k}) \to (g_t)$ in $\CG$.
\end{lemma}
\begin{proof}
Let $\psi_t^n = (g_t^n)^{-1}$.  Then the chain rule implies that for each $z \in \D$ and $t \geq 0$ we have that
\begin{equation}
\label{eqn::inverse_loewner}
\psi_t^n(z) = - z \int_{[0,t] \times \partial \D} (\psi_s^n)'(z) \Psi(u,z) d\varsigma^n(s,u) + z;
\end{equation} 
see \cite[Remark~4.15]{lawler2005conformally}.  The desired result follows because it is clear from the form of \eqref{eqn::inverse_loewner} that the family $(\psi_t^n)$ is equicontinuous when restricted to a compact subset of $\D$ and compact interval of time in $[0,\infty)$.
\end{proof}

\begin{proof}[Proof of Proposition~\ref{prop::solve_radial}]
We are first going to prove uniqueness of solutions to \eqref{eqn::radial_loewner_sigma}.  Suppose that we have two solutions $(g_t)$ and $(\wt{g}_t)$ to \eqref{eqn::radial_loewner_sigma}.  Fix $T \geq 0$.  Then the domain of $g_T$ (resp.\ $\wt{g}_T$) contains $B(0,\tfrac{1}{4} e^{-T})$ by the Koebe one-quarter theorem since time is parameterized by $\log$ conformal radius.  To show that $g_T = \wt{g}_T$, it suffices to show that $g_T(z) = \wt{g}_T(z)$ for all $z \in B(0,\tfrac{1}{16} e^{-T})$ because two conformal transformations with connected domain and whose values agree on an open set agree everywhere.  For $0 \leq s \leq t \leq T$, we let $g_{s,t} = g_t \circ g_s^{-1}$ and $\wt{g}_{s,t} = \wt{g}_t \circ \wt{g}_s^{-1}$.  From the form of the radial Loewner equation it follows that the maps $g_{s,t}$ are Lipschitz in $0 \leq s \leq t \leq T$ and $z \in B(0,\tfrac{1}{16} e^{-T})$ where the Lispchitz constant only depends on $T$.    By estimating $g_{s,r}$ (resp.\ $\wt{g}_{s,r}$) by $z$ in the integral below, it thus follows that there exists a constant $C > 0$ depending only on $T$ such that
\begin{align*}
     &|g_{s,t}(z)  - \wt{g}_{s,t}(z)|\\
\leq &\int_{[s,t] \times \partial \D} \left|\Phi(u,g_{s,r}(z)) - \Phi(u,\wt{g}_{s,r}(z))\right| d\varsigma(r,u)\\
\leq& C(t-s)^2.
\end{align*}
Fix $\delta > 0$ and let $t_\ell = \delta \ell$ for $\ell \in \N_0$.  Then
\begin{align*}
     g_T(z)
 &=     g_{t_1,T} \circ g_{t_1}(z)\\
 &= g_{t_1,T}  \circ \wt{g}_{t_1}(z) + \big( g_{t_1,T}(g_{t_1}(z)) - g_{t_1,T}(\wt{g}_{t_1}(z)) \big).
\end{align*}
By the previous estimate and the Lipschitz property, the second term is of order $O(\delta^2)$ as $\delta \to 0$ where the implicit constant depends only on $T$.  Iterating this procedure implies that $g_T(z) - \wt{g}_T(z) = O(\delta)$ as $\delta \to 0$ where the implied constant depends only on $T$.  This implies uniqueness.

We are next going to show that if $(\varsigma^n)$ is a sequence in $\CN$ converging to $\varsigma$ and, for each $n$, $(g_t^n)$ is the solution to the radial Loewner equation with driving function $(\varsigma^n)$, then $(g_t^n) \to (g_t)$ in $\CG$ where $(g_t)$ is the radial Loewner equation driven by $\varsigma$.  By possibly passing to a subsequence, Lemma~\ref{lem::loewner_tightness} implies that there exists a family of conformal maps $(g_t)$ such that $(f_t = g_t^{-1})$ is a locally uniform subsequential limit of $(f_t^n)$ in both space and time.  To finish the proof, we just need to show that $(g_t)$ satisfies the radial Loewner equation driven by $\varsigma$.  For each $t \geq 0$ and $z \in \D$ with positive distance from the complement of the domain of $g_t$, we can write:
\begin{align*}
   g_t^n(z)
&= \int_{[0,t] \times \partial \D} \Phi(u,g_s^n(z)) d\varsigma^n(s,u) + z\\
&= O( t \times \sup_{s \in [0,t]}| g_s^n(z) - g_s(z)|) + \int_{[0,t] \times \partial \D}  \Phi(u,g_s(z)) d\varsigma^n(s,u)  + z.
\end{align*}
Taking a limit as $n \to \infty$ of both sides proves the assertion.

It is left to prove existence.  In the case that the radial Loewner evolution is driven by a family of measures $t \mapsto \varsigma_t$ on $\partial \D$ which is piecewise continuous with respect to the weak topology, the existence of a solution to the radial Loewner equation $(g_t)$ driven by $(\varsigma_t)$ follows from standard existence results for ordinary differential equations (see, for example, \cite[Theorem~4.14]{lawler2005conformally}).  The result in the general case follows by combining the previous assertion with Lemma~\ref{lem::continuous_measures_dense}.  In particular, if $\varsigma \in \CN$, then we let $(\varsigma_t^n)$ be a sequence as in Lemma~\ref{lem::continuous_measures_dense}.  For each $n$, let $(g_t^n)$ be the radial Loewner evolution driven by $t \mapsto \varsigma_t^n$.  Then the previous assertion implies that $(g_t^n)$ converges in $\CG$ to the unique solution $(g_t)$ driven by~$\varsigma$.
\end{proof}

To finish the proof of Theorem~\ref{thm::measurehullprocesscorrespondence}, we need to show that we can associate a growing family of hulls $(K_t)$ in $\ol{\D}$ parameterized by $\log$ conformal radius with an element of $\CN$ using the radial Loewner evolution and that the convergence of hulls with respect to the Caratheodory topology is equivalent to the convergence of measures in $\CN$, also using radial Loewner evolution.  This is accomplished in the following two propositions.

\begin{proposition}
\label{prop::hulls_measure}
Suppose that $(K_t)$ is a family of hulls in $\ol{\D}$ parameterized by $\log$ conformal radius as seen from $0$.  That is, the conformal radius of $D_t = \D \setminus K_t$ as seen from $0$ is equal to $e^{-t}$ for each $t \geq 0$.  There exists a unique measure $\varsigma \in \CN$ such that if $(g_t)$ is the solution of the radial Loewner evolution driven by $\varsigma$ then, for each $t \geq 0$, $K_t$ is the complement in $\ol{\D}$ of the domain of $g_t$.
\end{proposition}

The main ingredient in the proof of Proposition~\ref{prop::hulls_measure} is the following lemma.

\begin{lemma}
\label{lem::hull}
Suppose that $K \subseteq \ol{\D}$ is a compact hull and let $T = -\log \confrad(0;\D \setminus K)$.  Then there exists a measure $\varsigma \in \CN_T$ such that if $(g_t)$ is the radial Loewner evolution driven by $\varsigma$ then $\D \setminus K$ is the domain of $g_T$.
\end{lemma}
\begin{proof}
Fix $\epsilon > 0$ and let $\gamma^\epsilon \colon [0,T_\epsilon] \to \ol{\D}$ be a simple curve starting from a point in $\partial \D$ such that the Hausdorff distance between $K$ and $\gamma^\epsilon([0,T_\epsilon])$ is at most $\epsilon$.  Then (the radial version of) \cite[Proposition~4.4]{lawler2005conformally} implies that there exists a continuous function $U^\epsilon \colon [0,T_\epsilon] \to \partial \D$ such that if $(g_t^\epsilon)$ is the radial Loewner evolution driven by $U^\epsilon$ then, for each $t \in [0,T_\epsilon]$, $\gamma^\epsilon([0,t])$ is the complement in $\D$ of the domain of $g_t^\epsilon$.  Let $\varsigma_t^\epsilon = \delta_{U^\epsilon(t)}$.  By possibly passing to a subsequence $(\epsilon_k)$ of positive numbers which decrease to $0$ as $k \to \infty$, we have that $d\varsigma_t^\epsilon dt$ converges in $\CN_T$ to $\varsigma \in \CN_T$.  Proposition~\ref{prop::solve_radial} implies that the radial Loewner evolution $(g_t)$ driven by $\varsigma$ has the property that the domain of $g_T$ is $\D \setminus K$.
\end{proof}

\begin{proof}[Proof of Proposition~\ref{prop::hulls_measure}]
The uniqueness component of the proposition is obvious, so we will just give the proof of existence.  Fix $\delta > 0$ and, for each $\ell \in \N_0$, let $K^{\delta,\ell} = g_{\delta \ell}(K_{\delta(\ell+1)})$.  Let $\varsigma^{\delta,\ell}$ be a measure on $[\delta,\delta(\ell+1)] \times \partial \D$ as in Lemma~\ref{lem::hull} with respect to $K^{\delta,\ell}$, let $\varsigma^\delta = \sum_{\ell=0}^\infty \delta_{[\delta \ell, \delta (\ell+1))}(t) \varsigma^{\delta,\ell}$, and let $(g_t^\delta)$ be the radial Loewner evolution driven by $\varsigma^\delta$.  Then the complement of the domain of $g_{\delta \ell}^\delta$ is equal to $K_{\delta \ell}$ for each $\ell \in \N_0$.  The result follows by taking a limit along a sequence $(\delta_k)$ of positive numbers which decrease to $0$ as $k \to \infty$ such that $\varsigma^{\delta_k}$ converges in $\CN$ to $\varsigma \in \CN$.
\end{proof}

\begin{proposition}
\label{prop::convergence_equivalent}
Let $(\varsigma^n)$ be a sequence in $\CN$.  Suppose that, for each $n \in \N$ and $t \geq 0$, $K_t^n$ is the complement in $\ol{\D}$ of the domain of $g_t^n$ where $t \mapsto g_t^n$ is the radial Loewner evolution driven by $\varsigma^n$.  Then $\varsigma^n$ converges to an element $\varsigma$ of $\CN$ if and only if $(K_t^n)$ converges with respect to the Caratheodory topology to the growing sequence of compact hulls $(K_t)$ in $\ol{\D}$ associated with the radial Loewner evolution driven by $\varsigma$.
\end{proposition}
\begin{proof}
That the convergence of $\varsigma^n \to \varsigma$ in $\CN$ implies the Caratheodory convergence of the corresponding families of compact hulls is proved in Proposition~\ref{prop::solve_radial}.  Therefore, we just have to prove the reverse implication.  That is, we suppose that for each $n$, $(K_t^n)$ is a family of compact hulls in $\D$ parameterized by $\log$ conformal radius as seen from $0$ which converge in the Caratheodory sense to $(K_t)$.  For each $n \in \N$, let $\varsigma^n$ be the measure which drives the radial Loewner evolution associated with $(K_t^n)$ and let $\varsigma$ be the measure which drives the radial Loewner evolution associated with $(K_t)$.  Let $\wt{\varsigma}$ be a subsequential limit in $\CN$ of $(\varsigma^n)$.  The Caratheodory convergence of $(K_t^n)$ to $(K_t)$ implies that $\wt{\varsigma}$ drives a radial Loewner evolution whose corresponding family of compact hulls is the same as $(K_t)$, therefore $\varsigma = \wt{\varsigma}$.  This implies that the limit of every convergent subsequence of $(\varsigma^n)$ is given by $\varsigma$, hence $\varsigma^n \to \varsigma$ as $n \to \infty$ as desired.
\end{proof}

\begin{proof}[Proof of Theorem~\ref{thm::measurehullprocesscorrespondence}]
Combine Proposition~\ref{prop::solve_radial}, Proposition~\ref{prop::hulls_measure}, and Proposition~\ref{prop::convergence_equivalent} as explained in the beginning of this section.
\end{proof}

\subsection{Approximations}
\label{subsec::approximation}

\begin{figure}[ht!]
\begin{center}
\includegraphics[scale=0.85]{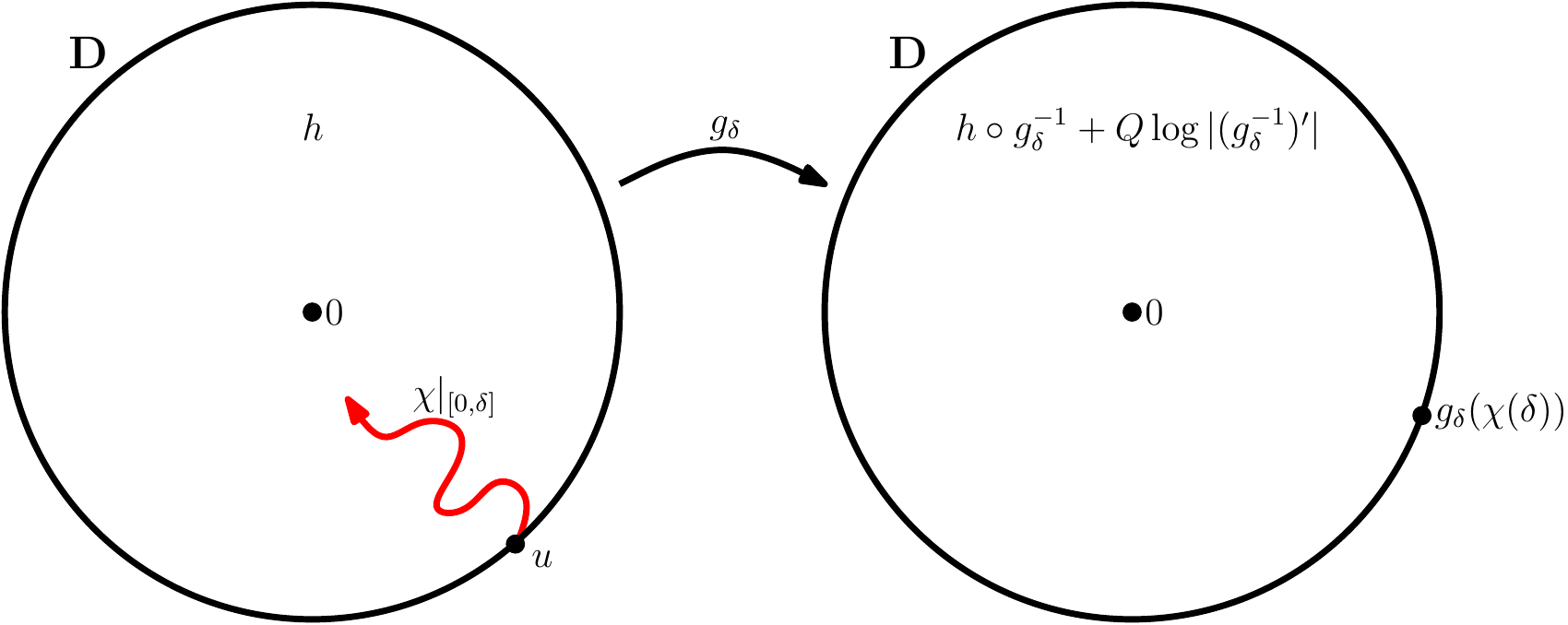}
\end{center}
\caption{\label{fig::zipper_stationary}  Fix $\kappa > 1$ and suppose that $(h,u)$ has the law as described in Proposition~\ref{prop::free_boundary_weighted} (where the role of $\gamma$ in the application of the proposition is played by $2\alpha_\kappa$) plus $-\tfrac{\kappa+6}{2\sqrt{\kappa}} \log|\cdot|$ and let $\nu_{h,\kappa}$ be the $-\tfrac{2}{\sqrt{\kappa}}$-quantum boundary length measure associated with $h$.  Then the conditional law of $h$ given $u$ is that of a free boundary GFF on $\D$ plus $-\tfrac{\kappa+6}{2\sqrt{\kappa}}\log|\cdot| + \tfrac{2}{\sqrt{\kappa}}\log|\cdot-u|$.  By Theorem~\ref{thm::radial_coupling_existence}, the law of the pair $(h,u)$ is invariant under the operation of sampling a radial $\SLE_\kappa$ process in $\D$ starting from $u$ and targeted at $0$ (which given $u$ is conditionally independent of $h$) up to some fixed ($\log$ conformal radius) time $\delta$, mapping back using the (forward) radial Loewner map $g_\delta$ as illustrated above, and applying the change of coordinates formula for quantum surfaces.  Here, $h$ is viewed as a modulo additive constant distribution.  This is the basic operation which is used to construct $\QLE$.}
\end{figure}

We are now going to describe an approximation procedure for generating $\QLE(\gamma^2,\eta)$.  Fix $\kappa > 1$.  Let $(h,u)$ have the law as described in Proposition~\ref{prop::free_boundary_weighted} (where the role of $\gamma$ in the application of the proposition is played by $2\alpha_\kappa$) plus $-\tfrac{\kappa+6}{2\sqrt{\kappa}} \log|\cdot|$ and let $\nu_{h,\kappa}(\partial \D)$ be the $-\tfrac{2}{\sqrt{\kappa}}$-quantum boundary length measure associated with $h$.  Fix $\delta > 0$.  We are now going to describe the dynamics of the triple $(\varsigma_t^\delta,g_t^\delta,\Fh_t^\delta)$ which will be an approximation to $\QLE(\gamma^2,\eta)$.  The random variables $\varsigma_t^\delta dt$, $(g_t^\delta)$, and $(\Fh_t^\delta)$ will take values in $\CN$, $\CG$, and $\CH$ respectively.  The basic operation is illustrated in Figure~\ref{fig::zipper_stationary}.  Consider the Markov chain in which we:

\begin{enumerate}
\item Pick $u \in \partial \D$ according to $\nu_{h,\kappa}$.  By Proposition~\ref{prop::free_boundary_weighted}, the conditional law of $h$ given $u$ is equal to that of the sum of a free boundary GFF on $\D$ plus $-\tfrac{\kappa+6}{2\sqrt{\kappa}} \log |\cdot|+\tfrac{2}{\sqrt{\kappa}}\log|\cdot-u|$.
\item Sample a radial $\SLE_\kappa$ in $\D$ starting from $u$ and targeted at $0$ taken to be conditionally independent of $h$ given $u$.  Let $(g_t)$ be the corresponding family of conformal maps which we assume to be parameterized by $\log$ conformal radius.
\item\label{step::new_h} Replace $h$ with $h \circ (g_\delta^{-1}) + Q \log|(g_\delta^{-1})'|$ where $Q = 2/\gamma+\gamma/2$ with $\gamma=\min(\sqrt{\kappa},\sqrt{16/\kappa})$.
\end{enumerate}
By Proposition~\ref{prop::free_boundary_weighted} and Theorem~\ref{thm::radial_coupling_existence}, we know that this Markov chain preserves the law of $h$.  We use this construction to define the processes $(\varsigma_t^\delta,g_t^\delta,\Fh_t^\delta)$ as follows.  We sample $U^{\delta,0}$ from $\nu_{h,\kappa} = \exp(\alpha_\kappa h)$ and let $W^{\delta,0} = \exp(i\sqrt{\kappa} B^{\delta,0})$ where $B^{\delta,0}$ is a standard Brownian motion independent of $h$, and take $g^\delta|_{[0,\delta)}$ to be the radial Loewner evolution driven by $U^{0,\delta} W^{\delta,0}$.  For each $t \in [0,\delta]$, we let $\Fh^\delta|_{[0,\delta)}$ be given by\footnote{We add the term $\tfrac{\kappa+6}{2\sqrt{\kappa}}\log|\cdot|$ back into the GFF whenever applying $\pHarm$ because $\pHarm$ as defined in Remark~\ref{rem::orthogonal_projection} is defined only for the GFF.} $\pHarm(h \circ (g_t^\delta)^{-1} + Q \log|((g_t^\delta)^{-1})'| + \tfrac{\kappa+6}{2\sqrt{\kappa}} \log|\cdot|)$ normalized so that $\Fh_t^\delta(0) = 0$ for $t \in [0,\delta]$.  Given that $(g_t^\delta)$ and $(\Fh_t^\delta)$ have been defined for $t \in [0,\delta k)$, some $k \in \N$, we sample $U^{\delta,k}$ from $\exp(\alpha_\kappa \Fh_{\delta k}^\delta)$ and let $W^{\delta,k} = \exp(i\sqrt{\kappa} B^{\delta,k})$ where $B^{\delta,k}$ is an independent standard Brownian motion defined in the time interval $[\delta k,\delta(k+1))$ (so that $B_{\delta k}^{\delta,k} = 0$).  We then take $\wt{g}^\delta|_{[\delta k,\delta(k+1))}$ to be the radial Loewner evolution driven by $U^{\delta,k} W^{\delta,k}$ and $g^\delta|_{[\delta k,\delta(k+1))} = \wt{g}^\delta|_{[\delta k,\delta(k+1))} \circ g_{\delta k}^\delta$.  Finally, we take $\Fh^\delta|_{[\delta,\delta(k+1))}$ to be given by $\pHarm(h \circ (g_t^\delta)^{-1}+ Q \log|((g_t^\delta)^{-1})'| + \tfrac{\kappa+6}{2\sqrt{\kappa}} \log|\cdot|)$ normalized so that $\Fh_t^\delta(0) = 0$ for $t \in [\delta k,\delta(k+1))$.

Since $h \circ (g_t^\delta)^{-1} + Q \log|((g_t^\delta)^{-1})'| + \tfrac{\kappa+6}{2\sqrt{\kappa}} \log|\cdot|$  for each $t \geq 0$ is a free boundary GFF on $\D$ plus $2/\sqrt{\kappa}$ times the log function centered at an independent point on $\partial \D$, the orthogonal projections used to define $\Fh_t^\delta$ are almost surely defined for Lebesgue almost all $t \geq 0$ simultaneously; recall Proposition~\ref{prop::gff_markov}.  We can extend the definition of $\Fh_t^\delta$ so that it makes sense almost surely for all $t \geq 0$ simultaneously as follows.  By induction, it is easy to see that the complement $K_t^\delta$ in $\ol{\D}$ of the domain of $g_t^\delta$ is a local set for (the GFF part of) $h$ for each $t \geq 0$.  Hence, Proposition~\ref{prop::cond_mean_continuous} implies that $\Fh_t^\delta$ is almost surely continuous as a function $[0,\infty) \times \partial \D \to \R$.  (This point is explained in more detail in Lemma~\ref{lem::local_tightness} below.)

Let
\begin{equation}
\label{eqn::approx_measure}
\varsigma_t^\delta = \sum_{\ell=0}^\infty \one_{[\delta \ell,\delta(\ell+1))}(t) \delta_{U^{\delta,\ell} W^{\delta,\ell}}.
\end{equation}
That is, $\varsigma_t^\delta$ for $t \in [\delta \ell,\delta(\ell+1))$ and $\ell \in \N$ is given by the Dirac mass located at $U^{\delta,\ell} W^{\delta,\ell} \in \partial \D$.  Then $(g_t^\delta)$ is the radial Loewner evolution driven by $\varsigma_t^\delta$.  That is, $(g_t^\delta)$ solves
\begin{equation}
\label{eqn::approx_radial}
\dot{g}_t^\delta(z) = \int_{\partial \D} \Phi(u,g_t^\delta(z)) d\varsigma_t^\delta(u), \quad\quad g_0^\delta(z) = z.
\end{equation}
(Recall \eqref{eqn::radial_functions}.)  We emphasize that by Theorem~\ref{thm::radial_coupling_existence} we have
\[ h \circ (g_t^\delta)^{-1} + Q\log|((g_t^\delta)^{-1})'| \stackrel{d}{=} h \quad\text{for all}\quad t \geq 0\]
as modulo additive constant distributions.  In particular, the law of $(\Fh_t^\delta)$ is stationary in $t$.

For our later arguments, it will be more convenient to consider the measure $d\varsigma_t^\delta dt$ on $[0,\infty) \times \partial \D$ in place of $\varsigma_t^\delta$, which for each $t \geq 0$ is a measure on $\partial \D$.  Note that this is random variable which takes values in $\CN$.  We also note that $(g_t^\delta)$ takes values in $\CG$ and $(\Fh_t^\delta)$ takes values in $\CH$.

\begin{definition}
\label{def::delta_qle}
We call the triple $(\varsigma_t^\delta,g_t^\delta,\Fh_t^\delta)$ constructed above the {\bf $\delta$-approximation} to $\QLE(\gamma^2,\eta)$.
\end{definition}

Note that the dynamics $(\varsigma_t^\delta, g_t^\delta, \Fh_t^\delta)$ satisfy two of the arrows from Figure~\ref{fig::QLEtriangle}.  Namely, $g_t^\delta$ is obtained from $\varsigma_t^\delta$ by solving the radial Loewner equation and $\Fh_t^\delta$ is obtained from $g_t^\delta$ using $\pHarm(h \circ (g_t^{\delta})^{-1} + Q \log|((g_t^\delta)^{-1})'| + \tfrac{\kappa+6}{2\sqrt{\kappa}}\log|\cdot|)$ (and then normalized to vanish at the origin).  However, $\varsigma_t^\delta$ is not obtained from $\Fh_t^\delta$ via exponentiation.  (Rather, $\varsigma_t^\delta$ is given by a Dirac mass at a point in $\partial \D$ which \emph{is sampled from} the measure given by exponeniating $\Fh_t^\delta$.)  In Section~\ref{subsec::tightness}, we will show that each of the elements of $(\varsigma_t^\delta,g_t^\delta,\Fh_t^\delta)$ is tight (on compact time intervals) with respect to a suitable topology as $\delta \to 0$.  In Section~\ref{subsec::existence}, we will show that both of the aforementioned arrows for the $\QLE(\gamma^2,\eta)$ dynamics still hold for the subsequentially limiting objects $(\varsigma_t, g_t, \Fh_t)$.  We will complete the proof by showing that $\varsigma_t$ is equal to the measure $\nu_t$ which is given by exponentiating $\Fh_t$, hence the triple $(\nu_t,g_t,\Fh_t)$ satisfies all three arrows of the $\QLE(\gamma^2,\eta)$ dynamics.

\subsection{Tightness}
\label{subsec::tightness}

\begin{figure}[ht!]
\begin{center}
\includegraphics[scale=0.85,page=1,trim = 0.15cm 3.cm 0.25cm 0]{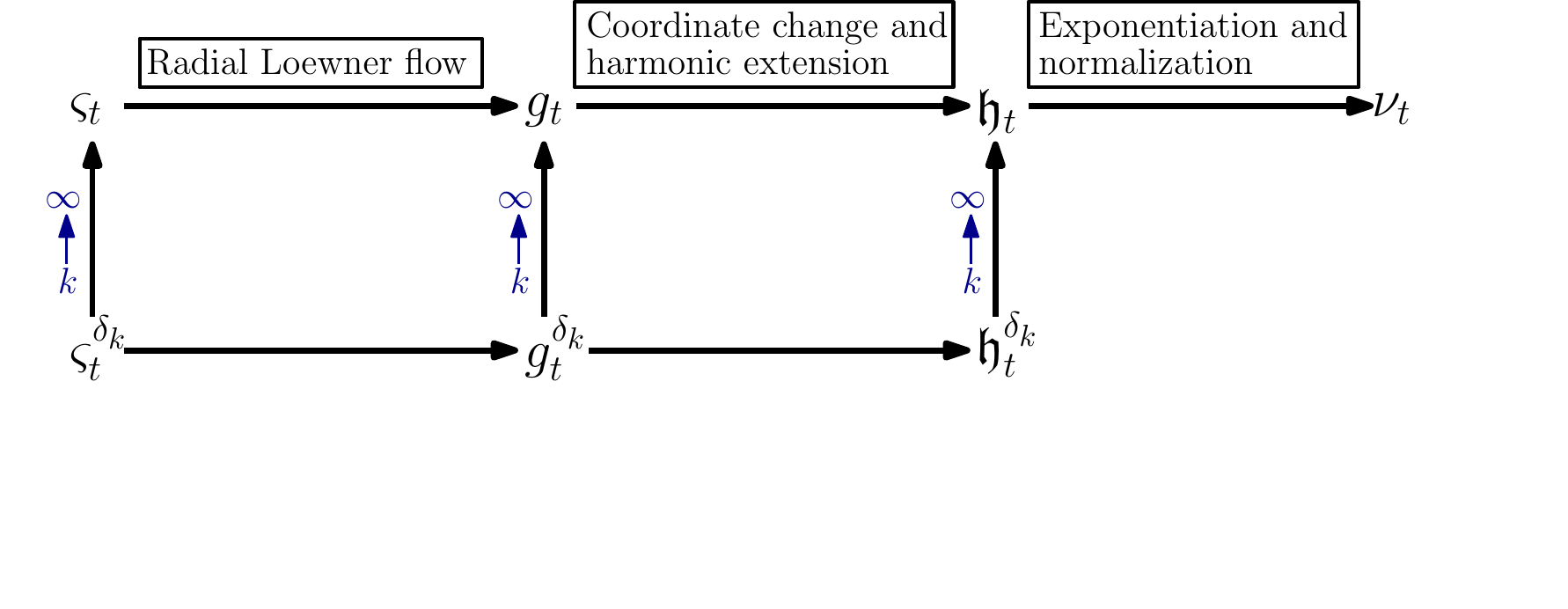}
\end{center}
\caption{\label{fig::tightness_scheme} In Proposition~\ref{prop::all_tight} of Section~\ref{subsec::tightness}, we prove the tightness of each of the elements of the triple $(\varsigma_t^\delta,g_t^\delta,\Fh_t^\delta)$ (on compact time intervals) in the $\delta$-approximation to $\QLE(\gamma^2,\eta)$ as $\delta \to 0$ with respect to the topologies introduced in Section~\ref{subsec::spaces}.  The subsequentially limiting objects $(\varsigma_t,g_t,\Fh_t)$ are related to each other in the same way as $(\varsigma_t^\delta,g_t^\delta,\Fh_t^\delta)$ and as indicated above.  Namely, $g_t$ is generated from $\varsigma_t$ by solving the radial Loewner equation and $\Fh_t$ is related to $g_t$ in that it is given by $\pHarm(h \circ g_t^{-1} + Q\log|(g_t^{-1})'|+\tfrac{\kappa+6}{2\sqrt{\kappa}}\log|\cdot|)$ (normalized to vanish at the origin).  The measure $\nu_t$ is obtained from $\Fh_t$ by exponentiation and is constructed in Proposition~\ref{prop::measure_limit_exists}.  Upon proving tightness, the existence of $\QLE(\gamma^2,\eta)$ is established by showing that $\nu_t = \varsigma_t$.  This is completed in Section~\ref{subsec::existence}.
}
\end{figure}

The purpose of this section is to establish Proposition~\ref{prop::all_tight}, which gives the existence of subsequential limits of the triple $(\varsigma_t^\delta, g_t^\delta, \Fh_t^\delta)$ viewed as a random variable taking values in $\CN \times \CG \times \CH$ as $\delta \to 0$.  We begin with the following two lemmas which are general results about local sets for the GFF.

\begin{lemma}
\label{lem::limits_are_local}
Suppose that $(h_n,K_n)$ is a sequence such that, for each $n$, $h_n$ is a GFF on $\D$ (with Dirichlet, free, or mixed boundary conditions and the same boundary conditions for each $n$) and $K_n \subseteq \ol{\D}$ is a local set for $h_n$.  Fix $a > 0$.  Assume that $(h_n,K_n)$ are coupled together so that $h_n \to h$ (resp.\ $K_n \to K$) almost surely as $n \to \infty$ in $\CD_a$ (resp.\ the Hausdorff topology) where $h$ is a GFF on $\D$ and $K \subseteq \ol{\D}$ is closed.  Then $K$ is local for $h$.  If $\CC_{K_n}$ (resp.\ $\CC_K$) denotes the conditional expectation of $h_n$ (resp.\ $h$) given $K_n$ and $h|_{K_n}$ (resp.\ $K$ and $h|_K$) and $\CC_{K_n} \to F$ locally uniformly almost surely for some function $F \colon \D \setminus K \to \R$, then $F = \CC_K$.
\end{lemma}
\begin{proof}
The proof that $K$ is a local set for $h$ is similar to that of \cite[Lemma~4.6]{ss2010contour}.  In particular, we will make use of the second characterization of local sets from \cite[Lemma~3.9]{ss2010contour}.  Fix a deterministic open set $B \subseteq \D$.  For each $n \in \N$, we let $S_n$ be the event that $K_n \cap B \neq \emptyset$ and let $\wt{K}_n = K_n$ on $S_n^c$ and $\wt{K}_n = \emptyset$ otherwise.  We also let $S$ be the event that $K \cap B \neq \emptyset$ and let $\wt{K} = K$ on $S^c$ and $\wt{K} = \emptyset$ otherwise.  For each $n \in \N$, let $h_n^1 = \pHarm(h_n;B)$ and $h_n^2 = \pSupp(h_n;B)$ and define $h^1,h^2$ analogously for $h$.  Since $h^1$ is independent of $h^2$ (recall Proposition~\ref{prop::gff_markov}), it suffices to show that $h^2$ is independent of the triple $(S,\wt{K},h^1)$.  Since $K_n$ is local for $h_n$, the second characterization of local sets from \cite[Lemma~3.9]{ss2010contour} implies that $h_n^2$ is independent of the triple $(S_n,\wt{K}_n,h_n^1)$ for each $n \in \N$.  The result therefore follows because this implies that the independence holds in the $n \to \infty$ limit.

Suppose that $\CC_{K_n} \to F$ locally uniformly almost surely for some $F \colon \D \setminus K \to \R$.  Then $F$ is almost surely harmonic since each $\CC_{K_n}$ is harmonic.  Since $K_n$ is local for $h_n$ we can write $h_n = \wt{h}_n + \CC_{K_n}$ where $\wt{h}_n$ is a zero-boundary GFF on $\D \setminus K_n$.  Fix $\epsilon > 0$.  Since $h_n \to h$ in $\CD_a$ it follows that $h_n \to h$ in $(-\Delta)^a L^2((1-\epsilon)\D \setminus K)$ as $n \to \infty$.  The local uniform convergence of $\CC_{K_n}$ to $F$ in $\D \setminus K$ as $n \to \infty$ implies that $\CC_{K_n} \to F$ in $(-\Delta)^a L^2(V)$ for all $V \subseteq \ol{\D} \setminus K$ with $\dist(V, K \cup \partial \D) >0$ as $n \to \infty$.  Combining, we have that $\wt{h}_n$ converges to some $\wt{h}$ in $(-\Delta)^a L^2(V)$ as $n \to \infty$ for such $V$.  Since this holds for all such $V$, we have that $h = \wt{h} + F$ and $\wt{h}$ is a zero-boundary GFF in $\D \setminus K$.  Since $K$ is local for $h$, $\wh{h} = h - \CC_K = \wt{h} + F - \CC_K$ is a zero-boundary GFF on $\D \setminus K$.  Rearranging, we have that $\wt{h} - \wh{h} =  \CC_K - F$.  If $\phi$ is harmonic in $\D \setminus K$, then we have that $(\CC_K - F,\phi)_\nabla = (\wt{h}-\wh{h},\phi)_\nabla = 0$.  Since this holds for all such $\phi$, we have that $\CC_K - F = 0$, desired.
\end{proof}

Proposition~\ref{prop::cond_mean_continuous} gives that if $(K_t)$ is an increasing family of local sets for a GFF $h$ on $\D$ parameterized by $\log$ conformal radius as seen from a given point $z \in \D$, then $\CC_{K_t}(z)$ evolves as a Brownian motion as $t$ varies but $z$ is fixed.  This in particular implies that $\CC_{K_t}(z)$ is continuous in $t$.  We are now going to extend this to show that $\CC_{K_t}(z)$ is continuous in both $t$ and $z$.

\begin{lemma}
\label{lem::local_tightness}
Suppose that $h$ is a GFF on $\D$ (with Dirichlet, free, or mixed boundary conditions) and that $(K_t)$ is an increasing family of local sets for $h$ parameterized so that $-\log \confrad(0;\D \setminus K_t) = t$ for all $t \in [0,T]$ where $T > 0$ is fixed.  Then the function $[0,T]  \times B(0,\tfrac{1}{16} e^{-T}) \to \R$ given by $(t,z) \mapsto \CC_{K_t}(z)$ has a modification which is H\"older continuous with any exponent strictly smaller than $1/2$.  The H\"older norm of the modification depends only on $T$ and and the boundary data for $h$.  (In the case that $h$ has free boundary conditions, we fix the additive constant for $h$ so that $\CC_{K_0}(0) = 0$.)
\end{lemma}

The reason that Lemma~\ref{lem::local_tightness} is stated for $z \in B(0,\tfrac{1}{16} e^{-T})$ is that the Koebe one-quarter theorem implies that $B(0,\tfrac{1}{4} e^{-T}) \subseteq \D \setminus K_t$ for all $t \in [0,T]$.  In particular, $B(0,\tfrac{1}{16} e^{-T})$ has positive distance from $K_t$ for all $t \in [0,T]$.  By applying Lemma~\ref{lem::local_tightness} iteratively, we see that $\CC_{K_t}(z)$ is in fact continuous for all $z,t$ pairs such that $z$ is contained in the component of $\D \setminus K_t$ containing the origin.

\begin{proof}[Proof of Lemma~\ref{lem::local_tightness}]
We are going to prove the result using the Kolmogorov-Centsov theorem.  Fix $0 \leq s \leq t \leq T$ and $z,w \in B(0,\tfrac{1}{16} e^{-T})$.  Since $K_t$ is local for $h$, we can write $h = h_t + \CC_{K_t}$ where $h_t$ is a zero-boundary GFF on $\D \setminus K_t$.  Re-arranging, we have that $\CC_{K_t} = h - h_t$.  Let $h_\epsilon$ (resp.\ $h_{\epsilon,t}$) be the circle average process for $h$ (resp.\ $h_t$).  Taking $\epsilon = \tfrac{1}{16} e^{-T}$ so that $z \in B(0,\tfrac{1}{16} e^{-T})$ implies $B(z,\tfrac{1}{16} e^{-T}) \subseteq B(0,\tfrac{1}{8} e^{-T})$ in what follows, we have that
\[ \CC_{K_t}(z) - \CC_{K_t}(w) = \big( h_\epsilon(z) - h_\epsilon(w) \big) - \big( h_{\epsilon,t}(z) - h_{\epsilon,t}(w) \big).\]
The same argument as in the proof of \cite[Proposition~3.1]{ds2011kpz} applied to both $h_\epsilon$ and $h_{\epsilon,t}$ implies that for each $p \geq 2$ there exists a constant $C > 0$ such that
\begin{equation}
\label{eqn::tightness_change_point}
\E\left[ \big( \CC_{K_t}(z) - \CC_{K_t}(w) \big)^p\right] \leq C|z-w|^{p/2}.
\end{equation}
Proposition~\ref{prop::cond_mean_continuous} also implies that for each $p \geq 2$ there exists a constant $C > 0$ such that
\begin{equation}
\label{eqn::tightness_change_time}
\E\left[ \big( \CC_{K_t}(z) - \CC_{K_s}(z) \big)^p \right] \leq C|t-s|^{p/2}.
\end{equation}
Combining~\eqref{eqn::tightness_change_point}, \eqref{eqn::tightness_change_time} with the inequality $(a+b)^p \leq 2^p (a^p + b^p)$ implies that for each $p \geq 2$ there exists a constant $C > 0$ such that
\begin{equation}
\label{eqn::tightness_change_both}
\E\left[ \big( \CC_{K_t}(z) - \CC_{K_s}(w) \big)^p \right] \leq C(|z-w|^{p/2} + |t-s|^{p/2}).
\end{equation}
The desired result thus follows from the Kolmogorov-Centsov theorem \cite{ry1999martingales,ks1991stoccalc}.
\end{proof}

\begin{proposition}
\label{prop::all_tight}
There exists a sequence $(\delta_k)$ in $(0,\infty)$ decreasing to $0$ such that the following is true.  There exists a coupling of the laws of $h_k$, $(\varsigma_t^{\delta_k})$, $(g_t^{\delta_k})$, and $(\Fh_t^{\delta_k})$ as $k \in \N$ varies --- where $h_k$ denotes the GFF used to generate $(\varsigma_t^{\delta_k})$, $(\Fh_t^{\delta_k})$, and $(g_t^{\delta_k})$ --- and limiting processes $h \in \CD_a$ (some $a > 0$), $\varsigma \in \CN$, $(g_t) \in \CG$, and $(\Fh_t) \in \CH$ such that $h_k$, $\varsigma_t^{\delta_k} dt$, $(g_t^{\delta_k})$, and $(\Fh_t^{\delta_k})$ almost surely converge to $h$, $\varsigma$, $(g_t)$, and $(\Fh_t)$ respectively, in $\CD_a$, $\CN$, $\CG$, and $\CH$.  Moreover, $(g_t)$ is the radial Loewner evolution generated by $\varsigma$ and $\Fh_t$ for each $t \geq 0$ is almost surely given by $\pHarm(h \circ g_t^{-1} + Q\log|(g_t^{-1})'|+\tfrac{\kappa+6}{2\sqrt{\kappa}}\log|\cdot|)$ (normalized to vanish at the origin).  Finally,
\begin{equation}
\label{eqn::little_g_stationary}
h \circ g_t^{-1} + Q\log|(g_t^{-1})'| \stackrel{d}{=} h \quad\text{for each}\quad t \geq 0
\end{equation}
as modulo additive constant distributions.
\end{proposition}
\begin{proof}
As explained in Section~\ref{subsec::spaces}, the law of the free boundary GFF has separable support; see also \cite[Lemma~4.2 and Lemma~4.3]{ss2010contour}.  It is also explained in Section~\ref{subsec::spaces} that the same holds for the laws of $\varsigma_t^\delta dt$, $(g_t^\delta)$, and $(\Fh_t^\delta)$ viewed as random variables taking values in $\CN$, $\CG$, and $\CH$, respectively.  The tightness of the law of $h$ is obvious as is the tightness of the law of $\varsigma_t^\delta dt$.  The tightness of the law of $(g_t^\delta)$ follows from Lemma~\ref{lem::loewner_tightness} and the tightness of the law of $(\Fh_t^\delta)$ follows from Lemma~\ref{lem::local_tightness}.  This implies the existence of a sequence $(\delta_k)$ of positive real numbers along which the law $\CL_\delta$ of $(h,\zeta_t^\delta dt,g_t^\delta,\Fh_t^\delta)$ has a weak limit.  The Skorohod representation theorem implies that we find a coupling $(h_k,\varsigma_t^{\delta_k} dt,g_t^{\delta_k},\Fh_t^{\delta_k})$ of the laws $\CL_{\delta_k}$ such that $h_k \to h$, $\varsigma_t^{\delta_k} dt \to \varsigma$, $(g_t^{\delta_k}) \to (g_t)$, and $(\Fh_t^{\delta_k}) \to (\Fh_t)$ almost surely as $k \to \infty$ in the senses described in the statement of the proposition.

It is left to show that $(h,\varsigma,g_t,\Fh_t)$ are related in the way described in the proposition statement.  Theorem~\ref{thm::measurehullprocesscorrespondence} implies that $(g_t)$ is obtained from $\varsigma$ by solving the radial Loewner equation.  Therefore we just need to show that
\begin{enumerate}[(i)]
\item\label{it::coordinate_orthogonal} $\Fh_t$ can be obtained from $g_t$ via coordinate change and applying $\pHarm$ and then
\item\label{it::stationarity} establish \eqref{eqn::little_g_stationary}.
\end{enumerate}

We will start with \eqref{it::coordinate_orthogonal}.  For each $t \geq 0$, let $K_t$ be the hull given by the complement in $\ol{\D}$ of the domain of $g_t$.  The first step is to show that $K_t$ is local for $h$.  Let $(K_t^{\delta_k})$ denote the corresponding family of hulls associated with $(g_t^{\delta_k})$.  By possibly passing to a subsequence of $(\delta_k)$ and using that the Hausdorff topology is compact hence separable, we can recouple so that, in addition to the above, we have that $K_t^{\delta_k} \to \wt{K}_t$ almost surely in the Hausdorff topology for all $t \in \Q_+$.  Lemma~\ref{lem::limits_are_local} then implies that $\wt{K}_t$ is local for $h$ for all $t \in \Q_+$.  Combining this with the first characterization of local sets given in \cite[Lemma~3.9]{ss2010contour} implies that $K_t$ is local for $h$ for all $t \geq 0$.  Lemma~\ref{lem::limits_are_local} implies that $\CC_{K_t^{\delta_k}} \to \CC_{K_t}$ locally uniformly almost surely for all $t \in \Q_+$.  Combining this with Lemma~\ref{lem::local_tightness} implies that $\CC_{K_t^{\delta_k}} \to \CC_{K_t}$ locally uniformly in both space and time.  Since $\Fh_t^{\delta_k}$ is given by $\CC_{K_t^{\delta_k}} \circ (g_t^{\delta_k})^{-1} + Q\log|((g_t^{\delta_k})^{-1})'| + \tfrac{\kappa+6}{2\sqrt{\kappa}}\log|\cdot|$ (normalized to vanish at the origin), we therefore have that $\Fh_t$ is given by $\CC_{K_t} \circ g_t^{-1} + Q\log| (g_t^{-1})'|+ \tfrac{\kappa+6}{2\sqrt{\kappa}}\log|\cdot|$ (normalized to vanish at the origin).

The construction of the $\delta$-approximation implies that
\[ h_k \circ (g_t^{\delta_k})^{-1} + Q \log| ((g_t^{\delta_k})^{-1})'| \stackrel{d}{=}  h_k  \quad\text{for each}\quad k \in \N \quad\text{and}\quad t \geq 0\]
as modulo additive constant distributions, hence the same holds in the limit as $k \to \infty$ due to the nature of the convergence.  This gives \eqref{it::stationarity}.
\end{proof}

\subsection{Subsequential limits solve the $\QLE$ dynamics}
\label{subsec::existence}

Throughout this section, we suppose that $(\delta_k)$ is a sequence in $(0,\infty)$ decreasing to~$0$ as in the statement of Proposition~\ref{prop::all_tight} and $(h_k,\varsigma_t^{\delta_k},g_t^{\delta_k},\Fh_t^{\delta_k})$ are coupled together on a common probability space such that $h_k \to h$ in $\CD_a$ for $a > 0$, $\varsigma_t^{\delta_k} dt \to \varsigma$ in $\CN$, $(g_t^{\delta_k}) \to (g_t)$ in $\CG$, and $(\Fh_t^{\delta_k}) \to (\Fh_t)$ in $\CH$ as in the statement of Proposition~\ref{prop::all_tight}.  The purpose of this section is to construct a family of probability measures $(\nu_t)$ on $\partial \D$ from $(\Fh_t)$ and then show that the triple $(\nu_t, g_t,\Fh_t)$ satisfies the $\QLE$ dynamics illustrated in Figure~\ref{fig::QLEtriangle}.  The measures $\nu_t$ will only be defined for almost all $t \geq 0$, so we will in fact think of $(\nu_t)$ as being given by a single measure $\nu \in \CN$.

We will accomplish the above in two steps.  We will first construct a measure $\nu \in \CN$ which, for a given time $t \geq 0$, should be thought of as the $-\tfrac{2}{\sqrt{\kappa}}$-quantum boundary length measure (Proposition~\ref{prop::measure_limit_exists}) generated from the boundary values of $\Fh_t$ (normalized to be a probability).  That is, $\nu \in \CN$ is formally given by $\CZ_t^{-1} \exp(\alpha_\kappa \Fh_t(u)) du dt$ where $\CZ_t$ is a normalization constant.  This step is carried out in Section~\ref{subsubsec::qle_rate}.  The second step (Proposition~\ref{prop::qle_solution}) is to show that $\varsigma = \nu$.  This is carried out in Section~\ref{subsubsec::loewner_rate_is_qle}.  As we explained earlier, this will complete the proof because it gives that $(\nu_t,g_t,\Fh_t)$ satisfies all three arrows of the $\QLE(\gamma^2,\eta)$ dynamics described in Figure~\ref{fig::QLEtriangle}.

\subsubsection{Construction of the $\QLE$ driving measure}
\label{subsubsec::qle_rate}

We begin by defining the approximations we will use to construct $\nu$.  We first approximate $\Fh_t$ by orthogonally projecting it to the subspace of $H(\D)$ (recall the definition of $H(\D)$ from Section~\ref{subsubsec::gff_free}) spanned by $\{f_1,\ldots,f_n\}$ where $(f_n)$ is an orthonormal basis of the subspace of functions of $H(\D)$ which are harmonic in $\D$.  In what follows in this section, the precise choice of basis is not important (i.e., the resulting measure $\nu$ does not depend on the choice of basis).  However, for our later arguments, it will be convenient to make a particular choice so that it is obvious that our approximations are continuous in $t$.  Thus for each $n \in \N$ which is even (resp.\ odd) we take $f_{n}(z) = \beta_n^{-1} \re(z^{n/2})$ (resp.\ $f_n = \beta_n^{-1} \im(z^{(n+1)/2})$) where $\beta_n = \| \re(z^{n/2}) \|_\nabla$ (resp.\ $\beta_n = \| \im(z^{(n+1)/2})\|_\nabla$) so that $\| f_n \|_\nabla = 1$.  Indeed, an elementary calculation implies that $(f_n)$ is orthonormal and that $(f_n)$ spans follows because every harmonic function in $\D$ is the real part of an analytic function on $\D$.  Note that $(f_n)$ is part of an orthonormal basis of all of $H(\D)$; we will use this in conjunction with \eqref{eqn::boundary_measure_onb} in what follows.  For each $n \in \N$ and $t \geq 0$, we let $\Fh_t^n$ be the orthogonal projection of $\Fh_t$ onto the subspace of $H(\D)$ spanned by $\{f_1,\ldots,f_n\}$, i.e.\ the real parts of polynomials in $z$ of degree at most~$n/2$.  We let
\begin{equation}
\label{eqn::nu_r_def}
 d\nu_t^n(u) = \frac{1}{\CZ_{n,t}} \exp(\alpha_\kappa \Fh_t^n(u)) d u \quad\text{for}\quad u \in \partial \D \quad\text{and}\quad t \geq 0
\end{equation}
where $\CZ_{n,t}$ is a normalizing constant so that $\nu_t^n$ has unit mass.  Note that $\Fh_t^n$ varies continuously in $t$ with respect to the uniform topology on continuous functions defined on $\ol{\D}$.  One way to see this is to note that since $\Fh_t$ is harmonic in $\D$ for each fixed $t$, it is equal to the real part of an analytic function $F_t$ on $\D$.  Then $\Fh_t^n$ is given by the real part of the terms up to degree $n/2$ in the power series expansion for $F_t$.  The claimed continuity follows because these coefficients for $F_t$ are a continuous function of $\Fh_t$ restricted to $\tfrac{1}{2}\D$ with respect to the uniform topology on continuous functions on $\tfrac{1}{2} \ol{\D} \to \R$.  We also let
\begin{equation}
\label{eqn::nu_r_time_def}
d\nu^n(t,u) = d\nu_t^n(u) dt \quad\text{for}\quad u \in \partial \D \quad\text{and}\quad t \geq 0.
\end{equation}
Then $\nu^n \in \CN$ for all $n \in \N$.

\begin{proposition}
\label{prop::measure_limit_exists}
There exists a sequence $(n_j)$ in $\N$ with $n_j \to \infty$ as $j \to \infty$ and a measure $\nu \in \CN$ such that $\nu^{n_j} \to \nu$ in $\CN$ almost surely.  That is, we almost surely have for each $T \geq 0$ and continuous function $\phi \colon [0,T] \times \partial \D \to \R$ that
\[ \lim_{j \to \infty} \int_{[0,T] \times \partial \D} \phi(s,u) d\nu^{n_j}(s,u) = \int_{[0,T] \times \partial \D} \phi(s,u) d\nu(s,u).\]
\end{proposition}
In the proof that follows and throughout the rest of this section, for measures $\varsigma_1,\varsigma_2$, we will use the notation $d(\varsigma_1-\varsigma_2)$ to denote integration against the signed measure $\varsigma_1-\varsigma_2$.
\begin{proof}[Proof of Proposition~\ref{prop::measure_limit_exists}]
Fix $n,n' \in \N$, $T \geq 0$, and a continuous function $\phi \colon [0,T] \times \partial \D \to \R$.  By Fubini's theorem, we have that
\begin{align}
   &\E\left[ \left|\int_{[0,T] \times \partial \D}  \phi(s,u) d\big(\nu^n(s,u) - \nu^{n'}(s,u)\big)\right| \right] \notag\\
\leq &\int_0^T \E\left[ \left|\int_{\partial \D}  \phi(s,u) d\big(\nu_s^n(u) - \nu_s^{n'}(u)\big)\right| \right] ds. \label{eqn::boundary_convergence}
\end{align}
By stationarity, the inside expectation does not depend on $s$.  Moreover, the integral inside of the expectation converges to zero as $n, n' \to \infty$ for any fixed $s \geq 0$ because $\nu_s^n$ converges weakly almost surely as $n \to \infty$ to the $-\tfrac{2}{\sqrt{\kappa}}$-quantum boundary measure on $\partial \D$ associated with $\Fh_s$ normalized to be a probability measure (recall \eqref{eqn::boundary_measure_onb}) and the quantity inside of the expectation is bounded by $2\| \phi\|_{L^\infty}$.  Therefore it follows from the dominated convergence theorem that the expression in \eqref{eqn::boundary_convergence} converges to zero as $n,n' \to \infty$.  Applying Markov's inequality and the Borel-Cantelli lemma gives the almost sure convergence of $\int_{[0,T] \times \partial \D} \phi(s,u) d\nu^n(s,u)$ provided we take a limit along a sequence $(n_j)$ in $\N$ which tends to $\infty$ sufficiently quickly.  By possibly passing to a further (diagonal) subsequence, this, in turn, gives us the almost sure convergence of $\int_{[0,T] \times \partial \D} \phi(s,u) d\nu^{n_j}(s,u)$ for any countable collection of continuous functions $\phi \colon [0,T] \times \partial \D \to \R$.  This proves the result because we can pick a countable dense subset of continuous functions $\phi \colon [0,T] \times \partial \D \to \R$ with respect to the uniform topology on $[0,T] \times \partial \D$ and then use the continuity of the aforementioned integral with respect to the uniform topology on continuous functions.  Passing to a final (diagonal) subsequence gives the convergence for all $T \geq 0$ simultaneously.
\end{proof}

\subsubsection{Loewner evolution driven by the $\QLE$ driving measure solves the $\QLE$ dynamics}
\label{subsubsec::loewner_rate_is_qle}

\begin{figure}[ht!]
\begin{center}
\includegraphics[scale=0.85,page=2]{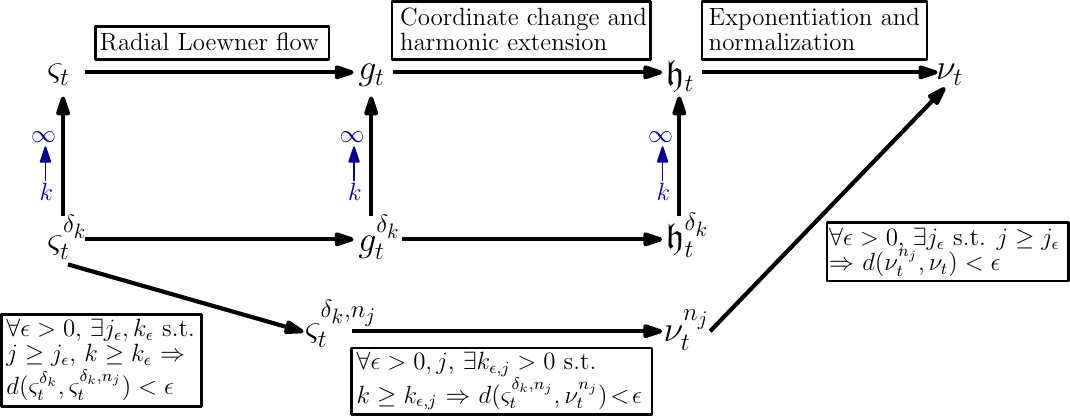}
\end{center}
\caption{\label{fig::qle_existence_scheme} (Continuation of Figure~\ref{fig::tightness_scheme}.)  Shown is the approximation scheme used to show that $\varsigma_t = \nu_t$ (Proposition~\ref{prop::qle_solution}) to complete the proof of the existence of $\QLE(\gamma^2,\eta)$ for $(\gamma^2,\eta)$ on one of the upper two curves from Figure~\ref{fig::etavsgamma}.  The statements in each of the three boxes along the bottom of the figure from left to right are proved in Lemma~\ref{lem::mu_delta_mu_delta_n}, Lemma~\ref{lem::mu_delta_n_nu_delta_n}, and Proposition~\ref{prop::measure_limit_exists}, respectively.  The symbol $d$ represents a notion of ``closeness'' which is related to the topology of $\CN$.  To show that $\varsigma_t = \nu_t$, we first pick $j$ very large so that $d(\nu_t^{n_j},\nu_t) < \epsilon$ and $d(\varsigma_t^{\delta_k},\varsigma_t^{\delta_k,n_j}) < \epsilon$.  We then pick $k$ very large so that $d(\varsigma_t^{\delta_k},\varsigma_t) < \epsilon$ and $d(\varsigma_t^{\delta_k,n_j},\nu_t^{n_j}) <~\epsilon$.}
\end{figure}

Throughout, we let $\nu$ be the (random) element of $\CN$ constructed in Proposition~\ref{prop::measure_limit_exists} and we let $\varsigma$ be the (random) element of $\CN$ which drives $(g_t)$.  As explained in Proposition~\ref{prop::all_tight}, we know that we can obtain $\Fh_t$ from $g_t$ by $\pHarm(h \circ g_t^{-1} + Q \log| (g_t^{-1})'| + \tfrac{\kappa+6}{2\sqrt{\kappa}}\log|\cdot|)$ (normalized to vanish at the origin) and that we can obtain $\nu$ by exponentiating $\Fh_t$.  Therefore the proof of Theorem~\ref{thm::existence} will be complete upon establishing the following.

\begin{proposition}
\label{prop::qle_solution}
We almost surely have that $\varsigma = \nu$.
\end{proposition}

A schematic illustration of the main steps in the proof of Proposition~\ref{prop::qle_solution} is given in Figure~\ref{fig::qle_existence_scheme}.  The strategy is to relate $\varsigma$ and $\nu$ using three approximating measures: $\varsigma_t^{\delta_k} dt$, $\varsigma_t^{\delta_k,n_j} dt$, and $\nu_t^{\delta_k,n_j} dt$.  We introduced $\varsigma_t^{\delta_k}$ in \eqref{eqn::approx_measure} and we introduced $\nu_t^{n_j}$ in \eqref{eqn::nu_r_def}.  We know that $\varsigma_t^{\delta_k} dt \to \varsigma$ as $k \to \infty$ and $\nu_t^{n_j} dt \to \nu$ as $j \to \infty$ in $\CN$.  In the rest of this section, we will introduce $\varsigma_t^{\delta_k,n_j} dt$ and then show that $\varsigma_t^{\delta_k,n_j} dt$ is close to both $\varsigma_t^{\delta_k} dt$ and $\nu_t^{n_j} dt$ for large $j$ and $k$.

We now give the definition of $\varsigma_t^{\delta,n}$.  Fix $n \in \N$ and let
\begin{equation}
\label{eqn::nu_t_delta_n}
\nu_t^{\delta,n} = \CZ_{n,t,\delta}^{-1} \exp(\alpha_\kappa \Fh_t^{\delta,n}(u))du
\end{equation}
where $\Fh_t^{\delta,n}$ is the orthogonal projection of $\Fh_t^\delta$ onto the subspace spanned by $\{f_1,\ldots,f_n\}$ as defined above and $\CZ_{n,t,\delta}$ is a normalization constant so that $\nu_t^{\delta,n}$ is a probability measure.  For each $\ell \in \N_0$, let $U^{\delta,\ell,n}$ be a point picked from $\nu_t^{\delta,n}$.  Fix $\zeta > 0$.  For each $t \geq 0$, it follows from \eqref{eqn::boundary_measure_onb} that 
\[ \nu_t^{\delta,n} \to \nu_t^\delta \quad\text{as}\quad n \to \infty\]
weakly almost surely.  By the stationarity of $\Fh_t^\delta$, the rate of convergence does not depend $t$ or $\delta$.  It thus follows that there exists non-random $n_0 \in \N$ depending only on $\zeta$ such that for all $n \geq n_0$ we can couple the sequences $(U^{\delta,\ell})$ and $(U^{\delta,\ell,n})$ together so that
\begin{equation}
\label{eqn::points_close}
\p[E_\ell^{\delta,n}] \leq \zeta \quad\text{where}\quad E_\ell^{\delta,n} = \{|U^{\delta,\ell} - U^{\delta,\ell,n}| \geq \zeta\}.
\end{equation}
We assume throughout that $U^{\delta,\ell,n}$ and $U^{\delta,n}$ are coupled as such.  Let
\[ \varsigma_t^{\delta,n} = \sum_{\ell =0}^\infty \one_{[\delta \ell,\delta(\ell+1))}(t) \delta_{U^{\delta,\ell,n}}.\]
That is, $\varsigma_t^{\delta,n}$ for $t \in [\delta \ell,\delta(\ell+1))$ with $\ell \in \N_0$ is given by the Dirac mass located at $U^{\delta,\ell,n}$.  Note that $\varsigma_t^{\delta,n}$ is defined analogously to $\varsigma_t^\delta$ except the $U^{\delta,\ell,n}$ are picked from $\nu_t^{\delta,n}$ in place of $\nu_t^\delta = \CZ_{t,\delta}^{-1} \exp(\alpha_\kappa \Fh_t^\delta(u))du$ and the Brownian motions have been omitted.

The proof of Proposition~\ref{prop::qle_solution} has two steps.

The first step (Lemma~\ref{lem::mu_delta_mu_delta_n}) is to show that for each $\epsilon > 0$ there exists $j_\epsilon, k_\epsilon \in \N$ such that $\varsigma_t^{\delta_k} dt$ and $\varsigma_t^{\delta_k,n_j} dt$ are $\epsilon$-close for all $j \geq j_\epsilon$ and $k \geq k_\epsilon$ (the result is stated for more general values of $\delta$ and $n$ because it is not necessary in the proof to work along the sequences $(\delta_k)$ and $(n_j)$).  We note that the choice of $k$ determines the speed at which the location of the Dirac mass is resampled while the choice of $j$ determines the expected fraction of the $(U^{\delta_k,\ell,n_j})$ which are close to the $(U^{\delta_k,\ell})$ (recall \eqref{eqn::points_close}).

The second step (Lemma~\ref{lem::mu_delta_n_nu_delta_n}) is to show that for each $\epsilon > 0$ and $j \in \N$ there exists $k_{\epsilon,j} > 0$ such that $\varsigma_t^{\delta_k,n_j} dt$ and $\nu_t^{n_j} dt$ are $\epsilon$-close for all $k \geq k_{\epsilon,j}$ (the result is stated for more general values of $n$ because in the proof it is not necessary to work along the sequence $(n_j)$).  The proof is by a law of large numbers argument.  By construction, we know that $t \mapsto \nu_t^{\delta_k,n_j}$ is continuous for a \emph{fixed} value of $j$ and the choice of $j$ controls our estimate its modulus of continuity.  When $\delta_k > 0$ is sufficiently small depending on $j$ so that the rate at which the points $(U^{\delta_k,\ell,n_j})$ are being sampled is much faster than the rate at which $t \mapsto \nu_t^{\delta_k,n_j}$ is changing, we can think of organizing the points $(U^{\delta_k,\ell,n_j})$ into groups each of which is close to being i.i.d.\ This is what leads to the law of large numbers effect.

Once these estimates have been established, we will pick $j$ very large so that both $\nu_t^{n_j} dt$ is close to $\nu$ and $\varsigma_t^{\delta_k,n_j} dt$ is close to $\varsigma_t^{\delta_k} dt$.  We will then choose $k$ to be very large so that $\varsigma_t^{\delta_k} dt$ is close to $\varsigma$ and $\varsigma_t^{\delta_k,n_j} dt$ is close to $\nu_t^{n_j} dt$.

\begin{lemma}
\label{lem::mu_delta_mu_delta_n}
Fix $T > 0$ and suppose that $\phi \colon [0,T] \times \partial \D \to \R$ is continuous.  For every $\epsilon > 0$ there exists $n_\epsilon \in \N$ and $\delta_{\epsilon} > 0$ such that $n \geq n_\epsilon$ and $\delta \in (0,\delta_\epsilon)$ implies that 
\[ \E\left| \int_{[0,T] \times \partial \D} \phi(s,u) d\big(\varsigma_s^\delta(u) - \varsigma_s^{\delta,n}(u) \big) ds\right| < \epsilon.\]
\end{lemma}
\begin{proof}
We are first going to explain how to bound the difference when $s = \ell \delta$ for some $\ell \in \N_0$.  Fix $\epsilon > 0$.  By the continuity of $\phi$, it follows from \eqref{eqn::points_close} that there exists $n_{1,\epsilon}$ such that $n \geq n_{1,\epsilon}$ implies that
\begin{equation}
\label{eqn::ej_error1}
 \E\left|\int_{\partial \D} \one_{(E_\ell^{\delta,n})^c} \phi(\delta \ell,u) d\big(\varsigma_{\ell \delta}^{\delta}(u) - \varsigma_{\ell \delta}^{\delta,n}(u) \big) \right| < \frac{\epsilon}{4}
\end{equation}
provided we choose $\zeta > 0$ small enough.  Since the integrand is bounded, it also follows from \eqref{eqn::points_close} that there exists $n_{2,\epsilon}$ such that $n \geq n_{2,\epsilon}$ implies that
\begin{equation}
\label{eqn::ej_error2}
\E\left|\int_{\partial \D} \one_{E_\ell^{\delta,n}} \phi(\delta \ell,u) d\big(\varsigma_{\ell \delta}^{\delta}(u) - \varsigma_{\ell \delta}^{\delta,n}(u) \big) \right| < \frac{\epsilon}{4}.
\end{equation}
Combining \eqref{eqn::ej_error1} and \eqref{eqn::ej_error2} gives that $n \geq n_\epsilon = \max(n_{1,\epsilon},n_{2,\epsilon})$ implies that 
\[
\E\left|\int_{\partial \D} \phi(\delta \ell,u) d\big(\varsigma_{\ell \delta}^{\delta}(u) - \varsigma_{\ell \delta}^{\delta,n}(u) \big) \right| < \frac{\epsilon}{2}.
\]
Using the continuity of Brownian motion, it follows that there exists $\delta_\epsilon > 0$ such that for all $n \geq n_\epsilon$ and $\delta \in (0,\delta_\epsilon)$ we have that
\[
\sup_{s \in [\delta \ell, \delta(\ell+1))} \E\left|\int_{\partial \D} \phi(s,u) d\big(\varsigma_s^{\delta}(u) - \varsigma_s^{\delta,n}(u) \big) \right| < \epsilon.
\]
This implies the desired result.
\end{proof}

\begin{lemma}
\label{lem::mu_delta_n_nu_delta_n}
Fix $T > 0$ and suppose that $\phi \colon [0,T] \times \partial \D \to \R$ is continuous.  For each $n \in \N$ there exists $k_{\epsilon,n} \in \N$ such that $k \geq k_{\epsilon,n}$ implies that
\begin{equation}
\label{eqn::measures_converge}
\E\left|\int_{[0,T] \times \partial \D} \phi(s,u) d\big(\varsigma_s^{\delta_k,n}(u) - \nu_s^n(u) \big) ds\right| < \epsilon.
\end{equation}
\end{lemma}

It is important that the limit in the statement of Lemma~\ref{lem::mu_delta_n_nu_delta_n} is along the sequence $(\delta_k)$ because then we have that $\Fh_t^{\delta_k} \to \Fh_t$ as $k \to \infty$ and $\nu_t^n$ is defined in terms of~$\Fh_t$.

\begin{proof}[Proof of Lemma~\ref{lem::mu_delta_n_nu_delta_n}]
Let (recall~\eqref{eqn::nu_t_delta_n})
\[ \ol{\varsigma}_t^{\delta_k,n}(u) du = \sum_{\ell=0}^\infty \one_{[\ell \delta_k,(\ell+1)\delta_k)}(t) \nu_{\ell \delta_k}^{\delta_k,n}(u)du.\]
Note that the increments of
\[ \int_0^{\delta_k \ell} \int_{\partial \D} \phi(s,u) d(\varsigma_s^{\delta_k,n}(u) - \ol{\varsigma}_s^{\delta_k,n}(u)) ds\]
as $\ell$ varies are uncorrelated given $\Fh_t^{\delta_k,n}$.  Consequently, we have that
\begin{align*}
 \E\left[ \left(\int_{[0,T] \times \partial \D} \phi(s,u) d(\varsigma_s^{\delta_k,n}(u) - \ol{\varsigma}_s^{\delta_k,n}(u)) ds \right)^2 \right] = O(\delta_k)
\end{align*}
where the implicit constant in $O(\delta_k)$ depends on $T$.  It thus suffices to prove \eqref{eqn::measures_converge} with $\ol{\varsigma}_t^{\delta_k,n}$ in place of $\varsigma_t^{\delta_k,n}$.  By the continuity of $\Fh_t$ and the local uniform convergence of $\Fh_t^{\delta_k}$ to $\Fh_t$ as $k \to \infty$, it is easy to see that
\[ \lim_{k \to \infty} \E\left|\int_{[0,T] \times \partial \D} \phi(s,u) d(\ol{\varsigma}_s^{\delta_k,n}(u) - \nu_s^{\delta_k,n}(u)) ds\right| = 0.\]
Combining gives \eqref{eqn::measures_converge}.
\end{proof}

\begin{proof}[Proof of Proposition~\ref{prop::qle_solution}]
Fix $T > 0$ and $\phi \colon [0,T] \times \partial \D \to \R$ continuous.  It suffices to show that
\begin{equation}
\label{eqn::nu_equals_sigma} \int_{[0,T] \times \partial \D} \phi(s,u) d(\nu(s,u) - \varsigma(s,u)) = 0
\end{equation}
almost surely.  Fix $\epsilon > 0$.  Then Proposition~\ref{prop::all_tight} implies that there exists $k_\epsilon \in \N$ such that $k \geq k_\epsilon$ implies that
\begin{equation}
\label{eqn::final_error1}
   \E \left| \int_{[0,T] \times \partial \D} \phi(s,u) \big(d \varsigma_s^{\delta_k}(u) ds - d \varsigma(s,u) \big) \right| < \frac{\epsilon}{4}.
\end{equation}
Lemma~\ref{lem::mu_delta_mu_delta_n} implies that there exists $j_\epsilon \in \N$ such that, by possibly increasing the value of $k_\epsilon$, we have that $j \geq j_\epsilon$ and $k \geq k_\epsilon$ implies that
\begin{equation}
\label{eqn::final_error2}
   \E \left| \int_{[0,T] \times \partial \D} \phi(s,u) d\big(\varsigma_s^{\delta_k}(u) - \varsigma_s^{\delta_k,n_j}(u) \big) ds \right| < \frac{\epsilon}{4}.
\end{equation}
Proposition~\ref{prop::measure_limit_exists} implies that, by possibly increasing the value of $j_\epsilon$, we have that $j \geq j_\epsilon$ implies that
\begin{equation}
\label{eqn::final_error3}
   \E \left| \int_{[0,T] \times \partial \D} \phi(s,u) d\big(\nu^{n_j}(s,u) - \nu(s,u)\big) \right| < \frac{\epsilon}{4}.
\end{equation}
Let $j = j_\epsilon$.  Lemma~\ref{lem::mu_delta_n_nu_delta_n} implies that there exists $k_{\epsilon,j} \in \N$ such that $k \geq k_{\epsilon,j}$ implies that
\begin{equation}
\label{eqn::final_error4}
   \E \left| \int_{[0,T] \times \partial \D} \phi(s,u) d\big(\varsigma_s^{\delta_k,n_j}(u) - \nu_s^{n_j}(u) \big) ds \right| < \frac{\epsilon}{4}.
\end{equation}
Using the triangle inequality, \eqref{eqn::final_error1}---\eqref{eqn::final_error4}, and that $\epsilon > 0$ was arbitrary implies \eqref{eqn::nu_equals_sigma}, as desired.
\end{proof}

We can now complete the proof of Theorem~\ref{thm::existence}.

\begin{proof}[Proof of Theorem~\ref{thm::existence}]
Proposition~\ref{prop::qle_solution} gives us that the limiting triple $(\nu_t,g_t,\Fh_t)$ satisfies all three arrows of the $\QLE$ dynamics as described in Figure~\ref{fig::QLEtriangle}.  That the limiting triple $(\nu_t,g_t,\Fh_t)$ satisfies $\eta$-DBM scaling as defined in Definition~\ref{def::eta_scaling} follows from the argument explained in Section~\ref{subsec::zipperspecialization}.  Combining gives the desired result.
\end{proof}

%% file: tex/dynamics.tex
We are now going to identify the limiting dynamics which govern the time evolution of the harmonic component $(\Fh_t)$ of the $\QLE(\gamma^2,\eta)$ processes $(\nu_t,g_t,\Fh_t)$ constructed in Section~\ref{sec::existence} to prove Theorem~\ref{thm::existence}.

As in Section~\ref{sec::couplings}, we will use~$\CP$ (resp.\ $\ol{\CP}$) to denote $2\pi$ times the Poisson (resp.\ conjugate Poisson) kernel on~$\D$ (recall~\eqref{eqn::poisson_kernel}).  We also let $\CP^\star = \CP - 1$ so that $\CP^*(0,u) = 0$ for all $u \in \partial \D$.  We use $a \cdot b$ to denote the standard dot product.  If one of $a$ or $b$ is a complex number, we will identify it with a vector in $\R^2$ when writing $a \cdot b$.  In particular, if $a = x + yi$ and $b = u + vi$ for $x,y,u,v \in \R$ then $a \cdot b = xu + yv$.

Let
\begin{equation}
\label{eqn::dtzu}
D_t(z,u) = -\nabla \Fh_t(z) \cdot \Phi(u,z) + \frac{1}{\sqrt{\kappa}}\CP^\star(z,u) + Q (\partial_\theta \ol{\CP})(z,u).
\end{equation}
In~\eqref{eqn::dtzu}, $\partial_\theta \ol{\CP}(z,u)$ for $z \in \D$ and $u \in \partial \D$ means the partial derivative of the map $\D \times \R \to \R$ given by $(z,\theta) \mapsto \ol{\CP}(z,e^{i\theta})$ with respect to $\theta$ evaluated at $(z,u)$; see~\eqref{eqn::deriv_p_theta} for an explicit formula.  We will show that the time-evolution of $\Fh_t$ is governed by the equation 
\begin{equation}
\label{eqn::qle_evolution}
 \dot{\Fh}_t(z) = \int_{\partial \D} \bigg(D_t(z,u) + \CP^\star(z,u) W(t,u) \bigg) d\nu_t(u)
\end{equation}
where $W(t,u)$ denotes a space-time white noise on $\partial \D \times [0,\infty)$.  The reason that $\CP^\star$ appears in~\eqref{eqn::qle_evolution} rather $\CP$ is that we have normalized $\Fh_t$ so that $\Fh_t(0) = 0$ for all $t \geq 0$.  If we had instead only normalized at time $0$ by setting $\Fh_0(0) = 0$, then the resulting equation would be the same except with $\CP$ in place of $\CP^\star$.

The evolution equation~\eqref{eqn::qle_evolution} is significant for two reasons.  First, if one were able to show that~\eqref{eqn::qle_evolution} has a unique solution then it would imply that the subsequential limits used to construct $\QLE(\gamma^2,\eta)$ for $(\gamma^2,\eta)$ on one of the top two curves from Figure~\ref{fig::etavsgamma} to prove Theorem~\ref{thm::existence} are actually true limits.  In other words, it is not necessary to pass along a subsequence $(\delta_k)$ as $\delta \to 0$.  Second, although its derivation takes as input the existence of $\QLE(\gamma^2,\eta)$ for $(\gamma^2,\eta)$ on one of the upper two curves from Figure~\ref{fig::etavsgamma} as proved in Theorem~\ref{thm::existence}, one can make the ansatz that~\eqref{eqn::qle_evolution} describes the dynamics for other $(\gamma^2,\eta)$ values.  It is then possible that a careful analysis of~\eqref{eqn::qle_evolution} could be used to determine the stationary distribution for the dynamics in these cases which in turn might suggest a way to construct these processes.

We will make~\eqref{eqn::qle_evolution} rigorous (Proposition~\ref{prop::sde}) by putting $\Fh_t$ into coordinates and then showing that the coordinate processes satisfy a certain infinite dimensional SDE in integrated form.  The particular choice of coordinates is not important for the proof; we choose $L^2(\tfrac{1}{2} \D)$ to be concrete.  We will start in Section~\ref{subsec::approximate_dynamics} by studying the dynamics of the $\delta$-approximations introduced in Section~\ref{subsec::approximation}.  We will then use this in Section~\ref{subsec::martingale_problem} to show that the $\delta \to 0$ subsequential limits must satisfy a certain martingale problem.  This, in turn, allows us to derive the SDE satisfied by the $\delta \to 0$ subsequential limits.

\subsection{SDE for the approximate dynamics}
\label{subsec::approximate_dynamics}

Let $(g_t^\delta)$ be the radial Loewner flow associated with the $\delta$-approximation to $\QLE(\gamma^2,\eta)$ using forward $\SLE$s as described in Section~\ref{subsec::approximation}.  In other words,
\begin{align*}
\dot{g}_t^\delta(z) = \int_{\partial \D} \Phi(u,g_t^\delta(z)) d\varsigma_t^\delta(u),\quad g_0^\delta(z) = z,\quad \varsigma_t^\delta = \sum_{\ell=0}^\infty \one_{[\delta \ell,\delta(\ell+1))}(t) \delta_{U^{\delta,\ell} W_t^{\delta,\ell}}.
\end{align*}
Recall that the points $U^{\delta,\ell}$ and Brownian motions $W^{\delta,\ell}$ in $\partial \D$ were defined in Section~\ref{subsec::approximation}.  In this section, we are going to describe the dynamics of the harmonic component $(\Fh_t^\delta)$ of the $\delta$-approximation $(\varsigma_t^\delta,g_t^\delta,\Fh_t^\delta)$.

For each $0 \leq s \leq t$ we let $g_{s,t}^\delta = g_t^\delta \circ (g_s^\delta)^{-1}$.  For each $\ell \in \N_0$, note that $g_{\delta \ell,\delta(\ell+1)}^\delta$ is the conformal transformation which maps away the $(\ell+1)$st radial $\SLE_\kappa$ process in the $\delta$-approximation to $\QLE(\gamma^2,\eta)$.  It will be convenient in what follows to describe the maps $g_{\delta \ell,\delta (\ell+1)}^\delta$ in terms of a reverse rather than forward Loewner flow.  We can accomplish this as follows.  For each $\ell  \in \N_0$, we let $(f_t^{\delta,\ell})$ be the family of conformal maps which solve the reverse radial Loewner equation
\begin{align*}
\dot{f}_t^{\delta,\ell}(z) = -\int_{\partial \D} \Phi(u,f_t^{\delta,\ell}(z)) d\varsigma_{\delta(2\ell+1)-t}^\delta(u), \quad f_{\delta \ell}^{\delta,\ell}(z) = z, \quad t \in [\delta \ell,\delta(\ell+1)].
\end{align*}
(Note that these reverse radial Loewner evolutions are not centered.)  Then $f_{\delta(\ell+1)}^{\delta,\ell} = (g_{\delta \ell,\delta(\ell+1)}^\delta)^{-1}$.  To see this, note that $q_t = f_{\delta(\ell+1)-t}^{\delta,\ell}$ satisfies the forward radial Loewner equation with $q_0(z) = f_{\delta(\ell+1)}^{\delta,\ell}(z)$, $q_\delta(z) = z$, and driving measure $\varsigma_{\delta \ell + t}$.  The claim follows because $g_{\delta \ell,\delta \ell+t}^\delta \circ f_{\delta(\ell+1)}^{\delta,\ell}$ satisfies the same equation, has the same initial condition, and solutions to this equation are unique. It is not in general true, however, that $f_t^{\delta,\ell} = (g_{\delta \ell,t}^\delta)^{-1}$ for intermediate values of $t$ in $(\delta \ell,\delta(\ell+1))$.

We also let
\begin{equation}
\label{eqn::v_delta_ell}
 V_t^{\delta,\ell} = U^{\delta,\ell} W_{\delta(2\ell+1)-t}^{\delta,\ell}.
\end{equation}

\begin{proposition}
\label{prop::approximate_evolution}
Fix $t \geq 0$, let $\ell \in \N_0$ be such that $t \in [\delta \ell,\delta(\ell+1))$, and let
\begin{align*}
   \wt{D}_t^\delta(z) &= (\nabla \Fh_{\delta \ell}^\delta)(f_t^{\delta,\ell}(z)) \cdot \dot{f}_t^{\delta,\ell}(z) +  \frac{1}{\sqrt{\kappa}} \CP^\star(f_t^{\delta,\ell}(z),V_t^{\delta,\ell}) + Q(\partial_\theta \ol{\CP})(f_t^{\delta,\ell}(z),V_t^{\delta,\ell}),\\
   \wt{\sigma}_t^\delta(z) &=  \CP^\star(f_t^{\delta,\ell}(z),V_t^{\delta,\ell}).
\end{align*}
There exists a standard Brownian motion $B$ and a process $\wt{\Fh}_t^\delta$ taking values in $\CH$ (as defined in Section~\ref{subsec::spaces}) which solves the SDE
\begin{equation}
\label{eqn::approximate_evolution}
d\wt{\Fh}_t^\delta(z) = \wt{D}_t^\delta(z) dt + \wt{\sigma}_t^\delta(z) d B_t \quad\text{for each} \quad z \in \D,
\end{equation}
and $\wt{\Fh}_t^\delta = \Fh_t^\delta$ for each $t$ of the form $\delta \ell$ for $\ell \in \N_0$.  Moreover, for each $\epsilon > 0$, $K \subseteq \D$ compact, and $T > 0$ we have
\[ \p\left[\sup_{\substack{t \in [0,T] \\ z \in K}}  |\Fh_t^\delta(z) - \wt{\Fh}_t^\delta(z)| \geq \epsilon \right] \to 0 \quad\text{as}\quad \delta \to 0.\]
Finally, $\wt{\Fh}_t^\delta \stackrel{d}{=} \Fh_t^\delta$ for each $t \geq 0$.
\end{proposition}

Recall that $\Fh_t^\delta$ is harmonic in $\D$ hence $C^\infty$ in $\D$ for each $t \geq 0$.  Consequently, the definition of $\wt{D}_t^\delta$ makes sense pointwise for $z \in \D$.  We will perform the calculation in several steps.  It suffices to prove the result for $\ell = 0$.  In order to avoid carrying around too much notation, we will write $f_t^\delta = f_t^{\delta,0}$ and $V_t^\delta = V_t^{\delta,0}$.  Let $du$ denote the standard length measure on $\partial \D$.  Since $\Fh_t^\delta$ is harmonic in $\D$ with $\Fh_t^\delta(0) = 0$, for $t = \delta$ we can write
\begin{equation}
\label{eqn::dht}
 \begin{split}
 \Fh_t^\delta(z) =& \frac{1}{2\pi} \int_{\partial \D} \CP(f_t^\delta(z),u) \Fh_0^\delta(u)  d u - \frac{\kappa+6}{2\sqrt{\kappa}} \left( \log|f_t^\delta(z)| - \log|z| \right) +\\
 & \frac{1}{2\pi}\int_{\partial \D} \CP^\star(z,u) (h \circ f_t^\delta)(u) du + Q \log|(f_t^\delta)'(z)| - \frac{1}{\sqrt{\kappa}} t
 \end{split}
\end{equation}
where $h$ is a zero-boundary GFF on $\D$ independent of $f_t^\delta$ and $\Fh_0^\delta$.  (The term $-(1/\sqrt{\kappa}) t$ arises because the two integrals above vanish at $0$, $(f_t^\delta)'(0) = e^{-t}$, and we have that $\Fh_t^\delta(0) = 0$.)  We can compute the time derivative of $2\pi$ times the first integral from~\eqref{eqn::dht} as follows.
\begin{align}
  d \left( \int_{\partial \D} \CP(f_t^\delta(z),u) \Fh_0^\delta(u)  d u \right)
 &=\int_{\partial \D} \bigg(  (\nabla \CP)(f_t^\delta(z),u) \cdot \dot{f}_t^\delta(z)  \bigg)  \Fh_0^\delta(u) du dt \notag\\
&= 2\pi (\nabla \Fh_0^\delta)(f_t^\delta(z)) \cdot \dot{f}_t^\delta(z) dt. \label{eqn::harmonic_change}
\end{align}
By first computing $d \log(f_t^\delta(z))$ and then taking real parts, it is also easy to see that 
\begin{equation}
\label{eqn::log_change}
d\left( \log|f_t^\delta(z)| - \log|z|\right) = -\CP(f_t^\delta(z),V_t^\delta) dt.
\end{equation}

\noindent We will now compute the time-derivative of the second to last term from~\eqref{eqn::dht}.

\begin{lemma}
\label{lem::mass_change}
For $t \in [0,\delta]$, we have that
\[ \frac{d}{dt} \log |(f_t^\delta)'(z)| =  
(\partial_\theta \ol{\CP})(f_t^\delta(z),V_t^\delta)  - \CP(f_t^\delta(z),V_t^\delta).\]
\end{lemma}
\begin{proof}
We have that,
\begin{align}
   &\frac{d}{dt} \log (f_t^\delta)'(z)
= \frac{1}{(f_t^\delta)'(z)} (\dot{f}_t^\delta)'(z) \notag\\
=& -\int_{\partial \D} \frac{u + f_t^\delta(z)}{u - f_t^\delta(z)} d\varsigma_{\delta-t}^\delta(u) - f_t^\delta(z) \int_{\partial \D} \frac{2u}{(u - f_t^\delta(z))^2} d\varsigma_{\delta-t}^\delta(u). \label{eqn::deriv_dyn}
\end{align}
For $\theta \in \R$, we note that
\[ \partial_\theta \left(\frac{e^{i \theta} + z}{e^{i \theta} - z}\right)
   =  \frac{-2iz e^{i \theta}}{(e^{i \theta} - z)^2}.\]
In particular,
\[  \frac{-2z e^{i \theta}}{(e^{i \theta} - z)^2} = -i \partial_\theta \left(\frac{e^{i \theta} + z}{e^{i \theta} - z}\right)\]
so that
\begin{equation}
\label{eqn::deriv_p_theta} 
 \re \left( \frac{-2z e^{i \theta}}{(e^{i \theta} - z)^2} \right) = \im\left(\partial_\theta \left(\frac{e^{i \theta} + z}{e^{i \theta} - z}\right)\right) = \partial_\theta \ol{\CP}(z,e^{i \theta}).
\end{equation}
This allows us to write the real part of~\eqref{eqn::deriv_dyn} more concisely as
\begin{align*}
& - \int_{\partial \D} \CP(f_t^\delta(z),u) d\varsigma_{\delta-t}^\delta(u) + \int_{\partial \D} (\partial_\theta\ol{\CP})(f_t^\delta(z),u)  d\varsigma_{\delta-t}^\delta(u) \notag\\
=& - \CP(f_t^\delta(z),V_t^\delta) + (\partial_\theta \ol{\CP})(f_t^\delta(z),V_t^\delta),
\end{align*}
which proves the lemma.
\end{proof}

The following lemma gives the law of the third term from~\eqref{eqn::dht}.  We emphasize that it does not describe the dynamics of this term in time.

\begin{lemma}
\label{lem::gff_harmonic_change}
Suppose that $h$ is a zero-boundary GFF on $\D$, let $(\psi_t)$ be the reverse radial Loewner flow driven by $e^{iW_t}$ where $W \colon [0,\infty) \to \R$ is a continuous function, and let $H_t$ be the harmonic extension of $(h \circ \psi_t)(z)$ from $\partial \D$ to $\D$.  For each $t \geq 0$ there exists a standard Brownian motion $B$ such that
\begin{equation}
\label{eqn::brownian_rep}
H_t(z) = \int_0^t \CP(\psi_s(z),e^{i W_s}) dB_s.
\end{equation}
\end{lemma}

\noindent We emphasize that in the statement of Lemma~\ref{lem::gff_harmonic_change}, the function $W$ should be thought of as deterministic. In particular, $B$ and $W$ are independent and so are $h$ and $W$.  The main ingredient in the proof of Lemma~\ref{lem::gff_harmonic_change} is the following analog of~\eqref{eqn::green_change} from Lemma~\ref{lem::green_change} with the Dirichlet Green's function in place of the Neumann Green's function on $\D$.

\begin{lemma}
\label{lem::green_change_dirichlet}  Suppose that $G$ is the Dirichlet Green's function on $\D$, let $(\psi_t)$ be the reverse radial Loewner flow driven by $e^{iW_t}$ where $W \colon [0,\infty) \to \R$ is a continuous function, and let
\[ G_t(z,w) = G(\psi_t(z),\psi_t(w)) \quad\text{for}\quad z,w \in \D \quad\text{and}\quad t \geq 0.\]
Then
\[ dG_t(z,w) = \CP(\psi_t(z),e^{iW_t}) \CP(\psi_t(w),e^{iW_t}) dt.\]
\end{lemma}
\begin{proof}
This follows from a calculation which is similar to that of Lemma~\ref{lem::green_change}.  In particular, one computes the difference between the expressions in~\eqref{eqn::green_change1} and~\eqref{eqn::green_change2} (rather than the sum, as in the case of Lemma~\ref{lem::green_change}).  See also the chordal version of this calculation carried out in \cite[Section~4.1]{sheffield2010weld}.
\end{proof}

\begin{proof}[Proof of Lemma~\ref{lem::gff_harmonic_change}]
Fix a Brownian motion $B$ and let $\wt{H}_t(z)$ be equal to the expression on the right side of~\eqref{eqn::brownian_rep} with respect to $B$.  Then $H_t(z)$ and $\wt{H}_t(z)$ are both mean zero Gaussian processes (with $t$ fixed and thought of as functions of $z$), so to complete the proof we just have to show that they have the same covariance function.  For $z,w \in \D$, we have that
\begin{align}
\label{eqn::brownian_cov}
      \cov( \wt{H}_t(z), \wt{H}_t(w))
&= \int_0^t \CP(\psi_s(z),e^{i W_s}) \CP(\psi_s(w),e^{i W_s}) ds
\end{align}
For each $r \in (0,1)$, let $H_t^r$ be the function which is harmonic in $\D$ whose boundary values are given by $(h \circ \psi_t)(r z)$.  Then $H_t^r \to H_t$ locally uniformly as $r \to 1$.  By Lemma~\ref{lem::green_change_dirichlet}, we have that
\begin{align*}
  &    \cov(H_t^r(z), H_t^r(w))\\
=& \frac{1}{(2\pi)^2} \int_{\partial \D} \int_{\partial \D} \CP(z,u) \CP(w,v) G(\psi_t(r u),\psi_t(r v)) d u d v\\
=& \frac{1}{(2\pi)^{2}} \int_0^t \int_{\partial \D} \int_{\partial \D}  \CP(z,u) \CP(w,v) \CP(\psi_s(r u),e^{i W_s})\CP(\psi_s(r v), e^{i W_s}) d u d v ds\\
&\quad\quad+ \frac{1}{(2\pi)^2} \int_{\partial \D} \int_{\partial \D} \CP(z,u) \CP(w,v) G(ru,rv) dudv.
\end{align*}
Since $z \mapsto \CP(\psi_s(rz),e^{iW_s})$ is harmonic in $\D$, we have that
\begin{align*}
 &\frac{1}{2\pi} \int_{\partial \D} \CP(z,u) \CP(\psi_s(ru), e^{i W_s}) du\\
=&  \CP(\psi_s(rz), e^{iW_s}) \to \CP(\psi_s(z),e^{i W_s}) \quad\text{as}\quad r \to 1.
\end{align*}
The limit above is locally uniform in $z \in \D$.  Note also that $(u,v) \mapsto G(ru,rv)$ converges to $0$ in $L^1( (\partial \D)^2)$.  The assertion of the lemma therefore follows because as $r \to 1$, the left side above converges to $\cov(H_t(z),H_t(w))$ and the right side converges to the same expression as in the right side of~\eqref{eqn::brownian_cov}.
\end{proof}

\begin{proof}[Proof of Proposition~\ref{prop::approximate_evolution}]
Combining the calculations from~\eqref{eqn::harmonic_change} and~\eqref{eqn::log_change} with Lemma~\ref{lem::mass_change} and Lemma~\ref{lem::gff_harmonic_change} implies that the following is true.  There exists a Brownian motion $B$ such that for all $z \in \D$ we have
\begin{align*}
 \Fh_\delta^\delta(z) = \Fh_0^\delta(z) + \int_0^\delta \wt{D}_s^\delta(z) ds + \int_0^\delta \wt{\sigma}_s^\delta(z) dB_s.
\end{align*}
For each $z \in \D$ and $t \in [0,\delta]$ we let
\begin{align*}
 \wt{\Fh}_t^\delta(z) = \Fh_0^\delta(z) + \int_0^t \wt{D}_s^\delta(z) ds + \int_0^t \wt{\sigma}_s^\delta(z) dB_s.
\end{align*}
Then $\wt{\Fh}_\delta^\delta = \Fh_\delta^\delta$.  Extending the definition of $\wt{\Fh}_t^\delta$ in the same way to all $t \geq 0$ defines a process which satisfies all of the assertions of the proposition.  The statement regarding the convergence of $\wt{\Fh}_t^\delta-\Fh_t^\delta$ to zero as $\delta \to 0$ in probability uniformly on compact subsets of $\D$ and compact intervals of time can be extracted from the continuity of the coefficients and that $\wt{\Fh}_t^\delta = \Fh_t^\delta$ for all times of the form $t = \delta \ell$ with $\ell \in \N_0$.
\end{proof}

\subsection{SDE for the limiting dynamics}
\label{subsec::martingale_problem}

We are now going to show that a subsequential limit $\Fh_t$ of $\Fh_t^\delta$ as $\delta \to 0$ satisfies~\eqref{eqn::qle_evolution}.  By Proposition~\ref{prop::approximate_evolution}, it suffices to work with $\wt{\Fh}_t^\delta$ in place of $\Fh_t^\delta$.  Throughout, when we refer to a $\delta \to 0$ limit we always mean along a subsequence such that the law of $(\Fh_t^\delta)$ hence $(\wt{\Fh}_t^\delta)$ has a weak limit as in Section~\ref{sec::existence}.  We shall also assume (using the Skorohod representation theorem for weak convergence) that we have coupled the laws of such a sequence together onto a common probability space so that $\wt{\Fh}_t^\delta \to \Fh_t$ as $\delta \to 0$ almost surely.

It will be useful in what follows to work in coordinates.  There are many possible choices which would work equally well.  To be concrete, we will make the particular choice of $L^2(\tfrac{1}{2}\D)$.  Note that $\Fh_t^\delta$ is determined by its values on $\tfrac{1}{2} \D$ since it is harmonic.  We will use $(\cdot,\cdot)$ to denote the $L^2$ inner product on $\tfrac{1}{2}\D$ and let $\| \cdot \|$ denote the corresponding norm.  Fix an orthonormal basis $(\psi^j)$ of this space such that $\psi^j$ is smooth for each $j$.  Let $D_t(z,u)$ be as in~\eqref{eqn::dtzu} and let $\nu$ be the driving measure associated with the limiting $\QLE(\gamma^2,\eta)$.  We will write $d\nu_t(u)$ for integration against the conditional measure $\nu(t,u)$ with $t \geq 0$ fixed.  We will now use integration against $d\nu_t(u)$ to define a number of objects.  These are each defined only for almost all $t \geq 0$ but make sense to integrate against.  Let
\begin{align*}
 D_t(z) = \int_{\partial \D} D_t(z,u) d\nu_t(u).
\end{align*}
For each $t \geq 0$ and $j \in \N$, let
\[ D_t^j = (D_t , \psi^j)\]
be the coordinates of $D_t$ with respect to $(\psi^j)$.  We also let
\[ \sigma_t^{ij} = \int_\D \int_\D \int_{\partial \D}   \psi^i(z)  \psi^j(w)  \CP^\star(z,u)  \CP^\star(w,u)  d\nu_t(u) dz dw.\]
We define the operator
\[ \CL_t = \frac{1}{2} \sum_{i,j} \sigma_t^{ij} \partial_{ij} + \sum_i D_t^i \partial_i.\]
Finally, we let $a_t^j = (\Fh_t,\psi^j)$ and $\wt{a}_t^{\delta,j} = (\wt{\Fh}_t^\delta,\psi^j)$ be the coordinates of $\Fh_t$ and $\wt{\Fh}_t^\delta$, respectively, with respect to $(\psi^j)$.

We begin by showing that $\Fh_t$ solves the martingale problem associated with the operator $\CL_t$.

\begin{proposition}
\label{prop::martingale_problem}
Fix $k \in \N$ and suppose that $F \colon \R^k \to \R$ is a smooth function.  Let $F_t = F(a_t^1,\ldots,a_t^k)$.  Then
\begin{equation}
\label{eqn::martingale_problem}
F_t - F_0 - \int_0^t \CL_s F_s ds
\end{equation}
is a continuous square-integrable martingale.
\end{proposition}

For $F$ as in the statement of Proposition~\ref{prop::martingale_problem}, we let
\begin{align}
\label{eqn::fs}
\begin{split}
F_t &= F(a_t^1,\ldots,a_t^k),\\
\partial_i F_t &= (\partial_i F)(a_t^1,\ldots,a_t^k), \quad\text{and}\\
\partial_{ij} F_t &= (\partial_{ij} F)(a_t^1,\ldots,a_t^k).
\end{split}
\end{align}
We also define $\wt{F}_t^\delta$, $\partial_i \wt{F}_t^\delta$, and $\partial_{ij} \wt{F}_t^\delta$ analogously except with $(\wt{a}_t^{1,\delta},\wt{a}_t^{2,\delta},\ldots)$ in place of $(a_t^1,a_t^2,\ldots)$.  Since $\wt{\Fh}_t^\delta \to \Fh_t$ locally uniformly, we have that $\wt{a}_t^{\delta,j} \to a_t^j$ locally uniformly.  Thus,
\begin{align}
\wt{F}_t^\delta \to F_t, \quad \partial_i \wt{F}_t^\delta \to \partial_i F_t, \quad\text{and}\quad \partial_{ij} \wt{F}_t^\delta \to \partial_{ij} F_t \label{eqn::f_convergence}
\end{align}
locally uniformly.

\begin{proof}[Proof of Proposition~\ref{prop::martingale_problem}]
We let $\wt{D}_t^\delta(z)$ (resp.\ $\wt{\sigma}_t^\delta(z)$) be the drift (resp.\ diffusion coefficient) of $\wt{\Fh}_t^\delta(z)$ as given in the statement of Proposition~\ref{prop::approximate_evolution}.  We also let 
\[ \wt{D}_t^{\delta,j} = (\wt{D}_t^\delta,\psi^j) \quad\text{and}\quad  \wt{\sigma}_t^{\delta,j} = (\wt{\sigma}_t^\delta,\psi^j)\]
be the coordinates of $\wt{D}_t^\delta$ and $\wt{\sigma}_t^{\delta,j}$.  By Proposition~\ref{prop::approximate_evolution} there exists a continuous local martingale $\wt{M}_t^\delta$ such that for each $t \in [\delta \ell,\delta(\ell+1))$ we have that
\begin{align*}
     \wt{F}_t^\delta - \wt{F}_{\delta \ell}^\delta
= \frac{1}{2} \sum_{i,j=1}^k \int_{\delta \ell}^{t} \wt{\sigma}_s^{\delta,i} \wt{\sigma}_s^{\delta,j} \partial_{ij} \wt{F}_s^\delta  ds +  \sum_{i=1}^k \int_{\delta \ell}^{t} \wt{D}_s^{\delta,i} \partial_i \wt{F}_s^\delta  ds+ \bigg( \wt{M}_t^\delta - \wt{M}_{\delta \ell}^\delta \bigg).
\end{align*}
Summing up, we consequently have that
\begin{equation}
\label{eqn::approx_mg}
\wt{F}_t^\delta - \wt{F}_0^\delta - \int_0^t \left(\frac{1}{2}\sum_{i,j=1}^k \wt{\sigma}_s^{\delta,i} \wt{\sigma}_s^{\delta,j} \partial_{ij} \wt{F}_s^\delta +  \sum_{i=1}^k \wt{D}_s^{\delta,i} \partial_i \wt{F}_s^\delta  \right) ds = \wt{M}_t^\delta.
\end{equation}

We are now going to argue that the integral on the left side of~\eqref{eqn::approx_mg} converges locally uniformly as $\delta \to 0$.  We will justify in detail the following:
\begin{align}
\label{eqn::dint_limit}
\int_0^t \sum_{i=1}^k \wt{D}_s^{\delta,i} \partial_i \wt{F}_s^\delta ds
  &\to \int_0^t \sum_{i=1}^k D_s^i \partial_i F_s ds  \quad\text{as}\quad \delta \to 0.
\intertext{A similar argument gives}
\label{eqn::sint_limit}
 \int_0^t \sum_{i,j=1}^k \wt{\sigma}_s^{\delta,i} \wt{\sigma}_s^{\delta,j}  \partial_{ij} \wt{F}_s^\delta ds
  &\to  \int_0^t \sum_{i,j=1}^k  \sigma_s^{ij}\partial_{ij} F_s ds \quad\text{as}\quad \delta \to 0.
\end{align}
Let $\wt{D}_s^\delta(z,u)$ be defined analogously to $D_s(z,u)$ from~\eqref{eqn::dtzu} except with $\wt{\Fh}_s^\delta$ in place of $\Fh_s$.  We have that
\begin{align*}
 \wt{D}_s^\delta(z) = \int_{\partial \D} \wt{D}_s^\delta(z,u) d\varsigma_{(s+\delta \ell)^-}^\delta(u) \quad\text{for}\quad s = \delta \ell \quad\text{and}\quad \ell \in \N_0.
\end{align*}
(Note that for $s = \delta \ell$ we have that $\varsigma_{(s+\delta)^-}^\delta$ is given by the Dirac mass located at $V_{\delta \ell}^{\delta, \ell}$.)  From the definitions of $\wt{D}_s^\delta(z,u)$ and $D_s(z,u)$, the convergence of $\wt{\Fh}_t^\delta \to \Fh_t$ in $\CH$, and the convergence of $\varsigma_t^\delta dt \to \nu$ in $\CN$ (Proposition~\ref{prop::qle_solution}) it follows that for each $T \geq 0$ fixed we have that $\wt{D}_s^\delta(z) \to D_s(z)$ with respect to the weak topology induced by continuous functions on $\tfrac{1}{2} \D \times [0,T]$ as $\delta \to 0$.  Therefore $\wt{D}_s^{\delta,j} \to D_s^j$ with respect to the weak topology on continuous functions on $[0,T]$ as $\delta \to 0$.  By~\eqref{eqn::f_convergence}, we know that $\wt{F}_t^\delta - \wt{F}_0^\delta \to F_t - F_0$ locally uniformly as $\delta \to 0$.  Combining all of these observations gives~\eqref{eqn::dint_limit} (and~\eqref{eqn::sint_limit} by analogy).

Combining~\eqref{eqn::dint_limit} with~\eqref{eqn::sint_limit} implies that the left side of~\eqref{eqn::approx_mg} converges locally uniformly as $\delta \to 0$.  It remains to show that the martingale $\wt{M}_t^\delta$ from~\eqref{eqn::approx_mg} converges locally uniformly to a continuous square-integrable martingale $M_t$.  Since the left side of~\eqref{eqn::approx_mg} converges locally uniformly as $\delta \to 0$, it follows that $\wt{M}_t^\delta \to M_t$ locally uniformly as $\delta \to 0$.  Recall that $\wt{\Fh}_t^\delta$ and $\Fh_t$ for each $t \geq 0$ are both distributed as the harmonic extension of the form of the free boundary GFF on $\D$ as described in Proposition~\ref{prop::free_boundary_weighted} (where the role of $\gamma$ in the application of the proposition is played by $2\alpha_\kappa$) from $\partial \D$ to $\D$ by the time stationarity of the evolution (recall Proposition~\ref{prop::approximate_evolution}).  In particular, for every $p \geq 1$ and $r \in (0,1)$ there exists a constant $C_{p,r} < \infty$ depending only on $p$ and $r$ such that
\[ \sup_{z \in r \D} \E\left[ \big(\wt{\Fh}_t^\delta(z)\big)^p \right] \leq C_{p,r}.\] (The same holds with $\Fh_t$ in place of $\wt{\Fh}_t^\delta$ since $\Fh_t \stackrel{d}{=} \wt{\Fh}_t^\delta$ for each $t \geq 0$.)  Indeed, to see that this is the case one first notes that it suffices to bound the second moment since $\wt{\Fh}_t^\delta$ can be written as the sum of a Gaussian and an independent $\log$ function centered at a point in $\partial \D$ and the second moment can be bounded by first representing $\wt{\Fh}_t^\delta$ as the harmonic extension of its boundary values and then bounding the covariances.  This implies that $\wt{M}_t^\delta$ is a square-integrable martingale and the convergence $\wt{M}_t^\delta \to M_t$ takes place in $L^2$.  Therefore $M_t$ is a square-integrable martingale, as desired.
\end{proof}

\begin{proposition}
\label{prop::sde}
There exists a filtration $(\CF_t)$ to which $\Fh_t$ is adapted and a family of $(\CF_t)$-standard Brownian motions $B^j$ with
\begin{equation}
\label{eqn::cross_variation}
 d\langle B^i, B^j \rangle_t = \frac{\sigma_t^{ij}}{\sqrt{ \sigma_t^{ii} \sigma_t^{jj}}} dt.
\end{equation}
such that
\begin{equation}
\label{eqn::coordinate_sde}
 d a_t^j = D_t^j dt + (\sigma_t^{jj})^{1/2} dB_t^j \quad\text{for each}\quad j \in \N.
\end{equation}
\end{proposition}
Note that it follow from the definition that $\sigma_t^{ii} > 0$ for all $t \geq 0$.  In particular, the expression in the right side of~\eqref{eqn::cross_variation} makes sense.
\begin{proof}[Proof of Proposition~\ref{prop::sde}]
We take $F(x_1,x_2,\ldots) = x_j$ and let $F_t$ be as described in~\eqref{eqn::fs} with this choice of $F$.  Then $F_t = a_t^j$ and $\CL_t F_t = D_t^j$.  By Proposition~\ref{prop::martingale_problem}, we know that
\begin{equation}
\label{eqn::markovian_mg1}
 M_t = a_t^j - \int_0^t D_s^j ds
\end{equation}
is a continuous square-integrable martingale.  Taking $F(x_1,x_2,\ldots) = x_j^2$ so that $F_t = (a_t^j)^2$, by Proposition~\ref{prop::martingale_problem}, we also know that
\begin{equation}
\label{eqn::markovian_mg2} \wt{M}_t = (a_t^j)^2 - \int_0^t \bigg(2 D_s^j a_s^j + \sigma_s^{jj} \bigg) ds
\end{equation}
is a continuous square-integrable martingale.  Combining~\eqref{eqn::markovian_mg1} and~\eqref{eqn::markovian_mg2} implies that
\[ M_t^2 - \int_0^t \sigma_s^{jj} ds\]
is a continuous square-integrable martingale.  By the martingale characterization of the quadratic variation \cite{ks1991stoccalc,ry1999martingales}, it therefore follows that $d \langle M \rangle_t = \sigma_t^{jj} dt$.  Note that $\sigma_t^{jj} > 0$.  We then let
\[ B_t^j = \int_0^t (\sigma_s^{jj})^{-1/2} dM_s.\]
Then $B^j$ is a continuous local martingale and
\begin{align*}
      d\langle B^j \rangle_t
 = \int_0^t (\sigma_s^{jj})^{-1} d\langle M \rangle_s
 = \int_0^t ds = t.
\end{align*}
Therefore it follows from the Levy characterization of Brownian motion that $B^j$ is a standard Brownian motion.  Moreover,
\[  a_t^j = \int_0^t D_s^j ds + M_t = \int_0^t D_s^j ds + \int_0^t (\sigma_s^{jj})^{1/2} dB_s^j.\]
This proves that $a_t^j$ solves~\eqref{eqn::coordinate_sde}.  That the Brownian motions $B^j$ satisfy~\eqref{eqn::coordinate_sde} follows by applying a similar argument with $F(x_1,x_2,\ldots) = x_i x_j$ for $i \neq j$ distinct.
\end{proof}

%% file: tex/sample_path_properties.tex
\subsection{Continuity}
\label{subsec::continuity}

The purpose of this section is to establish the following result which, as explained just after the statement, implies Theorem~\ref{thm::qle_continuity}.

\begin{theorem}
\label{thm::continuity}
Fix $\gamma > 0$ and let $Q = 2/\gamma+ \gamma/2$.  Suppose that $k,\ell \in \N_0$, $\gamma_i \in \R$ for $1 \leq i \leq k+\ell$, and let
\begin{equation}
\label{eqn::beta_definition}
\beta_\star = \max(2\sqrt{2},\gamma_1,\ldots,\gamma_{k+\ell}) \quad\text{and}\quad \beta^\star = \max(2, -\gamma_1,\ldots,-\gamma_k).
\end{equation}
Assume that $\beta^\star < Q$ and let
\begin{equation}
\label{eqn::holder_exponent}
 \ol{\Delta} = \frac{Q-\beta^\star}{Q+\beta_\star} \in (0,1).
\end{equation}
Fix $\Delta \in (0,\ol{\Delta})$.  Suppose that $h_1$ and $h_2$ are random modulo additive constant distributions each of which can be expressed as the sum of a free boundary GFF and an independent function of the form $\sum_{i=1}^{k+\ell} \gamma_i \log|\cdot-x_i|$ where $x_1,\ldots,x_k \in \partial \D$ and $x_{k+1},\ldots,x_{k+\ell} \in \D$ are distinct.  Let $A$ be a local set for $h_1$ which almost surely does not contain $0$ and let $D$ be the connected component of $\D \setminus A$ which contains~$0$.  Let $\varphi \colon \D \to D$ be the unique conformal transformation with $\varphi(0) = 0$ and $\varphi'(0) > 0$ and suppose that
\begin{equation}
\label{eqn::gff_law_equality}
h_1 \circ \varphi + Q \log|\varphi'| \stackrel{d}{=} h_2
\end{equation}
as modulo additive constant distributions.  Then $\varphi$ is almost surely H\"older continuous with exponent~$\Delta$.
\end{theorem}

We note that the value of $\beta^\star$ corresponding to a field $h$ associated with one of the $\QLE(\gamma^2,\eta)$ processes with $\gamma \in (0,2)$ constructed in Theorem~\ref{thm::existence} is given by $2 < Q$.  Therefore Theorem~\ref{thm::continuity} implies Theorem~\ref{thm::qle_continuity}.

Theorem~\ref{thm::continuity} gives a bound for the H\"older \emph{exponent} of $\varphi$ but does not give a bound for the H\"older \emph{norm} of $\varphi$.  As we will see in the proof of Theorem~\ref{thm::continuity}, bounding the H\"older norm is related to the additive constant which is implicit in~\eqref{eqn::gff_law_equality} as well as the proximity of the $x_i$ (locations of the $\log$ singularities) from each other.

By \cite[Theorem~1.2]{sheffield2010weld} and Theorem~\ref{thm::radial_coupling_existence}, Theorem~\ref{thm::continuity} can also be applied to the chordal and radial $\SLE_\kappa$ processes for $\kappa \neq 4$.  This gives an alternative proof of \cite[Theorem~5.2]{rs2005sle} which states that the complementary components of $\SLE_\kappa$ curves for $\kappa \neq 4$ are H\"older domains.  For completeness, we restate this result as the following corollary.

\begin{corollary}
\label{cor::sle_holder}
Suppose that $\eta$ is a chordal $\SLE_\kappa$ process for $\kappa \neq 4$ in $\D$ from $-i$ to $i$ and fix $T > 0$.  Let $D$ be a non-empty connected component of $\D \setminus \eta([0,T])$ and let $\varphi \colon \D \to D$ be a conformal transformation.  Fix $\Delta \in (0,\ol{\Delta})$ where $\ol{\Delta}$ is as in \eqref{eqn::holder_exponent} with $\gamma = \min(\sqrt{\kappa},\sqrt{16/\kappa})$ and $\gamma_1 = 2/\sqrt{\kappa}$.  Then $\varphi$ is almost surely H\"older continuous with exponent~$\Delta$.  The same likewise holds if $\eta$ is instead a radial $\SLE_\kappa$ process in $\D$ targeted at $0$.
\end{corollary}

We remark that the H\"older exponent obtained in Corollary~\ref{cor::sle_holder} is not the optimal value for $\SLE_\kappa$ \cite{rs2005sle,kang2008boundary,bs2009harmonicmeasure}.

The idea of the proof of Theorem~\ref{thm::continuity} is to exploit the identity \eqref{eqn::gff_law_equality} to bound the growth rate of $|\varphi'(z)|$ as $z \to \partial \D$.  To show that $\varphi$ is almost surely $\Delta$-H\"older in $\D$ for a given value of $\Delta \in (0,1)$ it suffices to show that there exists a constant $C > 0$ such that $|\varphi'(z)| \leq C(1-|z|)^{\Delta-1}$ for all $z \in \D$ (because we can integrate $\varphi'$).  The main step in carrying this out is to compute the maximal growth rate of the circle average process associated with a free boundary GFF on $\D$ as the radius $\epsilon > 0$ about which we average tends to zero.  This is accomplished in Section~\ref{subsec::gff_extremes}.  In Section~\ref{subsec::continuity_proofs} we combine this with \eqref{eqn::gff_law_equality} to constrain $|\varphi'|$ to complete the proof.

\subsubsection{Extremes of the free boundary GFF}
\label{subsec::gff_extremes}

Suppose that $h$ is a free boundary GFF on $\D$.  We fix the additive constant for $h$ by taking its mean on $\D$ to be equal to $0$.  For each $z \in \D$ and $\epsilon > 0$ such that $B(z,\epsilon) \subseteq \D$, let $h_\epsilon(z)$ denote the average of $h$ on $\partial B(z,\epsilon)$.  The purpose of this section is to prove the following (as well as Proposition~\ref{prop::free_boundary_max_log}, its generalization to the case in which we add $\log$ singularities to the free boundary GFF).

\begin{proposition}
\label{prop::free_boundary_max_growth}
Suppose that $h$ is a free boundary GFF on $\D$.  Then,
\begin{equation}
\label{eqn::free_boundary_max_growth_boundary}
 \p\left[ \limsup_{\epsilon \to 0} \sup_{z \in \partial B(0,1-\epsilon)} \frac{|h_\epsilon(z)|}{\log \epsilon^{-1}} \leq 2 \right] = 1.
\end{equation}
Moreover,
\begin{equation}
\label{eqn::free_boundary_max_growth_interior}
 \p\left[ \limsup_{\epsilon \to 0} \sup_{z \in B(0,1-\epsilon)} \frac{|h_\epsilon(z)|}{\log \epsilon^{-1}} \leq 2\sqrt{2} \right] = 1.
\end{equation}
\end{proposition}

This is the analog of the upper bound established in \cite{hmp2010thickpoints} with the free boundary GFF in place of the GFF with Dirichlet boundary conditions.  Note that the constant $2$ in \eqref{eqn::free_boundary_max_growth_boundary} is the same as the constant which appears in \cite{hmp2010thickpoints} for the GFF with Dirichlet boundary conditions after one adjusts for the difference in the normalization used for the GFF in this article and in \cite{hmp2010thickpoints}.  We will extract \eqref{eqn::free_boundary_max_growth_interior} from the corresponding result for the GFF with Dirichlet boundary conditions and the odd-even decomposition for the whole-plane GFF \cite[Section~3.2]{sheffield2010weld}; this is the reason that the constant $2\sqrt{2}$ rather than $2$ appears.

We begin by recording the following modulus of continuity result for the circle average process established in \cite[Proposition~2.1]{hmp2010thickpoints} except with the free boundary GFF in place of the GFF with Dirichlet boundary conditions.

\begin{proposition}
\label{prop::continuity}
Suppose that $h$ is a free boundary GFF on $\D$ and let $h_\epsilon(z)$ denote the corresponding circle average process.  Then $h_\epsilon(z)$ has a modification $\wt{h}_\epsilon(z)$ such that for every $\lambda \in (0,1/2)$ and $\zeta > 0$ there exists $M = M(\lambda,\zeta)$ (random) such that
\[ |\wt{h}_\epsilon(z) - \wt{h}_\delta(w)| \leq M \frac{|(z,\epsilon) - (w,\delta)|^\lambda}{\epsilon^{\lambda+\zeta}}\]
for $\epsilon,\delta \in (0,1]$ with $\tfrac{1}{2} \leq \epsilon/\delta \leq 2$ and $z,w \in B(0,1- \epsilon \vee \delta)$.
\end{proposition}
\begin{proof}
This was proved in \cite[Proposition~2.1]{hmp2010thickpoints} for the zero-boundary GFF.  It follows from the Markov property that the same holds for the whole-plane GFF (see \cite[Proposition~2.8]{ms2013imag4}).  Combining this with the odd-even decomposition for the whole-plane GFF into a sum of a zero-boundary GFF and a free-boundary GFF (see, for example, \cite[Section~3.2]{sheffield2010weld}) it follows that the same modulus of continuity estimate also holds for the free boundary GFF.
\end{proof}

Throughout, we shall always assume that we are working with such a modification and indicate it with $h_\epsilon$.  Proposition~\ref{prop::continuity} is proved by generalizing \cite[Proposition~3.1]{ds2011kpz} using a version of the Kolmogorov-Centsov theorem which bounds the growth of the H\"older norm for processes parameterized by $[0,\infty)$ (as opposed to a compact time interval, which is the setting of the usual Kolmogorov-Centsov theorem \cite{ks1991stoccalc,ry1999martingales}).  We will not provide an independent proof here.  We next record the following elementary Gaussian tail estimate.

\begin{lemma}
\label{lem::gaussian_tail}
Suppose that $Z \sim N(0,1)$.  Then
\[ \p[ Z \geq \lambda] \sim \frac{1}{\sqrt{2\pi} \lambda} \exp\left(-\frac{\lambda^2}{2} \right) \quad\text{as}\quad \lambda \to \infty.\]
\end{lemma}

We next record the following variance estimate for the circle average process associated with the free boundary GFF.

\begin{lemma}
\label{lem::free_boundary_variance}
Suppose that $h$ is a free boundary GFF on $\D$ with corresponding circle average process $h_\epsilon$.  For each $z \in \D$ and $\epsilon > 0$ such that $B(z,\epsilon) \subseteq \D$, we have that
\[ \var(h_\epsilon(z)) = \log \epsilon^{-1} - \log \dist(z,\partial \D) + O(1)\]
where the constant implicit in the $O(1)$ term is uniform in $z$ and $\epsilon$.
\end{lemma}
\begin{proof}
See, for example, the discussion in \cite[Section~6]{ds2011kpz}.
\end{proof}

We can now give the proof of Proposition~\ref{prop::free_boundary_max_growth}.

\begin{proof}[Proof of Proposition~\ref{prop::free_boundary_max_growth}]
We will start with \eqref{eqn::free_boundary_max_growth_boundary}.  Suppose that $z \in \D$ and $\epsilon > 0$ are such that $B(z,\epsilon) \subseteq \D$.  Fix $\delta > 0$.  By Lemma~\ref{lem::gaussian_tail} and Lemma~\ref{lem::free_boundary_variance} there exists a constant~$C$ (independent of $z$, $\epsilon$, and $\delta$) such that
\begin{align*}
        \p\left[ h_\epsilon(z) \geq (2+\delta) \log \epsilon^{-1} \right]
\leq \exp\left(- \frac{(2+\delta)^2(\log \epsilon^{-1})^2}{2(\log \epsilon^{-1} - \log \dist(z, \partial \D)) + C} \right).
\end{align*}
Fix $\xi > 0$, let $K = \xi^{-1}$, and, for each $n \in \N$, let $r_n = n^{-K}$. Note that $r_n^{1+\xi} = n^{-(1+K)}$.  For each $n \in \N$, let $\CD_{n,\xi}$ consist of those points $z \in (r_n^{1+\xi} \Z)^2 \cap \D$ with $r_n \leq \dist(z,\partial \D) < r_{n-1}$.  Note that $r_{n-1} - r_n$ is proportional to $r_n^{1+\xi}$ so that the number of elements in $\CD_{n,\xi}$ is $O(r_n^{-(1+\xi)})$.  It thus follows from a union bound that there exists a constant $C > 0$ such that
\begin{align*}
   \p\left[\bigcup_{z \in \CD_{n,\xi}} \left\{ h_{r_n}(z) \geq (2+\delta) \log r_n^{-1} \right\} \right]
\leq C \exp\left( \log r_n^{-(1+\xi)} - \frac{(2+\delta)^2 (\log r_n^{-1})^2}{4 \log r_n^{-1} + C}\right).
\end{align*}
It is easy to see from the above that there exists $\xi_0 > 0$ depending only on $\delta > 0$ such that for all $\xi \in (0,\xi_0)$ we have that
\[ \sum_{n=1}^\infty \p\left[ \bigcup_{z \in \CD_{n,\xi}} \left\{ h_{r_n}(z) \geq (2+\delta) \log r_n^{-1} \right\} \right] < \infty.\]
Therefore it follows from the Borel-Cantelli lemma that there almost surely exists $n_0 \in \N$ (random) such that
\begin{equation}
\label{eqn::grid_average_small}
h_{r_n}(z) \leq (2+\delta) \log r_n^{-1} \quad\text{for all}\quad z \in \CD_{n,\xi} \quad\text{and}\quad n \geq n_0.
\end{equation}
Now suppose that $z \in \D$ and that $\epsilon = \dist(z,\partial \D)$.  Fix $n \in \N$ such that $r_n \leq \epsilon < r_{n-1}$ and $z_n \in \CD_{n,\xi}$ such that $|z-z_n| \leq 4 r_n^{1+\xi}$.  Fix $\lambda \in (0,1/2)$ and let $\zeta = \lambda \xi$.  Using again that $r_{n-1} - r_n$ is proportional to $r_n^{1+\xi}$, Proposition~\ref{prop::continuity} implies that there exists a constant $M = M(\lambda,\zeta)$ (random but independent of $z$ and $\epsilon$) such that
\begin{align}
\label{eqn::grid_approximate_average}
       |h_\epsilon(z) - h_{r_n}(z_n)|
\leq M\left(\frac{r_n^{(1+\xi)\lambda}}{r_n^{\lambda+\zeta}} \right)
   = M.
\end{align}
Combining \eqref{eqn::grid_average_small} and \eqref{eqn::grid_approximate_average} implies \eqref{eqn::free_boundary_max_growth_boundary}, as desired.

We now turn to \eqref{eqn::free_boundary_max_growth_interior}.  We first note that the argument of \cite[Lemma~3.1]{hmp2010thickpoints} (which is very similar to the argument given just above) implies that if $h_\epsilon^0$ is the circle-average process corresponding to a zero-boundary GFF on $\D$ then
\[ \p\left[ \limsup_{\epsilon \to 0} \sup_{z \in B(0,1-\epsilon)} \frac{|h_\epsilon^0(z)|}{\log \epsilon^{-1}} \leq 2 \right] = 1.\]
It follows from the Markov property and scale-invariance that the same holds for the circle average process for the whole-plane GFF restricted to any bounded domain in $\C$ (see \cite[Proposition~2.8]{ms2013imag4}).  Combining this with the odd-even decomposition for the whole-plane GFF into a sum of a zero-boundary GFF and a free-boundary GFF (see, for example, \cite[Section~3.2]{sheffield2010weld}) gives~\eqref{eqn::free_boundary_max_growth_interior}.
\end{proof}

We are now going to generalize Proposition~\ref{prop::free_boundary_max_growth} to the setting in which we add $\log$ singularities to the field.

\begin{proposition}
\label{prop::free_boundary_max_log}
Fix $x_1,\ldots,x_k \in \partial \D$ and $x_{k+1},\ldots,x_{k+\ell} \in \D$ distinct and let $\gamma_1,\ldots,\gamma_{k+\ell} \in \R$.  Let
\[ \beta_\star = \max(2\sqrt{2},\gamma_1,\ldots,\gamma_{k+\ell}) \quad\text{and}\quad \beta^\star = \max(2, -\gamma_1,\ldots,-\gamma_k)\]
be as in \eqref{eqn::beta_definition}.  Suppose that $h$ is the sum of a free boundary GFF on $\D$ and $\sum_{i=1}^k \gamma_i \log|\cdot-x_i|$.  We assume that the additive constant for $h$ has been fixed by taking its average on $\D$ to be equal to $0$.  For each $z \in \D$ let $h_\epsilon(z)$ be the average of $h$ on $\partial B(z,\epsilon)$.  Then,
\begin{equation}
\label{eqn::free_boundary_log_max_growth_boundary}
 \p\left[ \limsup_{\epsilon \to 0} \sup_{z \in \partial B(0,1-\epsilon)} \frac{h_\epsilon(z)}{\log \epsilon^{-1}} \leq \beta^\star \right] = 1.
\end{equation}
Moreover,
\begin{equation}
\label{eqn::free_boundary_log_max_growth_interior}
 \p\left[ \limsup_{\epsilon \to 0} \inf_{z \in B(0,1-\epsilon)} \frac{h_\epsilon(z)}{\log \epsilon^{-1}} \geq -\beta_\star \right] = 1.
\end{equation}
\end{proposition}
\begin{proof}
By the absolute continuity properties of the free boundary GFF (see, for example, \cite[Proposition~3.2]{ms2010imag1}), it suffices to prove the result when $k=1$.  We will now explain the proof of \eqref{eqn::free_boundary_log_max_growth_boundary} when $\gamma_1 \in (-2,0]$ and $x_1 \in \partial \D$ so that $\beta^\star = 2$.  The case when $\gamma_1 \geq 0$ is obvious since adding a non-negative multiple of $\log|\cdot-x_1|$ can only decrease the asymptotic growth of the circle average as $\epsilon \to 0$.  We know that the law of the free boundary GFF normalized to have zero mean weighted by its $\gamma_1$-LQG boundary measure is (see the discussion just after \cite[Equation~(83)]{ds2011kpz} in the proof of \cite[Theorem~6.1]{ds2011kpz}):
\begin{enumerate}
\item absolutely continuous with respect to the law of the (unweighted) free boundary GFF and
\item can be written as the sum of $-\gamma_1 \log|\cdot-x|$ for $x \in \partial \D$ chosen uniformly from Lebesgue measure, a bounded, smooth function, and an independent free boundary GFF.
\end{enumerate}
Therefore the result in this case follows from Proposition~\ref{prop::free_boundary_max_growth}.  The case that $\gamma_1 \leq -2$ follows because we can think of first adding $-(2-\zeta)\log|\cdot-x_1|$ to $h$ for $\zeta > 0$ very small, applying the $(-2,0]$ case, and then adding $(\gamma_1+2-\zeta)\log|\cdot-x_1|$.  The proof of \eqref{eqn::free_boundary_log_max_growth_interior} when $x_1 \in \D$ is similar.
\end{proof}

\subsubsection{Proof of Theorem~\ref{thm::continuity}}
\label{subsec::continuity_proofs}

We can fix the additive constants for $h_1$ and $h_2$ by taking their means on $\D$ to be both equal to zero.  Then \eqref{eqn::gff_law_equality} implies that there exists an almost surely finite random variable $C$ such that
\begin{equation}
\label{eqn::gff_law_equality2}
 h_1 \circ \varphi + Q\log|\varphi'| \stackrel{d}{=} h_2 + C.
\end{equation}
(The reason we have written \eqref{eqn::gff_law_equality2} this way rather than as an equality modulo additive constant is to emphasize one of the sources of the H\"older norm.)  This implies that we may couple $h_1$, $h_2$, $\varphi$, and $C$ onto a common probability space so that the equality in distribution from \eqref{eqn::gff_law_equality2} is an almost sure equality.

Let $\Fh_{A,1}$ (resp.\ $\Fh_2$) denote the harmonic extension of the boundary values of $h_1$ (resp.\ $h_2$) from $\partial D$ to $D$ (resp.\ $\partial \D$ to $\D$).  We can think of defining these objects in two steps.
\begin{enumerate}
\item Apply $\pHarm$ to the GFF component of $h_1$ and $h_2$.  These projections are almost surely defined because $\ol{\D} \setminus D$ and $\partial \D$ are respectively local for the GFF components of $h_1$ and $h_2$ (recall Section~\ref{subsec::local_sets}).
\item Add the harmonic extension of the sum of $\log$ functions component of $h_1$ (resp.\ $h_2$) from $\partial D$ to $D$ (resp.\ $\partial \D$ to $\D$).
\end{enumerate}
We know that the harmonic extension of the values from $\partial \D$ to $\D$ of the left side of \eqref{eqn::gff_law_equality2} is almost surely equal to that of the right side (recall that we have coupled so that \eqref{eqn::gff_law_equality2} is an almost sure equality).  These are respectively given by $\Fh_{A,1} \circ \varphi + Q \log|\varphi'|$ and $\Fh_2 + C$, hence we almost surely have that
\begin{equation}
\label{eqn::gff_ce_law_equality}
\Fh_{A,1} \circ \varphi + Q \log|\varphi'| = \Fh_2 + C.
\end{equation}
Rearranging \eqref{eqn::gff_ce_law_equality}, we see that
\begin{equation}
\label{eqn::gff_ce_law_equality2}
Q \log|\varphi'| = \Fh_2 - \Fh_{A,1} \circ \varphi + C.
\end{equation}

For $j=1,2$ we let $h_{j,\epsilon}$ be the circle average process associated with $h_j$.  For each $\delta > 0$ we also let $B_\delta = B(0,1-\delta)$.  Proposition~\ref{prop::free_boundary_max_log} implies that
\begin{align}
\p\left[ \limsup_{\epsilon \to 0} \sup_{z \in B_\epsilon} \frac{\Fh_2(z)}{\log \epsilon^{-1}} \leq \beta^\star \right] &= 1.  \label{eqn::fb_gff_mean_bound_logs}
\end{align}
Indeed, this follows because we can write $h_2$ as the sum of $\Fh_2$, a zero-boundary GFF which is independent of $\Fh_2$, and a sum of $\log$ singularities located at points with a positive distance from $\partial \D$.  In particular, the contribution of the latter to the boundary behavior of $\Fh_2$ is bounded.  If the maximal value of $\Fh_2(z)$ for $z \in B_\epsilon$ as $\epsilon \to 0$ exceeded $\beta^\star(1+o(1)) \log \epsilon^{-1}$ then there would be at least a $1/2$ chance that the maximal value of $h_{2,\epsilon}(z)$ for $z \in \partial B_\epsilon$ as $\epsilon \to 0$ would exceed $\beta^\star(1+o(1))\log \epsilon^{-1}$, which would contradict~\eqref{eqn::free_boundary_log_max_growth_boundary}.

Fix $z \in \D$, let $\epsilon = \dist(z,\partial \D)$, and let $d = \dist(\varphi(z),\partial D)$.  By distortion estimates for conformal maps, we know that
\[ \frac{1}{4} |\varphi'(z)| \epsilon \leq d  \leq 4 |\varphi'(z)| \epsilon.\]
Since $\Fh_{A,1}$ is harmonic in $D$, we have that $\Fh_{A,1}(\varphi(z))$ is equal to the average of its values on $\partial B(\varphi(z),d)$.  It follows from Proposition~\ref{prop::free_boundary_max_log} using the same argument that we used to justify \eqref{eqn::fb_gff_mean_bound_logs} that
\begin{align}
\p\left[ \liminf_{\epsilon \to 0} \inf_{z \in B_\epsilon} \frac{\Fh_{A,1}(\varphi(z))}{\log (|\varphi'(z)|\epsilon)^{-1}} \geq -\beta_\star \right] = 1.  
\label{eqn::fb_gff_mean_bound_local}
\end{align}
Combining \eqref{eqn::fb_gff_mean_bound_logs} and \eqref{eqn::fb_gff_mean_bound_local} with \eqref{eqn::gff_ce_law_equality2} we have uniformly in $z \in B_\epsilon$ that
\begin{align}
\begin{split}
 Q  \log |\varphi'(z)| &\leq  (\beta_\star+\beta^\star) (1+o(1))\log \epsilon^{-1} - \beta_\star (1+o(1)) \log |\varphi'(z)| + C \label{eqn::deriv_bound}
\end{split}
\end{align}
with probability tending to $1$ and the $o(1)$ terms tending to $0$ as $\epsilon \to 0$.  Let $\ol{\Delta} = (Q-\beta^\star)/(Q+\beta_\star)$ be as in \eqref{eqn::holder_exponent} so that $1-\ol{\Delta} = (\beta_\star+\beta^\star)/(Q+\beta_\star)$.  Rearranging the terms in \eqref{eqn::deriv_bound} to solve for $\log|\varphi'|$ implies that
\begin{equation}
\label{eqn::deriv_rearrange_bound}
 \p\left[ \limsup_{\epsilon \to 0} \sup_{z \in B_\epsilon} \frac{\log |\varphi'(z)|}{\log \epsilon^{-1}} \leq 1-\ol{\Delta} \right] = 1.
\end{equation}
Since $\varphi$ is a conformal transformation, distortion estimates imply that $|\varphi'|$ is bounded when evaluated at points which have a positive distance from $\partial \D$.  Therefore we just need to control $|\varphi'(z)|$ as $z \to \partial \D$.  Fix $\zeta > 0$ such that $\ol{\Delta}-\zeta > 0$.  Then \eqref{eqn::deriv_rearrange_bound} implies that there exists $C_1,C_2 < \infty$ (random) such that
\begin{equation}
\label{eqn::deriv_bound2}
|\varphi'(z)| \leq C_1(1-|z|)^{\ol{\Delta}-1-\zeta} + C_2 \quad\text{for all}\quad z \in \D.
\end{equation}
The H\"older property follows from this by integrating $\varphi'$.

\qed

\subsection{Phases}
\label{subsec::phases}

The purpose of this section is to discuss the problem of establishing the different phases for the sample path of $\QLE(\gamma^2,\eta)$ as described in Figure~\ref{fig::etavsgamma} (though we will not provide a rigorous proof here).  We use the notation of Section~\ref{subsec::mainresults} in that for any $(\gamma^2, \eta)$ pair that lies on one of the two upper lines of Figure~\ref{fig::etavsgamma}, we write $\QLE(\gamma^2,\eta)$ to denote one of the subsequential limits whose existence is established in Theorem~\ref{thm::existence} (and explained in more detail in Section~\ref{sec::sample_path}).

Recall that the solutions on the middle curve were constructed as $\delta \to 0$ subsequential limits of ``reshuffled'' radial $\SLE_\kappa$ with $\kappa = 16/\gamma^2 \in (4,\infty)$.  The solutions on the upper curve were constructed as $\delta \to 0$ subsequential limits of ``reshuffled'' radial $\SLE_\kappa$ where $\kappa = \gamma^2 \in (1,4)$.  We remind the reader of the phases of radial $\SLE_\kappa$ \cite{rs2005sle}:
\begin{enumerate}
\item When $\kappa \in [0,4]$ a radial $\SLE_\kappa$ is almost surely a simple curve.
\item When $\kappa \in (4, 8)$ a radial $\SLE_{\kappa}$ is almost surely a continuous curve that hits itself but fills zero Lebesgue measure.
\item When $\kappa \geq 8$, a radial $\SLE_{\kappa}$ curve is almost surely space-filling.
\end{enumerate}
For a given GFF $h$, let $\mu_h^\gamma$ denote the $\gamma$-LQG measure.  Let $K_t$ denote the hull at time $t$ of a radial $\SLE_\kappa$.  From these results, one may easily deduce the following for the coupling of $\SLE_\kappa$ and the GFF used in the $\QLE(\gamma^2,\eta)$ construction:
\begin{enumerate}
\item When $\kappa \in [0,4]$ the Lebesgue measure of $K_t$ and the quantum measure $\mu_h^\gamma(K_t)$ are both almost surely zero for all $t$.
\item When $\kappa \in (4,8)$ the Lebesgue measure of $K_t$ and the quantum measure $\mu_h^\gamma(K_t)$ are both increasing functions of $t$ that have discontinuities (because bubbles can be swallowed instantaneously).
\item When $\kappa \geq 8$, the Lebesgue measure of $K_t$ and the quantum measure $\mu_h^\gamma(K_t)$ are almost surely continuously increasing functions.
\end{enumerate}

At first glance, it may seem obvious that the three statements just above also apply in the ``reshuffled'' $\delta\to 0$ subsequential limit: that is, they still apply when $(K_t)$ is the $\QLE(\gamma^2,\eta)$ growth process associated to the given $\kappa >  1$ value as constructed in Theorem~\ref{thm::existence}.  In the approximating process $(\varsigma_t^\delta,g_t^\delta,\Fh_t^\delta)$ described in Section~\ref{sec::existence}, the maps $g_t^\delta \colon \D \setminus K_t^\delta \to \D$ exactly describe an $\SLE_\kappa$ evolution except that the seed location is re-randomized according to some rule after each $\delta$ units of capacity time.  In particular, the process $\mu_h^\gamma(K^\delta_t)$ is zero, discontinuously increasing, or continuously increasing precisely when this is true for the corresponding $\SLE_\kappa$.

Going further, one may recall the ``reshuffling'' discussion of Section~\ref{sec::discrete}.  Ordinary $\SLE_\kappa$, stopped at $\delta$ increments of time, induces a Markov chain on quantum surfaces, and ``reshuffling'' this Markov chain yields the $\delta$-approximation to QLE.  Note however, that if $k>0$ is fixed, then by Proposition~\ref{prop::reshuffling} the reshuffling procedure does not change the law of the configuration (the field on $\D$ plus the seed location) that one has after $k$ steps, nor does it change the law of the transition step that takes place between $k$ and $k+1$.  The probability that the $\delta$-approximation absorbs more than $M$ units of Lebesgue (or quantum) area during the $k$th step is precisely the same as the probability that ordinary SLE absorbs more than $M$ units of Lebesgue (or quantum) area between time $k\delta$ and $(k+1)\delta$.  In a certain sense, this suggests that the ``rate'' at which quantum mass is being swallowed should be the same for $\SLE$ as it is for the $\QLE$ approximation --- and one might expect the same to hold in the $\delta \to 0$ limit as well.

However, it is important to recall that the {\em joint} law of the amount of mass absorbed at step $k$ and at another step (say $j$) may be very different for the $\delta$-increment $\QLE$ and $\SLE$ processes.  If $\kappa > 4$, one could worry that even though $\mu_h^\gamma(K_t^\delta)$ is strictly increasing for each $\delta$, it might be that in the $\delta \to 0$ limit, the quantum area increase tends to ``concentrate'' on increasingly rare QLE instances, so that the limiting process is almost surely zero.  When $\kappa \in (1,4)$, there is another source of worry: a $\delta \to 0$ limit of $K^\delta_T$ processes, each of which absorbs zero quantum area, could in principle swallow a positive amount of quantum area, for example if a certain region became closer and closer to getting ``pinched off'' by $K^\delta_T$ as $\delta$ tended to zero.  Thus, in the $\kappa \in (1,4)$ setting, one might worry that $\mu_h^\gamma(K_t)$ could be increasing despite the fact that $\mu_h^\gamma(K_t^\delta)$ is almost surely zero.

%% file: tex/open_questions.tex
In this section, we present a number of open questions that naturally arise from this work.  They are divided into three categories: existence and uniqueness, sample path properties, and connections to discrete models.  Of course, this collection of problems is far from exhaustive.  Most of the basic questions one would think to ask about QLE remain open.

\subsection*{Existence and uniqueness}

\begin{question}
\label{q::existence}
In Theorem~\ref{thm::existence}, we proved the existence of a solution to the $\QLE(\gamma^2,\eta)$ dynamics when $(\gamma^2,\eta)$ is on one of the two upper curves from Figure~\ref{fig::etavsgamma} by realizing these processes as subsequential $\delta \to 0$ limits of certain approximations.  Was it necessary to pass to subsequences, or does the limit exist non-subsequentially?  Assuming the limit exists and is unique, is it the only solution to the $\QLE$ dynamics for the given $(\gamma^2,\eta)$ pair?  Suppose that in the approximations, instead of flowing by a fixed amount $\delta$ of capacity time in between ``tip reshufflings'' we instead flowed by a fixed amount of quantum length of the exploring path (measured in some reasonable sense) using the procedures described in \cite{sheffield2010weld}.  Would we then still obtain the same process (up to a time change) in the $\delta \to 0$ limit?
\end{question}

To solve the first part of the question, one could fix $h$ and then try to explicitly couple $\delta$ and $\delta'$ approximations so that the triples $(\nu^\delta_t, g^\delta_t, \Fh^\delta_t)$ and $(\nu^{\delta'}_t, g^{\delta'}_t, \Fh^{\delta'}_t)$ (which agree at $t = 0$) remain close for $t > 0$.  To solve the second part, one could try to show that the infinite dimensional SDE that describes the dynamics of the $(\Fh_t)$ process described in Section~\ref{sec::dynamics} has a unique solution.

The latter part of the question might be particularly interesting for points on the dotted curve in Figure~\ref{fig::etavsgamma} for which $\eta$ is not {\em too} large; for these points, it may be that the ``tip rererandomization'' always leaves the tip fixed if it is done at fixed capacity time increments, but that it moves the tip in an interesting way if rerandomizations are done at fixed ``quantum length'' increments.

\begin{question}
\label{q::alphavseta}
What can one say about the relationship between $\alpha$ and $\eta$?  For each $\eta$ and $\gamma$ pair, is there at most one choice of $\alpha$ for which there exists a solution to the QLE dynamics that has the $\eta$-DBM scaling property (as discussed at the beginning of Section \ref{sec::continuum})?
\end{question}

\begin{question}
\label{q::off_the_curves}
Is it possible to generalize our approximation procedure to make sense of the $\QLE(\gamma^2,\eta)$ processes for $(\gamma^2,\eta)$ pairs which are \emph{not} on the top two curves from Figure~\ref{fig::etavsgamma}?  Are there such $(\gamma^2,\eta)$ pairs for which there is a stationary solution with $(\Fh_t)$ given by the harmonic extension of the boundary values of a form of the free boundary GFF from $\partial \D$ to $\D$?
\end{question}

There are various ways one might attempt to explore this problem, including the following:
\begin{enumerate}
\item Carefully analyze the infinite dimensional SDE from Section~\ref{sec::dynamics} and look for hints as to what types of stationary solutions might arise.
\item Look for some clever variant of the reshuffling trick --- perhaps something involving $\SLE_\kappa$ or $\SLE_{\kappa}(\rho)$ processes on $\gamma$-LQG surfaces (with $\kappa$ not necessarily equal to $\gamma^2$ or $16/\gamma^2$).
\item Fix $\gamma$ and then try to combine the known solutions (corresponding to the two or three special $\eta$ values) in some way to obtain solutions for other $\eta$ values.
\item Try to grow a QLE approximation from many points simultaneously and understand some limiting law for the corresponding point process.
\item Consider one of the understood stationary QLE processes but take the first $n$ coordinates in an expansion of the GFF $h$ to have fixed variance that is smaller or larger than usual (and note that since this modified model has a law absolutely continuous to the original, the original QLE growth process can still be defined).  Try to take some sort of $n \to \infty$ limit to get a different multiple of the GFF, and control what happens to the QLE growth along the way.
\end{enumerate}

All of these ideas are hindered by the fact that we do not yet have a clear sense of what the stationary dynamics of $\Fh_t$ and $\nu_t$ {\em should} be when the pair $(\gamma^2, \eta)$ fails to lie on one of the special curves.

\begin{question}
\label{q::determined}
What are the sources of randomness in the $\QLE(\gamma^2,\eta)$ processes?  Does the GFF determine the QLE growth process for some $(\gamma^2,\eta)$ values but not for others? \end{question}

We remark that level lines and imaginary geometry flow lines are examples of local sets that have been proved to be almost surely determined by the GFF instance they are coupled with \cite{ss2010contour,dub2009gff,ms2010imag1,ms2013imag4}.  Our guess is that for any $\gamma \in (0,2)$ the sets are determined by the field when $\eta = 0$, since in this case QLE should describe growing balls in a metric determined by the field.  On the other hand, when $\gamma = 0$, there is no randomness from the field at all, so if the $\eta$-DBM scaling limits corresponding to $\gamma=0$ are given by a (non-deterministic) form of QLE, they will have to have a source of randomness other than the field.  On the other hand, it is still conceivable that QLE is always determined by the field when $\gamma \not = 0$.  We leave it to the reader to decide what intuition (if any) can be drawn from Figures~\ref{fig::large_metric2_seeds} and~\ref{fig::dlaseeds}.

\begin{question}
\label{q::variants}
What variants of $\QLE$ can one rigorously construct by only allowing growth from an interval of the boundary, or by growing at different speeds from different intervals?  Are there interesting variants that involve replacing radial $\SLE_\kappa$ with some sort of $\SLE_{\kappa}(\rho)$ process (or perhaps a process growing from a fixed number of tips at once) before applying a reshuffling procedure?
\end{question}

There are many interesting questions along these lines that one can explore without leaving the special $(\gamma^2, \eta)$ curves described in this paper.

\subsection*{Sample path properties}

\begin{question}
\label{q::chunks}
In the QLE approximations, the set of added ``chunks'' has a natural tree structure.  (Each chunk is a child of the chunk on whose boundary it started growing.)  Can one take a limit of this tree structure, to define geodesics, for general $\gamma$ and $\eta$ and are the geodesics almost surely simple curves?  Are they almost surely removable?
\end{question}

\begin{question}
\label{q::cle_qle}
Is there a natural variant of $\QLE$ in which an underlying conformal loop ensemble $\CLE_\kappa$ is fixed, with $\kappa \in (8/3,4)$, and the growth process absorbs entire loops instantaneously?
\end{question}

To address the above question, one might replace the radial $\SLE_\kappa$ process (as used to construct QLE solutions in this paper) with the $\SLE_\kappa(\rho)$ process corresponding to $\rho = \kappa-6$, which is used to construct conformal loop ensembles in \cite{MR2494457, MR2979861}.  In the latter process, there are special times at which the force point and the tip are at the same place (which correspond to times at which the exploration process is {\em not} partway through creating a loop) and there is a notion of local time for this set.  To produce a loop variant of QLE, instead of running the $\SLE_\kappa(\rho)$ process for $\delta$ units of capacity time in between rerandomizations of the tip, one could run it for $\delta$ units of this local time in between tip rerandomizations.  Taking a $\delta \to 0$ limit may yield an interesting LQG-based variant of the loop exploration process described for the $\kappa=4$ case in \cite{MR3057185}.

\begin{question}
\label{q::dimension_of_the_boundary_and_trace}
What is the dimension of the $\QLE$ trace?  What is the Euclidean dimension of the boundary of the domain of $(g_t)$ at a generic time $t$?  For $(\gamma^2,\eta)$ on one of the upper two curves from Figure~\ref{fig::etavsgamma}, is the former dimension larger than the Euclidean dimension of the $\SLE$ curve used to construct the $\QLE$?  Is the latter dimension smaller or larger than the dimension of the outer boundary of the $\SLE$ used in the construction?
\end{question}

Note that it is natural to expect the quantum scaling dimension of the $\QLE$ trace to be the same as that of the SLE curve used to approximate it, which can be computed from the Euclidean dimension of SLE using the KPZ formula (although actually proving this fact, using some precise notion of quantum dimension, would presumably require some work).  On the other hand, one cannot use the KPZ formula to deduce the Euclidean dimension of the $\QLE$ trace from its quantum dimension, because the $\QLE$ trace is not independent of the field $h$ used to defined the underlying quantum surface.

At least heuristically, the number of colored squares in a figure like Figure \ref{fig::qletrace} should scale like a power of $\delta$ as $\delta \to 0$, and this should be the same exponent as for the number of squares hit by an $\SLE_6$ drawn independently on top of the surface.  We expect the $\QLE(8/3,0)$ trace to look like the $\delta \to 0$ scaling limit of the discrete random set shown in Figure \ref{fig::qletrace}.  Although $\SLE_6$ has dimension $7/4$, intuitively, we would guess that the Euclidean dimension of the object in Figure \ref{fig::qletrace} is larger than $7/4$, because the $\QLE$ may have a greater tendency to hit big squares (and avoid small squares) than the independently drawn $\SLE_6$ does.  (If you put an independent $\SLE_6$ on top of a quantum surface, of course it will generally be more likely to hit a given big square than a given small square; on the other hand, the $\QLE$ trace seems, intuitively, to be {\em actively drawn} towards the bigger squares and away from smaller ones.)

\begin{question}
\label{q::holder_regularity}
Suppose that $(g_t)$ is the family of conformal maps which correspond to one of the $\QLE(\gamma^2,\eta)$ processes constructed in Theorem~\ref{thm::existence} with $\gamma \in (0,2)$.  In Theorem~\ref{thm::qle_continuity}, we showed that for each $t \geq 0$ we almost surely have that $g_t^{-1}$ is H\"older continuous with a given exponent $\Delta$.  The proof of Theorem~\ref{thm::qle_continuity} in Section~\ref{sec::sample_path} is based on a ``worst-case'' thick points analysis and is therefore unlikely to yield the \emph{exact} H\"older exponent.  This exponent has been exactly computed in the case of $\SLE$ and the bound we determine for the corresponding $\QLE$ is smaller than this value.  (This is the case since we extract the result for $\QLE$ from a more general result which includes both $\SLE$ and $\QLE$ as special cases.)  Is it possible to determine the corresponding exponent for $\QLE$?  Is $\QLE$ more or less irregular than the $\SLE$s used to construct it?
\end{question}

\begin{question}
\label{q::continuous_at_all_times}
In Theorem~\ref{thm::qle_continuity}, we proved that for each \emph{fixed} $t \geq 0$ that the domain $(g_t)$ of a $\QLE(\gamma^2,\eta)$ process for $(\gamma^2,\eta)$ on one of the top two curves from Figure~\ref{fig::etavsgamma} with $\gamma \in (0,2)$ is almost surely a H\"older domain.  Does this hold almost surely for all times $t \geq 0$ simultaneously?
\end{question}

\begin{question}
\label{q::gamma_2_continuity}
We proved the existence of the $\QLE(4,1/4)$ processes in Section~\ref{sec::existence} (so $\gamma=2$), however Theorem~\ref{thm::qle_continuity} is restricted to the case that $\gamma \in (0,2)$.  Do the maps $(g_t)$ associated with the $\QLE(4,1/4)$ processes extend continuously to $\partial \D$?
\end{question}

\subsection*{Connections to discrete models}

\begin{question}
\label{q::dla}
Can we understand discrete DLA on tree-weighted planar graphs in a deeper way?  In particular, can we say how much information about the ``shape'' comes from the randomness in the underlying planar map, versus the randomness in the growth process?  Equivalently, if we independently draw two DLA processes on top of the same planar map, how ``correlated'' are they with each other?
\end{question}

\begin{question}
\label{q::laplacian}
Is there a nice discrete story that relates $\eta$-LRW to the corresponding $\eta$-DBM models --- a story that somehow generalizes the relationship between LERW and DLA described in this paper?
\end{question}

\begin{question}
\label{q::randomness}
One can define the Eden model on a uniformly random triangulation by assigning an exponential weight to each edge, using the weights to define a metric, and considering increasing balls in that metric.  Can one show that the randomness that arises from the exponential weights on edges does not significantly change the law of the overall metric (at long distances) on the random planar map (recall Figure \ref{fig::large_metric2_seeds})?  How much does it change it typically?  Is there a KPZ scaling result for random planar maps, maybe with some other power in place of $1/3$?
\end{question}

\subsection*{Work in progress}

We briefly mention three projects that the authors are actively working on:

\begin{enumerate}
\item A joint work with Bertrand Duplantier (announced some time ago) extending the quantum zipper results in \cite{sheffield2010weld} to describe weldings of different types of quantum wedges.  One outcome of this work will be a Poissonian description of the bubbles cut off by a quantum gravity zipper in the case that $\kappa \in (4,8)$.
\item A joint work with Ewain Gwynne and Xin Sun on the phase transitions for QLE, expanding on the ideas sketched in Section \ref{subsec::phases}.
\item A work using the results of the paper with Duplantier to give a Poissonian description of bubbles that appear in the QLE models that correspond to $\kappa \in (4,8)$ (somewhat formalizing and extending the ``slot machine'' story).  This work will also study the $\gamma^2 = 8/3$ case specifically in more detail.  The ultimate aim of this project is to rigorously construct the metric space structure of the corresponding LQG surface and to show that the random metric space obtained this way agrees in law with a form of the Brownian map.
\end{enumerate}